\documentclass[oneside]{amsart}

\usepackage{microtype}
\usepackage{amsmath,amssymb, amsfonts,  accents}
\usepackage{amsthm}
 \usepackage{aliascnt}
\usepackage{xstring}
\usepackage{xcolor}

\usepackage[mathscr]{euscript}

\usepackage[T1]{fontenc}
\usepackage[utf8]{inputenc}

\usepackage{url}
\usepackage{latexsym}
\usepackage{enumitem}
\usepackage[colorlinks=true, citecolor={blue!50!black}, linkcolor={blue!50!black}, pdfhighlight=/O, 
ocgcolorlinks=true]{hyperref}
\hypersetup{bookmarksdepth=2}
\usepackage{graphicx}
\usepackage{color}
\usepackage{tikz-cd}

\usepackage{rotating}

\usepackage{verbatim}
\usepackage{geometry}

\usepackage{xy}
\input xy
\xyoption{all}

\newcommand{\xdownarrow}[1]{
  \ensuremath{\begin{turn}{90}{\tiny${#1 }$}\end{turn}
    \left\downarrow\vbox to 0.4cm{}\right.\kern-\nulldelimiterspace
  }
}

\newcommand{\xuparrow}[1]{
  \ensuremath{
    \begin{turn}{90}{\tiny ${#1}$}\end{turn} \left\uparrow\vbox to 0.4cm{}\right. \kern-\nulldelimiterspace
  } 
}

\newcommand{\rt}{\to}
\newcommand{\hooklongrightarrow}{\lhook\joinrel\longrightarrow}
 \DeclareMathOperator{\weirdleq}{\substack{  _{} \vspace{-3.8pt} \\ \eqslantless }}

\newcommand{\mf}[1]{\mathfrak{#1}}
\newcommand{\mc}[1]{\mathcal{#1}}

\newcommand{\mm}[1]{\mathrm{#1}}
\newcommand{\cat}[1]{
\StrLen{#1}[\mystrlen]
\ifnum\mystrlen=1 \mathscr{#1}
\else \mathrm{#1}
\fi}

\newcommand{\KK}{{R}}  
  
\DeclareMathOperator{\dap}{dap}  
\newcommand{\ZZ}{\mathbb{Z}}

\newcommand{\EE}{\mathbb{E}}

\newcommand{\CC}{\mathbb{C}}

\newcommand{\LL}{\mathbb{L}}

\usepackage{upgreek}

\DeclareMathOperator*{\colim}{colim}

\newcommand*\cocolon{%
        \nobreak
        \mskip6mu plus1mu
        \mathpunct{}%
        \nonscript
        \mkern-\thinmuskip
        {:}%
        \mskip2mu
        \relax
}

\DeclareMathOperator{\an}{an} 

\DeclareMathOperator{\ft}{ft} 
\newcommand{\op}[0]{\mm{op}}
\newcommand{\aft}[0]{\mm{aft}}
\newcommand{\laft}[0]{\mm{laft}}
\newcommand{\coft}[0]{\mm{coft}}
\newcommand{\prored}[0]{\mm{prored}}
\newcommand{\red}[0]{\mm{red}}
\newcommand{\conv}{\mathrm{conv}}
\newcommand{\coh}{\mathrm{coh}}
 
\DeclareMathOperator{\aug}{aug} 
\DeclareMathOperator{\der}{der}

\DeclareMathOperator{\Zar}{Zar}

\newcommand{\et}{\mathrm{et}}

\DeclareMathOperator{\id}{id}

\DeclareMathOperator{\Mod}{Mod}

\newcommand{\Vect}{\cat{Vect}}

\newcommand{\perf}{\mathrm{Perf}}

\newcommand{\Coh}{\cat{Coh}}
\newcommand{\QC}{\cat{QC}}
\newcommand{\APerf}{\cat{APerf}}
\newcommand{\dAPerf}{\cat{dAPerf}}

\newcommand{\Flat}{\cat{Flat}}

\newcommand{\cAMod}[2]{\cat{Der}_{#1/#2}}
\newcommand{\caMod}{\cat{Der}_{/\KK}}
\newcommand{\cAAPerf}[2]{\cat{Der}_{#1/#2, \geq 0}^{\mm{aperf}}}

\DeclareMathOperator{\Alg}{Alg} 
\DeclareMathOperator{\CAlg}{CAlg}  
\NewDocumentCommand{\DAlg}{e{^_}}{\cat{CAlg}^{\mathrm{der}\IfValueT{#1}{,#1}}_{\IfValueT{#2}{#2}}}

\NewDocumentCommand{\SCR}{e{^_}}{\cat{CAlg}^{\mathrm{an}\IfValueT{#1}{,#1}}_{\IfValueT{#2}{#2}}}

\NewDocumentCommand{\Art}{e{^_}}{\mathrm{Art}^{\IfValueT{#1}{,#1}}_{\IfValueT{#2}{#2}}}
\NewDocumentCommand{\AArt}{e{^_}}{\mathrm{AArt}^{\IfValueT{#1}{,#1}}_{\IfValueT{#2}{#2}}}

\NewDocumentCommand{\FilSCR}{e{^_}}{\cat{FilCAlg}^{\mathrm{an}\IfValueT{#1}{,#1}}_{\IfValueT{#2}{#2}}}

\NewDocumentCommand{\FilDAlg}{e{^_}}{\cat{FilCAlg}^{\der\IfValueT{#1}{,#1}}_{\IfValueT{#2}{#2}}}

\NewDocumentCommand{\grSCR}{e{^_}}{\cat{GrCAlg}^{\an\IfValueT{#1}{,#1}}_{\IfValueT{#2}{#2}}}
\NewDocumentCommand{\grDAlg}{e{^_}}{\cat{GrCAlg}^{\der\IfValueT{#1}{,#1}}_{\IfValueT{#2}{#2}}}
\NewDocumentCommand{\SurSCR}{e{^_}}{\cat{CAlg}^{\an, \Delta^1\IfValueT{#1}{,#1}}_{\IfValueT{#2}{#2}}}
\NewDocumentCommand{\SurDAlg}{e{^_}}{\cat{CAlg}^{\der, \Delta^1\IfValueT{#1}{,#1}}_{\IfValueT{#2}{#2}}}

\newcommand{\GrMod}{\cat{GrMod}}

\newcommand{\FilMod}{\cat{FilMod}}

\newcommand{\sS}{\cat{S}}
\newcommand{\Sp}{\cat{Sp}}

\newcommand{\Cat}{\cat{Cat}}

\newcommand{\PrStk}{\mathrm{PrStk}}
\newcommand{\PreSt}{\PrStk}
\DeclareMathOperator{\Stk}{Stk} 
\newcommand{\Aff}{\cat{Aff}}
\newcommand{\Sch}{\mm{Sch}}

\newcommand{\fmp}{\cat{Moduli}}
\newcommand{\FMP}{\fmp}

\newcommand{\modulistk}{\mathrm{ModuliStk}}
\newcommand{\Poly}{\mathrm{Poly}}

\NewDocumentCommand{\LieAlg}{e{_}}{%
\cat{LieAlg}^\pi_{\Delta\IfValueT{#1}{,#1}}%
}
\NewDocumentCommand{\LieAlgd}{e{_}}{%
\cat{LieAlgd}^\pi_{\Delta\IfValueT{#1}{,#1}}%
}

\DeclareMathOperator{\fib}{fib} 
\DeclareMathOperator{\cofib}{cofib} 

\DeclareMathOperator{\triv}{triv}

\DeclareMathOperator{\Free}{Free}  
\newcommand{\free}{\Free}

\DeclareMathOperator{\forget}{forget}

\DeclareMathOperator{\End}{End} 
 \renewcommand{\hom}{\mathrm{Hom}}
 \DeclareMathOperator{\Map}{Map} 
 \newcommand{\Fun}{\mathrm{Fun}}

 \newcommand{\sym}{\mathrm{Sym}}
\newcommand{\LSym}{\mathbb{L}\sym}

 \DeclareMathOperator{\Tot}{Tot}

\newcommand{\gr}{\mathrm{Gr}} 
\newcommand{\sqz}{\mathrm{sqz}}

\newcommand{\adic}{\mathrm{adic}}

\DeclareMathOperator{\Spec}{Spec} 

\newcommand{\Spf}{\mathrm{Spf}}

\DeclareMathOperator{\Lie}{Lie} 
\DeclareMathOperator{\coLie}{coLie}

\DeclareMathOperator{\Ind}{Ind}

\newcommand*{\defeq}{\mathrel{\vcenter{\baselineskip0.5ex \lineskiplimit0pt
                     \hbox{\scriptsize.}\hbox{\scriptsize.}}}%
                     =}

\usepackage{cleveref}

\theoremstyle{definition}
\newtheorem{definition}{Definition}[section]

\newaliascnt{cons}{definition}

\aliascntresetthe{cons}

\newaliascnt{construction}{definition}
\newtheorem{construction}[construction]{Construction}
\aliascntresetthe{construction}

\newaliascnt{example}{definition}
\newtheorem{example}[example]{Example}
\aliascntresetthe{example}
\newtheorem*{example*}{Example}

\newaliascnt{notation}{definition}
\newtheorem{notation}[notation]{Notation}
\aliascntresetthe{notation}

\newaliascnt{remark}{definition}
\newtheorem{remark}[remark]{Remark}
\aliascntresetthe{remark}

\theoremstyle{plain}

\newaliascnt{proposition}{definition}
\newtheorem{proposition}[proposition]{Proposition}
\aliascntresetthe{proposition}

\newaliascnt{lemma}{definition}
\newtheorem{lemma}[lemma]{Lemma}
\aliascntresetthe{lemma}

\newaliascnt{corollary}{definition}
\newtheorem{corollary}[corollary]{Corollary}
\aliascntresetthe{corollary}

\newaliascnt{warning}{definition}
\newtheorem{warning}[warning]{Warning}
\aliascntresetthe{warning}

\newaliascnt{observation}{definition}
\newtheorem{observation}[observation]{Observation}
\aliascntresetthe{observation}

\newaliascnt{claim}{definition}

\aliascntresetthe{claim}

\newaliascnt{exercise}{definition}

\aliascntresetthe{exercise}

\newaliascnt{theorem}{definition}
\newtheorem{theorem}[theorem]{Theorem}
\aliascntresetthe{theorem}

\newaliascnt{superlemma}{definition}

\aliascntresetthe{superlemma}

\newaliascnt{fact}{definition}

\aliascntresetthe{fact}

\newtheorem*{theorem*}{Theorem}

\crefname{definition}{Definition}{Definitions}
\crefname{question}{Question}{Questions}
\crefname{cons}{Construction}{Constructions}
\crefname{construction}{Construction}{Constructions}
\crefname{example}{Example}{Examples}
\crefname{notation}{Notation}{Notations}
\crefname{remark}{Remark}{Remarks}
\crefname{proposition}{Proposition}{Propositions}
\crefname{lemma}{Lemma}{Lemmas}
\crefname{corollary}{Corollary}{Corollaries}
\crefname{warning}{Warning}{Warnings}
\crefname{observation}{Observation}{Observations}
\crefname{claim}{Claim}{Claims}
\crefname{exercise}{Exercise}{Exercises}
\crefname{theorem}{Theorem}{Theorems}
\crefname{superlemma}{Super Lemma}{Super Lemmas}
\crefname{fact}{Fact}{Facts}

\title{Formal Integration of Derived Foliations}
\author{Lukas Brantner}
\address{Oxford University, Universit\'{e} Paris–Saclay (CNRS)}
\email{brantner@maths.ox.ac.uk} 
\author{Kirill Magidson}
\address{Northwestern University}
\email{kirill.magidson@northwestern.edu} 
\author{Joost Nuiten}
\address{Institut de Mathématiques de Toulouse, Université de Toulouse}
\email{joost.nuiten@math.univ-toulouse.fr}

\begin{document} 
	\begin{abstract}
		Frobenius' theorem in differential geometry asserts that 
		every involutive subbundle of
		the tangent bundle of a manifold $M$ integrates to a decomposition of $M$ into smooth leaves.
		
		We   prove an infinitesimal   analogue  of this result
		for  locally coherent qcqs schemes $X$ over coherent rings. More precisely, we integrate partition Lie algebroids on $X$ to formal moduli \mbox{stacks $X \rightarrow \widehat{S}$,} where $\widehat{S}$ is the formal leaf space and the fibres of $X \rightarrow \widehat{S}$ are  the formal leaves.  	
		We  deduce that deformations of $X$-families of algebro-geometric objects
		are  controlled by partition Lie algebroids on $X$.
		Combining our integration equivalence with a result of Fu, we deduce that To\"{e}n--Vezzosi's infinitesimal derived foliations (under suitable finiteness hypotheses)  are formally integrable.

	\end{abstract}
	\setcounter{tocdepth}{1}
	\maketitle 
	\tableofcontents

	\section{Introduction}
	A  {smooth  algebraic foliation} $\mathcal{F}$ on  a  smooth  complex  variety $X$  
	is a subbundle $$ E_{\mathcal{F}} \subseteq T_X$$ which is closed under the commutator  bracket of vector fields, i.e.\ a locally free Lie algebroid whose anchor map is the inclusion of a subbundle.
	We   call  $ \mathcal{F}$   \textit{algebraically integrable} if there is a decomposition 
	$$ X = \coprod_{\alpha  } \mathcal{F}_{\alpha} $$ of $X$ 
	into smooth complex subvarieties $\mathcal{F}_{\alpha}$ --  the leaves -- whose tangent vectors span  $E_{\mathcal{F}} \subset T_X $. Note that algebraic integrability is  a highly nontrivial condition.
	One can approach the problem of integrating a given smooth foliation $ \mathcal{F}$  in three steps.

	First, we integrate $ \mathcal{F}$  \textit{formally}, which  means that for every   point $x \in M$, we find power series $$f_1,\ldots, f_k \in \CC[[z_1,\ldots,z_n]]$$
	in local coordinates $z_1,\ldots, z_n$ such that  a local vector field $\alpha $  lies in $E_{\mathcal{F}}$ \mbox{if and only if   $\alpha . df_i = 0$  for all $i$. }
	The formal leaf   passing through $x$ is then encoded by the
	functor  $$\mathcal{F}^\wedge_x:  A \mapsto  \{a_1,\ldots a_n \in A \ | \  \forall i: \ f_i(a_1,\ldots a_n ) = 0\} $$ from local Artinian $\CC$-algebras to sets.
		Varying $x \in X$,  we obtain a `quotient' map of prestacks $$p:X \rightarrow \widehat{S}$$ to the formal leaf space such that $\mathcal{F}^\wedge_x$ sends   $A$ to the set of (isomorphism classes of) dotted lifts
	\begin{equation*}\begin{tikzcd}
			\Spec(\CC)\arrow[r, "x"]\arrow[d] & X\arrow[d]\\
			\Spec(A)\arrow[r,  "\overline{x}"]\arrow[ru, dotted] & \widehat{S}.
	\end{tikzcd}\end{equation*}
	
	Next, we check the local convergence of these power series
	to obtain   complex manifolds $\mathcal{F}_x$ passing through all points of $X$, thereby verifying  Frobenius theorem for complex manifolds.
	
	Finally, we examine  whether or not the various leaves $\mathcal{F}_x$ are algebraic subvarieties of $X$. If this is the case, we obtain a  submersion  of Deligne--Mumford stacks $p: X \rightarrow S$ such that $X \rightarrow \widehat{S}$ is the formal completion of $X$ along $p$.

	\begin{figure}[h!]
		\centering
		\includegraphics[width=0.35\textwidth]{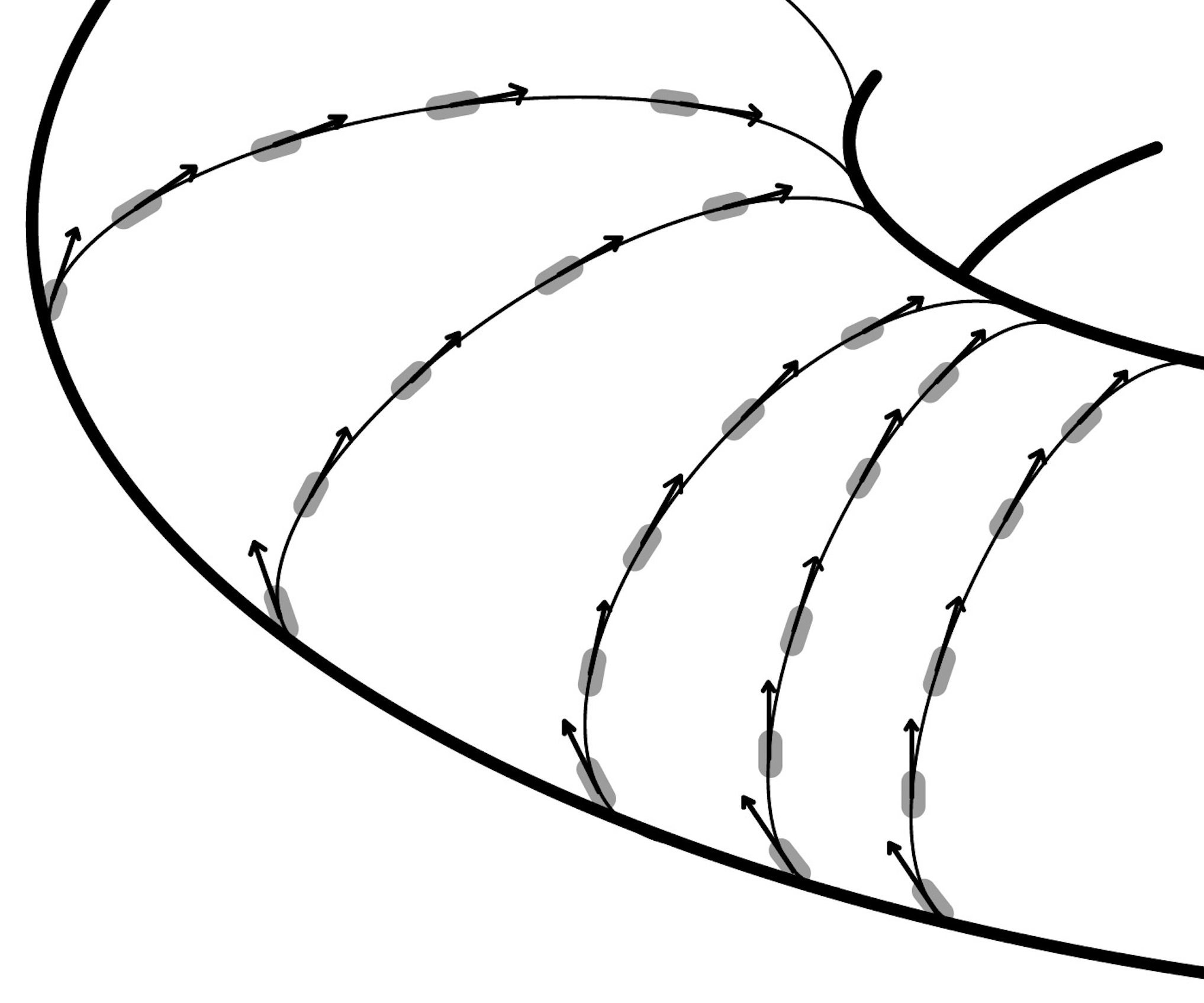}
		\caption{A foliation (depicted by tangent vectors) along with its formal  leaves (in gray) and leaves (in black).}
	\end{figure}

	In this article, we generalise the first step to more general Lie algebroids   on more general schemes.
	More precisely, we replace complex varieties by locally coherent qcqs schemes $X$  over \mbox{coherent rings $R$,} both possibly derived, and formally integrate \textit{partition Lie algebroids} $$\rho: E \rightarrow T_{X/R}[1]$$  with $\rho$ not necessarily injective and $E$ not \mbox{necessarily locally free.} 
	The result of our integration procedure is a \textit{formal  moduli stack on $X$ } --  around each point $x$ in $X$, we obtain a formal moduli problem $\mathcal{F}_x^{\wedge}$ encoding the  formal leaf  passing  \mbox{through $x$}, which can be  non-smooth, stacky, and derived.  
	
	\begin{example}
	Let us illustrate our integration procedure in the following example suggested by Abramovich. Consider a foliation on $\mathbb{A}^2$ generated by a single vector field, such as
	\[
	v = (x+y) \frac{\partial}{\partial x} + y^2 \frac{\partial}{\partial y}.
	\] 
First, assume that the base field $k$ has characteristic 0. Then this foliation can be encoded by the free Lie algebroid generated by $v$ in $T_{\mathbb{A}^2}$. To integrate it, we note that a (formal) curve $\gamma(t) = (x(t),y(t))$ in $\mathbb{A}^2$ is tangent to $v$ if $x’(t) = x(t)+y(t)$ and $y’(t) = y(t)^2$.
Setting $x(t) = \sum_n a_n t^n$ and $y(t) = \sum_n b_n t^n$, we obtain the equations
\[
(n+1) a_{n+1} = a_n + b_n\qquad\qquad\qquad (n+1) b_{n+1} = \sum_{i+j=n} b_i b_j.
\]
Solving recursively in terms of $a_0$ and $b_0$, we obtain elements $x(t),y(t)\in k[a_0,b_0][\![t]\!]$, which give rise to a map $k[x,y]\to k[a_0, b_0][\![t]\!]$ and hence a map
\[
\widehat{\mathbb{G}}_a \times \mathbb{A}^2 \to \mathbb{A}^2.
\]
Here $\widehat{\mathbb{G}}_a = \Spf(k[\![t]\!])$ denotes the formal additive group. This map in fact determines an action of $\widehat{\mathbb{G}}_a$ on $\mathbb{A}^2$, and the formal moduli stack $\mathbb{A}^2\to S$ corresponding to the Lie algebroid generated by the vector field $v$ is the homotopy quotient of this action, i.e., the geometric realisation of the simplicial diagram
\[\begin{tikzcd}
\dots \arrow[r, yshift=1.5ex]\arrow[r, yshift=0.5ex]\arrow[r, yshift=-1.5ex]\arrow[r, yshift=-0.5ex] & \widehat{\mathbb{G}}_a \times \widehat{\mathbb{G}}_a \times\mathbb{A}^2  \arrow[r, yshift=1ex]\arrow[r, yshift=0ex]\arrow[r, yshift=-1ex] & \widehat{\mathbb{G}}_a \times \mathbb{A}^2 \arrow[r, yshift=-0.5ex]\arrow[r, yshift=0.5ex]&  \mathbb{A}^2.
\end{tikzcd}\] 
The formal leaf passing through any point $z$ is $\mathcal{F}^\wedge_z\cong \widehat{\mathbb{G}}_a$. For a generic point $z$ the map $\mathcal{F}^\wedge_z\to \mathbb{A}^2$ is a formal immersion, but at the origin the map is simply constant.

When $k$ has characteristic $p>0$, the foliation gives rise to the free \emph{partition} Lie algebroid on one generator $v$. To integrate it to a formal moduli stack, we replace $\widehat{\mathbb{G}}_a$ by 
\[
\widehat{\mathbb{G}}_a^{\sharp} = \Spf(k\langle t\rangle^\wedge) = \colim_m \Spec\big(k\langle t\rangle/(\gamma_n(t) \colon n>m)\big),
\]
where $k\langle t\rangle^\wedge$ denotes the completion of the free divided power algebra on one generator $t$. This time, we set $x(t) = \sum_n a_n \gamma_n(t)$ and $y = \sum_n b_n \gamma_n(t)$, and the equations $x’(t)=x(t)+y(t)$ and $y’(t)=y(t)^2$ become 
\[
a_{n+1} = a_n + b_n \qquad\qquad\qquad b_{n+1} = \sum_{i+j = n} {n \choose i} b_i b_j.
\]
Solving recursively, we obtain elements in $k[a_0,b_0]\langle t\rangle^\wedge$, which give rise to an action $\widehat{\mathbb{G}}_a^{\sharp}\times\mathbb{A}^2\to\mathbb{A}^2$. 
Taking the homotopy quotient of this action as above gives the desired formal moduli stack $\mathbb{A}^2\to S$, and the formal leaf passing through any point is $\widehat{\mathbb{G}}_a^{\sharp}$. Note that, in contrast to the characteristic zero case, these leaves are all singular.
	\end{example}
	
	\begin{remark}
	The theory of partition Lie algebroids on $\mathbb{F}_p$-schemes is more subtle than that of Lie algebroids on $\mathbb{Q}$-schemes. For example, consider the formal moduli stack $\mathbb{A}^1 \to \mathbb{A}^1/\widehat{\mathbb{G}}_m$ obtained as the quotient of the weight $n$ action of the formal multiplicative group $\widehat{\mathbb{G}}_m\times\mathbb{A}^1\to \mathbb{A}^1; \ (z, x)\mapsto z^n x$. Explicitly, this action is given by the ring map 
	$$
	k[x]\to k[x][\![t]\!] ; \qquad x\mapsto (1+t)^nx,
	$$
	where $z=1+t$. The formal moduli stack $\mathbb{A}^1 \to \mathbb{A}^1/\widehat{\mathbb{G}}_m$ arises as the integration of a partition Lie algebroid with underlying anchor map $\rho \colon \mathcal{O}_{\mathbb{A}^1}\to T_{\mathbb{A}^1}$ sending the generator to a multiple of the Euler vector field $n\cdot x\frac{\partial}{\partial x}$. These partition Lie algebroids exhibit the following phenomena that are not present over $\mathbb{Q}$:
	\begin{enumerate}
	\item The partition Lie algebroids for the $\mathbb{G}_m$-actions of weight $n$ and weight $nl$ are equivalent when $l$ is coprime to $p$: they are identified by restricting the action along the étale map $(-)^l\colon \mathbb{G}_m\to \mathbb{G}_m$. However, the partition Lie algebroids of weights $p^k$ are not equivalent.
	
	\item For the $\mathbb{G}_m$-action of weight $p^k$, the underlying anchor map $\mathcal{O}_{\mathbb{A}^1}\to T_{\mathbb{A}^1}$ is zero even though the $\widehat{\mathbb{G}}_m$-action on $\mathbb{A}^1$ is non-trivial.
	\end{enumerate}
	 The second point shows that it is not possible to obtain the formal moduli stack $\mathbb{A}^1 \to \mathbb{A}^1/\widehat{\mathbb{G}}_m$ as the formal integration of a classical Lie algebroid or $p$-restricted Lie algebroid, as considered in \cite{Ekedahl}: since the anchor map is zero, any such integration would come with a retraction to $\mathbb{A}^1$.
	\end{remark}
	
	Formal moduli stacks $X \rightarrow \widehat{S}$  also arise in deformation theory, where they capture 
	infinitesimal derived deformation functors of $X$-indexed families of algebro-geometric objects  
	\begin{equation*}\begin{tikzcd}
			M\arrow[d]\\
			X.
	\end{tikzcd}\end{equation*} 
	Our integration equivalence then implies that every such deformation functor is governed by an essentially unique partition Lie algebroid  on $X$, generalising various earlier equivalences in characteristic $0$ (\cite{drinfeld1988letter, deligne1986letter, hinich2001dg, DAG-X, pridham2010unifying,  Hennion, GR, nuiten2019koszul}) and $p$ (\cite{BM19, brantner2020purely}).
	
	Foliations on  smooth complex varieties $X$ can also be defined as  quotient bundles $$\Omega_X \twoheadrightarrow \Omega_{\mathcal{F}}$$  inducing maps of commutative differential graded algebras $$ (\Lambda^\ast \Omega_{X}, d_{\mathrm{dR}}) \rightarrow (\Lambda^\ast \Omega_{\mathcal{F}}, \overline{d}_{\mathrm{dR}}).$$ 
This definition was   generalised  by To\"{e}n--Vezzosi \cite{toen2023infinitesimal}    to  the non-smooth derived setting -- they call the resulting structure an 
  \textit{infinitesimal derived foliation}.  \mbox{Fu} \cite{FuThesis} recently constructed an equivalence  between  almost perfect
	infinitesimal derived foliations   and    dually almost perfect partition Lie algebroids. Combining  Fu's   equivalence with our main result, we deduce that  almost perfect   infinitesimal derived foliations  (in the sense of To\"{e}n--Vezzosi) are formally integrable.

	\subsection{Statement of results}
	A Lie algebroid on a smooth manifold $X$ 
	consists of an anchor map $$\rho: E \rightarrow T_X$$   of vector bundles  
	together with a bilinear Lie bracket $[-,-]$ on the space $\Gamma(X,E)$ of global sections such that for all global sections $s_1, s_2 \in \Gamma(M,E)$ and all smooth functions $f$, we have $$[s_1,fs_2 ]= \rho(s_1)(f) \cdot s_2 + f[s_1,s_2].$$
	Lie algebroids were  first introduced by Pradines \cite{pradines1967theorie} and have since found numerous applications.
	
	In this article, we will  first generalise the theory of Lie algebroids   the setting of algebraic geometry over \mbox{very general bases,} and then formally integrate them.
	
	\subsubsection*{Setup.} Let us fix a  coherent base  ring $R$, i.e.\  a commutative ring  whose  finitely generated ideals are   also  finitely presented.
	Let $X$ be a    quasi-compact quasi-separated (`qcqs') derived $R$-scheme   which is locally coherent,
meaning that it	admits an open cover by affine derived schemes 
	$\Spec(B)$ with $\pi_0(B)$ coherent and $\pi_i(B)$ \mbox{finitely presented over $\pi_0(B)$. }
	
	Every such scheme admits a  quasi-coherent cotangent complex $L_{X/R} $ and a   {tangent complex} $$ T_{X/R}  = L_{X/R}^{\vee}.$$
This tangent complex belongs to the $\infty$-category $ \QC^{\vee}_X = \Ind(\Coh_X^{\op})$  of pro-coherent sheaves on $X$, which is the natural home of duals of quasi-coherent sheaves.
Pro-coherent sheaves  were first introduced by  Deligne \cite[Appendix]{hartshorne1966residues};  we refer to  \Cref{procohappendix} for further details.
	
	\subsubsection*{Partition Lie algebroids.} 
	In   \Cref{PartitionLieAlgebroids}, we
	construct the $\infty$-category 	of \textit{partition Lie algebroids}
	$$ \LieAlgd_{X/\KK}$$ 
 on the scheme $X$. These homotopical objects exhibit many expected properties:
	\begin{enumerate}
		\item (Forgetful functor) There is a monadic forgetful functor $\LieAlgd_{X/\KK} \rightarrow (\QC^{\vee}_X)_{/T_{X/R}[1]}$; hence partition Lie algebroids are maps of pro-coherent sheaves 
		$$ \mathfrak{g} \xrightarrow{\ \ \rho \ \ } T_{X/R}[1]$$
		equipped  with additional Lie algebraic structure.
		\item (Global sections) Taking derived global sections gives rise to  a functor $$ \LieAlgd_{X/\KK} \rightarrow  (\LieAlg_{\KK})_{/R\Gamma(X,T_{X/R}[1])},$$
		where $\LieAlg_{\KK} =  \LieAlgd_{\Spec(\KK)/\KK}$ denotes the $\infty$-category of $R$-partition Lie algebras   and 
		$$R\Gamma(X,T_{X/R}[1])$$ is the Kodaira--Spencer partition Lie algebra encoding infinitesimal deformations of the $R$-scheme $X$ (cf.\ \Cref{sec:underlying lie}).  
 
		\item (Descent) Every \'{e}tale map $$Y \rightarrow X$$ induces a functor $ \LieAlgd_{X/\KK} \rightarrow  \LieAlgd_{Y/\KK}$ lifting the pullback functor on the level of pro-coherent sheaves.
		Moreover, there is a canonical equivalence of $\infty$-categories $$
		\LieAlgd_{X/\KK} \xrightarrow{\ \simeq \ } \lim_{U\subseteq X \text{ affine open}} \LieAlgd_{U/\KK}.
		$$  
		
	\end{enumerate}
		Before proceeding to formal integration, let us highlight various special cases of our construction:

	\begin{enumerate}[label=(\alph*)]
		\item For $k$ a field of characteristic $0$ and $X = \Spec(k)$, the $\infty$-category $ \LieAlgd_{{X/k}} =\LieAlg_k $ is modelled by  (shifted)  differential graded Lie algebras over $k$.    
		For $X = \Spec(A)$ an affine $k$-scheme, $ \LieAlgd_{{X/k}}  $ arises from  the category of differential graded Lie algebroids on $X$ with the tame semi-model structure, \mbox{see \cite[Variant 3.12]{nuiten2019homotopical}.}
		
		\item For $k$  a field of characteristic $p$ and  $X = \Spec(k)$, we recover the  $\infty$-category of  {partition Lie algebras}   introduced in \cite{BM19}. These can
		be modelled by cosimplicial-simplicial $k$-vector spaces with additional operations parametrised by nested chains of partitions, see \cite[Definition 5.43]{BCN21}. For $X = \Spec(F)$ with $F/k$ a 
		finite purely inseparable field extension, $\LieAlgd_{{X/k}}$ recovers the $\infty$-category of $F/k$-partition Lie algebroids appearing in the purely inseparable Galois correspondence \cite{brantner2020purely}.
		For $R= A $  complete local Noetherian with residue field $k$ and $X = \Spec(k)$, the $\infty$-category $\LieAlgd_{{X/A}}$ is equivalent to the mixed partition Lie algebras appearing in \cite[Definition 6.22]{BM19}.

		\item For $R$ a coherent ring and $X = \Spec(R)$, the $\infty$-category $\LieAlgd_{{X/R}}=\LieAlg_R$ appears in  \cite[Section 3]{BCN21}. It admits an explicit model in terms of simplicial-cosimplicial objects, given in \cite[Theorem 5.42]{BCN21}.  	
	\end{enumerate}  
	
	\subsubsection*{Formal moduli stacks}
Given $X$ and $R$ as above, we will now introduce    \textit{formal moduli stacks}
	$$ X \rightarrow \widehat{S}$$ under $X$. Informally, these   can be   thought of  
	in at least two different   ways:
	\begin{enumerate}
		\item as  decompositions of $X$ into of formal leaves $\mathcal{F}^{\wedge}_x$ parametrised by the points $x$ of $\widehat{S}$;
		\item as  infinitesimal   deformation functors of   $X$-indexed families of  algebro-geometric objects.
	\end{enumerate}
	To formalise the notion of a formal moduli stack, we will use  \textit{animated rings}.
	\begin{notation}
	We write  $\SCR = \cat{P}_{\Sigma}(\Poly)$ for the  $\infty$-category of animated rings, which is obtained by freely adding sifted colimits to the category of finitely generated polynomial rings. Note that $\SCR $ is also sometimes  denoted  by $\mathrm{Ani}$, $\cat{DAlg}_{\geq 0},$ $\CAlg^{\Delta}$, or $\mathrm{SCR}$ in the literature. The last notation emphasises that animated rings can be  modelled by simplicial commutative rings.

	\end{notation}
Formal moduli stacks will belong to the  $\infty$-category of \textit{$R$-prestacks} $$\PrStk_R $$ i.e.\ the $\infty$-category of  (accessible) functors    from animated $R$-algebras to spaces. Every derived $R$-scheme $X$ gives a prestack $ R \mapsto \Map_{\Sch^{\der}}(\Spec(R),X)$, which we  \mbox{denote by the same name.}

	\begin{definition}[Formal moduli stacks]
		A \textit{formal moduli stack} $\mathcal{F}$ under an $R$-prestack $X$ is a map of $R$-prestacks $$ X \longrightarrow \widehat{S} $$
		satisfying the following conditions:
		\begin{enumerate}
			\item $\widehat{S} $ has deformation theory,  which means that it preserves limits of Postnikov towers and pullbacks along nilpotent extensions of animated $R$-algebras, see \Cref{hasdeftheory};
			\item $X\rightarrow \widehat{S} $ is locally almost of finite presentation, see \Cref{laftdef};
			\item $X \rightarrow \widehat{S} $ restricts to an equivalence on all reduced affine  schemes.
		\end{enumerate} 
Write $\modulistk_{X/R} \subset (\PrStk_R)_{X/}$ for the full subcategory of formal moduli stacks under $X$.
	\end{definition}

	\begin{remark}[Formal leaves of formal moduli stacks] 
		Given an $R$-point $x$ in  $X$ and a 	formal moduli stack 
		$X \rightarrow \widehat{S}$, the formal leaf $\mathcal{F}^{\wedge}_x$ passing  through $x$ is the formal moduli problem 
		sending an augmented local Artinian animated $R$-algebra $A$ to the space of dotted lifts
		\begin{equation*}\begin{tikzcd}
				\Spec(R)\arrow[r, "x"]\arrow[d] & X\arrow[d]\\
				\Spec(A)\arrow[r,  "\overline{x}"]\arrow[ru, dotted] & \widehat{S}.
		\end{tikzcd}\end{equation*}
	\end{remark}

	\begin{remark}[Deformation functors and formal moduli stacks]
		Let $B$ be a coherent eventually coconnective animated $R$-algebra. Formal deformation functors of (suitably nice) $B$-indexed families of algebro-geometric objects  over $R$ are then parametrised by functors
		$$\begin{tikzcd}
			F\colon \Art_{\KK/B}\arrow[r] & \sS
		\end{tikzcd}$$
		from the $\infty$-category $$\Art_{\KK/B} \subset \SCR_{\KK/B}$$ of Artinian extensions of $B$ (cf.\ \Cref{def:Artinian}) to the $\infty$-category $\cat{S}$ of spaces
		satisfying
		\begin{enumerate}
			\item $F(B)\simeq \ast$;
			\item $F$ preserves pullbacks along nilpotent extensions.
		\end{enumerate}
		
		In \Cref{sec:formal integration}, we show that the $\infty$-category of such functors is equivalent to the $\infty$-category of formal moduli stacks under $X= \Spec(B)$; this statement is essentially contained in \cite[5.1.2]{GRII}.
		Indeed, if $X\rt  \widehat{S}$ is a formal moduli stack, then the corresponding functor $\Art_{R/B}\rt \sS$ sends an Artinian extension $A\rt B$ over $R$ to the space of dotted lifts in the following diagram:
		$$\begin{tikzcd}
			{X}\arrow[r]\arrow[d] &  \widehat{S}\arrow[d]\\
			\Spec(A)\arrow[r]\arrow[ru, dotted] & \Spec(\KK).
		\end{tikzcd} $$ 	 \end{remark}

	\subsubsection*{The main theorem} 
 Given a formal moduli stack $X \rightarrow \widehat{S}, $
we can consider the    sheaf of `derived  tangent vectors'    \mbox{to the formal leaves $\mathcal{F}_x^{\wedge}$.} 
 To make this idea precise, 
 we use that the $\infty$-category $\QC^{\vee}_X$ of pro-coherent sheaves can be identified with the $\infty$-category of   functors
  $\Coh_{X,\geq 0} \rightarrow \cat{S}$ 
 preserving   terminal objects  and pullbacks along $\pi_0$-surjections. Here $\Coh_{X,\geq 0} $ denotes the $\infty$-category of connective coherent sheaves on $X$.
 
 The  pro-coherent tangent sheaf $$T_{X/\widehat{S}} $$ then
 corresponds to the functor   sending $I \in \Coh_{X,\geq 0} $ to  the space of dotted arrows in the following square:
	\begin{equation*}\begin{tikzcd}
			X\arrow[r,  ]\arrow[d] & X\arrow[d]\\
			\sqz_X(I)\arrow[r,   ]\arrow[ru, dotted] & \widehat{S}.
	\end{tikzcd}\end{equation*}
	Here $\sqz_X(I)$ denotes the   derived scheme  $(X, \mathcal{O}_X \oplus I)$, where $\mathcal{O}_X \oplus I$ is the trivial square-zero extension of the structure sheaf $\mathcal{O}_X$ by $I$. There is a natural map $$T_{X/\widehat{S}}  \rightarrow T_{X/R}, $$
 which captures the derived tangents to the formal leaves. We then prove:
	\begin{theorem}[Main theorem] \label{main} 
		Let $X$ be a locally coherent qcqs derived $\KK$-scheme over a coherent animated ring $R$. 
			The shifted tangent fibre functor $$\modulistk_{X/R} \longrightarrow (\QC^{\vee}_X)_{T_{X/R}[1]} $$
			$$ (X \rightarrow Y) \ \  \mapsto  \ \ (T_{X/Y}[1] \rightarrow  T_{X/R}[1])$$
			lifts to an equivalence of $\infty$-categories $$T_{X/-}[1]: \modulistk_{X/R} \xrightarrow{\simeq} \LieAlgd_{X/\KK}.$$
			Hence every partition Lie algebroid on $X$ integrates uniquely to a  formal moduli \mbox{stack under $X$.}
			
		 	This equivalence is functorial with respect to almost finitely presented maps.   
	\end{theorem}

	A key ingredient in the proof of \Cref{main} is a Koszul equivalence for animated rings. Before stating it, we introduce some notation.

	\begin{definition}
		An animated $R$-algebra $A$ over $B$ is said to be \mbox{\textit{complete almost finitely augmented} if }
		\begin{enumerate}
			\item the induced map $\pi_0(A) \rightarrow \pi_0(B)$ is surjective on $\pi_0$ with  kernel $I$;
			\item  the animated ring $A$ is derived $I$-complete;
			\item the $A$-module $B$ is almost perfect as an $A$-module.
		\end{enumerate}  
		Write $\SCR^{ \aft,\wedge}_{R/B} \subset \SCR_{R/B}$ for the full subcategory spanned by those animated $R$-algebras which are complete almost finitely augmented over $B$.
	\end{definition} 
	
	\begin{example}
		For $R=B$ a complete local Noetherian animated ring, we can identify $\SCR^{ \aft,\wedge}_{R/B }$ with the $\infty$-category of augmented animated $R$-algebras  which are complete, local, and Noetherian.  
	\end{example}
	
	In \Cref{PartitionLieAlgebroids}, we lift the assignment $A \mapsto (L_{B/A}^{\vee}[1]\rightarrow L_{B/R}^{\vee}[1])$ to a functor $$\mathfrak{D}: \SCR_{R/B} \rightarrow \Lie^{\pi}_{\Delta, X/R}$$
	for $X = \Spec(B)$ and  establish the following result:
	\begin{theorem}\label{mainaffine}
		The functor $\mathfrak{D}$ restricts to an equivalence between the full subcategory  $$\SCR^{ \aft,\wedge}_{R/B } \subset \SCR_{R/B } $$ of complete almost finitely augmented objects and the full subcategory
		$$\Lie^{\pi,\dap, \weirdleq 0}_{\Delta, X/R}  \subset \Lie^{\pi}_{\Delta, X/R}$$
		of partition Lie algebroids $\mathfrak{g} \rightarrow L_{B/R}^{\vee}[1]$ with $\mathfrak{g}$ dually  almost perfect of Tor-amplitude $\leq 0$.  
	\end{theorem}

	\Cref{main} and \Cref{mainaffine} generalise  various earlier results:
	\begin{enumerate}[label=(\alph*)]
		\item Let $k$ be a field of characteristic $0$. For $R=k$  and $X = \Spec(k)$ a point, our results specialise to Lurie--Pridham's equivalence between formal moduli problems and differential graded Lie algebras \cite{DAG-X,pridham2010unifying}, which is in turn extends earlier work work by Deligne \cite{deligne1986letter}, Drinfel'd \cite{drinfeld1988letter},  Hinich \cite{hinich2001dg}, Manetti \cite{manetti2009differential},   and others.
		
		Now assume $A$ is a  Noetherian and eventually coconnective animated $k$-algebra. For $R=k$ and $X = \Spec(A)$, our result is contained in \cite{nuiten2019koszul}, whereas for $R = A$ and $X = \Spec(A)$, it was proven by Hennion \cite{Hennion}. 
		
		In characteristic $0$, the monadicity of the tangent fibre functor  is also proven in the work of Gaitsgory--Rozenblyum  \cite[Chapter 8]{GRII}. Note, however, that their proof 
		crucially relies on the fact that  looped Lie algebras in characteristic $0$ are abelian --- this fails away from characteristic $0$ due to the presence  of power operations.
		
		\item Let $k$ be a field of characteristic $p$. For    $X = \Spec(k)$ and  $R=k$ or, more generally,  $R=A$ a complete local Noetherian ring with residue field $k$, our results are contained in   \cite{BM19}. For $X=\Spec(L)$ with $L$ a finite purely inseparable field extension over $R=k$,  our Koszul statement \Cref{mainaffine} recovers  \cite[Theorem 3.15]{brantner2020purely}.
	\end{enumerate}

	\subsection{Acknowledgements}
	We are grateful to Dan Abramovich, Jiaqi Fu, Sof\'{i}a Marlasca Aparicio, Bertrand To\"{e}n, and Nikola Tomi\'{c} 	for various discussions related to the content of this paper.
	
	L.B.\ was supported as a Royal Society University Research Fellow at Oxford University  through grant URF$\backslash$R1$\backslash$211075 and by the Centre national de la recherche scientifique (CNRS) at Orsay.
 K.M.\ was supported by NSF grants DMS-2152235 and DMS-21002010. J.N.\ was supported by the CNRS, under the programme PEPS JCJC, and the ANR project ANR-24-CE40-5367-01 (``LieDG'').

\newpage

\section{Affine Koszul duality}\label{section:affineKD} 
We write $\DAlg_{\KK}$ for the $\infty$-category of derived commutative rings and
$\SCR_{\KK} $  for the $\infty$-category of animated $\KK$-algebras; a summary of their theory is given in \Cref{sec:deralg}.
Note that animated rings are the same as connective derived rings:
 $$
 `\text{animated } = \text{ derived }+\text{ connective'}.
 $$

The goal of this section will be to provide some results about derived rings that are derived $I$-adically complete for ideals $I$ that are of ``sufficiently finite type''; most importantly, we prove a comonadicity result (\Cref{thm:comonad}) generalising results from \cite{BM19}.

We will start by recalling a derived analogue of endowing a ring with the $I$-adic filtration with respect to an ideal $I$. Instead of working directly with ideals $I\subseteq A$, it is more convenient to work with the corresponding closed immersion $A\rt A/I=B$. We will refer to objects $A \rt A/I=B$ as \textit{arrows} and use the notation $(A\rightarrow B)$.

\subsection{Derived filtered and graded rings}
To describe the adic filtration associated to an arrow $(A\rightarrow B)$, we first recall how one can obtain the $\infty$-category of derived filtered rings as the $\infty$-category of $\LSym$-algebras in the $\infty$-category of filtered modules.
\begin{notation}
We will consider the following $\infty$-categories:
\begin{enumerate}[label=(\alph*)]
\item The $\infty$-category $\Mod_{\ZZ}^{\Delta^{1}}:=\Fun(\Delta^1, \Mod_{\ZZ})$ of maps of $\ZZ$-modules $(X \to Y)$. 

\item The $\infty$-category $\FilMod_{\ZZ}=\Fun\big((\mathbb{Z}_{\geq 0}, \leq)^{\op}, \Mod_{\ZZ}\big)$ of non-negatively filtered $\ZZ$-modules, i.e. diagrams in $\Mod_{\ZZ}$ of the form
$$
\dots\to F^1V\to F^0V=V.
$$
We will denote objects of $\FilMod_{\ZZ}$ as $F^{\star}V$, and often denote the underlying object $F^{0}V$ of the filtration simply as $V$. 

\item The $\infty$-category $\GrMod_{\ZZ}=\Fun((\ZZ_{\geq 0}, =), \Mod_{\ZZ})$ of non-negatively graded $\ZZ$-modules. We use the notation $V^{\bullet}$ for objects of $\GrMod_{\ZZ}$ to emphasise the grading.
\end{enumerate}
We will be interested in the following three functors relating these $\infty$-categories:
$$\begin{tikzcd}
\GrMod_{\ZZ} &\FilMod_{\ZZ}\arrow[l, "\gr"{swap}] \arrow[r, "\aug"{swap}, yshift=-1ex]  &  \Mod^{\Delta^{1}}_{\ZZ} \arrow[l, "\ker"{swap}, yshift=1ex].
\end{tikzcd}$$
Here $\gr^{\bullet}(V)$ is the associated graded of a filtered module $F^{\star}V$, i.e.\ $\gr^i(V)=F^iV/F^{i+1}V$. The functor $\aug$ sends a filtered module $F^{\star}V$ to the natural augmentation $(V\to \gr^0(V) )$ to its $0$-th graded piece. The functor $\ker$ is its \textit{left} adjoint and sends an arrow $(p\colon V\to V_0)$ to the filtered module $\dots\to 0 \to 0\to \fib(p)\to V$.
\end{notation}
We endow each of these $\infty$-categories with a $t$-structure making it a stably projectively generated $\infty$-category in the sense of  \Cref{def:proj gend}, as well as a closed symmetric monoidal structure:
\begin{enumerate}[label=(\alph*)]
\item We endow $\Mod_{\ZZ}^{\Delta^1}$ with the \emph{surjective $t$-structure}, in which an object $X \rightarrow Y$ is connective if $Y$ and $\mm{fib}(X\to Y)$ (and hence also $X$) are connective $\ZZ$-modules. In other words, $X\to Y$ is connective if $X$ and $Y$ are connective and $\pi_0(X)\to \pi_0(Y)$ is surjective. An object $X\to Y$ is coconnective if $X$ and $\mm{fib}(X \rightarrow Y)$ are coconnective $\ZZ$-modules. The monoidal structure is the levelwise tensor product 
$$
(X\to Y)\otimes (X'\to Y') = (X\otimes X'\to Y\otimes Y').
$$

\item We endow $\FilMod_{\ZZ}$ with the Day convolution product and the $t$-structure in which $F^\star V$ is (co)connective if each $F^iV$ is (co)connective.

\item Likewise, $\GrMod_{\ZZ}$ is equipped with the Day convolution product and the $t$-structure in which $V^\bullet$ is (co)connective if each $V^\bullet$ is (co)connective.
\end{enumerate}
The generators of $\Mod_{\ZZ}^{\Delta^1}$ are the surjections $\mathbb{Z}^{\oplus m+n}\to \mathbb{Z}^{\oplus n}$ and the generators of $\FilMod_{\ZZ}$ and $\GrMod_{\ZZ}$ are the finitely generated free abelian groups, where each generator is of a certain (filtration) weight. Each of these $\infty$-categories is then a derived algebraic context in the sense of  \Cref{def:alg context} (see \cite[Section 4.2]{R20}) and hence comes with a notion of derived commutative algebra: these are  algebras over the monad obtained by right-left extending the (strict) symmetric algebra monad on the categories of surjections of flat abelian groups and of flat filtered/graded abelian groups. 
 
\begin{definition}
We define the $\infty$-categories of \emph{animated} (resp.\ \emph{derived}) \emph{filtered rings}, \emph{animated} (resp.\ \emph{derived}) \emph{graded rings} and \emph{animated} (resp.\ \emph{derived}) \emph{surjections} to be the $\infty$-categories
\begin{align*}
\FilSCR &=\SCR(\FilMod_{\ZZ, \geq 0}) & \FilDAlg&=\DAlg(\FilMod_{\ZZ})\\
\grSCR &=\SCR(\cat{GrMod}_{\ZZ, \geq 0}) & \grDAlg&=\DAlg(\cat{GrMod}_{\ZZ})\\
\SurSCR&= \SCR(\Mod_{\ZZ, \geq 0}^{\Delta^1}) & \SurDAlg &=\DAlg(\Mod_{\ZZ}^{\Delta^1})
\end{align*}
of animated and derived commutative algebras in the sense of  \Cref{def:derived comm}.
\end{definition}
\begin{remark}
An animated surjection is simply a map of animated rings $A\to B$ which is surjective on $\pi_0$; these have been introduced by Z.\ Mao in \cite{M21}, where they are called animated pairs.
\end{remark}
\begin{remark}
By Remark \ref{rem:Eoo}, every derived filtered or graded ring has an underlying $\mathbb{E}_\infty$-algebra with respect to  Day convolution, and modules are  modules over this underlying $\mathbb{E}_\infty$-algebra.
\end{remark}
Taking associated graded objects and augmentations commutes with the $\LL\sym$-monads, so that we obtain functors
$$\begin{tikzcd}
\grDAlg & \FilDAlg\arrow[l, "\gr"{swap}]\arrow[r, "\aug"] &  \SurDAlg.
\end{tikzcd}$$
Both functors preserve limits, colimits and connective objects, and the functor $\aug$ sends a filtered algebra $F^{\star}A$ to the augmentation arrow $F^0A\to \gr^0(A)$.
\begin{definition}
\mbox{The \textit{adic filtration} functor $\adic\colon \SurDAlg\rt \FilDAlg$ is the left adjoint to $\aug$.}
\end{definition}
Since the augmentation functor is left $t$-exact, taking adic filtrations preserves connective objects. We therefore obtain a diagram of the form
$$\begin{tikzcd}
\grSCR & \FilSCR\arrow[l, "\gr"{swap}] \arrow[r, "\aug"{swap}, yshift=-1ex]  & \SurSCR \arrow[l, "\adic"{swap}, yshift=1ex].
\end{tikzcd}$$
where all functors are left adjoints. Note that taking the associated graded does not preserve limits at the connective level. 

We give a familiar description of the adic filtration on the level of $0$-th homotopy groups. 

\begin{proposition}\label{prop:I-adic on pi_0}
Let $(A \to B)$ be an animated surjection, and $I=\ker(\pi_{0} A \rightarrow \pi_{0}B)$. Then $\pi_{0}\big(\adic(A\to B)\big)$ is given by the commutative algebra in $\Fun((\mathbb{Z}, \geq), \Mod_{\KK}^\heartsuit)$
$$\begin{tikzcd}[column sep=1pc]
\dots\arrow[r] & \sym^3_{\pi_0(A)}(I)\arrow[r] & \sym^2_{\pi_0(A)}(I)\arrow[r] & I \arrow[r] & \pi_0(A).
\end{tikzcd}$$
\end{proposition}
In the situation where $I$ is flat as a $\pi_0(A)$-module, we have that $\sym^n_{\pi_0(A)}(I)\simeq I^n$. Each map is then an inclusion and $\pi_{0}\big(\adic(A\to B)\big)$ is given by the classical $I$-adic filtration on $\pi_0(A)$.
\begin{proof}
We can assume that both $A$ and $B$ are discrete. The filtered algebra $\pi_{0}\adic(A\to B)$ has the universal property that given any discrete filtered algebra $F^{\star} R'$, any map $(A \to B) \rightarrow (R' \to \gr^{0}(R')) $ extends uniquely to a filtered map $\pi_{0}\adic(A \to B) \rightarrow F^{\star} R'$. Note that the category of discrete filtered algebras is equivalent to the category of commutative algebras in $\FilMod_{\KK}^\heartsuit=\Fun((\mathbb{Z}, \geq), \Mod_{\KK}^\heartsuit)$; in other words, $F^{n+1}R'\to F^nR'$ need not be injective.

The data of a map of surjections $(A\to B) \rightarrow (R' \to \gr^{0}(R'))$ is the same as the data of a map $(A,I) \rightarrow (R', F^{1}R')$ of algebras equipped with an ideal. Let $f_{1}\colon I \rightarrow F^{1}R'$ be induced map on ideals. This extends uniquely to a map of discrete filtered algebra $\sym_A^\star I \rt F^\star R'$, given in weight $n$ by
$$\begin{tikzcd}
\sym_A^n(f)\colon \sym^{n}_{A}(I) \arrow[r]& F^{n}R; \qquad f^{\otimes n} (x_{1}\otimes \dots \otimes x_{n}) =  f_{1}(x_{1})\dots  f_{1}(x_{n}),
\end{tikzcd}$$
where the product is taken inside $F^{\star}R$. 
\end{proof}

Finally, recall that a filtered module is \textit{complete} if $\lim F^nV=0$ and that a derived filtered commutative ring is complete if it is complete as a filtered module. Equivalently, a filtered module is complete if it is local with respect to all maps inducing equivalences on the associated graded. Since taking symmetric algebras preserves such graded equivalences, it follows that the completion of a derived filtered ring is again a derived filtered ring. In good situations, one can relate the completeness of a filtered derived algebra $F^\star A$ to completeness of the underlying algebra $F^0A$ with respect to its augmentation ideal:
\begin{lemma}\label{lem:completeness}
Let $A$ be an animated filtered ring and let $I$ denote the kernel of the augmentation $\pi_0(F^0A)\rt \pi_0(\gr^0(A))$. Suppose that $I$ is a finitely generated ideal of $\pi_0(F^0A)$. Then the following hold:
\begin{enumerate}
\item If $M$ is a complete filtered $A$-module, then the underlying $F^0A$-module $F^0M$ is $I$-complete \cite[Definition 7.3.1.1]{SAG}.

\item Suppose that $A=\adic(F^0A\to \gr^0(A))$ is the adic filtration of an animated surjection, and let $A^\wedge$ be its completion. If $F^0A$ is $I$-complete and $\gr^0(A)\otimes_{F^0A} F^0A^\wedge\rt \gr^0(A)$ is an equivalence, then $A\simeq A^\wedge$.
\end{enumerate}
\end{lemma}
The proof of  \Cref{lem:completeness} makes use of the following version of the ``derived Nakayama lemma'':
\begin{lemma}\label{lem:nakayama}
Let $A\to B$ be a map of connective $\mathbb{E}_2$-rings inducing a surjection on $\pi_0$ and suppose that $I=\ker(\pi_0(A)\to \pi_0(B))$ is a finitely generated ideal. If $M$ is a connective $I$-complete $A$-module such that $B\otimes_A M\simeq 0$, then $M\simeq 0$.
\end{lemma}
\begin{proof}
Suppose that there exists a smallest integer $n$ such that $\pi_n(M)\neq 0$. Then $\pi_n(M)$ is an $I$-complete $\pi_0(A)$-module \cite[Theorem 7.3.4.1]{SAG} such that $\pi_n(M)\otimes_{\pi_0(A)}\pi_0(B)\simeq 0$. But \cite[Proposition 6.5]{dwyer2002complete} would then imply that $\pi_n(M)=0$.
\end{proof}
\begin{proof}[Proof of  \Cref{lem:completeness}]
If $M$ is a complete filtered module, then $F^0M\simeq \lim F^0M/F^nM$ is a limit of $I$-complete $F^0A$-modules: since each $x\in I$ is of weight $\geq 1$, $x^n\colon F^0M/F^nM\to F^0M/F^nM$ is null.

For (2), the map $A\rt A^\wedge$ induces an equivalence on the associated graded, so that it suffices to verify that $F^0A\rt F^0A^\wedge$ is an equivalence. To see this, let us start by observing that since $A$ arises as an adic filtration,  \Cref{prop:I-adic on pi_0} implies that
$$
\pi_0(F^0A^\wedge)\cong \lim_n \mm{coker}\big(\pi_0(F^nA)\to \pi_0(F^0A)\big) \cong \pi_0(F^0A)^\wedge_I
$$
agrees with the (classical) $I$-adic completion of $\pi_0(F^0A)$. Since $\pi_0(F^0A)$ is $I$-complete, the map $\pi_0(F^0A)\rt \pi_0(F^0A^\wedge)$ is surjective. Consequently, the kernel $I'=\ker\big(\pi_0(F^0A^\wedge)\rt \pi_0(\gr^0(A))\big)$ is the image of $I$; in particular, it is finitely generated. It follows from part (1) that $F^0A^\wedge$ is $I'$-complete and hence $I$-complete as an $A$-module.  \Cref{lem:nakayama} then implies that $F^0A\rt F^0A^\wedge$ is an equivalence if and only if $\gr^0(A)\rt \gr^0(A)\otimes_{F^0A} F^0A^\wedge$, is an equivalence. This map is a section of the natural map $\gr^0(A)\otimes_{F^0A} F^0A^\wedge\rt \gr^0(A)$.
\end{proof}

\subsection{Cotangent fibre}
The usual cotangent complex formalism for derived rings extends in an evident way to the graded/filtered/arrow setting; we refer to Appendix \ref{sec:comm} for a general discussion. We will be mainly interested in the relative cotangent complexes of augmentation maps:

\begin{definition}[Cotangent fibre]
For every $A$ in   $\grDAlg$ or  $\FilDAlg$, we define the \textit{cotangent fibre} to be the desuspended relative cotangent complex 
$$
\cot(A)=L_{\gr^0(A)/A}[-1].
$$
Here we view $\gr^0(A)$ as a derived graded/filtered algebra concentrated in weight $0$. Note that $\cot(A)$ is concentrated in weights $\geq 1$, i.e.\ $\gr^0\big(\cot(A)\big)=0$.

Likewise, for every derived arrow $(A \to B)$ we define the \textit{cotangent fibre} to be the desuspended relative cotangent complex
$$
\cot(A \to B)=L_{B/A}[-1].
$$
One can also identify this with the relative cotangent complex of the map of arrows $(A \to B)\rt ( B \to B)$, which is given by the $(B \to B)$-module $\big(\cot(A \to B), 0\big)$.
\end{definition}
Note that the cotangent fibre is connective for each animated graded/filtered ring and each animated surjection.
\begin{remark}\label{rem:cotangent adic}
For any filtered derived ring $A$ and any derived arrow $(B' \to B)$, there are natural equivalences
\begin{align*}
\gr(\cot(A))&\simeq \cot(\gr(A))\\
\aug(\cot(A))&\simeq \cot(\aug(A))\\
F^{\star\leq 1}\cot(B' \to B)&\simeq \cot(\adic(B' \to B))
\end{align*}
where $F^{\star\leq 1}\cot(B' \to B)$ denotes the $B$-module $\cot(B' \to B)$ put in filtration weight $1$. The first two equivalences follow from the fact that the functors $\gr$ and $\aug$ preserve cotangent complexes (since they commute with taking derived symmetric algebras) and the last equivalence follows by adjunction from the fact that $\aug$ preserves square zero extensions.
\end{remark} 

\begin{lemma}[{\cite[Remark 4.25]{BM19}}]\label{lem:graded cot}
Let $f\colon A\rt B$ be a map between derived graded rings. Then the following are equivalent:
\begin{enumerate}
\item $f$ is an equivalence.
\item $A^{0} \rt B^{0}$ is an equivalence and $\cot(A)\rt \cot(B)$ is an equivalence of graded modules.
\end{enumerate}
Similarly, a map between complete animated filtered rings $f\colon A\rt B$ is an equivalence if and only if $\gr^0(A)\rt \gr^0(B)$ is an equivalence and $\cot(A)\rt \cot(B)$ induces an equivalence on the associated graded.
\end{lemma}
\begin{proof} 
The complete filtered case follows from the graded case since equivalences between complete derived animated rings are detected on the associated graded and because the associated graded commutes with taking cotangent complexes. For the graded case, suppose that $f$ satisfies condition (2) and that there exists a lowest weight $n\geq 1$ such that $\cofib(A^n\to B^n)$ is not zero. Let $A^{\geq n}$ and $B^{\geq n}$ denote the weight $\geq n$ parts of $A$ and $B$, both of which are graded $A$-modules, and consider the natural map
$$
B'=A\otimes_{\LSym_A(A^{\geq n})} \LSym_A(B^{\geq n})\rt B
$$
induced by $f$ and the inclusion $B^{\geq n}\to B$. Using that $\LSym^{\geq 2}_A(B^{\geq n})$ and $\LSym^{\geq 2}_A(B^{\geq n})$ are concentrated in weights $\geq 2n$ (hence in weights $\geq n+1$), one sees that $B'\to B$ is an equivalence in weights $\leq n$. It follows that $L_{B/B'}$ is zero in weights $\leq n$; this can be seen for instance by using the bar resolution to resolve $B$ by free $B'$-algebras, each of which agrees with $B'$ in weight $\leq n$. Consequently, in weight $n$ we find that $L_{B/A}^n\simeq L_{B'/A}^n\simeq \cofib(A^n\to B^n)$ is non-zero. By part (1) of  \Cref{lem:graded nakayama}, this implies that $B^0\otimes_B L_{B/A}\simeq \cofib(\cot(A)\to \cot(B))$ is non-zero, contradicting the assumption that $\cot(A)\simeq \cot(B)$.
\end{proof}
 
\begin{lemma}[{\cite[Proposition 4.26]{BM19}}]\label{lem:graded adic}
For any derived arrow $(A \to B)$, there is a natural equivalence of derived graded rings
$$
\LSym_{B}(\cot(A\to B)(1))\xrightarrow{\sim} \gr(\adic(A\to B))
$$
where $\cot(A\to B)(1)$ denotes the cotangent fibre concentrated in weight $1$.
\end{lemma}
\begin{proof}
For any graded algebra $B^{\bullet}$, the map $\LSym_{B^{0}}(B^{1})\rt B^{\bullet}$ induces an equivalence in degrees $\leq 1$ and hence induces a natural equivalence of $B^{0}$-modules $B^{1}\rt \cot(B)^{1}$. Since 
$$
\cot\big(\gr(\adic(A\to B))\big)\simeq \cot(A\to B)(1)
$$
is concentrated in weight $1$, we obtain an equivalence $\cot(A\to B)\simeq \gr(\adic(A \to B))^{1}$ of $B$-modules. This induces a map of derived graded rings as in the lemma, which is an equivalence by  \Cref{lem:graded cot}.
\end{proof}

\subsection{Almost finite type conditions}
We will use the cotangent fibre to study finiteness properties of animated (graded or filtered) rings. To this end, let us recall from  \Cref{def:almost finite} that a map of animated (graded or filtered) rings $A\to B$ is \emph{almost finitely presented} if $B$ is an almost compact object in the $\infty$-category of animated  (graded or filtered) $A$-algebras.
 
\begin{lemma}[{\cite[Proposition 5.15]{BM19}}]\label{lem:graded aft}
\mbox{The following are equivalent for an animated graded ring $A$:}
\begin{enumerate}
\item $\cot(A)$ is an almost perfect graded $A^{0} $-module.
\item the map of graded algebras $A\rt A^{0}$ is almost finitely presented.
\item the map of graded algebras $A^{0} \rt A$ is almost finitely presented.
\end{enumerate}
\end{lemma}
\begin{proof} 
It is clear that (3) $\Rightarrow$ (2) $\Rightarrow$ (1), so assume that $\cot(A)$ is almost perfect. We can then define a sequence of animated graded algebras $A^0=A_{(0)}\to A_{(1)}\to \dots$ with colimit $A$, where
$$
A_{(n)}=A_{(n-1)}\otimes_{\LSym_{A^0}(A_{(n-1)}^{n})} \LSym_{A^0}(A^{n})
$$
and $A_{(0)}=A^0$. The natural map $A_{(n)}\to A$ is an equivalence in weight $\leq n$, so that the colimit is indeed $A$. Furthermore, an inductive argument shows that $\cot(A_{(n)})$ is precisely the weight $\leq n$ part of $\cot(A)$. Using this, there are equivalences of $A^0$-modules concentrated in weight $n$
$$
(A/A_{(n-1)})^n\simeq L_{A/A_{(n-1)}}^n\simeq (A^0\otimes_A L_{A/A_{(n-1)}})^n\simeq \cot(A)^n.
$$
Here the first two equivalences use that $A_{(n-1)}\to A$ is an equivalence in weight $\leq n-1$. The last equivalence uses the cofibre sequence $\cot(A_{(n-1)})\to \cot(A)\to A^0\otimes_A L_{A/A_{(n-1)}}$ and the fact that $\cot(A_{(n-1)})$ coincides with the weight $\leq n-1$ part of $\cot(A)$. 
By assumption, each $\cot(A)^n$ is an almost perfect $A^0$-module that becomes increasingly connective as $n$ tends to $\infty$. It follows that each $A_{(n-1)}\to A_{(n)}$ is almost finitely presented, with cofibre becoming increasingly connective as $n$ tends to $\infty$. This implies that the colimit $A$ is almost finitely presented as well.
\end{proof}

\begin{proposition}\label{prop:filtered aft}
Let $A\rt B$ be a map of animated filtered rings inducing a surjection on $\pi_0$. Then the following are equivalent:
\begin{enumerate}
\item $B$ is an almost finitely presented filtered $A$-algebra.
\item $L_{B/A}$ is an almost perfect filtered $B$-module and the map of filtered algebras $\pi_0(A)\rt \pi_0(B)$ is of finite presentation.
\item $B$ is almost perfect as a filtered $A$-module.
\end{enumerate}
The same assertions apply to animated rings and animated graded rings.
\end{proposition}
\begin{proof}
This is a special case of  \Cref{prop:aft conditions}.
\end{proof}
 
\begin{definition}\label{def:complete aft}
We will say that:
\begin{enumerate}
\item an animated graded algebra $A$ is \textit{almost finitely augmented} if it satisfies the equivalent conditions of  \Cref{lem:graded aft}.
\item an animated filtered algebra $A$ is \textit{complete almost finitely augmented} if it is complete and the (filtered) map $A\rightarrow \gr^{0}(A)$ satisfies the equivalent conditions of  \Cref{prop:filtered aft}.
\item an animated surjection $(A\to B)$ is \textit{complete almost finitely augmented} if $\adic(A\to B)$ is complete almost finitely augmented.
\end{enumerate}
\end{definition}

\begin{lemma}\label{lem:aperf over complete}
Let $A$ be a complete filtered animated ring and let $M$ be a connective filtered $A$-module. Then $M$ is an almost perfect $A$-module if and only if it is complete and $\gr(M)$ is an almost perfect graded $\gr(A)$-module.
\end{lemma}
\begin{proof}
This is a special case of  \Cref{lem:aperf filtered}.
\end{proof}
We now verify that  this agrees with the notions introduced in \cite{BM19} (in the pointed setting):
\begin{lemma}[{\cite[Proposition 5.28]{BM19}}]\label{lem:complete aft}
Let $A$ be a complete animated filtered ring. Then $A$ is complete almost finitely augmented if and only if $\gr(A)$ is an almost finitely augmented graded ring.
\end{lemma}
\begin{proof}
By  \Cref{prop:filtered aft}, $A$ is almost finitely augmented if and only $\gr^0(A)$ is almost perfect as a filtered $A$-module. By  \Cref{lem:aperf over complete}, this equivalent to $\gr^0(A)$ being an almost perfect graded $\gr(A)$-module, which in turn is equivalent to $\gr(A)$ being almost finitely augmented by  \Cref{prop:filtered aft}.
 
\end{proof}
 
\begin{proposition}\label{prop:complete aft} 
Let $(A \to B)$ be an animated surjection and let $I=\ker(\pi_0(A)\rightarrow \pi_0(B))$. Then the following are equivalent:
\begin{enumerate}
\item The map $(A \to B)$ is complete almost finitely augmented.
\item The map $(A\to B)$ is almost finitely presented and $A$ is $I$-complete.
\item $B$ is almost perfect as an $A$-module and $A$ is $I$-complete.
\end{enumerate}
\end{proposition} 
\begin{proof}
Write $A_{\mm{ad}}=\adic(A\to B)$. Assertions (2) and (3) are equivalent by  \Cref{prop:filtered aft}. For (1) $\Rightarrow$ (2), suppose that the adic filtration $A_{\mm{ad}}$ is complete almost finitely augmented and note that $A\to B$ is equivalent to $F^0A_{\mm{ad}}\to \gr^0(A_{\mm{ad}})$. To see that this map is almost finitely presented, note that $F^0\colon \FilDAlg_{\KK}\rt \DAlg_{\KK}$ preserves almost finitely presented maps. Indeed, its right adjoint, sending a derived algebra $A'$ to the filtered algebra $\dots\to A'\to A'\to A'$, preserves filtered colimits and $n$-coconnective objects.

 Since $A_{\mm{ad}}\to \gr^0(A_\mm{ad})$ is almost finitely presented, we conclude that $F^0A_\mm{ad}\to \gr^0(A_\mm{ad})$ is almost finitely presented as well. Finally, $A$ is $I$-complete by  \Cref{lem:completeness}.

Let us now assume that condition (2) holds and show that $A_{\mm{ad}}$ is complete almost finitely augmented. The functor $\adic$ preserves almost finitely presented objects, since its right adjoint preserves filtered colimits and sends $n$-coconnective objects to $(n+1)$-coconnective objects. Since the map of animated surjections $(A\to B)\rt (B\to B)$ is almost finitely presented, we conclude that $A_{\mm{ad}}\to \gr^0(A_\mm{ad})$ is almost finitely presented. It follows that $A^\wedge_{\mm{ad}}$ is complete, connective (by the Milnor sequence) and $\gr(A^\wedge_\mm{ad})\simeq \gr(A_\mm{ad})$ is almost finitely augmented.  \Cref{lem:complete aft} therefore implies that $A^\wedge_\mm{ad}$ is almost finitely augmented.

It therefore remains to show that $A_\mm{ad}\rt A^{\wedge}_{\mm{ad}}$ is an equivalence. By  \Cref{lem:completeness} and the fact that $A_\mm{ad}$ is $I$-complete, it suffices to verify that $$f\colon \gr^0(A)\otimes_A F^0A^\wedge_\mm{ad}\rt \gr^0(A)$$ is an equivalence. This map is an isomorphism on $\pi_0$ because $\pi_0(F^0A_\mm{ad}^\wedge)\cong \pi_0(A)^\wedge_I$ and $\pi_0(\gr^0(A))\cong \pi_0(A)/I$. It then suffices to verify that its cotangent complex vanishes. Unravelling the definitions, the cotangent complex of $f$ is the underlying module of the filtered module $\cofib(\cot(A_\mm{ad})\to \cot(A^\wedge_{\mm{ad}}))$, which is indeed null: its associated graded is null and $\cot(A_\mm{ad})$ and $\cot(A^\wedge_\mm{ad})$ are both complete, by Remark \ref{rem:cotangent adic} and  \Cref{prop:complete aft} (and since  almost perfect $\gr^0(A)$-modules are complete).
\end{proof}

\begin{corollary}[{\cite[Proposition 5.23]{BM19}}]\label{cor:completeness axiom b}
Let $A$ be an animated filtered ring that is complete almost finitely augmented. Then the induced pair $F^0A\to \gr^0 A$ is complete almost finitely presented.
\end{corollary}
\begin{proof}
Since the functor $\aug$ preserves free algebras, it sends almost finitely presented maps of filtered animated rings to almost finitely presented maps of pairs. In particular, $F^0A\to \gr^0 A$ is almost finitely presented.  \Cref{lem:completeness} implies that $A$ is complete with respect to the kernel of $\pi_0(F^0A)\to \pi_0\gr^0(A)$, and we conclude by  \Cref{prop:complete aft}. 
\end{proof}
\begin{lemma}\label{lem:sqz is aft}
Let $B$ be an animated (filtered, graded) ring, let $M$ be an almost perfect connective (filtered, graded) $B$-module and let $f\colon B_\alpha=B\oplus_\alpha M\to B$ be a square zero extension of $B$ by $M$ classified by a map $\alpha\colon L_{B}[-1]\rt M$. Then the following assertions hold:
\begin{enumerate}
\item Let $\cat{K}$ be the thick subcategory of (filtered, graded) $B_\alpha$-modules generated by the $B_\alpha$-modules underlying an almost perfect $B$-module. Then $\cat{K}=\APerf_{B_\alpha}$.
\item $f$ is almost of finite presentation.
\end{enumerate}
\end{lemma}
\begin{proof} 
Assuming (1), we see that $B$ is almost perfect as a $B_\alpha$-module, so that (2) follows from  \Cref{prop:filtered aft}. To see that $\APerf_{B_\alpha}\subseteq \cat{K}$, notice that every almost perfect $B_\alpha$-module $N$ fits into a cofibre sequence where the outer terms arise from almost perfect $B$-modules
$$\begin{tikzcd}
M\otimes_B (B\otimes_{B_\alpha} N)\simeq M\otimes_{B_\alpha} N\arrow[r] & N\arrow[r] & B\otimes_{B_\alpha} N.
\end{tikzcd}$$
To see that $\cat{K}\subseteq \APerf_{B_\alpha}$,   observe that any object $N\in \cat{K}$ is eventually connective and that $\pi_k(N)$ is a finitely generated module over $\pi_0(B_\alpha)$ if $N$ is $k$-connective. Indeed, these properties hold for almost perfect $B$-modules and are stable under (co)fibres and retracts. Consequently, every $k$-connective $N\in \cat{K}$ fits into a cofibre sequence $B_\alpha[k]^{\oplus n}\to N\to N'$ where $N'$ is $(k+1)$-connective. Notice that $B_\alpha$ is contained in $\cat{K}$, since it fits into a cofibre sequence $M\to B_\alpha\to B$ where both $M$ and $B$ are almost perfect $B$-modules. It follows that $N'$ defines a $(k+1)$-connective object in $\cat{K}$. 

Given a module $N\in \cat{K}$ and $n\geq 0$, we can therefore apply this construction iteratively to obtain a cofibre sequence $K\to N\to N/K$ where $K$ is a perfect $B_\alpha$-module and $N/K$ is $n$-connective. This implies that $N$ is an almost perfect $B_\alpha$-module.
\end{proof}
\begin{corollary}\label{cor:sqz is complete aft}
Let $B$ be an animated ring, $M\in \APerf(B)_{\geq 0}$ and $B\oplus_\alpha M$ a square zero extension of $B$ by $M$ classified by a map $\alpha\colon L_{B}[-1]\rt M$. Then the arrow $(B\oplus_\alpha M \to B)$ is complete almost finitely augmented.
\end{corollary}
\begin{proof}
Let $I$ be the kernel of $\pi_0(B\oplus_\alpha M)\rt \pi_0(B)$. This is a finitely generated ideal, with generators arising from the generators of $\pi_0(M)$.
By  \Cref{prop:complete aft}, it suffices to prove that $B\oplus_\alpha M\to B$ is almost of finite presentation and that $B\oplus_\alpha M$ is $I$-complete. The first condition follows from  \Cref{lem:sqz is aft} and the second condition is immediate since $I^2=0$. 
\end{proof}
\begin{corollary}\label{cor:sqz aft}
Let $(A \to B)$ be an animated arrow which is complete almost finitely presented. For any almost perfect connective $A$-module $M$ and any square zero extension of $A$ by $M$, the arrow $(A\oplus_\alpha M \to B)$ is again complete almost finitely presented. 
\end{corollary}
\begin{proof}
We apply  \Cref{prop:complete aft}. As the composite $A\oplus_\alpha M\rt A\rt B$ of two maps that are almost finitely presented, $A\oplus_\alpha M\rt B$ is almost finitely presented. Let $I$ be the kernel of $\pi_0(A\oplus_\alpha M)\rt \pi_0(B)$. This is generated by finitely many generators $x_i$ of $\pi_0(M)$, which square to zero, and finitely many elements $y_i$ that lift the generators of $\ker(\pi_0(A)\to \pi_0(B))$. Then $A\oplus_\alpha M$ is evidently $(x_i)$-complete, and it is $(y_i)$-complete since it is an extension of $A$ by $M$, both of which are $(y_i)$-complete. We conclude that $A\oplus_\alpha M$ is $I$-complete.
\end{proof}

\subsection{Comonadicity of the cotangent fibre}
Throughout this section, we fix a base animated ring $\KK$. For each animated surjection $(A\to B)$ of $\KK$-algebras, the cotangent fibre $\cot(A \to B)=L_{B/A}[-1]$ is a connective $B$-module that comes equipped with a natural map 
$$
L_{B/\KK}[-1]\rt \cot(A\to B).
$$
The goal of this section will be to show that when $(A\to B)$ is complete almost finitely augmented, its structure can be completely encoded in terms of its cotangent fibre.
\begin{notation}
Let $B$ be an animated $\KK$-algebra. We will write
$$\begin{tikzcd}
\SCR^{\wedge \aft}_{\KK/B}\arrow[r, hookrightarrow] & \SCR_{\KK/B}
\end{tikzcd}$$
for the full subcategory of connective derived $\KK$-algebras $A$ endowed with a map $A\to B$ that exhibits $A$ as complete almost finitely augmented (\Cref{def:complete aft}).
\end{notation}
\begin{definition}\label{def:derivation}
Let $B$ be a derived $\KK$-algebra. We define a \emph{derivation} $(M, \alpha)$ on $B$ to be the datum of a $B$-module $M$, together with a map $\alpha\colon L_{B/\KK}[-1]\rt M$. Write
$$
\cAMod{B}{\KK} = (\Mod_B)_{L_{B/\KK}[-1]/}
$$
for the $\infty$-category of derivations $(M, \alpha)$ on $B$. If $B$ is connective,   let $\cAAPerf{B}{\KK}\subseteq \cAMod{B}{\KK}$ denote the full subcategory spanned by the derivations $(M, \alpha)$, where $M$ is an almost perfect connective.
\end{definition}
Our goal will then be the following:
\begin{theorem}\label{thm:comonad}
Let $B$ be an animated $\KK$-algebra. Then the the cotangent fibre defines the left adjoint in a comonadic adjunction
\begin{equation}\label{diag:ft cot-sqz adjunction}\begin{tikzcd}
\cot\colon \SCR^{\wedge \aft}_{\KK/B}\arrow[r, yshift=1ex] & \cAAPerf{B}{\KK}\colon \sqz.\arrow[l, yshift=-1ex]
\end{tikzcd}\end{equation}
\end{theorem}
Let us first construct the desired adjoint pair. In fact, for later purposes we will construct a version of this adjunction that also encodes naturality in the animated $\KK$-algebra $B$.
\begin{construction}\label{cons:modules under cotangent}
Let $\Mod$ be the $\infty$-category of pairs $(B, M)$ consisting of a derived $\KK$-algebra and a $B$-module (cf.\ Construction \ref{cons:cotangent complex}). We then have a commuting diagram as follows:
$$\begin{tikzcd}
\DAlg^{\Delta^{1}}_\KK \arrow[rr, yshift=1ex, "L"]\arrow[rd, "\mm{ev}_1"{swap}] & &\Mod\times_{\DAlg} \DAlg_{\KK} \arrow[ll, yshift=-1ex, "\mm{triv}"]\arrow[ld]\\
& \DAlg_\KK.
\end{tikzcd}$$
Here the two downwards pointing functors send a map of derived rings $A\to B$, resp.\ a derived ring with a module $(B, M)$, to $B$; both functors are cocartesian and cartesian fibrations. The pair $(L, \triv)$ is the relative adjunction given by the cotangent complex and trivial square zero extension functors
$$
L(A\to B)=\big(B, L_{A/\KK}\otimes_A B\big) \qquad\qquad \triv(B, M)=(B\oplus M\to B).
$$
 
We abuse notation and also write $L\colon \DAlg_{\KK} \rt \Mod$ for the functor sending $A\mapsto (A, L_{A/\KK})$. We then define the $\infty$-category $\caMod$ of \emph{derivations} to be the fibre product
$$\begin{tikzcd}
\caMod\arrow[r]\arrow[d] & \Fun([1], \Mod)\times_{\Fun([1], \DAlg)} \DAlg_{\KK}\arrow[d, "\mm{ev}_0"]\\
\DAlg_{\KK} \arrow[r, "{L[-1]}"] & \Mod\times_{\DAlg} \DAlg_{\KK}.
\end{tikzcd}$$
An object of $\caMod$ is given by a tuple $\big(B, M, \alpha\big)$ consisting of a derived ring $B$ and a derivation on $B$, i.e.\ a $B$-module $M$ and a map $\alpha\colon L_{B/\KK}[-1]\rt M$. By adjunction, one can also identify this with a tuple $(B, M, \alpha)$ of a derived ring, a module and a map of derived pairs $\alpha\colon (B, B) \rt \big(B\oplus M[1], B)$ which is the identity on the codomains.
\end{construction}
\begin{remark}\label{rem:pushout along cotangent}
The projection $\caMod\rt \DAlg_\KK;\ (B, M, \alpha)\longmapsto B$ is a cartesian fibration by construction, classifying the functor 
$$
B\longmapsto \cAMod{B}{\KK}.
$$
Each map $f\colon B\rt B'$ induces a functor $f_*\colon \cAMod{B'}{\KK}\rt \cAMod{B}{\KK}$ by restriction of scalars, sending a derivation $(N, \beta)$ to $f_*(N, \beta)=(f_*N, \beta')$ with
$$\begin{tikzcd}
\beta'\colon L_{B/\KK}[-1]\arrow[r] & f_*L_{B'/\KK}[-1]\arrow[r, "f_*\beta"] & f_*N
\end{tikzcd}$$
This admits a left adjoint, sending a derivation $(M, \alpha)$ to the derivation computed as the pushout of the diagram of $B'$-modules
$$\begin{tikzcd}
L_{B'/\KK}[-1] & f^*L_{B/\KK}[-1]\arrow[r, "f^*\alpha"]\arrow[l] & f^*M.
\end{tikzcd}$$
Consequently, the projection $\caMod\rt \DAlg_\KK$ is a cocartesian fibration as well.
\end{remark}
\begin{remark}\label{rem:cocartesian etale case}
Let $f\colon B\rt B'$ be a map of derived rings with $L_{B'/B}\simeq 0$. Then the induced adjoint pair from Remark \ref{rem:pushout along cotangent} has a simpler description: the left adjoint sends $(M, \alpha)$ to the derivation $(f^*M, f^*\alpha)$, with $f^*\alpha\colon L_{B'/\KK}[-1]\simeq f^*L_{B/\KK}[-1]\rt f^*M$. In other words, the adjoint pair is simply given by the functors
$$\begin{tikzcd}
f^*\colon \cAMod{B}{\KK}\arrow[r, yshift=1ex] & \cAMod{B'}{\KK}\arrow[l, yshift=-1ex]\colon f_*
\end{tikzcd}$$
taking base change and direct image at the level of modules.
\end{remark}
\begin{definition}\label{def:sqz-cot adjunction}
We will write
\begin{equation}\label{diag:sqz-cot adjunction}\begin{tikzcd}
\DAlg_{\KK}^{\Delta^{1}}\arrow[rr, "\cot", yshift=1ex]\arrow[rd, "\mm{ev}_1"{swap}] & & \caMod\arrow[ld]\arrow[ll, yshift=-1ex, "\sqz"] \\
& \DAlg_\KK.
\end{tikzcd}\end{equation}
for the relative adjunction defined as follows:
\begin{enumerate}
\item The right adjoint $\sqz$ sends a derivation $\big(B, M, \alpha\big)$ to the square zero extension $B\oplus_\alpha M\to B$ classified by $\alpha$. More precisely, this is given by the fibre product in the $\infty$-category of maps of derived $\KK$-algebras
$$\begin{tikzcd}
{[B\oplus_\alpha M\to B]}\arrow[d]\arrow[r] & {[B\to B]}\arrow[d, "0"]\\
{[B\to B]}\arrow[r, "\alpha"] & {[B\oplus M[1]\to B]}.
\end{tikzcd}$$
\item The left adjoint $\cot$ is given by
$$\begin{tikzcd}
{[A\to B]}\arrow[r, mapsto] & \big(B, \  L_{B/\KK}[-1]\to L_{B/A}[-1]\big).
\end{tikzcd}$$
\end{enumerate}
\end{definition}

\begin{proof}[Proof of  \Cref{thm:comonad}]
The proof is essentially the same as in \cite{BM19}, although one cannot apply the axiomatic argument from loc.\ cit.\ (nor the modification in Section 6) stricto sensu. To start, notice that Corollary \ref{cor:sqz is complete aft} implies that for each animated $\KK$-algebra $B$, the adjoint pair \eqref{diag:sqz-cot adjunction} restricts to an adjoint pair $\cot\colon \SCR_{\KK/B}^{\wedge \aft}\leftrightarrows \cAAPerf{B}{\KK}\colon \sqz$. We have to verify that this adjoint pair is comonadic. To see that $\cot$ is conservative, let $(A\to B)\rt (A'\to B)$ be a map in $\SCR_{\KK/B}^{\wedge \aft}$ that induces an equivalence on cotangent fibres. The induced map $\adic(A\to B)\rt \adic(A'\to B)$ then induces an equivalence on the associated graded by  \Cref{lem:graded adic}. Since the adic filtrations are complete by definition, this implies that $A\rt A'$ is an equivalence.

Next, let $A^\bullet$ be a cosimplicial diagram in $\SCR^{\wedge \aft}_{\KK/B}$ such that $\cot(A^\bullet)$ admits a splitting in $\cAAPerf{B}{\KK}$. Let $\Tot(A^\bullet)$ be the totalisation of $A^\bullet$, computed in the $\infty$-category $\DAlg_{\KK/B}$ of (not necessarily connective) derived $\KK$-algebras augmented over $B$. We need to show that $\Tot(A^\bullet)$ is complete almost finitely presented and that the natural map $\cot(\Tot(A^\bullet))\rt \Tot(\cot(A^\bullet))$ is an equivalence.

To see this, let $A^\bullet_{\mm{ad}}=\adic(A^\bullet\to B)$ denote the associated diagram of complete almost finitely presented derived filtered algebras. Recall that there are functors
$$\begin{tikzcd}
\grDAlg_\KK & \FilDAlg_\KK \arrow[r, "\aug"]\arrow[l, "\gr"{swap}] & \SurDAlg_\KK
\end{tikzcd}$$
preserving both limits (computed in unbounded $\KK$-modules) and colimits, and commuting with the functors taking cotangent fibres. In particular, $\aug(\Tot(A^\bullet_\mm{ad}))\simeq \Tot(\aug(A^\bullet_\mm{ad}))\simeq \Tot(A^\bullet)$ and there is a commuting square
$$\begin{tikzcd}
\aug\big(\cot(\Tot(A^\bullet_\mm{ad}))\big)\arrow[d, "\sim"{swap}]\arrow[r] & \aug\big(\Tot(\cot(A^\bullet_\mm{ad}))\big) \arrow[d, "\sim"]\\
\cot(\Tot(A^\bullet))\arrow[r] & \Tot(\cot(A^\bullet)).
\end{tikzcd}$$
It therefore suffices to verify that:
\begin{enumerate}
\item[(a)] the derived filtered ring $\Tot(A^\bullet_\mm{ad})$ is complete almost finitely augmented (hence connective).
\item[(b)] the filtered cotangent fibre preserves this totalisation (in unbounded filtered modules).
\end{enumerate}
For (a), note that $\Tot(A^\bullet_\mm{ad})$ is complete as   a limit of complete algebras. Its associated graded is 
$$
\gr(\Tot(A^\bullet_\mm{ad}))\simeq \Tot\big(\gr(A^\bullet_\mm{ad})\big)\simeq \Tot\big[\LSym_{B}\big(\cot(A^\bullet \to B)\big)\big]\simeq \LSym_{B}\big(\Tot[\cot(A^\bullet\to B)]\big).
$$ 
Here the second equivalence is  \Cref{lem:graded adic} and the last equivalence uses that the augmented cosimplicial diagram $\Tot[\cot(A^\bullet \to B)]\rt \cot(A^\bullet \to B)$ admits a splitting, so that $\LSym_B$ (or any other functor) preserves this totalisation. Since $\Tot[\cot(A^\bullet \to B)]$ is almost perfect by assumption, we conclude that $\gr\big(\Tot(A^\bullet_\mm{ad})\big)$ is an almost finitely augmented graded algebra (in particular connective). It follows that the complete derived filtered algebra $\Tot(A^\bullet_\mm{ad})$ is connective as well, and  \Cref{lem:complete aft} implies that it is complete almost finitely augmented.

For (b), note that there is a commuting square
$$\begin{tikzcd}
\gr\big(\cot(\Tot(A^\bullet_\mm{ad}))\big)\arrow[d, "\sim"{swap}]\arrow[r] & \gr\big(\Tot(\cot(A^\bullet_\mm{ad}))\big) \arrow[d, "\sim"]\\
\cot\big(\Tot[\gr(A^\bullet_\mm{ad})]\big)\arrow[r] & \Tot\big[\cot(\gr(A^\bullet_\mm{ad}))\big]
\end{tikzcd}$$
since taking the associated graded commutes with totalisations (in the unbounded setting) and with cotangent complexes. Since the cosimplicial diagram $\gr(A^\bullet_\mm{ad})\simeq \LSym_{B}\big(\cot(A^\bullet \to B)\big)$ has a \textit{split} totalisation (preserved by any functor), the bottom map is an equivalence. Consequently, the map 
\begin{equation}\label{eq:filtered tot-cot}
\cot(\Tot(A^\bullet_\mm{ad}))\rt \Tot(\cot(A^\bullet_\mm{ad}))
\end{equation}
induces an equivalence on the associated graded. The domain is an almost perfect connective filtered module: this follows from  \Cref{prop:filtered aft}, since $\Tot(A^\bullet_\mm{ad})$ is complete almost finitely presented by part (a). In particular, $\cot(\Tot(A^\bullet_\mm{ad}))$ is a complete filtered module. The codomain of the map \eqref{eq:filtered tot-cot} is a limit of complete (in fact, almost perfect connective) filtered $B$-modules and is therefore complete itself. We conclude that \eqref{eq:filtered tot-cot} is an equivalence, as desired.
\end{proof}

In particular,  \Cref{thm:comonad} shows that for every animated surjection $(A\to B)$ which is complete almost finitely augmented, there exists a cosimplicial diagram of square zero extensions of $B$ whose totalisation is equivalent to $A$: this is simply the comonadic cobar resolution of $A$. For later purposes, we record the following variant of this construction:
\begin{corollary}\label{cor:comonadic cobar}
Let $A\to B\in \SCR_{\KK/B}^{\wedge, \aft}$ and let $M\in \APerf_{B, \geq 0}$. For any map $\alpha\colon A\rt B\oplus M$ of animated $\KK$-algebras augmented over $B$, there exists a natural coaugmented cosimplicial diagram 
$$\begin{tikzcd}
A\arrow[r] & A^\bullet\arrow[r] & B\oplus M
\end{tikzcd}$$
in $(\SCR_{\KK/B}^{\wedge, \aft})_{/B\oplus M}$ with the following properties:
\begin{enumerate}
\item For each $i$, the map $A^i\to B\oplus M$ is equivalent to $\sqz(M^i, \alpha^i)\to \sqz(M, 0)$ for a map $(M^i, \alpha^i)\to (M, 0)$ in $\cAAPerf{B}{\KK}$.
\item The augmented cosimplicial diagram $\cot(A)\to \cot(A^\bullet)$ in $\cAAPerf{B}{\KK}$ is split.
\item $A\simeq \Tot(A^\bullet)$.
\end{enumerate}
\end{corollary}
\begin{proof}
Suppose that $U\colon \cat{D}\leftrightarrows \cat{C}\colon R$ is a comonadic adjunction. For every object $c\in \cat{C}$, this induces an adjoint pair $U\colon \cat{D}_{/R(c)} \leftrightarrows \cat{C}_{/c}\colon C$ on over-categories, which is comonadic as well. It follows that every object $\alpha\colon d\to R(c)$ in $\cat{D}_{/R(c)}$ admits a comonadic cobar resolution of the form 
$$\begin{tikzcd}
d\arrow[r] & (RU)^\bullet(d) \arrow[r] & R(c).
\end{tikzcd}$$
Here each $\alpha^n\colon (RU)^n(d)\to R(c)$ is the image under $R$ of the map $U(RU)^{n-1}(d)\to UR(c)\to c$, where the first map is the image of $\alpha^{n-1}$ under the forgetful functor $U$. This has the property that the augmented cosimplicial object $U(d)\to U(RU)^\bullet(d)$ in $\cat{C}_{/c}$ is split and (hence) that $d\simeq \Tot((RU)^\bullet(d))$. We now apply this to the situation where $(U, R)=(\cot, \sqz)$ is the comonadic adjunction \eqref{diag:ft cot-sqz adjunction}, $c$ is the derivation $(M, 0)\in \cAAPerf{B}{\KK}$ and $d\rt R(c)$ is the map $\alpha\colon A\rt B\oplus M$.
\end{proof}

\newpage

\section{Partition Lie algebroids}\label{PartitionLieAlgebroids}
Recall that a \emph{Lie algebroid} over a smooth scheme $Z$ over a field $k$ is a quasi-coherent sheaf $\mathfrak{g}\in \QC^{\heartsuit}_Z$ equipped with an $\mathcal{O}_Z$-linear anchor map $\rho\colon \mf{g}\to T_{Z/\KK}$ to the tangent sheaf and a $k$-linear Lie bracket $[-, -]\colon \mf{g}\otimes_k \mf{g}\to \mf{g}$ satisfying the Leibniz rule. The goal of this section will be to introduce a derived (and shifted) variant of such Lie algebroids over a qcqs locally coherent derived $\KK$-scheme $X$, which we will refer to as \emph{partition Lie algebroids}. More precisely, we will define a partition Lie algebroid to be given by a map $\rho\colon \mathfrak{g}\rt T_{X/\KK}[1]$ in the $\infty$-category
$$
\QC^\vee_X=\text{Ind}(\Coh_X^{\op})
$$
of pro-coherent sheaves on $X$, together with the structure of an algebra over a certain monad
$$\begin{tikzcd}
\Lie^\pi_{\Delta, X/\KK}\colon \big(\QC^\vee_X\big)_{/T_{X/\KK}[1]}\arrow[r] & \big(\QC^\vee_X\big)_{/T_{X/\KK}[1]}.
\end{tikzcd}$$
Here $T_{X/\KK}=L_{X/\KK}^\vee$ denotes the pro-coherent dual of the cotangent sheaf of $X$ (Example \ref{ex:dual cotangent}). One can therefore think of $\rho\colon \mf{g}\to T_{X/\KK}[1]$ as the \emph{anchor map} underlying the partition Lie algebroid.

Our construction of the monad $\Lie^\pi_{\Delta, X/\KK}$ will be based on two ingredients:
\begin{enumerate}\setlength{\itemsep}{5pt}
\item Let $\cAAPerf{X}{\KK}$ denote the $\infty$-category of derivations $\alpha\colon L_{X/\KK}[-1]\to M$ with values in a connective, almost perfect quasi-coherent sheaf on $X$. We then consider the comonad
$$\begin{tikzcd}
\coLie^\pi_{\Delta, X/\KK}\colon \cAAPerf{X}{\KK}\arrow[r] & \cAAPerf{X}{\KK}; \quad (M, \alpha)\arrow[r] & \cot\big(\mc{O}_X\oplus_\alpha M\big)
\end{tikzcd}$$ 
sending each derivation $\alpha\colon L_{X/\KK}[-1]\to M$ to the cotangent fibre of the square zero extension $\mc{O}_X\oplus_\alpha M$.

\item Taking pro-coherent duals determines a fully faithful functor (\Cref{lem:slice duality})
$$\begin{tikzcd}
(-)^\vee\colon \cat{Der}_{X/\KK, \geq 0}^{\mm{aperf}, \op} \arrow[r, hookrightarrow]  & \big(\QC^\vee_X\big)_{/T_{X/\KK}[1]}.
\end{tikzcd}$$
Its essential image consists of maps $F\to T_{X/\KK}[1]$ where $F$ is dually almost perfect (\Cref{def:dually almost perfect}) and of non-positive tor-amplitude (\Cref{def:tor-ampl}).
\end{enumerate}
Our main aim will then be to prove the following:
\begin{theorem}\label{thm:pla scheme}
Let $X$ be a locally coherent qcqs derived $\KK$-scheme. There exists a unique extension
$$\begin{tikzcd}[column sep=4pc]
\cat{Der}_{X/\KK, \geq 0}^{\mm{aperf}, \op}\arrow[r, "(\coLie^\pi_{\Delta, X/\KK})^{\op}"]\arrow[d, hookrightarrow, "(-)^\vee"{swap}] & \cat{Der}_{X/\KK, \geq 0}^{\mm{aperf}, \op}\arrow[d, hookrightarrow, "(-)^\vee"]\\ 
\big(\QC^\vee_X\big)_{/T_{X/\KK}[1]}\arrow[r, "\Lie^\pi_{\Delta, X/\KK}", dotted] & \big(\QC^\vee_X\big)_{/T_{X/\KK}[1]}
\end{tikzcd}$$
to a monad $\Lie^\pi_{\Delta, X/\KK}$ that preserves sifted colimits.
\end{theorem}
\begin{definition}\label{def:pla scheme}
We will refer to the monad $\Lie^\pi_{\Delta, X/\KK}$ of  \Cref{thm:pla scheme} as the \emph{partition Lie algebroid monad} and define the $\infty$-category of \emph{partition Lie algebroids}
$$
\LieAlgd_{X/\KK} = \Alg_{\Lie^\pi_{\Delta, X/\KK}}\Big(\big(\QC^\vee_X\big)_{/T_{X/\KK}[1]}\Big)
$$
to be its $\infty$-category of algebras.
\end{definition}
We will prove the affine case of  \Cref{thm:pla scheme} in Section \ref{sec:pla affine}. In Section \ref{sec:pla functorial} we will study the naturality of the $\infty$-category $\LieAlgd_{B/\KK}$ in the coherent animated $\KK$-algebra $B$. More precisely, we will prove the following in  \Cref{prop:lie algebroid naturality}:
\begin{enumerate}
\item[(a)] For each map $f\colon B\rt B'$ of that is almost finitely presented, there exists a functor
$$\begin{tikzcd}
f^\sharp\colon \LieAlgd_{B/\KK}\arrow[r] & \LieAlgd_{B'/\KK}.
\end{tikzcd}$$

\item [(b)] When $f$ is étale, the functor $f^\sharp=f^*$ is given by the usual inverse image at the level of the underlying pro-coherent module, and admits a right adjoint.
\end{enumerate}
We will then use this in Section \ref{sec:pla schemes} to prove  \Cref{thm:pla scheme} by a descent argument.

\subsection{Partition Lie algebroids on affine schemes}\label{sec:pla affine}
Let $B$ be an animated $\KK$-algebra and recall the adjoint pair from  \Cref{def:sqz-cot adjunction}
$$\begin{tikzcd}
\cot\colon \DAlg_{\KK/B}\arrow[r, yshift=1ex] & \cAMod{B}{\KK}\cocolon \sqz.\arrow[l, yshift=-1ex]
\end{tikzcd}$$
Here $\cAMod{B}{\KK}$ is the $\infty$-category of derivations $(M, \alpha)=(M, \alpha\colon L_{B/\KK}[-1]\to M)$ on $B$ (\Cref{def:derivation}). 
\begin{notation}\label{not:coLie algebroid comonad}
Let us write $\coLie^{\pi}_{\Delta, B/\KK}:=\cot\circ \sqz$ for the composite comonad
$$\begin{tikzcd}
\coLie^{\pi}_{\Delta, B/\KK}=\cot\circ \sqz\colon \cAMod{B}{\KK}\arrow[r] & \cAMod{B}{\KK}.
\end{tikzcd}$$
\end{notation}
\begin{remark}\label{rem:colie filtration}
Given a derivation $(M, \alpha)$, one can endow $\coLie^\pi_{\Delta, B/\KK}(M, \alpha)$ with a natural decreasing filtration, as follows. Consider $L_{B/\KK}[-1]$ as a filtered module in weight $0$ (i.e.\ $F^1=0$) and consider $M$ as a filtered module in weight $1$, i.e., $F^2M=0$ and $F^1M=M$. The derivation $\alpha$ then lifts to a derivation $\alpha\colon L_{B/\KK}[-1]\to F^\star M$ and we obtain a filtration
$$
F^\star\coLie^\pi_{\Delta, B/\KK}(M, \alpha) = \cot(B\oplus_\alpha F^\star M).
$$
Note that $L_{B/\KK}[-1]\to F^\star M$ is null on the associated graded, so that the associated graded of this filtration is equivalent to $\coLie^\pi_{\Delta, B/\KK}(M, 0)$.
\end{remark}

We will use the following notation throughout this paper: 
\begin{definition}\label{def:almost eventually constant}
	Let $\cat{C}$ be an $\infty$-category such that the inclusion $\tau_{\leq n}\cat{C}\subseteq \cat{C}$ of its $n$-truncated objects admits a left adjoint $\tau_{\leq n}$ for each $n\in \mathbb{N}$. We will say that a tower $$\dots\to X_1\to X$$ is \emph{almost eventually constant} if for each $n\in \mathbb{N}$, there exists an $m$ such that $$\tau_{\leq n}X_{i+1}\to \tau_{\leq n}X_i$$ is an equivalence for all $i\geq m$.
\end{definition}

\begin{lemma}\label{lem:coLie properties}
The comonad $\coLie^{\pi}_{\Delta, B/\KK}\colon \cAMod{B}{\KK}\rt \cAMod{B}{\KK}$ has the following properties:
\begin{enumerate}
\item It preserves derivations $(M, \alpha)$ where $M$ is an almost perfect connective $B$-module.
\item It preserves sifted colimits.
\item For each finite cosimplicial diagram $(M^\bullet, \alpha^\bullet)$ in $\cAMod{B}{\KK}$ such that $\Tot(M^\bullet)$ and all $M^t$ are in $\APerf_{B,\geq 0}$, there is a natural equivalence
$$
\coLie^\pi_{\Delta, B/\KK}\big(\Tot(M^\bullet, \alpha^\bullet)\big)\simeq \Tot\big(\coLie^\pi_{\Delta, B/\KK}(M^\bullet, \alpha^\bullet)\big).
$$
\item For any tower $\dots\to (M_2, \alpha_2)\to (M_1, \alpha_1)$ in $\cAMod{B}{\KK}$ such that 
$ \dots\to M_2\to M_1 $ is eventually constant (i.e.\ 
the connectivity of $\fib(M_{n+1}\to M_n)$ tends to $\infty$), there is an equivalence
$$
\coLie^\pi_{\Delta, B/\KK}(\lim_n (M_n, \alpha_n))\simeq \lim_n\big(\coLie^\pi_{\Delta, B/\KK}(M_n, \alpha_n)\big).
$$
\end{enumerate}
\end{lemma}
\begin{proof}
Property (1) follows from  \Cref{lem:sqz is aft} and (2) follows from the fact that $\cot$ and $\sqz$ both preserve sifted colimits. For (3), we use the filtration from Remark \ref{rem:colie filtration}. Since limits in $\cAMod{B}{\KK}$ are computed on the underlying $B$-modules, it then suffices to show that the natural map
$$
\mu\colon \cot\big(\Tot(B\oplus_{\alpha^\bullet} F^\star M^\bullet)\big)\rt \Tot\big(\cot(B\oplus_{\alpha^\bullet} F^\star M^\bullet)\big)
$$
is an equivalence of filtered modules. \Cref{prop:filtered aft} and  \Cref{lem:sqz is aft} imply that $\cot\big(B\oplus_{\alpha} F^\star M\big)$ is an almost perfect filtered $B$-module whenever $M$ is almost perfect. In particular, the domain and codomain of $\mu$ are complete, and it remains to verify that the induced map on the associated graded is an equivalence. As the maps $\alpha^\bullet$ are null on the associated graded, it will be enough to show that the functor
$$\begin{tikzcd}
F\colon \Mod_B\arrow[r] & \GrMod_B^{\geq 1}; \quad M\arrow[r] & \cot\big(B\oplus M(1)\big)
\end{tikzcd}$$
preserves finite totalisations, where $\GrMod_B^{\geq 1}$ is the $\infty$-category of \textit{strictly} positively graded $B$-modules. Using the bar resolution, the functor $F$ can be obtained via colimits and compositions from the functors $\LSym^n_B\colon \GrMod_B^{\geq 1}\rt \GrMod_B^{\geq 1}$. These are $n$-excisive and hence preserve finite totalisations \cite[Proposition 3.37]{BM19}. Since the functors preserving finite totalisations are stable under composition and colimits, the result follows.

For (4), note that if $\fib(M_{n+1}\to M_n)$ is $i$-connective, then $B\oplus_{\alpha_{n+1}} M_{n+1}\rt B\oplus_{\alpha_n} M_n$ has $i$-connective fibres. Consequently, its relative cotangent complex is $(i+1)$-connective \cite[Corollary 25.3.6.4]{SAG}. Taking cotangent fibres, one finds that $\coLie^\pi_\Delta(M_{n+1})\rt \coLie^\pi_\Delta(M_n)$ has $i$-connective fibres. The tower therefore stabilises after $k$-truncation for any $k$, from which the result follows.
\end{proof}
Let us now suppose that $B$ is a coherent animated $\KK$-algebra and consider the dual picture, using linear duality in the $\infty$-category $\QC^\vee_B=\mm{Ind}(\Coh_B^{\op})$ of pro-coherent $B$-modules. We refer to Appendix \ref{sec:dag} for the theory of pro-coherent sheaves and just recall that there are left adjoint functors (cf.\ Observation \ref{obs:t-structure} and  \Cref{def:pro-coh dual})
$$\begin{tikzcd}
\upiota\colon \Mod_B\arrow[r] & \QC^\vee_B & &  (-)^\vee\colon \Mod_B\arrow[r] & \QC^{\vee, \op}_B.
\end{tikzcd}$$
The first functor is fully faithful on eventually connective $B$-modules and the second is fully faithful on almost perfect $B$-modules. We will write $\dAPerf_B\subseteq \QC^\vee_B$ for the full subcategory of \emph{dually almost perfect} pro-coherent $B$-modules (\Cref{def:dually almost perfect}). Explicitly, one can identify $\QC^\vee_B$ with the $\infty$-category of exact functors $F\colon \Coh_B\rt \Sp$. The pro-coherent module $\upiota(M)$ then sends $K\mapsto K\otimes_B M$ and the pro-coherent dual $M^\vee$ sends $K\mapsto \hom_B(M, K)$.
\begin{notation}
Given a coherent animated $\KK$-algebra $B\in \SCR_{\KK}$, we will write $T_{B/\KK}=L_{B/\KK}^\vee$ for the pro-coherent dual of the cotangent complex and refer to it as the (pro-coherent) \emph{tangent complex}. 
\end{notation}
By  \Cref{lem:slice duality}, pro-coherent duality determines an equivalence
$$\begin{tikzcd}
(-)^\vee\colon \cat{Der}_{B/\KK}^{\mm{aperf}, \op}=\big((\APerf_{B})_{L_{B/\KK}[-1]/}\big)^{\op}\arrow[r, "\sim"] & \big(\dAPerf_{B}\big)_{/T_{B/\KK}[1]}.
\end{tikzcd}$$

Using this, we construct the partition Lie algebroid monad in the affine setting:
\begin{proposition}\label{prop:pla affine}
Let $B$ be a coherent animated $\KK$-algebra. Then there exists a unique extension
$$\begin{tikzcd}[column sep=4pc]
\cat{Der}_{B/\KK, \geq 0}^{\mm{aperf}, \op}\arrow[r, "(\coLie^\pi_{\Delta, B/\KK})^{\op}"]\arrow[d, hookrightarrow, "(-)^\vee"{swap}] & \cat{Der}_{B/\KK, \geq 0}^{\mm{aperf}, \op}\arrow[d, hookrightarrow, "(-)^\vee"]\\ 
\big(\QC^\vee_B\big)_{/T_{B/\KK}[1]}\arrow[r, "\Lie^\pi_{\Delta, B/\KK}", dotted] & \big(\QC^\vee_B\big)_{/T_{B/\KK}[1]}
\end{tikzcd}$$
to a monad $\Lie^\pi_{\Delta, B/\KK}$ that preserves sifted colimits.
\end{proposition}
\begin{definition}\label{def:pla affine}
Following  \Cref{def:pla scheme}, we will write $\LieAlgd_{B/\KK}$ for the $\infty$-category of partition Lie algebroids over $B$, that is, algebras over the monad $\Lie^\pi_{\Delta, B/\KK}$.
\end{definition}
Before turning to the proof of  \Cref{prop:pla affine}, we record some first consequences of the above definition.
\begin{example}\label{ex:partition lie algebra}
If $B=\KK$ is a coherent animated ring, then the monad $\Lie^\pi_{\Delta, B/B}$ coincides with the partition Lie algebra monad from \cite{BM19, BCN21}, so that $\LieAlgd_{B/B}\simeq \LieAlg_{B}$ is equivalent to the $\infty$-category of partition Lie algebras over $B$. In particular, the free partition Lie algebra on a dually almost perfect $B$-module can be computed explicitly in terms of the partition complex, using the derived functor of $\Sigma_r$-invariants (see \cite[Section 3.6]{BCN21})
$$
\Lie^\pi_{\Delta, B/B}(M)\simeq \sum_{r} \big(\widetilde{C}^*(\Sigma|\Pi_r|^{\diamond}, B)\otimes_B M^{\otimes r}\big)^{\Sigma_r}.
$$
\end{example}
\begin{remark}\label{rem:free partition Lie algebroid}
For a coherent animated $\KK$-algebra $B$, the free partition Lie algebroid $\Lie^\pi_{\Delta, B/\KK}(M, \alpha)$ can be viewed as a deformation of the free $B$-linear partition Lie algebra on $M$. More precisely, $\Lie^\pi_{\Delta, B/\KK}(M, \alpha)$ admits a natural exhaustive increasing filtration whose associated graded is equivalent to the free partition Lie algebra $\Lie^\pi_{\Delta, B/\KK}(M, 0)\simeq \Lie^\pi_{\Delta, B/B}(M)$. Indeed, Remark \ref{rem:colie filtration} and  \Cref{lem:aperf filtered} provide such a filtration when $M$ is the pro-coherent dual of a connective almost perfect $B$-module; one then extends by sifted colimits. The resulting filtration is analogous to the filtration on free Lie algebroids constructed in \cite{kapranov2007free}.
\end{remark}
\begin{proposition}\label{prop:fib anchor}
Let $B$ be a coherent animated $\KK$-algebra. Then the following hold:
\begin{enumerate}
\item There is an equivalence $(\LieAlgd_{B/\KK})_{/0}\simeq \LieAlg_B$.

\item Taking the fibre of the anchor map defines a sifted colimit preserving right adjoint functor $\LieAlgd_{B/\KK}\to \LieAlg_{B}$.
\end{enumerate}
\end{proposition}
It follows that one can identify $\LieAlgd_{B/\KK}$ with the $\infty$-category of algebras over a monad on $\LieAlg_{B}$ whose underlying endofunctor sends $\mf{g}$ to $\mf{g}\times T_{B/\KK}$.
\begin{proof}
For (1), note that the forgetful functor $(\LieAlgd_{B/\KK})_{/0}\to \QC^ \vee_B$ exhibits $(\LieAlgd_{B/\KK})_{/0}$ as the $\infty$-category of algebras over a sifted colimit preserving monad on $\QC^\vee_B$. Unravelling the definitions, this monad coincides with $\Lie^\pi_{\Delta, B/B}$, so that the result follows from Example \ref{ex:partition lie algebra}. Part (2) follows from part (1) and the fact that taking pullbacks along the map of partition Lie algebroids $0\to T_{B/\KK}[1]$ preserves limits and sifted colimits.
\end{proof}
Finally, we record the following reformulation of  \Cref{thm:comonad}, using pro-coherent duality:
\begin{corollary}\label{cor:complete KD}
Let $B$ be a coherent animated $\KK$-algebra. Then there is a fully faithful functor
$$\begin{tikzcd}
\mathfrak{D}\colon \SCR^{\wedge \aft}_{\KK/B}\arrow[r, "\cot"] & \cat{Coalg}_{\mm{coLie}^\pi_\Delta}(\cAAPerf{B}{\KK})\arrow[r, "{(-)^\vee}"] & (\LieAlgd_{B/\KK})^{\op}
\end{tikzcd}$$
whose essential image consists of the partition Lie algebroids whose underlying pro-coherent $B$-module is dually almost perfect of non-positive tor-amplitude.
\end{corollary}

Let us now turn to the proof of  \Cref{prop:pla affine}. We will show that the $\infty$-category $(\QC^\vee_B)_{/T_{B/\KK}[1]}$ is generated under sifted colimits by the full subcategory of $M\to T_{B/\KK}[1]$ where $M$ is dually almost perfect of nonpositive tor-amplitude (\Cref{def:tor-ampl}). To make this more precise,  recall the process of adding colimits to $\infty$-categories from \cite[Section 5.3.6]{HTT}:
\begin{definition}\label{def:colimit completion}
Let $\cat{K}$ be a class of small $\infty$-categories. An $\infty$-category $\cat{V}$ is said to be $\cat{K}$-complete if it admits all colimits of diagrams indexed by $\infty$-categories in $\cat{K}$.

Suppose that $\cat{C}$ is a small $\infty$-category equipped with a set of colimiting cocones $\cat{R}=\{K_\alpha^{\rhd}\to \cat{C}\}$ where each $K_\alpha\in \cat{K}$. We will say that a functor $j\colon \cat{C}\rt \cat{V}$ exhibits $\cat{V}$ as the \emph{$\cat{K}$-completion of $\cat{C}$ relative to $\cat{R}$} if it is the initial $\cat{K}$-complete $\infty$-category with a functor from $\cat{C}$ that preserves the colimit diagrams in $\cat{R}$.
\end{definition}
By \cite[Proposition 5.3.6.2]{HTT}, the $\cat{K}$-completion of $\cat{C}$ relative to $\cat{R}$ always exists and the functor $j\colon \cat{C}\to \cat{V}$ is fully faithful.
\begin{proposition}\label{prop:colimit completion slice}
Let $\cat{K}$ be a class of small $\infty$-categories, each of which is contractible. Suppose that $j\colon \cat{C}\hookrightarrow \cat{V}$ exhibits $\cat{V}$ as the $\cat{K}$-completion of $\cat{C}$ relative to a set of colimit diagrams $\cat{R}=\{K_\alpha^{\rhd}\to \cat{C}\}$. For any object $X\in \cat{V}$, the induced fully faithful functor
$$\begin{tikzcd}
j\colon \cat{C}_{/X}\arrow[r, hookrightarrow] & \cat{V}_{/X}
\end{tikzcd}$$
exhibits $\cat{V}_{/X}$ as the $\cat{K}$-completion of $\cat{C}_{/X}$ relative to the set $\cat{R}_{/X}$ of colimit diagrams $K_\alpha^{\rhd}\to \cat{C}_{/X}$ whose image in $\cat{C}$ is contained in $\cat{R}$.
\end{proposition}
\begin{proof}
Let $\cat{P}_{\cat{R}}(\cat{C})\subseteq \cat{P}(\cat{C})$ be the $\infty$-category of presheaves on $\cat{C}$ that send the colimit diagrams in $\cat{R}$ to limits. By the proof of \cite[Proposition 5.3.6.2]{HTT}, we can then identify $\cat{V}\subseteq \cat{P}_{\cat{R}}(\cat{C})$ with the smallest full subcategory that contains the representable presheaves and is closed under $\cat{K}$-indexed colimits. Using this, we can identify $X$ with a presheaf $X\colon \cat{C}^{\op}\rt \sS$ and $\cat{C}_{/X}$ is the full subcategory of representable presheaves over $X$. 

Now recall that the inclusion $\cat{C}_{/X}\hookrightarrow \cat{P}(\cat{C})_{/X}$ induces a natural equivalence $\cat{P}(\cat{C}_{/X})\rt \cat{P}(\cat{C})_{/X}$ \cite[Corollary 5.1.6.12]{HTT} (the proof in loc.\ cit.\ also applies to non-representable $X$). This equivalence restricts to an equivalence $\cat{P}_{\cat{R}}(\cat{C})_{/X}\simeq \cat{P}_{\cat{R}_{/X}}(\cat{C}_{/X})$. Since the forgetful functor $\cat{P}_{\cat{R}_{/X}}(\cat{C}_{/X})\rt \cat{P}_{\cat{R}}(\cat{C})$ detects colimits of $\cat{K}$-indexed diagrams (which are all contractible), it follows that $\cat{V}_{/X}\subseteq \cat{P}_{\cat{R}_{/X}}(\cat{C}_{/X})$ is the smallest full subcategory that contains the representable presheaves and is closed under $\cat{K}$-indexed colimits. The result then follows from \cite[Proposition 5.3.6.2]{HTT}.
\end{proof}

\begin{proof}[Proof of  \Cref{prop:pla affine}]
Let $\dAPerf_{B}^{\weirdleq 0}\subseteq \QC^\vee_B$  denote the full subcategory of dually almost perfect objects of non-positive tor-amplitude (\Cref{def:tor-ampl}). This $\infty$-category is equivalent to $\APerf_{B, \geq 0}^{\op}$ by pro-coherent duality.
It now follows from \cite[Proposition 4.20]{BCN21} that $\dAPerf_{B}^{\weirdleq 0}\rt \QC^\vee_B$ exhibits $\QC^\vee_B$ as the sifted colimit completion of $\dAPerf_{B}^{\weirdleq 0}$ relative to the set $\cat{R}_B$ of colimit diagrams in $\dAPerf_{B}^{\weirdleq 0}$ that are dual to the following diagrams in $\APerf_{B, \geq 0}$:
\begin{enumerate}[label=(\alph*)]
\item Augmented cosimplicial diagrams $M^\bullet\colon\Delta_+\rt \APerf_{B, \geq 0}$ such that $M^{-1}\simeq \Tot(M^\bullet)$ (computed in $\Mod_B$).

\item Limits $M_\infty\to \dots \to M_2\to M_1$ of towers in $\APerf_{B, \geq 0}$ that are almost eventually constant (see \Cref{def:almost eventually constant}).
\end{enumerate}
By  \Cref{prop:colimit completion slice}, $(\QC^\vee_B)_{/T_{B/\KK}[1]}$ is the sifted colimit completion of $(\dAPerf_{B}^{\weirdleq 0})_{/T_{B/\KK}[1]}$ relative to $(\cat{R}_B)_{/T_{B/\KK}[1]}$. This implies that restriction along $j$ determines an equivalence between:
\begin{enumerate}
\item the (monoidal) $\infty$-category of sifted colimit preserving endofunctors of $(\QC^\vee_B)_{/T_{B/\KK}[1]}$ that preserve the full subcategory $(\dAPerf_B^{\weirdleq 0})_{/T_{B/\KK}[1]}$.

\item the (monoidal) $\infty$-category of endofunctors of $(\dAPerf_B^{\weirdleq 0})_{/T_{B/\KK}[1]}$ that preserve all sifted colimits of the form (a) and (b) above.
\end{enumerate}
Using that $(\dAPerf_B^{\weirdleq 0})_{/T_{B/\KK}[1]}\simeq \cat{Der}^{\mm{aperf}, \op}_{B/R, \geq 0}$ by pro-coherent duality (\Cref{lem:slice duality}), the result then follows from  \Cref{lem:coLie properties}.
\end{proof}

\subsection{Base change}\label{sec:pla functorial}
We will now study the functoriality of the $\infty$-categories $\LieAlgd_{B/\KK}$ of partition Lie algebroids with respect to the animated $\KK$-algebra $B$. We will do this by directly construct the cartesian fibration classifying $B\longmapsto \LieAlgd_{B/\KK}$. This is somewhat technical, and we refer to  \Cref{prop:lie algebroid naturality} for the final result.

We start by studying the functoriality of the $\infty$-category $(\QC^\vee_B)_{/T_{B/\KK}[1]}$ in the coherent animated $\KK$-algebra $B$. To do this, we will organise these $\infty$-categories into a cartesian fibration over the opposite of a certain subcategory of animated $\KK$-algebras.
\begin{notation}
Let $\SCR_{\KK}^{\coft}\hookrightarrow \SCR_{\KK}$ denote the subcategory whose objects are coherent animated $\KK$-algebras and whose morphisms are maps $A\rt B$ that are almost finitely presented.
More generally, for any functor $\cat{M}\rt \SCR_{\KK}$ or $\cat{N}\rt \SCR_{\KK}^{\op}$, we denote
$$
\cat{M}^{\coft} = \cat{M}\times_{\SCR_{\KK}} \SCR_{\KK}^{\coft}\hspace{40pt} \cat{N}^{\coft} = \cat{N}\times_{\SCR_{\KK}^{\op}} \SCR_{\KK}^{\coft, \op}.
$$
\end{notation}
Recall the $\infty$-category $\caMod$ of derivations $(A, M, \alpha)=(A, \alpha\colon L_{A/\KK}[-1]\to M)$ from Construction \ref{cons:modules under cotangent}, and let $\cat{Der}_{/\KK, \geq 0}^\mm{aperf}\subseteq \caMod$ denote the full subcategories of derivations $(A, M, \alpha)$ where $M$ is a connective almost perfect $A$-module.
\begin{lemma}\label{lem:derivations cocart}
The projection $\pi\colon \mm{Der}_{/\KK, \geq 0}^{\mm{aperf}, \coft}\to \SCR_{\KK}^{\coft}$ is a cocartesian fibration.
\end{lemma}
\begin{proof}
Let $f\colon A\rt B$ be a map of animated $\KK$-algebras. Recall from Remark \ref{rem:pushout along cotangent} that a map $(A, M, \alpha)\rt (B, N, \beta)$ covering $f$ is a cocartesian arrow in $\caMod$ if the induced square of $B$-modules
$$\begin{tikzcd}
f^*L_{A/\KK}[-1]\arrow[r]\arrow[d, "{f^*(\alpha)}"{swap}] & L_{B/\KK}[-1]\arrow[d, "\beta"]\\
f^*M\arrow[r] & N
\end{tikzcd}$$
is cocartesian. If $f$ is almost finitely presented, then the cofibres of the horizontal maps are almost perfect connective $B$-modules. Consequently, $N$ will be an almost perfect $B$-module if $M$ is an almost perfect connective $A$-module. We conclude that $\mm{Der}_{/\KK, \geq 0}^{\mm{aperf}, \coft}\hookrightarrow \caMod$ is stable under taking cocartesian lifts of maps in $\SCR_{\KK}^{\coft}$, so that $\pi$ is indeed a cocartesian fibration.
\end{proof}
\begin{construction}\label{con:anchor functorial}
Let $\mm{Der}_{/\KK, \geq 0}^{\mm{aperf}, \coft, \op}\to \SCR_{\KK}^{\coft, \op}$ be the opposite of the cocartesian fibration from  \Cref{lem:derivations cocart}. This classifies a functor $\SCR_{\KK}^{\coft}\rt \Cat_\infty$ sending $A\mapsto \cat{Der}^{\mm{aperf}, \op}_{A/\KK, \geq 0}$ and sending each map $f\colon A\to B$ to the functor 
\begin{equation}\label{diag:coanchor functorial}\begin{tikzcd}
\big(L_{A/\KK}[-1]\to M\big)\arrow[r, mapsto] & \Big(L_{B/\KK}[-1] \to L_{B/\KK}[-1]\oplus_{f^*L_{A/\KK}[-1]} f^*M\Big).
\end{tikzcd}\end{equation}
Write $\cat{R}_A$ for the collection of colimit diagrams in $\cat{Der}^{\mm{aperf}, \op}_{A/\KK, \geq 0}$ opposite to pullback squares of derivations along $\pi_0$-surjections, as well as limits of towers of derivations that are almost eventually constant (cf.\ \Cref{def:almost eventually constant}).

 By  \Cref{prop:colimit completion slice} and  \Cref{lem:pro-coh over coherent ring}, the fully faithful functor
\begin{equation}\label{diag:coanchor to anchor}\begin{tikzcd}
(-)^\vee\colon \cat{Der}^{\mm{aperf}, \op}_{A/\KK, \geq 0}\arrow[r] & \big(\QC^\vee_A\big)_{/T_{A/\KK}[1]}
\end{tikzcd}\end{equation}
exhibits $(\QC^\vee_A)_{/T_{A/\KK}[1]}$ as the completion of $\cat{Der}^{\mm{aperf}, \op}_{A/\KK, \geq 0}$ under colimits relative to the set of colimit diagrams $\cat{R}_A$ from the proof of  \Cref{prop:pla affine} (see  \Cref{def:colimit completion}). 

Each base change functor $\cat{Der}^{\mm{aperf}, \op}_{A/\KK, \geq 0}\rt \cat{Der}^{\mm{aperf}, \op}_{B/\KK, \geq 0}$ sends colimit diagrams in $\cat{R}_A$ to colimit diagrams in $\cat{R}_B$. The universal property of $(\QC^\vee_A)_{/T_{A/\KK}[1]}$ therefore implies that the maps \eqref{diag:coanchor to anchor} assemble into a natural transformation between two diagrams of $\infty$-categories. Explicitly, this sends $f\colon A\to B$ to a commuting square
$$\begin{tikzcd}
\cat{Der}^{\mm{aperf}, \op}_{A/\KK, \geq 0}\arrow[r, hookrightarrow, "(-)^\vee"]\arrow[d] & \big(\QC^\vee_A\big)_{/T_{A/\KK}[1]}\arrow[d, "f^\sharp"]\\
\cat{Der}^{\mm{aperf}, \op}_{B/\KK, \geq 0}\arrow[r, hookrightarrow, "(-)^\vee"] & \big(\QC^\vee_A\big)_{/T_{B/\KK}[1]}
\end{tikzcd}$$
where the right vertical functor $f^\sharp$ is the unique colimit preserving extension of the left vertical functor, given by
$$
f^\sharp\big(F\to T_{A/\KK}[1]\big) = f^*F\times_{(f^*L_{A/\KK})^\vee[1]} T_{A/\KK}[1].
$$
Note that this indeed preserves colimits and is dual to \eqref{diag:coanchor functorial}. 
\end{construction}
\begin{definition}\label{def:anchor functorial}
We will write $\QC^\vee_{/T[1]}\rt \SCR_{\KK}^{\coft, \op}$ for the cartesian fibration classified by the functor $A\longmapsto (\QC^\vee_A)_{/T_{A/\KK}[1]}$ from Construction \ref{con:anchor functorial}, and
$$\begin{tikzcd}[row sep=1pc]
\mm{Der}_{/\KK, \geq 0}^{\mm{aperf}, \coft, \op}\arrow[rd]\arrow[rr, hookrightarrow, "(-)^\vee"] & & \QC^\vee_{/T[1]}\arrow[ld]\\
& \SCR_{\KK}^{\coft, \op}
\end{tikzcd}$$
for the unstraightening of the natural transformation from Construction \ref{con:anchor functorial}.
\end{definition}
Our next goal will be to construct a version of the partition Lie algebroid monad in families. To this end, recall from  \Cref{def:sqz-cot adjunction} that taking cotangent fibres and square zero extensions defines a relative adjunction
\begin{equation}\label{diag:cot-sqz adj 2}\begin{tikzcd}
\DAlg_{\KK}^{\Delta^{1}}\arrow[rr, "\cot", yshift=1ex]\arrow[rd, "\mm{ev}_1"{swap}] & & \caMod\arrow[ld]\arrow[ll, yshift=-1ex, "\sqz"] \\
& \DAlg_\KK
\end{tikzcd}\end{equation}
We will write $\coLie^{\pi}_{\Delta}=\cot\circ \sqz$ for the composite comonad relative to the base $\DAlg_\KK$. On the fibre over each animated $\KK$-algebra $B$, this restricts to the comonad $\coLie^{\pi}_{\Delta, B/\KK}$ from Notation \ref{not:coLie algebroid comonad}.
\begin{lemma}\label{lem:cot-sqz etale base change}
The adjoint pair \eqref{diag:cot-sqz adj 2} enjoys the following properties:
\begin{enumerate}
\item Let $f\colon (A'\to A)\rt (B' \to B)$ be a map of animated surjections such that $A\otimes_{A'} B'\simeq B$ and $L_{B/A}\simeq 0$. Then $\cot(A'\to A)\rt \cot(B' \to B)$ is a cocartesian arrow in $\caMod$.

\item Let $(A, M, \alpha)\rt (B, N, \beta)$ be a cocartesian arrow in $\caMod$ such that $A, M, B, N$ are all connective and $L_{B/A}\simeq 0$. Then the induced square of animated rings
$$\begin{tikzcd}
A\oplus_\alpha M\arrow[d] \arrow[r] & B\oplus_\beta N\arrow[d]\\
A\arrow[r] & B
\end{tikzcd}$$
is cocartesian. 
\end{enumerate}
\end{lemma}
\begin{proof}
For property (1), note that the left commutative square of animated $\KK$-algebras gives rise to the right square of $B$-modules:
$$\begin{tikzcd}
A'\arrow[d, "{f'}"{swap}]\arrow[r] & A\arrow[d, "f"]\\
B'\arrow[r] & B
\end{tikzcd}\hspace{40pt} \begin{tikzcd} f^*L_{A/\KK}\arrow[r]\arrow[d] & f^*L_{A/A'} \arrow[d]\\
L_{B/\KK}\arrow[r] & L_{B/B'}
\end{tikzcd}$$
Since left square is a pushout of derived rings, the right vertical map $f^*L_{A/A'}\rt L_{B/B'}$ is an equivalence. The left vertical map is an equivalence as well, since $f\colon A\rt B$ was formally étale. It follows that the right square is a pushout square of $B$-modules, which means precisely that $\cot(A'\to A)\rt \cot(B' \to B)$ is a cocartesian arrow in $\caMod$ by Remark \ref{rem:cocartesian etale case}.

Property (2) follows from deformation theory. Indeed, consider the cartesian square of animated rings
$$\begin{tikzcd}
A\oplus_\alpha M\arrow[r, "q_1"]\arrow[d, "q_2"]\arrow[rd, "q_0"] & A\arrow[d, "0"]\\
A\arrow[r, "\alpha"] & A\oplus M[1].
\end{tikzcd}$$
and note that we need to show that the natural map $q_2^*(B\oplus_{\beta} N)\rt B$ is an equivalence of connective $A$-modules. By \cite[Theorem 16.2.0.2]{SAG}, the above square induces an equivalence
$$\begin{tikzcd}
{(q_1^*, q_2^*)}\colon \Mod_{A\oplus_\alpha M, \geq 0} \arrow[r, "\sim"] & \Mod_{A, \geq 0} \times_{\Mod_{A\oplus M[1], \geq 0}} \Mod_{A, \geq 0}
\end{tikzcd}$$
with inverse sending $(E_1, E_2, 0^*E_1\simeq \alpha^*E_2)$ to the fibre product $q_{1*}(E_1)\times_{q_{0*}0^*(E_1)} q_{2*}(E_2)$. 
Since $L_{B/A}\simeq 0$, the cocartesian morphism $(A, M, \alpha)\to (B, N, \beta)$ induces an equivalence $N\simeq B\otimes_A M$ (Remark \ref{rem:cocartesian etale case}). Using this, one sees that the $A\oplus_\alpha M$-module $B\oplus_\beta N$ corresponds under the above equivalence to the triple $(B, B, 0^*(B)\simeq \alpha^*B)$, where the equivalence is induced by $\beta$. This implies that $q_2^*(B\oplus_\beta N)\rt B$ is an equivalence.
\end{proof}
\begin{lemma}\label{lem:coLie naturality laft}
The comonad $\coLie^\pi_\Delta$ restricts to a fibrewise comonad
$$\begin{tikzcd}
\mm{Der}_{/\KK, \geq 0}^{\mm{aperf}, \coft}\arrow[rr, "\mm{coLie}^\pi_\Delta"]\arrow[rd, "\pi"{swap}] & & \mm{Der}_{/\KK, \geq 0}^{\mm{aperf}, \coft}\arrow[ld, "\pi"]\\
& \SCR_{\KK}^{\coft}
\end{tikzcd}$$
where $\pi$ sends a tuple $(A, M, \alpha)$ to the underlying coherent animated $\KK$-algebra $A$. Furthermore, it preserves $\pi$-cocartesian maps covering étale maps.
\end{lemma}
\begin{proof}
The comonad restricts by  \Cref{lem:coLie properties} and preserves $\pi$-cocartesian morphisms covering étale maps by  \Cref{lem:cot-sqz etale base change}.
\end{proof}
We then have the following families version of  \Cref{prop:pla affine}:
\begin{proposition}\label{prop:pla natural}
There exists a unique extension
$$\begin{tikzcd}[column sep=4.5pc]
\cat{Der}_{/\KK, \geq 0}^{\mm{aperf}, \coft, \op}\arrow[r, "{(\coLie^\pi_{\Delta})^{\op}}"]\arrow[d, hookrightarrow, "{(-)^\vee}"{swap}] & \cat{Der}_{/\KK, \geq 0}^{\mm{aperf}, \coft, \op}\arrow[d, hookrightarrow, "(-)^\vee"]\\ 
\big(\QC^\vee\big)_{/T[1]}\arrow[r, "\Lie^\pi_{\Delta}", dotted] & \big(\QC^\vee\big)_{/T[1]}
\end{tikzcd}$$
to a monad $\Lie^\pi_{\Delta}$ relative to $\SCR_{\KK}^{\coft, \op}$, that is, a monad in $\Cat_{\infty/\SCR_{\KK}^{\coft, \op}}$, that preserves sifted colimits fibrewise.
\end{proposition}
\begin{remark}\label{rem:pla oplax natural}
The relative monad $\Lie^\pi_{\Delta}$ restricts to a monad $(\QC^\vee_B)_{/T_{B/\KK}[1]}\rt (\QC^\vee_B)_{/T_{B/\KK}[1]}$ on the fibre over a coherent animated $\KK$-algebra $B$. This monad preserves sifted colimits and extends the monad $(\coLie^\pi_{\Delta, B/\KK})^{\op}$, so that it is naturally equivalent to the partition Lie algebroid monad $\Lie^{\pi}_{\Delta, B/\KK}$ of  \Cref{prop:pla affine}. Consequently,  \Cref{prop:pla natural} shows that the partition Lie algebroid monads $\Lie^{\pi}_{\Delta, B/\KK}$ depend oplax naturally on the coherent animated $\KK$-algebra $B$.
\end{remark}
As  the proof of  \Cref{prop:pla natural} is somewhat technical, we first describe the consequences of  \Cref{prop:pla natural} at the level of algebras.
\begin{definition}\label{def:pla fibration}
We define the $\infty$-category $\LieAlgd_{/\KK}$ of \textit{partition Lie algebroids} to be the $\infty$-category of algebras for the monad $\Lie^\pi_{\Delta}\colon \QC^\vee_{/T[1]}\rt \QC^\vee_{/T[1]}$. 
\end{definition}
Since $\Lie^\pi_{\Delta}$ covers the identity monad on $\SCR_{\KK}^{\coft, \op}$, we obtain a natural projection
$$\begin{tikzcd}
\LieAlgd_{/\KK}\arrow[r] & \Alg_{\id}\big(\SCR_{\KK}^{\coft, \op}\big)\simeq \SCR_{\KK}^{\coft, \op}.
\end{tikzcd}$$
For each coherent animated $\KK$-algebra $B$, the fibre over $B$ is equivalent to the $\infty$-category of partition Lie algebroids over $B$ from  \Cref{def:pla affine}.
\begin{proposition}\label{prop:lie algebroid naturality}
The following assertions hold:
\begin{enumerate}
\item The forgetful functor defines a map of cartesian fibrations preserving cartesian arrows
\begin{equation}\label{diag:lie algebroids pullback}\begin{tikzcd}
\LieAlgd_{/\KK}\arrow[rr, "\forget"]\arrow[rd] & & \QC^\vee_{/T[1]}\arrow[ld]\\ 
& \SCR_{\KK}^{\coft, \op}.
\end{tikzcd}\end{equation}
In particular, every almost finitely presented map $f\colon A\rt B$ of animated $\KK$-algebras induces a commuting square
$$\begin{tikzcd}
\LieAlgd_{A/\KK}\arrow[r, "f^\sharp"]\arrow[d, "\forget"{swap}] & \LieAlgd_{B/\KK}\arrow[d, "\forget"]\\
\QC^\vee(A)_{/T_{A/\KK}[1]}\arrow[r, "f^\sharp"] & (\QC^\vee_B)_{/T_{B/\KK}[1]}
\end{tikzcd}$$
where the bottom functor is given by $f^\sharp(\mf{g}) = f^*(\mf{g})\times_{(f^*L_{A/\KK})^\vee[1]} T_{B/\KK}[1]$.

\item If $f\colon A\rt B$ is étale, then the square \eqref{diag:lie algebroids pullback} is left adjointable, i.e.\ $f^\sharp$ intertwines the free partition Lie algebroid functors. Consequently, $f^\sharp\colon \LieAlgd_{A/\KK}\rt \LieAlgd_{B/\KK}$ preserves colimits and is given by $\mf{g}\longmapsto f^*\mf{g}$ on the underlying pro-coherent modules.

\item The cartesian fibration $\LieAlgd_{/\KK}\to \SCR^{\coft, \op}$ is classified by a functor $A\longmapsto \LieAlgd_{A/\KK}$ satisfying étale descent.
\end{enumerate}
\end{proposition}
\begin{remark}
Note that for a general finitely presented map $f\colon A\rt B$, the functor $f^\sharp$ preserves sifted colimits, but it need not preserve limits and it typically does not preserve all colimits.
\end{remark}
\begin{proof}
Assertion (1) is a formal property of categories of algebras over a fibrewise monad: if $p\colon \cat{Y}\rt \cat{X}$ is a cartesian fibration and $T\in \mm{End}_{/\cat{X}}(\cat{Y})$ is a fibrewise monad on $\cat{Y}$, then the forgetful functor $\Alg_T(\cat{Y})\rt \cat{Y}$ is a map of cartesian fibrations over $\cat{X}$ that preserves cartesian edges. To see this, let $\cat{O}$ be the operad for algebras and modules and consider the map of $\cat{O}$-monoidal categories $\big(\mm{End}_{/\cat{X}}(\cat{Y}), \cat{Y}\big)\rt \big(\ast, \cat{X}\big)$. Taking algebras, we then obtain
$$\begin{tikzcd}
\Alg_{\cat{O}}\big(\mm{End}_{/\cat{X}}(\cat{Y}), \cat{Y}\big)\arrow[rr]\arrow[rd] & & \Alg_{\EE_1}\big(\mm{End}_{/\cat{X}}(\cat{Y})\big)\times \cat{Y}\arrow[ld]\\
 & \cat{X}.
\end{tikzcd}$$
{It follows }from \cite[Corollary 3.2.2.3]{HA} that this is a map of cartesian fibrations over $\cat{X}$ preserving cartesian arrows. Taking the fibre over a fixed monad $T\in \Alg_{\EE_1}\big(\mm{End}_{/\cat{X}}(\cat{Y})\big)$, one obtains that $\Alg_T(\cat{Y})\rt \cat{X}$ is a cartesian fibration as well (with cartesian arrows detected in $\cat{Y}$). 

Assertion (2) follows from  \Cref{lem:coLie properties} and (3) follows from (2) and étale descent for $A\longmapsto \QC^\vee(A)$.
\end{proof}

We now turn to the proof of  \Cref{prop:pla natural}. We will need the following technical category-theoretic result about completing families of $\infty$-categories under colimits:
\begin{definition}
Let $p\colon \cat{C}\rt \cat{B}$ be a cartesian fibration over a small $\infty$-category and let $\cat{K}$ be a class of small $\infty$-categories.
\begin{itemize}
\item Let us say that a colimit diagram $K^{\rhd}\to \cat{C}_b$ in a fibre is \emph{conserved} if it is preserved by the base change functor $f^*\colon \cat{C}_{b}\to \cat{C}_{b'}$ for each $b'\to b$ in $\cat{B}$.
\item We will say that $p$ is \emph{$\cat{K}$-complete} if each fibre $\cat{E}_b$ is $\cat{K}$-complete and if all of these colimits are conserved.

\item By a \emph{collection of conserved colimit diagrams} we will mean a collection $\cat{R}=\cup_{b\in \cat{B}} \cat{R}_b$, where  each $\cat{R}_b=\big\{f_\alpha\colon K_\alpha^{\rhd}\rt \cat{C}_{b}\big\}$ is a set of conserved colimit diagrams in the fibre $\cat{C}_b$ and each base change functor sends $\cat{R}_b$ to $\cat{R}_{b'}$.

\item Given another cartesian fibration $q\colon \cat{W}\rt \cat{B}$,  write $\Fun_{\cat{R}/\cat{B}}(\cat{C}, \cat{W})$ for the full subcategory of $\Fun_{/\cat{B}}(\cat{C}, \cat{W})$ on those functors $F\colon \cat{C}\rt \cat{W}$ over $\cat{B}$ such that each map on fibres $F_b\colon \cat{C}_b\rt \cat{W}_b$ preserves the colimits in $\cat{R}_b$. If $\cat{R}$ is the collection of all colimits indexed by diagrams from $\cat{K}$, we will simply write $\Fun_{\cat{K}/\cat{B}}(\cat{C}, \cat{W})$.
\end{itemize}
\end{definition}
\begin{proposition}\label{prop:oplax adding colimits}
Let $\cat{B}$ be a small $\infty$-category and consider a fully faithful map of cartesian fibrations, preserving cartesian arrows
$$\begin{tikzcd}[row sep=1pc]
\cat{C}\arrow[rd, "p"{swap}]\arrow[rr, hookrightarrow, "j"] & & \cat{V}\arrow[ld, "\hat{p}"]\\
& \cat{B}.
\end{tikzcd}$$
Let $\cat{K}$ be a collection of small $\infty$-categories and let $\cat{R}$ be a set of conserved colimit diagrams, indexed by $\infty$-categories in $\cat{K}$. Suppose that the following condition holds:
\begin{enumerate}
\item $\hat{p}\colon \cat{V}\rt \cat{B}$ is $\cat{K}$-complete.

\item Each $j_b\colon \cat{C}_b\rt \cat{V}_b$ exhibits $\cat{V}_b$ as the $\cat{K}$-completion of $\cat{C}_b$ relative to $\cat{R}_b$.
\end{enumerate}
Then for every $\cat{K}$-complete cartesian fibration $q\colon \cat{W}\rt \cat{B}$, restriction along $j$ induces an equivalence of $\infty$-categories
$$
\Fun_{\cat{K}/\cat{B}}(\cat{V}, \cat{W})\simeq \Fun_{\cat{R}/\cat{B}}(\cat{C}, \cat{W}).
$$ 
\end{proposition}
\begin{proof}
We will need some details about the construction of the $\cat{K}$-completion $\cat{V}_b$ of $\cat{C}_b$ relative to $\cat{R}_b$ from the proof of \cite[Proposition 5.3.6.2]{HTT}. Let $\cat{P}_{\cat{R}_b}(\cat{C}_b)\subseteq \cat{P}(\cat{C}_b)$ be the $\infty$-category of presheaves on $\cat{C}_b$ sending all colimits in $\cat{R}_b$ to limits and recall that we can identify $\cat{V}_b\subseteq \cat{P}_{\cat{R}_b}(\cat{C}_b)$ with the smallest subcategory that contains the representable presheaves and is closed under $\cat{K}$-indexed colimits. Let $L_b\colon \cat{P}(\cat{C}_b)\rt \cat{P}_{\cat{R}_b}(\cat{C}_b)$ be the left adjoint to the inclusion and let $\cat{E}_b\subseteq \cat{P}(\cat{C}_b)$ be the inverse image $L_b^{-1}(\cat{V}_b)$; note that $\cat{E}_b$ is closed under $\cat{K}$-indexed colimits in $\cat{P}(\cat{C}_b)$. The functor $j_b\colon \cat{C}_b\rt \cat{V}_b$ then factors as
$$\begin{tikzcd}
j_b\colon \cat{C}_b\arrow[r, "h_b"] & \cat{E}_b\arrow[r, "L_b"] & \cat{V}_b
\end{tikzcd}$$
where the first functor is the Yoneda embedding and the second is a localisation (at the $L_b$-local equivalences). All of this depends functorially on the pair $(\cat{C}_b, \cat{R}_b)$. Consequently, under unstraightening it gives rise to a diagram of cartesian fibrations and maps preserving cartesian arrows
$$\begin{tikzcd}[row sep=1pc]
\cat{C}\arrow[rd, "p"{swap}]\arrow[r, "h"] & \cat{E}\arrow[d] \arrow[r, "L"] & \cat{V}\arrow[ld, "\hat{p}"]\\
& \cat{B}
\end{tikzcd}$$
where $\cat{E}$ and $\cat{V}$ are both $\cat{K}$-complete. By \cite{hinich2013dwyer}, the map $L$ is a localisation, inverting the $L_b$-local equivalences in each fibre $\cat{E}_b$. Since each localisation $L_b$ preserves $\cat{K}$-indexed colimits and admits a fully faithful right adjoint, it follows that restriction along $L$ defines a fully faithful embedding
$$\begin{tikzcd}
\Fun_{\cat{K}/\cat{B}}(\cat{V}, \cat{W})\arrow[r, hookrightarrow] & \Fun_{\cat{K}/\cat{B}}(\cat{E}, \cat{W}).
\end{tikzcd}$$
Write $\Fun'_{\cat{K}/\cat{B}}(\cat{E}, \cat{W})$ for the essential image; it consists of those functors $F\colon \cat{E}\rt \cat{W}$ over $\cat{B}$ such that each $F_b\colon \cat{E}_B\rt \cat{W}_b$ preserves $\cat{K}$-indexed colimits and sends the $L_b$-local equivalences to equivalences in $\cat{W}_b$. We now claim that the following two assertions hold:
\begin{enumerate}[label=(\alph*)]
\item Let $F\in \Fun_{/\cat{B}}(\cat{E}, \cat{W})$ such that its restriction $F\big|\cat{C}$ is contained in $\Fun_{\cat{R}/\cat{B}}(\cat{C}, \cat{W})$. Then $F\in \Fun'_{\cat{K}/\cat{B}}(\cat{E}, \cat{W})$ is the $q$-left Kan extension of $F\big|\cat{C}$ along $h$ if and only if $F$ sends $L$-local equivalences and preserves $\cat{K}$-indexed colimits fibrewise. 

\item If $F_0\in \Fun_{\cat{R}/\cat{B}}(\cat{C}, \cat{W})$, then there exists a $q$-left Kan extension $F\in \Fun_{\cat{K}/\cat{B}}(\cat{E}, \cat{W})$ of $F_0$ along $h$.
\end{enumerate} 
Given (a) and (b), the result follows: restriction along the localisation $L$ yields an equivalence $\Fun_{\cat{K}/\cat{B}}(\cat{V}, \cat{W})\simeq \Fun'_{\cat{K}/\cat{B}}(\cat{E}, \cat{W})$ and restriction along $h$ yields an equivalence $\Fun'_{\cat{K}/\cat{B}}(\cat{E}, \cat{W})\simeq \Fun_{\cat{R}/\cat{B}}(\cat{C}, \cat{W})$ with inverse given by $q$-left Kan extension.

It remains to verify claims (a) and (b). In the case where $\cat{B}=\ast$, these assertions are proven in the proof of \cite[Proposition 5.3.6.2]{HTT}; we will reduce to this case by a cofinality argument.
To this end, let us fix a presheaf $X\in \cat{E}_b\subseteq\cat{P}(\cat{C}_b)$ in the fibre over $b\in \cat{B}$. Then the functor
\begin{equation}\label{diag:cofinal}\begin{tikzcd}
(\cat{C}_b)_{/X}:=\cat{C}_b\times_{\cat{E}_b} (\cat{E}_b)_{/X}\arrow[r] & \cat{C}\times_{\cat{E}} \cat{E}_{/X} =: \cat{C}_{/X}
\end{tikzcd}\end{equation}
is cofinal. To see this, we have to show that for each $g\colon c\to X$ in $\cat{C}_{/X}$ covering a map $f\colon b'\to b$, the $\infty$-category $\big((\cat{C}_b)_{/X}\big)_{g/}$ is contractible. Unravelling the definitions, the left fibration $\big((\cat{C}_b)_{/X}\big)_{g/}\to (\cat{C}_b)_{/X}$ classifies the functor
$$\begin{tikzcd}
(\cat{C}_b)_{/X}\arrow[r] & \sS; & (c'\to X)\arrow[r] & \Map_{\cat{E}_{b'}}(c, f^*c')\times_{\Map_{\cat{E}_{b'}}(c, f^*X)} \{g\}.
\end{tikzcd}$$
The colimit of this diagram is contractible: indeed, $f^*\colon \cat{E}_b\to \cat{E}_{b'}$ and $\Map_{\cat{E}_{b'}}(c, -)\colon \cat{E}_{b'}\rt \sS$ both preserve colimits (recall that $\cat{E}'_b\subseteq \cat{P}(\cat{C}'_b)$ and that $f^*$ is the colimit-preserving extension of $f^*\colon \cat{C}_b\to \cat{C}_b'$ by definition).

Using \cite[Corollary 4.3.1.16]{HTT} and the fact that \eqref{diag:cofinal} is cofinal, we find that:
\begin{enumerate}[label=(\alph*')]
\item $F\in \Fun_{/\cat{B}}(\cat{E}, \cat{W})$ is the $q$-left Kan of its restriction to $\cat{C}$ if and only if each $F_b\colon \cat{E}_b\to \cat{W}_b$ is the left Kan extension of its restriction to $\cat{C}_b$.

\item $F_0\in \Fun_{/\cat{B}}(\cat{C}, \cat{W})$ admits a $q$-left Kan extension along $h$ if and only if each $F_{0b}\colon \cat{C}_b\to \cat{W}_b$ admits a left Kan extension along $h_b$.
\end{enumerate}
Using this, it suffices to verify assertions (a) and (b) fibrewise, where they hold by the proof of \cite[Proposition 5.3.6.2]{HTT}.
\end{proof}
\begin{proof}[Proof of  \Cref{prop:pla natural}]
This follows the same strategy as the proof of  \Cref{prop:pla affine}. By  \Cref{prop:oplax adding colimits} and the proof of  \Cref{prop:pla affine} (applied fibrewise),  restriction defines an equivalence of monoidal $\infty$-categories between:
\begin{enumerate}
\item endofunctors of $\QC^\vee_{/T[1]}$ relative to $\SCR_{\KK}^{\coft, \op}$ that preserve sifted colimits fibrewise and that preserve the full subcategory $\cat{Der}_{/\KK, \geq 0}^{\mm{aperf}, \coft, \op}$.

\item endofunctors of $\cat{Der}_{/\KK, \geq 0}^{\mm{aperf}, \coft, \op}$ relative to $\SCR_{\KK}^{\coft, \op}$ that preserve the sifted colimit diagrams (a) and (b) from the proof of  \Cref{prop:pla affine} fibrewise.
\end{enumerate}
The result now follows from the fact that the monad $(\coLie^\pi_{\Delta})^{\op}$ defines an algebra in the $\infty$-category (2), by  \Cref{lem:coLie properties}.
\end{proof}

\subsection{Partition Lie algebroids on schemes}\label{sec:pla schemes}
We will now deduce  \Cref{thm:pla scheme} from  \Cref{prop:lie algebroid naturality} by a descent argument. To establish unicity of the partition Lie algebroid monad $\Lie^{\pi}_{\Delta, X/\KK}$, we will need the following observation:
\begin{lemma}\label{lem:lke procoh scheme}
Let $X$ be a locally coherent qcqs scheme, $T\in \QC^\vee_X$ and  write $\cat{C}_T\subseteq (\QC^\vee_X)_{/T}$ for the full subcategory of $M\to T$ where $M$ is dually almost perfect of non-positive tor-amplitude. If $\cat{D}$ be an $\infty$-category with sifted colimits and $F\colon (\QC^\vee_X)_{/T}\rt \cat{D}$ preserves sifted colimits, then $F$ is left Kan extended from its restriction to $\cat{C}_T\subseteq (\QC^\vee_X)_{/T}$.
\end{lemma}
\begin{proof}
The proof of  \Cref{lem:pro-coh over coherent ring} shows that the $\infty$-category $\QC^\vee_X\simeq \mm{Ind}(\Coh_X^{\op})$ is equivalent to the $\infty$-category of functors $\big(\APerf_{X, \geq 0}\big)^{\op}\rt \sS$ that are reduced excisive and preserve limits of almost eventually constant towers (cf.\ \Cref{def:almost eventually constant}). Under this equivalence, the dually almost perfect objects of non-positive tor-amplitude simply correspond to the corepresentable functors. Consequently, for every $M\in \QC^\vee_X$ the canonical map $\colim_{N\in \cat{C}_M} N\rt M$ is an equivalence.

Now let $M\in (\QC^\vee_X)_{/T}$ and notice that there is an equivalence $(\cat{C}_T)_{/M}\simeq \cat{C}_M$. To see that $F$ is left Kan extended from $\cat{C}_T$, we need to show that
$$\begin{tikzcd}
\colim_{N\in (\cat{C}_T)_{/M}} F(N)\arrow[r] & F\big(\colim_{N\in (\cat{C}_T)_{/M}} N\big)\arrow[r] & F(M)
\end{tikzcd}$$
are both equivalences. Now note that $(\cat{C}_T)_{/M}\simeq \cat{C}_M$, so that the last map is an equivalence. Since $\cat{C}_M$ is sifted (it has finite coproducts) and $F$ preserves sifted colimits, the first map is an equivalence as well.
\end{proof}
\begin{proof}[Proof of  \Cref{thm:pla scheme}]
It follows directly from  \Cref{lem:lke procoh scheme} that space of sifted colimit preserving extensions is either empty or contractible. It therefore remains to prove the existence of the monad $\Lie^\pi_{\Delta, X/\KK}$. Writing $\mathfrak{U}$ for the poset of affine opens of $X$, \Cref{prop:lie algebroid naturality} provides a natural diagram
$$\begin{tikzcd}
\mathfrak{U}^{\op}\arrow[r] & \Fun^{\mm{RAd}}(\Delta^1, \Cat); \quad U\arrow[r, mapsto] & \Big((\QC^\vee_U)_{T_{U/\KK}[1]}\xrightarrow{\text{free}} \LieAlgd_{\mc{O}(U)/\KK)}\Big)
\end{tikzcd}$$
to the $\infty$-category of left adjoints and right adjointable squares between them \cite[Definition 4.7.4.16]{HA}. The limit of this diagram then yields an adjoint pair \cite[Proposition 4.7.4.19]{HA}
$$\begin{tikzcd}
\Lie^\pi_\Delta\colon \lim\limits_{U\in \mathfrak{U}^{\op}} (\QC^\vee_U)_{/T_{U/\KK}[1]}\arrow[r, yshift=1ex] & \lim\limits_{U\in \mathfrak{U}^{\op}} \LieAlgd_{\mc{O}(U)/\KK} \colon \forget\arrow[l, yshift=-1ex]
\end{tikzcd}$$
where the left (right) adjoint is given by the limit of the diagram of free (forgetful) functor for all affine open subspaces. Consequently, both adjoints preserve sifted colimits. Since pro-coherent sheaves satisfy descent (Corollary \ref{cor:descent}), one obtains a sifted colimit-preserving monad
$$\begin{tikzcd}
\Lie^\pi_{\Delta, X/\KK}\colon (\QC^\vee_X)_{/T_{X/\KK}[1]}\arrow[r] & (\QC^\vee_X)_{/T_{X/\KK}[1]}.
\end{tikzcd}$$
By construction, the restriction of this monad to the full subcategory $\cat{Der}_{X/\KK, \geq 0}^{\mm{aperf}, \op}$ is the limit of the monads
$$\begin{tikzcd}
\big(\cot\circ\mm{triv}(-^\vee)\big)^\vee\colon \cat{Der}_{\mc{O}(U)/\KK, \geq 0}^{\mm{aperf}, \op}\arrow[r] & \cat{Der}_{\mc{O}(U)/\KK, \geq 0}^{\mm{aperf}, \op}
\end{tikzcd}$$
for all affine opens $U\subseteq X$. By descent for almost perfect complexes, this monad naturally equivalent to $\big(\cot\circ\mm{triv}(-^\vee)\big)^\vee$, as desired.
\end{proof}
\begin{remark}\label{rem:pla scheme descent}
The proof of  \Cref{thm:pla scheme} shows that the  $\infty$-category  $\LieAlgd_{X/\KK}$ from  \Cref{def:pla scheme} coincides with the definition as the limit of the categories $\LieAlgd_{U/\KK}$ from  \Cref{def:pla affine} for all affine open subschemes.
\end{remark}

\newpage

\section{Infinitesimal deformations of families} \label{sec:infdefprob}
In this section, we will show that the $\infty$-category of partition Lie algebroids over a coherent animated $\KK$-algebra $B$ is equivalent to the $\infty$-category of deformation functors, or \emph{formal moduli problems} in the terminology of Lurie \cite[Chapter 12]{SAG}, defined on a certain $\infty$-category  of Artinian extensions of $B$. In Section \ref{sec:underlying lie}, we will use this to show that every partition Lie algebroid over $B$ has an underlying $\KK$-linear partition Lie algebra.

\subsection{Artinian extensions}
We begin by  introducing the following variant of Artin local rings relative to $B$:
\begin{definition}[Artinian extensions] \label{def:Artinian}
Let $B$ be a coherent animated $\KK$-algebra. We will say that a map of animated $\KK$-algebras $A\rt B$ is an \emph{Artinian extension} if it satisfies the following three conditions:
\begin{enumerate}
\item $A\to B$ is a nilpotent extension, that is, $\pi_0(A)\rt \pi_0(B)$ is surjective with nilpotent kernel.

\item $A\rt B$ is almost of finite presentation.

\item the fibre of $A\rt B$ is eventually coconnective.
\end{enumerate}
A map of animated $\KK$-algebras $A\rt B$ is said to be an \emph{almost Artinian extension} if it satisfies conditions (1) and (2). We will write $\Art_{\KK/B}\subseteq \AArt_{\KK/B}\subseteq \SCR_{\KK/B}$ for the full subcategory on the Artinian extensions of $B$.
\end{definition}
 
\begin{remark}\label{rem:coherent rings classical}
Let $A$ be a discrete commutative ring and let $I\subseteq A$ be a finitely presented nilpotent ideal such that $A/I$ is a coherent ring. Let us recall the following facts:
\begin{enumerate}
\item $A$ is coherent \cite[Theorem 4.1.1]{glaz1989coherent}.
\item For every $k\geq 0$, the ideal $I^k$ is finitely generated, and hence finitely presented over $A$.
\item If $M$ is a finitely presented $A$-module, then $I^kM/I^{k+1}M$ is a finitely presented $A$-module and hence a finitely presented $A/I$-module \cite[Theorem 2.1.8, Theorem 2.2.1]{glaz1989coherent}. 
\end{enumerate}
\end{remark}
\begin{proposition}\label{prop:artinian small}
Let $A'\to A$ be a map in $\SCR_{\KK/B}$ with $A\in \Art_{\KK/B}$. Then the following assertions are equivalent:
\begin{enumerate}
\item $A'\in \Art_{\KK/B}$ and $A'\to A$ is a nilpotent extension.

\item There exists a finite chain of maps $A'=A_n\to A_{n-1}\to \dots \to A_0=A$ in $\SCR_{\KK/B}$ where each $A_i\to A_{i-1}$ is a square zero extension by a connective coherent $B$-module, viewed as an $A_{i-1}$-module via the canonical map $A_{i-1}\to B$.
\end{enumerate}
\end{proposition}
\begin{proof}
If $A'\to A$ satisfies condition (2), it is clearly a nilpotent extension. We will prove by induction on $i$ that the map $A_i\to B$ exhibits $B$ as an almost perfect $A_i$-module. This is evident for $i=0$. Assume that $B\in \APerf_{A_{i-1}}$, so that every almost perfect $B$-module is almost perfect as an $A_{i-1}$-module. It follows that $A_i\to A_{i-1}$ is a square zero extension by an almost perfect $A_{i-1}$-module.  \Cref{lem:sqz is aft} now implies that every almost perfect $A_{i-1}$-module is almost perfect as an $A_i$-module, so that $B\in \APerf_{A_i}$ as well.  \Cref{prop:filtered aft} now implies that $A_i\to B$ is almost finitely presented. Since $A_i\to B$ is clearly a nilpotent extension with eventually coconnective fibre, we conclude that $A_i\to B$ is an Artinian extension, so that (1) follows.

For the converse, we first suppose that $A'\to A$ is a nilpotent extension with $A'\in \Art_{\KK/B}$ such that $A\in \APerf_{A'}$. Write
$$
I=\ker(\pi_0(A')\rt \pi_0(B)) \qquad\text{and}\qquad J=\ker(\pi_0(A')\to \pi_0(A)).
$$
Both $I$ and $J$ are finitely presented nilpotent ideals, since $A'\to A$ and $A'\to B$ are almost finitely presented. In particular, $\pi_0(A')$ is a coherent ring by Remark \ref{rem:coherent rings classical}, and likewise for $\pi_0(A)$.

We now fix $m\geq 0$ such that $I^m=0$ and $t\geq 1$ such that $\cofib(A'\to A)$ is $t$-coconnective; such $t$ exists since $A'\to B$ and $A\to B$ both have eventually coconnective fibres. 
We will construct a sequence of animated $A'$-algebras $A'= A_{(t+1)}\to \dots\to A_{(3)}\to A_{(2)}\to A_{(1)}\to A_{(0)}=A$ with the following properties:
\begin{enumerate}[label=(\alph*)]
\item For each $i$, $\cofib(A\to A_{(i)})$ is a $i$-connective and $t$-coconnective almost perfect $A'$-module.
\item Each $A_{(i+1)}\to A_{(i)}$ decomposes into $m$ square zero extensions, each by the $i$-fold suspension of a finitely presented discrete $\pi_0(B)$-module.
\end{enumerate}
Since finitely presented discrete $\pi_0(B)$-modules are coherent as $B$-modules (\Cref{def:coherent modules}), this provides the desired decomposition.

To begin, we define $A_{(1)}=\pi_0(A')\times_{\pi_0(A)} A$ and note that $A_{(1)}\to A_{(0)}=A$ is the base change of the sequence of square zero extensions
$$\begin{tikzcd}[column sep=1pc]
\pi_0(A') \arrow[r] & \pi_0(A')/I^{m-1}J\arrow[r] & \dots \arrow[r] & \pi_0(A')/IJ\arrow[r] & \pi_0(A')/J=\pi_0(A)
\end{tikzcd}$$
by the finitely presented discrete $\pi_0(B)$-modules $I^kJ/I^{k+1}J$ (by Remark \ref{rem:coherent rings classical}). Note that $\pi_0(A')\cong \pi_0(A_{(1)})$ and that $\cofib(A'\to A_{(1)})$ is $1$-connective and $t$-coconnective, as desired. 

For $i\geq 1$, suppose that we have constructed $A_{(i)}$ and let $C=\cofib(A'\to A_{(i)})\in \APerf_{A'}$. Since $C$ is $i$-connective and $\pi_0(A')\cong \pi_0(A_{(i)})$, $L_{A_{(i)}/A'}$ is $i$-connective and $\pi_i(L_{A_{(i)}/A'})\cong \pi_i(C)$ is a finitely presented $\pi_0(A')$-module. We define $A_{(i+1)}$ to be the square zero extension of $A_{(i)}$ classified by the map $L_{A_{(i)}/A'}\to \pi_i(C)[i]$. Then $\cofib(A'\to A_{(i+1)})\simeq \tau_{\geq i+1}C$ is $(i+1)$-connective and $t$-coconnective. The finitely presented $\pi_0(A')$-module $\pi_i(C)$ fits into a sequence of extensions
$$\begin{tikzcd}[column sep=1.2pc]
\pi_i(C)\arrow[r] & \pi_i(C)/I^{m-1}\arrow[r] & \dots\arrow[r] & \pi_i(C)/I^2\arrow[r] & \pi_i(C)\otimes_{\pi_0(A')} \pi_0(B)
\end{tikzcd}$$
by finitely presented discrete $\pi_0(B)$-modules. This implies that $A_{(i+1)}\to A_{(i)}$ has property (b).

It now remains to verify by induction that $A_{(i+1)}$ is indeed an almost perfect $A'$-module for all $i\geq -1$. For $i=-1$, this holds by hypothesis and for $i\geq 0$ we constructed $A_{(i+1)}$ as an iterated extension of $A_{(i)}$ by coherent $B$-modules, all of which are in $\APerf_A$ by inductive hypothesis.

Finally, let us show that if $A'\to A$ is a map in $\Art_{\KK/B}$, then $A\in \APerf_{A'}$. To this end, note that $A\to B$ exhibits $B$ as an almost perfect $A$-module because $A\to B$ is an Artinian extension. The previous part of the proof therefore shows that $A$ is an iterated square zero extension of $B$ by coherent $B$-modules. Because $A'\in \Art_{\KK/B}$, these are all almost perfect as $A'$-modules, so that $A\in \APerf_{A'}$ as well.
\end{proof}
\begin{corollary}\label{cor:art small}
The subcategory $\Art_{\KK/B}\subset \SCR_{\KK/B}$ is the smallest full subcategory satisfying the following two conditions:
\begin{enumerate}
\item It contains the terminal object $B\to B$.

\item Given $M\in \Coh_{B, \geq 0}$ and a pullback square exhibiting $A_\eta$ as the square zero extension of $A$ by $M$
\begin{equation}\label{diag:sqz art}\begin{tikzcd}
A_\eta\arrow[d]\arrow[r] & B\arrow[d]\\
A\arrow[r, "\eta"] & B\oplus M[1]
\end{tikzcd}\end{equation}
we have that $A_\eta\in \Art_{\KK/B}$ as soon as $A\in \Art_{\KK/B}$.
\end{enumerate}
Furthermore, for each diagram $A_1\to A_0\leftarrow A_2$ in $\Art_{\KK/B}$ where one arrow is a nilpotent extension, the pullback is again contained in $\Art_{\KK/B}$.
 
\end{corollary}
\begin{proof}
The first assertion follows from  \Cref{prop:artinian small} applied to $A=B$: this shows that $A'\to B$ is an Artinian extension if and only if it decomposes into a finite sequence of square zero extensions by connective coherent $B$-modules. The second assertion follows from  \Cref{prop:artinian small} by decomposing the nilpotent extension into square zero extensions by coherent $B$-modules and using the first assertion.
 
\end{proof}
\begin{corollary}\label{cor:artin remains coherent}
Let $B$ be a coherent animated ring and let $A\rt B$ be an Artinian extension. Then $A$ is coherent as well.
\end{corollary}
\begin{proof}
By  \Cref{prop:artinian small}, it suffices to treat the case where $A$ is a square zero extension of $B$ by a coherent $B$-module. In this case,  \Cref{lem:sqz is aft} implies that the truncation of an almost perfect $A$-module remains almost perfect.
\end{proof}
\begin{corollary}\label{cor:maps of art are aft}
Let $B$ be a coherent animated ring and let $A\to A'$ be a map of Artinian extensions of $B$. Then $A'$ is almost finitely presented as an animated $A$-algebra.
\end{corollary}
\begin{proof}
By Proposition \ref{prop:aft conditions}, it suffices to verify that $A'$ is an almost perfect $A$-module. Since $A'$ arises from $B$ by square zero extensions by coherent $B$-modules, it suffices to verify that $B$ is almost perfect as an $A$-module, which holds by assumption (and Proposition \ref{prop:aft conditions}).
\end{proof}
The $\infty$-category of almost Artinian extensions $\AArt_{\KK/B}$ is essentially obtained from the $\infty$-category of Artinian extensions by adding limits of relative Postnikov towers.
\begin{lemma}\label{lem:relative postnikov}
Let $\SCR^{\mm{sur}}_{\KK/B}\subseteq \SCR_{\KK/B}$ denote the full subcategory of animated surjections $A\to B$. For each $(A\to B)\in \SCR^{\mm{sur}}_{\KK/B}$, there exists an initial map $(A\to B)\rt (\tau_{\leq n}(A/B)\to B)$ to an animated surjection with $n$-coconnective fibre.
\end{lemma}
We will refer to $\tau_{\leq n}(A/B)$  as the \emph{fibrewise $n$-truncation} of $A$ over $B$.
\begin{proof}
Consider the full subcategory $\cat{C}\subset \Fun(\Delta^1, \SCR_{\KK})$ spanned by the animated surjections, and recall that $\cat{C}$ has compact projective generators of the form $R[x_1, \dots, x_{n+m}]\to R[x_1, \dots, x_n]$. Using this, one sees that an object $(A\to B)$ of $\cat{C}$ is $n$-truncated if and only if $A$ and $\fib(A\to B)$ are $n$-truncated (which implies that $B$ is $(n+1)$-truncated). Furthermore, each animated surjection $(A\to B)$ admits an $n$-truncation $\tau_{\leq n}(A\to B)$ whose domain and fibre are $\tau_{\leq n}A$ and $\tau_{\leq n}\fib(A\to B)$;  write $\tau'_{\leq n}B$ for the codomain, which need not be the $n$-truncation of $B$. Given an animated surjection $A\to B$, we then define
$$
\tau_{\leq n}(A/B) = \tau_{\leq n}A\times_{\tau'_{\leq n} B} B.
$$
This determines a functor $\tau_{\leq n}\colon \SCR^{\mm{sur}}_{\KK/B}\rt \SCR^{\mm{sur}}_{\KK/B}$ equipped with a natural transformation $\id\rt \tau_{\leq n}$. By construction, the map $\fib(A\to B)\rt \fib(\tau_{\leq n}(A/B)\to B)$ exhibits its target as the $n$-truncation of the domain. In particular, this implies that $\tau_{\leq n}$ is a localisation and that the $\tau_{\leq n}$-local objects are the animated surjections with $n$-coconnective fibre.
\end{proof}
\begin{lemma}\label{lem:almost artinian}
Let $A\to B$ be an animated surjection. Then the following are equivalent:
\begin{enumerate}
\item $A\to B$ is an almost Artinian extension.
\item For each $n\geq 0$, the fibrewise $n$-truncation $\tau_{\leq n}(A/B)\to B$ is an Artinian extension.
\end{enumerate}
Consequently, $\AArt_{B/\KK}\subseteq \SCR_{\KK/B}$ is closed under square zero extensions by almost perfect connective $B$-modules.
\end{lemma}
\begin{proof}
Note that $\pi_i(A)\cong \pi_i(\tau_{\leq n}(A/B))$ for all $i\leq n$.
Assuming (2), it therefore follows that $A\to B$ is a nilpotent extension and that $A$ is a coherent animated ring. Since each $\pi_i(B)$ is a finitely presented discrete module over $\pi_0(A)\cong \pi_0(\tau_{\leq 0}(A/B))$, we conclude that $B$ is an almost perfect $A$-module, so that $A\to B$ is almost finitely presented.

Assuming (1), the argument from  \Cref{prop:artinian small} shows that there exists a tower $\dots \to A_n\to A_{n-1}\to \dots \to B$ of Artinian extensions such that $\fib(A\to A_n)$ becomes increasingly connective. In particular, choosing $n$ large enough we see that $\pi_0(A)\cong \pi_0(A_n)$ is a coherent ring and that each $\pi_i(A)\cong \pi_i(A_n)$ is a finitely presented $\pi_0(A)$-module. It follows that $A$ is coherent, so that each truncation $\tau_{\leq n}\fib(A\to B)$ is an almost perfect $A$-module. Using this and  \Cref{prop:filtered aft}, one sees that each $A\to \tau_{\leq n}(A/B)$ is almost finitely presented. It follows that $\tau_{\leq n}(A/B)\to B$ is almost finitely presented, and it is clearly a nilpotent extension with $n$-coconnective fibre.
\end{proof}

\begin{definition}\label{def:artinian etale}
Let $\Art_\KK\hookrightarrow \AArt_{\KK}\hookrightarrow \Fun(\Delta^1, \SCR_{\KK})$ denote the subcategories whose:
\begin{enumerate}
\item objects are (almost) Artinian extensions $A\rt B$ where $B$ is coherent.
\item morphisms $(\tilde{f}, f)\colon (A\to B)\rt (A'\to B')$ are commuting squares in which $f$ is étale.
$$\begin{tikzcd}
A\arrow[d]\arrow[r, "\tilde{f}"] & A'\arrow[d]\\
B\arrow[r, "f"] & B'
\end{tikzcd}$$
\end{enumerate}
\end{definition}
\begin{proposition}\label{prop:art etale}
Let $\SCR_{\KK}^{\coh, \mm{et}}\hookrightarrow \SCR_{\KK}$ denote the subcategory of coherent animated $\KK$-algebras and étale maps between them. Then the codomain projections define cocartesian fibrations
$$\begin{tikzcd}
\Art_{\KK}\arrow[r] & \SCR_{\KK}^{\coh, \et} & & \AArt_{\KK}\arrow[r] & \SCR_{\KK}^{\coh, \et}.
\end{tikzcd}$$
Furthermore, for each étale map $f\colon B\to B'$, the induced functor $f^*\colon \AArt_{\KK/B}\to \AArt_{\KK/B'}$ preserves pullbacks along nilpotent extensions, as well as limits of fibrewise Postnikov towers.
\end{proposition}
\begin{proof}
A commuting square as in  \Cref{def:artinian etale} defines a locally cocartesian arrow in $\AArt_{\KK}$ if $A\to A'$ is étale. Indeed, any map of Artinian extensions $(A\to B)\rt (A''\to B')$ covering the map $B\to B'$ factors uniquely over $(A'\to B')$ because the étale map $A\to A'$ has the unique lifting property against the nilpotent extension $A''\to B'$. It is clear that these types of locally cocartesian arrows are stable under composition.

It therefore remains to verify that for every Artinian extension $(A\to B)$ and every étale map $f\colon B\to B'$, there exists a lift $(\tilde{f}, f)\colon (A\to B)\rt (A'\to B')$ such that $\tilde{f}\colon A\rt A'$ is étale. For every animated ring $B$, the functor $B'\mapsto \pi_0(B)\otimes_B B'$ defines an equivalence between the categories of étale $B$-algebras and étale $\pi_0(B)$-algebras. In turn, the functor $B''\mapsto B''\otimes_{\pi_0(B)} \pi_0(B)/I$ determines an equivalence between the (ordinary) categories of étale $\pi_0(B)$-algebras and $\pi_0(B)/I$-algebras whenever $I$ is a nilpotent ideal \cite[Tag 039R]{stacks-project}. Using this, one sees that there is a unique lift of $f$ to an étale map $\tilde{f}\colon A\rt A'$. 

To see that the resulting map $A'\to B'$ is (almost) Artinian, note that it is almost of finite presentation since both $A'$ and $B'$ are almost of finite presentation over $A$. Furthermore, $\ker(\pi_0(A')\to \pi_0(B'))\cong \pi_0(A')\otimes_{\pi_0(A)} \ker(\pi_0(A)\to \pi_0(A'))$ is the base change of a nilpotent ideal along a flat map, and hence a nilpotent ideal itself. Likewise, $\fib(A'\to B')\simeq A'\otimes_A \fib(A\to B)$ is eventually coconnective if $\fib(A\to B)$ is eventually coconnective, as $A\to A'$ is flat.

Finally,  consider the change of fibre functor $f^*\colon \AArt_{B/\KK}\to \AArt_{B'/\KK}$. For any map $A_1\to A_2$ of (almost) Artinian extensions of $B$, we have that $A_i\to f^*(A_i)$ is the unique étale map covering $f$. This implies that $f^*(A_2)\simeq f^*(A_1)\otimes_{A_1} A_2$. Using this, one readily verifies that $f^*$ preserves pullbacks along nilpotent extensions. Furthermore, the fact that $A\to f^*(A)$ is flat implies that $f^*(A)\otimes_A \tau_{\leq n}(A/B)\simeq \tau_{\leq n}(f^*(A)/B')$, so that $f^*$ preserves limits of fibrewise Postnikov towers as well.
\end{proof}

\subsection{Formal moduli problems}\label{sec:formal moduli problems}
\mbox{We  recall the notion of a formal moduli problem \cite[Chapter 12]{SAG}:}
\begin{definition}\label{def:formal moduli}
A functor $X\colon \Art_{\KK/B}\rt \sS$ is a \emph{formal moduli problem} if it satisfies the following two conditions:
\begin{enumerate}
\item $X(B)\simeq \ast$.
\item For any pullback square in $\Art_{\KK/B}$ of the form
$$\begin{tikzcd}
A\arrow[r]\arrow[d] & A_0\arrow[d]\\
A_1\arrow[r] & A_{01}
\end{tikzcd}$$
where $A_0\to A_{01}$ is a nilpotent extension, $X(A)\to X(A_0)\times_{X(A_{01})} X(A_1)$ is an equivalence.
\end{enumerate}
We will write $\FMP_{B/\KK}$ for the $\infty$-category of formal moduli problems.
\end{definition}
\begin{remark}\label{rem:fmp in terms of sqz}
Using  \Cref{prop:artinian small}, one sees that $X\colon \Art_{\KK/B}\rt \sS$ is a formal moduli problem if and only if $X(B)\simeq \ast$ and $X$ preserves each pullback square \eqref{diag:sqz art} describing a square zero extension of $A\in \Art_{\KK/B}$ by a coherent connective $B$-module.
\end{remark}
\begin{example}\label{ex:corep fmp}
Let $B'\to B$ be a map of derived $R$-algebras. We will write $\mm{Spf}(B')$ for the corresponding \emph{corepresentable formal moduli problem}
$$\begin{tikzcd}
\mm{Spf}(B')\colon \Art_{\KK/B}\arrow[r] & \sS; \quad A\arrow[r, mapsto] & \Map_{\DAlg_{\KK/B}}(B', A).
\end{tikzcd}$$
This defines the right adjoint in an adjoint pair
$$\begin{tikzcd}
\mc{O}\colon \FMP_{B/\KK}\arrow[r, yshift=1ex] & \DAlg_{\KK/B}^{\op}\cocolon \mm{Spf}\arrow[l, yshift=-1ex]
\end{tikzcd}$$
where $\mc{O}(X)$ can be viewed as the global sections of the formal moduli problem $X$.
\end{example}
\begin{lemma}\label{lem:fmp almost artin}
Restriction along $i\colon \Art_{\KK/B}\hookrightarrow \AArt_{\KK/B}$ defines an equivalence between $\FMP_{B/\KK}$ and the $\infty$-category of functors $X\colon \AArt_{\KK/B}\rt \sS$ satisfying the following conditions:
\begin{enumerate}
\item $X(B)\simeq \ast$.
\item $X$ preserves pullbacks along nilpotent extensions in $\AArt_{\KK/B}$.
\item $X$ preserves limits of fibrewise Postnikov towers.
\end{enumerate}
\end{lemma}
\begin{proof}
 \Cref{lem:almost artinian} implies that for any functor $Y\colon \Art_{\KK/B}\rt \sS$, its right Kan extension along $i$ is given by $i_*Y(A)=\lim_n Y(\tau_{\leq n}(A/B))$. The result follows readily from this.
\end{proof}
Recall that every formal moduli problem has a tangent complex:
\begin{construction}\label{con:tangent}
Let $\cat{Der}^{\mm{coh}}_{B/\KK, \geq 0}$ denote the $\infty$-category of derivations $\alpha\colon L_{B/\KK}[-1]\to M$ where $M$ is a connective coherent $B$-module.
 
Given a formal moduli problem $X\colon \Art_{\KK/B}\rt \sS$,   write
$$\begin{tikzcd}
T_{B/X}[1]\colon \cat{Der}^{\mm{coh}}_{B/\KK, \geq 0}\arrow[r] & \sS; \quad (M, \alpha)\arrow[r, mapsto] & X(B\oplus_\alpha M)
\end{tikzcd}$$
for the restriction of $X$ to the square zero extensions of $B$. Because $X$ is a formal moduli problem, this functor preserves pullbacks of $\pi_0$-surjections. Consequently,  \Cref{prop:colimit completion slice} and  \Cref{lem:pro-coh over coherent ring} (or Example \ref{ex:def functor linear} below) imply that $T_{B/X}[1]$ is classified by an object of $\QC^\vee_{B/T_{B/\KK}[1]}$ that we will refer to as the \emph{tangent fibre} of the formal moduli problem $X$, and denote by $T_{B/X}[1]$ as well.
\end{construction}
\begin{example}\label{ex:tangent fibre corep}
Let $B'\to B$ be a map of derived $R$-algebras and let $\mm{Spf}(B')$ be the associated corepresentable formal moduli problem (Example \ref{ex:corep fmp}). The tangent fibre of $\mm{Spf}(B')$ is then given by the functor $\cat{Der}^{\mm{coh}}_{B/\KK, \geq 0}\rt \sS$ corepresented by the cotangent complex $L_{B/\KK}[-1]\to L_{B/B'}[-1]$. Consequently, $T_{B/\mm{Spf}(B')}\simeq L^\vee_{B/B'}[1]$ is the pro-coherent dual of the cotangent fibre.
\end{example}

Finally, we record the functoriality of the $\infty$-category $\FMP_{B/\KK}$ with respect to étale maps:
\begin{construction}\label{con:fmp global}
Write $\cat{A}\to \cat{B}$ for the cocartesian fibration $\Art_{\KK}\rt \SCR^{\coh, \et}_{\KK}$ from  \Cref{prop:art etale} and consider the relative functor $\infty$-category  $\Fun_{/\cat{B}}(\cat{A}, \sS)\to \cat{B}$. Recall that this is the $\infty$-category over $\cat{B}$ defined by the universal property that functors $\cat{D}\to \Fun_{/\cat{B}}(\cat{A}, \sS)$ are naturally equivalent to pairs consisting of a functor $\cat{D}\to \cat{B}$ and a functor $\cat{A}\times_{\cat{B}} \cat{D}\to \sS$. Then projection $\Fun_{/\cat{B}}(\cat{A}, \sS)\to \cat{B}$ is a cartesian and cocartesian fibration, classifying the functor sending $B\mapsto \Fun(\Art_{\KK/B}, \sS)$ and a map $f\colon B\to B'$ to restriction and left Kan extension along $\Art_{\KK/B}\rt \Art_{\KK/B'}$ \cite[Proposition 7.3]{GHN17}. We write
$$
\FMP_{/\KK}\subseteq \Fun_{/\cat{B}}(\cat{A}, \sS)
$$
for the full subcategory spanned by the tuples $(B, X\colon \Art_{\KK/B}\rt \sS)$ where $X$ is a formal moduli problem. Since each $\Art_{\KK/B}\rt \Art_{\KK/B'}$ preserves pullbacks along nilpotent extensions, the projection $\FMP_{/\KK}\rt \SCR^{\coh, \et}_{\KK}$ is a cartesian and cocartesian fibration.
\end{construction}

\subsection{Deformation theory}\label{sec:deformation theory}
Let $B$ be a coherent animated $\KK$-algebra. Our goal will be to show that taking the tangent complex of a formal moduli problem (Construction \ref{con:tangent}) refines to an equivalence
$$\begin{tikzcd}
T_{B/-}[1]\colon \FMP_{B/\KK}\arrow[r, "\sim"] & \LieAlgd_{B/\KK}.
\end{tikzcd}$$
In fact, we will construct a version of this equivalence that is functorial in $B$ with respect to étale maps. To this end, we  introduce the following notation:
\begin{notation}
Consider the functor $\SCR_{\KK}^{\coh, \et}\rt \Cat_\infty$ sending $B\mapsto \LieAlgd_{B/\KK}$ and sending each étale map $f\colon B\to B'$ to $f^*\colon \LieAlgd_{B/\KK}\rt \LieAlgd_{B'/\KK}$. This classifies a cartesian fibration and a cocartesian fibration that we will denote by
$$\begin{tikzcd}
p\colon \cat{LieAlgd}^{\pi, \et}_{\Delta/\KK}\arrow[r] & \SCR_{\KK}^{\coh, \et, \op} & p^\vee\colon \cat{LieAlgd}^{\Delta/\KK}_{\pi}\arrow[r] &\SCR_{\KK}^{\coh, \et}.
\end{tikzcd}$$
More precisely, $p$ is the restriction of the cartesian fibration from  \Cref{prop:lie algebroid naturality} to the subcategory of étale maps, and we define $p^\vee$ to be the cocartesian fibration classified by the same diagram of $\infty$-categories. Part (2) of  \Cref{prop:lie algebroid naturality} implies that $p^\vee$ is a cartesian fibration as well.
\end{notation}
\begin{theorem}\label{thm:def functor vs Lie}
There is an equivalence of $\infty$-categories
$$\begin{tikzcd}[row sep=0.8pc]
\cat{LieAlgd}^{\Delta/\KK}_{\pi, \et}\arrow[rr, "\sim"]\arrow[rd] & & \FMP_{/\KK}\arrow[ld]\\
& \SCR_{\KK}^{\coh, \et}.
\end{tikzcd}$$
Furthermore, the induced equivalence between fibres over a coherent animated $\KK$-algebra $B$ fits into a commuting diagram
$$\begin{tikzcd}
\LieAlgd_{B/\KK}\arrow[r, "\sim"]\arrow[d, "\forget"{swap}] & \FMP_{B/\KK}\arrow[d, "{T_{B/-}[1]}"]\\
\QC^\vee_{B/T_{B/\KK}[1]}\arrow[r, hookrightarrow] & \Fun\big(\cat{Der}^{\mm{coh}}_{B/\KK, \geq 0}, \sS\big).
\end{tikzcd}$$ 
\end{theorem}
We will first construct the functor $\cat{LieAlgd}^{\Delta/\KK}_{\pi, \et}\rt \FMP_{/\KK}$ and then use the language of deformation theories from \cite[Chapter 12]{SAG} to prove that it is an equivalence.
\begin{construction}
If an animated surjection $A\to B$ is an (almost) Artinian extension, then it is in particular complete almost finitely augmented (\Cref{def:complete aft}). Using Corollary \ref{cor:complete KD}, we thus obtain a composite functor 
$$\begin{tikzcd}
\mathfrak{D}\colon \Art_{\KK}\arrow[r, "\cot"] & \cat{Coalg}_{\mm{coLie}^\pi_\Delta}(\cAAPerf{}{\KK})\arrow[r, "{(-)^\vee}"] & (\LieAlgd_{/\KK})^{\op}.
\end{tikzcd}$$
sending each Artinian extension $A\to B$ to the partition Lie algebroid $\mf{D}(A)=\cot(A)^\vee$ over $B$. Since the maps in $\Art_{\KK}$ all cover étale maps $B\to B'$, this assembles into a functor
$$\begin{tikzcd}[row sep=1pc]
\Art_{\KK}\arrow[rr, "\mf{D}"]\arrow[rd] & & (\cat{LieAlgd}_{\Delta/\KK}^{\pi, \et})^{\op}\arrow[ld]\\
& \SCR_{\KK}^{\coh, \et}.
\end{tikzcd}$$
The functor $\mathfrak{D}$ preserves cocartesian arrows: indeed, for any cocartesian arrow $(A\to B)\to (A'\to B')$ in $\Art_{\KK}$ covering an étale map $f\colon B\to B'$, we have that $B'\simeq A'\otimes_A B$ so that $\mf{D}(A')\simeq f^*\mf{D}(A)$.
The functor $\mf{D}$ induces a fully faithful functor on fibres by Corollary \ref{cor:complete KD} and is hence fully faithful itself, since it preserves cocartesian arrows.  
\end{construction}
\begin{proposition}\label{preservation_of_pullbacks}
Let $B$ be a coherent animated ring and let $A\to B\oplus M[1]$ be a map in $\Art_{\KK/B}$ classifying a square zero extension $A_\eta\to A$ by a coherent connective $B$-module. Then induced map of partition Lie algebroids over $B$
\begin{equation}\label{eq:pushout pres}\begin{tikzcd}
\mathfrak{D}(A)\amalg_{\mathfrak{D}(B\oplus M[1])} \mathfrak{D}(B)\arrow[r] & \mathfrak{D}(A_\eta)
\end{tikzcd}\end{equation}
is an equivalence.
\end{proposition}
 
\begin{lemma}\label{lem:pro-coh dual of aperf filtered}
Let $B$ be a coherent animated ring and let $F^\star M\in \FilMod_B$ be an almost perfect filtered $B$-module. Then the diagram of pro-coherent duals
$$\begin{tikzcd}
(F^0M/F^1M)^\vee\arrow[r] & (F^0M/F^2M)^\vee\arrow[r] & \dots \arrow[r] & (F^0M)^\vee
\end{tikzcd}$$
is a colimit diagram in $\QC^\vee_B$.
\end{lemma}
\begin{proof}
Recall that the pro-coherent dual of a $B$-module $M$ can be described by the left exact functor $\Coh_B\rt \sS$ corepresented by $M$. It therefore suffices to verify that for any coherent $B$-module $N$, the map $\colim \Map_B(F^0M/F^nM, N)\rt \Map_B(F^0M, N)$ is an equivalence of spaces. Suppose that $N$ is $k$-coconnective. Because $F^\star M$ is almost perfect, there exists an $m$ such that $F^nM$ is $k$-connective for all $n\geq m$. It follows that $\Map_B(F^0M/F^nM, N)\to \Map_B(F^0M, N)$ is an equivalence for all $n\geq m$.
\end{proof}
\begin{proof}[Proof of  \Cref{preservation_of_pullbacks}]
Let us say that an object $A\in \DAlg^{\wedge, \aft}_{\KK/B}$ is \emph{adequate} if for every $N\in \APerf_{B, \geq 0}$, the natural map
$$\begin{tikzcd}
\mathfrak{D}(A)\sqcup \mathfrak{D}(B\oplus N)\arrow[r] & \mathfrak{D}(A\oplus N)
\end{tikzcd}$$
is an equivalence, where $B\oplus N$ and $A\oplus N=A\times_B (B\oplus N)$ denote the trivial square zero extensions by $N$. To prove the proposition, it will suffice to verify the following assertions:
\begin{enumerate}
\item Every square zero extension $B\oplus_\alpha M$ of $B$ by $M\in \APerf_{B, \geq 0}$ is adequate.

\item Suppose that $A$ is adequate and consider a pullback square \eqref{diag:sqz art} classifying a square zero extension $A_\eta\to A$ by $M\in \APerf_{B, \geq 0}$. Then the following hold:
\begin{enumerate}
\item[(2a)] the map \eqref{eq:pushout pres} is an equivalence.

\item[(2b)] $A_\eta$ is adequate.
\end{enumerate}
\end{enumerate}
Indeed, Corollary \ref{cor:art small} then shows by induction that every $A\in \Art_{\KK/B}$ is adequate, so that $\mathfrak{D}$ sends every pullback square \eqref{diag:sqz art} to a pushout of partition Lie algebroids.

Recall that for any square zero extension $B\oplus_\alpha M$ classified by $\alpha\colon L_{B/\KK}[-1]\rt M$, the partition Lie algebroid $\mathfrak{D}(B\oplus_\alpha M)$ is the \emph{free} partition Lie algebroid on $\alpha^\vee\colon M^\vee\rt T_{B/\KK}[1]$. Consequently, $\mathfrak{D}((B\oplus_\alpha M)\oplus N)$ is the free partition Lie algebroid on $(\alpha^\vee, 0)\colon M^\vee\oplus N\rt T_{B/\KK}[1]$. Assertion (1) then follows from the fact that taking free Lie algebroids is a left adjoint.

We will simultaneously verify assertions (2a) and (2b). To this end, suppose that $A$ is adequate and let $A_\eta\to A$ be a square zero extension of $A$ classified by $\alpha\colon A\rt B\oplus M[1]$. Let $A\to A^\bullet\to B\oplus M[1]$ be the cosimplicial resolution provided by Corollary \ref{cor:comonadic cobar}. Given $N\in \APerf_{B, \geq 0}$, we then obtain a diagram of complete almost finitely augmented algebras over $B$
\begin{equation}\label{diag:sqz resol}\begin{tikzcd}
A_\eta\oplus N\arrow[d]\arrow[r] & A^\bullet_\eta\oplus N\arrow[r]\arrow[d] & B\oplus N\arrow[d, "{(\mm{id}, 0)}"]\\
A\arrow[r] & A^\bullet\arrow[r] &  B\oplus M[1].
\end{tikzcd}\end{equation}
where the top row is obtained from the bottom row by base change. This gives rise to a diagram of partition Lie algebroids of the form
\begin{equation}\label{diag:real square}\begin{tikzcd}
\big|\mathfrak{D}(A^\bullet)\amalg_{\mathfrak{D}(B\oplus M[1])} \mathfrak{D}(B\oplus N)\big|\arrow[r]\arrow[d] & \mathfrak{D}(A)\amalg_{\mathfrak{D}(B\oplus M[1])} \mathfrak{D}(B\oplus N)\arrow[d]\\
\big|\mathfrak{D}(A_\eta^\bullet\oplus N)\big|\arrow[r] & \mathfrak{D}(A_\eta\oplus N).
\end{tikzcd}\end{equation}
It suffices to verify that the right vertical map is an equivalence: if this is the case, then taking $N=0$ yields assertion (2a) and assertion (2b) then follows from the equivalences
$$\begin{tikzcd}
\big(\mathfrak{D}(A)\amalg_{\mathfrak{D}(B\oplus M[1])} \mathfrak{D}(B)\big)\amalg \mathfrak{D}(B\oplus N)\arrow[r, "\sim"]\arrow[d, "\sim"] &\mathfrak{D}(A)\amalg_{\mathfrak{D}(B\oplus M[1])} \mathfrak{D}(B\oplus N)\arrow[d, "\sim"]\\
\mathfrak{D}(A_\eta)\amalg \mathfrak{D}(B\oplus N)\arrow[r] & \mathfrak{D}(A_\eta\oplus N).
\end{tikzcd}$$
To see that the right vertical map in \eqref{diag:real square} is an equivalence, we will show that each of the other three maps is an equivalence.

\bigskip

\noindent\textit{Top horizontal map.} To see that the top horizontal map in \eqref{diag:real square} is an equivalence, it suffices to show that the map $|\mathfrak{D}(A^\bullet)|\rt \mathfrak{D}(A)$ is an equivalence. Since geometric realisations of partition Lie algebroids are computed in $(\QC^\vee_B)_{/T_{B/\KK}[1]}$, it suffices to verify this at the level of pro-coherent modules. Now recall that the cosimplicial resolution $A\to A^\bullet$ from Corollary \ref{cor:comonadic cobar} has the property that the augmented cosimplicial diagram $\cot(A)\to \cot(A^\bullet)$ in $\cAAPerf{B}{\KK}$ is split. Taking pro-coherent duals, one obtains that the augmented simplicial diagram $\mathfrak{D}(A^\bullet)\to \mathfrak{D}(A)$ is split in $(\QC^\vee_B)_{/T_{B/\KK}[1]}$, so that $|\mathfrak{D}(A^\bullet)|\simeq \mathfrak{D}(A)$.

\bigskip

\noindent\textit{Left vertical map.} It suffices to verify that each square of partition Lie algebroids
$$\begin{tikzcd}
\mathfrak{D}(B\oplus M[1])\arrow[r]\arrow[d] & \mathfrak{D}(B\oplus N)\arrow[d]\\
\mathfrak{D}(A^i)\arrow[r] & \mathfrak{D}(A^i_\eta\oplus N)
\end{tikzcd}$$
is a pushout. The cosimplicial resolution $A^\bullet$ from Corollary \ref{cor:comonadic cobar} has the property that $A^i\simeq B\oplus_\alpha K$ is a square zero extension of $B$ and that the map $A^i\simeq B\oplus_{\alpha} K\rt B\oplus M[1]$ arises from a map $e\colon K\rt M[1]$ in $\cAAPerf{B}{\KK}$. It follows that the above square is the image of a pushout square in $(\QC^\vee_B)_{/T_{B/K}[1]}$ under the free partition Lie algebroid functor, and hence a pushout square itself.

\bigskip

\noindent\textit{Bottom horizontal map.} It suffices to check that the augmented simplicial diagram $\mathfrak{D}(A^\bullet_\eta\oplus N)\rt \mathfrak{D}(A_\eta\oplus N)$ is a colimit diagram in $\QC^\vee_B$. Recall that this is the pro-coherent dual of the augmented cosimplicial diagram of $B$-modules $\cot(A_\eta\oplus N)\rt \cot(A_\eta^\bullet\oplus N)$. We will refine this to a diagram of filtered $B$-modules as follows. The diagram \eqref{diag:sqz resol} of derived $\KK$-algebras over $B$ lifts to a diagram of filtered $\KK$-algebras over $B$: we endow $B\oplus N$, $A$ and $A^\bullet$ with the filtration where $F^i=0$ for all $i\geq 1$, and $B\oplus M[1]$ with the filtration
$$
F^1(B\oplus M[1])=M[1]\qquad\qquad F^i(B\oplus M[1])=0 \qquad\text{for }i\geq 2.
$$
Taking fibre products, we then obtain (finite) filtrations on $A_\eta\oplus N$ and $A^\bullet_\eta\oplus N$.

The maps of animated filtered algebras $F^\star(A_\eta\oplus N)\to B$ and $F^\star(A_\eta^i\oplus N)\to B$ are almost of finite presentation. Indeed, each $A^i\to B$ (concentrated in weight $0$) is almost of finite presentation and $A_\eta^i\oplus N\to A^i$ is a square zero extension by an almost perfect filtered $A^i$-module, and thus almost of finite presentation by  \Cref{lem:sqz is aft}. By  \Cref{prop:filtered aft}, we therefore obtain an augmented cosimplicial diagram of almost perfect filtered $B$-modules
$$\begin{tikzcd}
\cot(F^\star(A_\eta\oplus N))\arrow[r] & \cot(F^\star(A_\eta^\bullet\oplus N)).
\end{tikzcd}$$
At the level of $F^0$, this is given by $\cot(A_\eta\oplus N)\to \cot(A_\eta^\bullet\oplus N)$ and at the level of the associated graded, it is given by
$$\begin{tikzcd}
\cot(A\oplus M(1)\oplus N)\arrow[r] & \cot(A^\bullet\oplus M(1)\oplus N).
\end{tikzcd}$$
Here $A\oplus M(1)\oplus N$ denotes the derived graded $\KK$-algebra given by the trivial square zero extension of $A$ by $N$ (in weight $0$) and $M$ (in weight $1$). 

Using  \Cref{lem:pro-coh dual of aperf filtered}, we therefore obtain a certain augmented simplicial diagram of pro-coherent $B$-modules equipped with an \emph{increasing} filtration
\begin{equation}\label{diag:realisation filtered}\begin{tikzcd}
F_\star \mathfrak{D}(A_\eta^\bullet\oplus N)\arrow[r] & F_\star \mathfrak{D}(A_\eta\oplus N)
\end{tikzcd}\end{equation}
At the level of the underlying objects (i.e.\ taking the colimit as $\star\to \infty$), this yields the augmented simplicial diagram $\mathfrak{D}(A^\bullet_\eta\oplus N)\rt \mathfrak{D}(A_\eta\oplus N)$. To see that this is a colimit diagram, it now suffices to show that \eqref{diag:realisation filtered} is a colimit diagram at the level of the associated graded. The associated graded of \eqref{diag:realisation filtered} is equivalent to the augmented simplicial diagram (with a certain weight grading that will not play a role)
$$
\mathfrak{D}(A^\bullet\oplus M\oplus N)\rt  \mathfrak{D}(A\oplus M\oplus N).
$$
Since $A$ and all $A^i$ (which are of the form $B\oplus_\alpha I$) are adequate, this is equivalent to the map
$$
\mathfrak{D}(A^\bullet)\amalg \mathfrak{D}(B\oplus M\oplus N)\rt \mathfrak{D}(A)\amalg \mathfrak{D}(B\oplus M\oplus N).
$$
This is an equivalence because (as we already saw), $|\mathfrak{D}(A^\bullet)|\simeq \mathfrak{D}(A)$.
\end{proof}
By  \Cref{preservation_of_pullbacks}, any partition Lie algebroid $\mf{g}$ over $B$ determines a formal moduli problem
\begin{equation}\label{eq:fmp from liealgd}
\Psi_\mf{g}\colon \Art_{\KK/B}\rt \sS; \quad \Psi_\mf{g}(A)=\Map_{\LieAlgd_{B/\KK}}(\mf{D}(A), \mf{g}).
\end{equation}
Using this, we will now construct the putative equivalence of  \Cref{thm:def functor vs Lie}.
\begin{construction}
Consider the functor
$$\begin{tikzcd}[column sep=3pc]
\Art_{\KK}\times_{\SCR_{\KK}^{\coh, \et}} \cat{LieAlgd}_{\pi, \et}^{\Delta/\KK} \arrow[r, "{(\mf{D}, \id)}"] & \big(\cat{LieAlgd}^{\pi, \et}_{\Delta/\KK})^{\op}\times_{\SCR_{\KK}^{\coh, \et}} \cat{LieAlgd}_{\pi, \et}^{\Delta/\KK}\arrow[r, "\Map"] & \sS
\end{tikzcd}$$
where the last functor is the fibrewise mapping space functor. Using Construction \ref{con:fmp global}, this is adjoint to a functor $\cat{LieAlgd}_{\pi, \et}^{\Delta/\KK}\rt \Fun_{/\SCR_{\KK}^{\coh, \et}}(\Art_{\KK}, \sS)$. Unravelling the definitions, this sends each $\mf{g}\in \LieAlgd_{B/\KK}$ to the formal moduli problem $\Psi_\mf{g}$ from \eqref{eq:fmp from liealgd}. We thus obtain a functor
$$\begin{tikzcd}[row sep=0.8pc]
\cat{LieAlgd}^{\Delta/\KK}_{\pi, \et}\arrow[rr, "\Psi"]\arrow[rd] & & \FMP_{/\KK}\arrow[ld]\\
& \SCR_{\KK}^{\coh, \et}.
\end{tikzcd}$$
Because $\mf{D}$ preserved cocartesian fibrations, the functor $\Psi$ is a map of cartesian fibrations that preserves cartesian arrows.
\end{construction}

It now remains to verify that $\Psi$ restricts to an equivalence on fibres. For this we will use Lurie's formalism of deformation theories, slightly reformulated as follows:
\begin{theorem}[{cf.\ \cite[Theorem 12.3.3.5]{SAG}}]\label{thm:lurie}
Let $\cat{B}$ be a presentable $\infty$-category equipped with a set of adjunctions $L_\alpha\colon \Sp \leftrightarrows \cat{B}\cocolon R_\alpha$ such that each $R_\alpha$ preserves sifted colimits and the $R_\alpha$ are jointly conservative. Let $\cat{B}_0\subseteq \cat{B}$ be the smallest full subcategory of $\cat{B}$ satisfying the following conditions:
\begin{enumerate}
\item It contains the initial object $0$.
\item If $K\in \cat{B}_0$, the for each pushout square in $\cat{B}$ of the form
\begin{equation}\label{diag:cell attachment}\begin{tikzcd}
L_\alpha(\mathbb{S}[n])\arrow[r]\arrow[d] & K\arrow[d]\\
0\arrow[r] & K'
\end{tikzcd}\end{equation}
for some $\alpha$ and $n<0$, the object $K'$ is contained in $\cat{B}_0$ as well.
\end{enumerate}
Then the restricted Yoneda embedding $h\colon \cat{B}\rt \cat{P}(\cat{B}_0); K\longmapsto \Map_{\cat{B}}(-, K)$ is fully faithful, with essential image consisting of those presheaves $X$ on $\cat{B}_0$ such that $X(0)\simeq \ast$ and such that \eqref{diag:cell attachment} is sent to a pullback diagram of spaces.
\end{theorem}
\begin{proof}
Consider the adjoint pair $\mathfrak{D}^*\colon \cat{B}\leftrightarrows \cat{P}(\cat{B}_0)^{\op}\cocolon \mathfrak{D}_*$ where $\mathfrak{D}^*(K) = \Map_{\cat{B}}(K, -)$. The conditions imply that this is a deformation theory in the sense of \cite[Definition 12.3.3.2]{SAG}, and the result follows from \cite[Theorem 12.3.3.5]{SAG}.
\end{proof}
\begin{example}\label{ex:def functor linear}
Let $B$ be a coherent animated $\KK$-algebra. For each $M\in \Coh_{B, \geq 0}$, consider
$$\begin{tikzcd}
R_M\colon \QC^\vee_{B/T_{B/\KK}[1]}\arrow[r] & \Sp; \quad (\rho\colon N\to T_{B/\KK}[1])\arrow[r, mapsto] & \hom_{\QC^\vee(B)}(M^\vee, \fib(\rho)).
\end{tikzcd}$$
This is a right adjoint functor preserving sifted colimits since $M^\vee\in \QC^\vee_B$ is compact. Since the collection of $M^\vee$ with $M\in \Coh_{B, \geq 0}$ generates $\QC^\vee_B$, this satisfies the conditions of  \Cref{thm:lurie}. The full subcategory $\cat{B}^\mm{lin}_0\subseteq \QC^\vee_{B/T_{B/\KK}[1]}$ appearing in  \Cref{thm:lurie} can then be identified as follows. Note that
$$
L_M(\mathbb{S})\simeq \big(0\colon M^\vee \to T_{B/\KK}[1]\big).
$$
For any $\rho\colon N^\vee\to T_{B/\KK}[1]$ with $N\in \Coh_{B, \geq 0}$, we then have that:
\begin{enumerate}[label=(\alph*)]
\item For any $n<0$ and $M^\vee[n]\to N^\vee$, the cofibre $\fib(N\to M[-n])^\vee$ is dual to a coherent connective $B$-module.

\item it can be obtained as a pushout of a diagram $0\leftarrow L_{N}(\mathbb{S}[-1])\to 0$ in $\QC^\vee_{B/T_{B/\KK}[1]}$.
\end{enumerate}
It follows that $\cat{B}^\mm{lin}_0$ is the full subcategory of $\rho\colon N^\vee\to T_{B/\KK}[1]$ where $N^\vee$ is the pro-coherent dual of a coherent connective $B$-module. Using  \Cref{lem:aperf over complete}, we can identify $\cat{B}^{\mm{lin}, \op}_0\simeq \cat{Der}^{\mm{coh}}_{B/\KK, \geq 0}$ with the full subcategory of derivations $\alpha\colon L_{B/\KK}[-1]\to M$ for which $M\in \Coh_{B, \geq 0}$. Consequently, there is a fully faithful functor
$$\begin{tikzcd}
\QC^\vee_{B/T_{B/\KK}}\arrow[r, hookrightarrow] & \Fun\big(\cat{Der}^{\mm{coh}}_{B/\KK, \geq 0}, \sS\big)
\end{tikzcd}$$
whose essential image consists of functors $X$ such that $X(0)\simeq \ast$ and preserving pullbacks along all maps $0\to (M[1], \alpha)$ in $\cat{Der}^{\mm{coh}}_{B/\KK, \geq 0}$ with $M\in \Coh_{B, \geq 0}$.
\end{example}

\begin{proof}[Proof of  \Cref{thm:def functor vs Lie}]
We will mimic the strategy from Example \ref{ex:def functor linear}. For each $M\in \Coh_{B, \geq 0}$, consider the functor
$$\begin{tikzcd}
R_M\colon \LieAlgd_{B/\KK}\arrow[r] & \Sp; \quad \big(\rho\colon \mf{g}\to T_{B/\KK}[1]\big)\arrow[r, mapsto] & \hom_{\QC^\vee_B}(M^\vee, \fib(\rho)).
\end{tikzcd}$$
 \Cref{prop:pla affine} implies that the forgetful functor $\LieAlgd_{B/\KK}\rt \QC^\vee_{B/T_{B/\KK}[1]}$ is a right adjoint that preserves sifted colimits and detects equivalences. Using this and Example \ref{ex:def functor linear}, it follows that the set of functors $R_M$ satisfies the conditions of  \Cref{thm:lurie}.

To identify the full subcategory $\cat{B}_0\subseteq \LieAlgd_{B/\KK}$, note that $L_M(\mathbb{S})\simeq \Lie^\pi_{\Delta, B/\KK}(0\colon M^\vee\to T_{B/\KK})$ is the free partition Lie algebroid generated by the zero map from $M^\vee$. By  \Cref{prop:pla affine}, this means that there are natural equivalences
$$
L_M(\mathbb{S})\simeq \Lie^\pi_{\Delta, B/\KK}(0\colon M^\vee\to T_{B/\KK})\simeq \mathfrak{D}(B\oplus M)
$$
for all $M\in \Coh_{B, \geq 0}$. It follows that $\cat{B}_0\subseteq \LieAlgd_{B/\KK}$ is the smallest full subcategory containing the initial partition Lie algebroid $0\simeq \mathfrak{D}$ which is closed under pushouts along the maps $\mathfrak{D}(B\oplus M[1])\to \mathfrak{D}(B)$ for all $M\in \Coh_{B, \geq 0}$. Corollary \ref{cor:art small} and  \Cref{preservation_of_pullbacks} now imply that there is a commuting diagram
$$\begin{tikzcd}[row sep=1.5pc, column sep=3pc]
\mm{Der}^{\coh}_{B/\KK, \geq 0}\arrow[d, "\sqz"{swap}]\arrow[r, "(-)^\vee", "\sim"{swap}] & \cat{B}_0^{\mm{lin}, \op} \arrow[r, hookrightarrow]\arrow[d] & \big(\QC^\vee_{B/T_{B/\KK}[1]}\big)^{\op}\arrow[d, "\Lie^\pi_{\Delta, B/\KK}"]\\
\Art_{\KK/B}\arrow[r, "\mathfrak{D}", "\sim"{swap}] & \cat{B}^{\op}_0\arrow[r, hookrightarrow] & (\LieAlgd_{B/\KK})^{\op}
\end{tikzcd}$$
where equivalence $\mf{D}$ sends each pullback square \eqref{diag:sqz art} to the opposite of a pushout square of the form \eqref{diag:cell attachment}. Using this,  \Cref{thm:lurie} then implies that $\Psi\colon \cat{LieAlgd}^{\Delta/\KK}_{\pi, \et}\rt \FMP_{/\KK}$ restricts to an equivalence $\LieAlgd_{B/\KK}\simeq \FMP_{B/\KK}$ on fibres with the desired properties.
\end{proof}
\begin{remark}
Combining the equivalence of  \Cref{thm:def functor vs Lie} with the adjoint pair from Example \ref{ex:corep fmp}, we obtain an adjoint pair
$$\begin{tikzcd}
C^*\colon \LieAlgd_{B/\KK} \arrow[r, yshift=1ex] & \DAlg^{\op}_{\KK/B} \cocolon\mathfrak{D} \arrow[l, yshift=-1ex]
\end{tikzcd}$$
where $\mathfrak{D}(B')=T_{B/\mm{Spf}(B')}$ with the partition Lie algebroid structure arising from  \Cref{thm:def functor vs Lie}. Example \ref{ex:tangent fibre corep} implies that the underlying object of $\mathfrak{D}(B')$ in $\QC^\vee_{B/T_{B/\KK}[1]}$ is naturally equivalent to $L_{B/B'}^\vee[1]\to L_{B/\KK}^\vee[1]=T_{B/\KK}[1]$.

The right adjoint $C^*$ can be thought of as a version of the \emph{Chevalley--Eilenberg complex} for partition Lie algebroids. Explicitly, $C^*(\mf{g})$ is the derived algebra of functions on the formal moduli problem classified by $\mf{g}$. One can now check that the adjoint pair $(C^*, \mathfrak{D})$ also defines a deformation theory in the sense of \cite[Definition 12.3.3.2]{SAG}.
\end{remark}
\begin{remark}\label{rem:lie vs fmp global}
Suppose that $X$ is a locally coherent qcqs derived scheme. By Remark \ref{rem:pla scheme descent}, the $\infty$-category of partition Lie algebroids on $X$ can be identified with the $\infty$-category of dotted sections
$$\begin{tikzcd}
\mm{Open}^\mm{aff}(X)^{\op}\arrow[rd, "\mc{O}"{swap}]\arrow[r, dotted, "\mf{g}"] & \cat{LieAlgd}^{\Delta/\KK}_{\pi, \et}\arrow[d, "p^\vee"]\arrow[r, "\sim"] & \FMP_{/\KK}\arrow[ld]\\
 & \SCR_{\KK}^{\coh, \et}
\end{tikzcd}$$
sending each map to a cocartesian arrow. Using  \Cref{thm:def functor vs Lie} and unravelling Construction \ref{con:fmp global}, one sees that the $\infty$-category of partition Lie algebroids on $X$ is equivalent to the $\infty$-category of functors $F\colon \Art_{\KK}\times_{\SCR_{\KK}^{\coh, \et}} \mm{Open}^\mm{aff}(X)^{\op}\rt \sS$ satisfying the following two properties:
\begin{enumerate}
\item For each affine open $U\subseteq X$, $F$ restricts to a formal moduli problem $F_U\colon \Art_{\KK/\mc{O}(U)}\rt \sS$.
\item For each inclusion of affine opens $V\subseteq U\subseteq X$, $F_V$ is the initial formal moduli problem equipped with a map from $F_U$ to its restriction along $\Art_{\KK/\mc{O}(U)}\to \Art_{\KK/\mc{O}(V)}$.
\end{enumerate}
\end{remark}

\subsection{Underlying partition Lie algebra}\label{sec:underlying lie}
Suppose that $\KK$ is coherent and that $B$ is a coherent animated $\KK$-algebra. We will use  \Cref{thm:def functor vs Lie} to show that the underlying pro-coherent $\KK$-module of a partition Lie algebroid $\mf{g}$ over $B$ can be endowed with a ($\KK$-linear) partition Lie algebra structure. To this end, we begin  by recalling the following properties of the $\infty$-categories $\SCR_A$ of animated $A$-algebras.
\begin{proposition}\label{prop:deforming animated algebras}
 
Given a cartesian square of animated rings on the left in which $f\colon A_0\to A_{01}$ induces a surjection on $\pi_0$, the right square of $\infty$-categories is cartesian as well:
$$\begin{tikzcd}
A\arrow[r, "g'"]\arrow[d, "f'"{swap}] & A_0\arrow[d, "f"] \\
A_1\arrow[r, "g"] & A_{01}
\end{tikzcd} \qquad \qquad\qquad  \begin{tikzcd} \SCR_A\arrow[r, "g'^*"]\arrow[d, "f'^*"{swap}] & \SCR_{A_0}\arrow[d, "f^*"]\\
\SCR_{A_1}\arrow[r, "g^*"] & \SCR_{A_{12}}.
\end{tikzcd}$$
 
\end{proposition}
\begin{proof}
Write $h\colon A\to A_{01}$ for the composite map and consider the following diagram of $\infty$-categories:
$$\begin{tikzcd}[column sep=3.5pc]
\SCR_A\arrow[r, yshift=1ex, "{(g'^*, f'^*)}"]\arrow[d, "\forget"{swap}] & \SCR_{A_0}\times_{\SCR_{A_{01}}}\SCR_{A_1}\arrow[d, "\forget"]\arrow[l, yshift=-1ex]\\
\Mod_{A, \geq 0}\arrow[r, yshift=1ex, "{(g'^*, f'^*)}"] & \Mod_{A_0, \geq 0}\times_{\Mod_{A_{01}, \geq 0}}\Mod_{A_1, \geq 0}.\arrow[l, yshift=-1ex]
\end{tikzcd}$$
Here the rows are both adjoint pairs: the left adjoint $(g'^*, f'^*)$ sends an animated $A$-algebra (resp.\ $A$-module) $B$ to the matching triple of objects $(A_0\otimes_A B, \ A_{01}\otimes_A B, \ A_1\otimes_A B)$ and the right adjoint sends a matching triple of animated algebras (resp.\ modules) $(B_0, B_{01}, B_1)$ to the fibre product $g'_*B_0\times_{h_*B_{01}} f'_*B_1$ of $A$-algebras (resp.\ $A$-modules). In particular, the vertical forgetful functors commute both with the left adjoints and with the right adjoints. Since the forgetful functor detects equivalences, the unit and counit of the top adjunction are an equivalence if the unit and counit of the bottom adjunction are equivalences. This follows from \cite[Theorem 16.2.0.2]{SAG}, so that the top adjunction is an equivalence. 
\end{proof}
\begin{definition}\label{def:B-defs}
Suppose that $\KK$ is coherent and let $B\in \SCR_\KK$ be a $\KK$-algebra. We will write $\cat{Def}_{B/\KK}$ for the full subcategory of $\Fun(\Delta^1, \SCR_{\KK})_{/\KK\to B}$ spanned by those commuting squares
$$\begin{tikzcd}
A\arrow[r]\arrow[d] & \KK\arrow[d]\\
B_A\arrow[r] & B
\end{tikzcd}$$
that are cocartesian and in which $A\to \KK$ is an Artinian extension. We will typically write $A\to B_A$ for an object in $\cat{Def}_{B/\KK}$, omitting the data of the cocartesian square. 
\end{definition}
There is a canonical functor $\cat{Def}_{B/\KK}\rt \Art_{\KK/\KK}$ sending each such square to $A\to \KK$. We then have the following corollary of  \Cref{prop:deforming animated algebras}:
\begin{lemma}\label{lem:def FMP}
Suppose that $\KK$ is coherent and let $B$ be an animated $\KK$-algebra. Then the functor $\cat{Def}_{B/\KK}\to \Art_{\KK/\KK}$ is a left fibration, classified by a formal moduli problem $\mm{def}_{B}\colon \Art_{\KK/\KK}\to \sS$.
\end{lemma}
\begin{proof}
The functor $\cat{Def}_{B/\KK}\to \Art_{\KK/\KK}$ is a cocartesian fibration, classified by $\mm{def}_{B}\colon \Art_{\KK/\KK}\to \Cat_\infty$ sending $A\mapsto \SCR_A\times_{\SCR_{\KK}} \{B\}$.  \Cref{prop:deforming animated algebras} implies that $\mm{def}_{B}$ is a formal moduli problem. In particular, $\mm{def}_{B}(\KK\oplus M)\simeq \Omega\mm{def}_{B}(\KK\oplus M[1])$ is a space for every coherent $\KK$-module $M$. By Artinian induction, this implies that $\mm{def}_{B}(A)$ is a space for all $A$, so that the result follows.
\end{proof}
\begin{notation}
Write $\Gamma(T_{B/\KK}[1])\in \LieAlg_{\KK}$ for the partition Lie algebra classified by the formal moduli problem $\mm{def}_B$. 
\end{notation}
One can identify the complex underlying $\Gamma(T_{B/\KK}[1])$ as follows:
\begin{construction}\label{con:sqz defs}
Given a map of coherent animated rings $f\colon \KK\to B$,  consider the functor $f^*\colon \Coh_{\KK, \geq 0}\to \Mod_B$ and define
$$
\cat{Def}_{B/\KK}^{\sqz} = \Coh_{\KK, \geq 0}\times_{\Mod_{B}} (\Mod_B)_{L_{B/\KK}[-1]/}
$$
to be the $\infty$-category of tuples $(M, \alpha)$ where $M\in \Coh_{\KK, \geq 0}$ and $\alpha\colon L_{B/\KK}[-1]\to f^*M$ is a derivation. This fits into a commuting square of $\infty$-categories
\begin{equation}\label{diag:linear defs}\begin{tikzcd}
\cat{Def}_{B/\KK}^\mm{sqz}\arrow[d]\arrow[r, "\sqz"] & \cat{Def}_{B/\KK}\arrow[d]\\
\Coh_{\KK, \geq 0}\arrow[r, "\triv"] & \Art_{\KK/\KK}
\end{tikzcd}\end{equation}
in which the vertical functors are left fibrations. The bottom horizontal functor sends $M\in \Coh_{\KK, \geq 0}$ to the trivial square zero extension $\KK\oplus M$ and the top horizontal map sends $(M, \alpha)$ to $(\KK\oplus M\to B\oplus_{\alpha} f^*M)$. This indeed defines a deformation of $B$ over $\KK\oplus M$: it is obtained as the pullback of the diagram of deformations 
$$\begin{tikzcd}
(\KK\to B)\arrow[r, "\alpha"] & (\KK\oplus M[1]\to B\oplus f^*M[1])& \arrow[l, "0"{swap}] (\KK\to B)
\end{tikzcd}$$
where the left map arises from the derivation $\alpha\colon B\to B\oplus f^*M[1]$ and the right map from the zero derivation. By the proof of  \Cref{prop:deforming animated algebras}, the fibre product then defines a deformation of $B$ as well.
\end{construction}
\begin{lemma}
Let $f\colon \KK\to B$ be a map of coherent animated rings. Then there is an equivalence of pro-coherent $\KK$-modules $\Gamma(T_{B/\KK}[1])\simeq f_*T_{B/\KK}[1]$.
\end{lemma}
\begin{proof}
By construction, the left fibration $\cat{Def}^{\sqz}_{B/\KK}\to \Coh_{\KK, \geq 0}$ is classified by the functor $\Coh_{\KK, \geq 0}\to \sS$ sending $M$ to the space of $B$-linear maps $\Map_{B}(L_{B/\KK}[-1], f^*M)$. By  \Cref{lem:pro-coh over coherent ring}, this functor is classified by a pro-coherent $\KK$-module; unravelling the definitions using Example \ref{ex:pro-coh dual} and  \Cref{def:procoh functoriality}, one sees that this pro-coherent $\KK$-module is precisely $f_*T_{B/\KK}[1]$.

It now suffices to verify that this pro-coherent $\KK$-module is equivalent to the module underlying the partition Lie algebra $\Gamma(T_{B/\KK}[1])$. In terms of formal moduli problems, this is equivalent to the square \eqref{diag:linear defs} being cartesian. Since the base change
$$\begin{tikzcd}
P=\cat{Def}_{B/\KK}\times_{\Art_{\KK/\KK}} \Coh_{\KK, \geq 0}\arrow[r] & \Coh_{\KK, \geq 0}
\end{tikzcd}$$
is a left fibration classified by a reduced excisive functor $\Coh_{\KK, \geq 0}\to \sS$, it suffices to verify the following: for each $M\in\Coh_{\KK, \geq 0}$, the induced map on fibres $\Map_B(L_{B/\KK}[-1], f^*M)\to P_M$ induces an equivalence on loop spaces at the canonical (zero) basepoint. This map of loop spaces can be identified with the canonical map
$$\begin{tikzcd}
\Map_B(L_{B/\KK}, f^*M)\arrow[r] & \Map_{\SCR_{\KK\oplus M/B}}\big(B\oplus f^*M, B\oplus f^*M\big)\simeq \Map_{\SCR_{\KK/B}}(B, B\oplus f^*M)
\end{tikzcd}$$
which is an equivalence by definition of the cotangent complex.
\end{proof}

\begin{remark}\label{rem:fmp slice}
Suppose that $\KK$ is coherent and let $\mf{g}\in \LieAlg_{\KK}$ be a partition Lie algebra. Write $X\colon \Art_{\KK/\KK}\to \sS$ for the formal moduli problem classified by $\mf{g}$ and $\int X\to \Art_{\KK/\KK}$ for its unstraightening. We then have fully faithful inclusions
$$\begin{tikzcd}
(\LieAlg_{\KK})_{/\mf{g}}\arrow[r, "\simeq"] & (\FMP_{\KK/\KK})_{/X}\arrow[r, hookrightarrow] & \Fun(\Art_{\KK/\KK}, \sS)_{/X}\arrow[r, "\simeq"] & \Fun(\int X, \sS).
\end{tikzcd}$$
The last equivalence follows from the fact that, under unstraightening, both $\infty$-categories are equivalent to the $\infty$ left fibrations over $\int X$. Unravelling the definitions, one sees that the essential image of the above inclusion consists of those functors $\int X\to \sS$ that preserve the terminal object, as well as every pullback square in $\int X$ whose image in $\Art_{\KK/\KK}$ is as in  \Cref{def:formal moduli}.
\end{remark}
\begin{construction}
Let $B$ be a coherent $\KK$-algebra. Given an object $A\to B_A$ in $\cat{Def}_{B/\KK}$, the map $B_A\to B$ is an almost Artinian extension. Indeed, it is the base change of the almost finitely presented map $A\to \KK$, hence almost finitely presented, and the kernel of $\pi_0(B_A)\to \pi_0(B)$ is generated by the images of the (nilpotent) generators of the kernel of $\pi_0(A)\to \pi_0(\KK)$. We thus obtain a functor $t\colon \cat{Def}_{B/\KK}\to \AArt_{\KK/B}$ given by $t(A\to B_A)=B_A$. Now suppose that $A_0\to A_{01}$ is a nilpotent extension and consider a pullback square in $\cat{Def}_{B/\KK}$
$$\begin{tikzcd}
(A\to B_A)\arrow[r]\arrow[d] & (A_0\to B_{A_0})\arrow[d]\\
(A_1\to B_{A_1})\arrow[r] & (A_{01}\to B_{A_{01}}).
\end{tikzcd}$$
Then $B_{A_0}\to B_{A_{01}}$ is a nilpotent extension as well, and the proof of  \Cref{prop:deforming animated algebras} shows that $B_A\simeq B_{A_0}\times_{B_{A_{01}}} B_{A_1}$. It follows that restriction along $t$ preserves formal moduli problems. Using Remark \ref{rem:fmp slice}, we therefore obtain a functor
$$\begin{tikzcd}
\Gamma\colon \LieAlgd_{B/\KK}\arrow[r, "\sim"] & \FMP_{B/\KK}\arrow[r, "t^*"] & \big(\FMP_{\KK/\KK}\big)_{/\mm{def}_B}\arrow[r, "\sim"] & \big(\LieAlg_\KK\big)_{/\Gamma(T_{B/\KK}[1])}.
\end{tikzcd}$$
\end{construction}
\begin{proposition}
Let $f\colon \KK\to B$ be a map of coherent animated rings. Then there is a commuting diagram of right adjoint functors
\begin{equation}\label{diag:lie forget}\begin{tikzcd}
\LieAlgd_{B/\KK}\arrow[d, "\forget"{swap}] \arrow[r, "\Gamma"] & (\LieAlg_{\KK})_{/\Gamma(T_{B/\KK}[1])}\arrow[d]\\
(\QC^\vee_B)_{/T_{B/\KK}[1]}\arrow[r, "f_*"] & (\QC^\vee_\KK)_{/\Gamma(T_{B/\KK}[1])}.
\end{tikzcd}\end{equation}
In particular, $\Gamma(T_{B/\KK}[1])\simeq f_*(T_{B/\KK}[1])$ as pro-coherent $\KK$-modules.
\end{proposition}
\begin{proof}
Consider the commuting square of $\infty$-categories
$$\begin{tikzcd}
\cat{Def}_{B/\KK}^{\sqz}\arrow[d]\arrow[r, "{t'}"] & \cAAPerf{B}{\KK}\arrow[d]\\
\cat{Def}_{B/\KK}\arrow[r, "t"] & \AArt_{\KK/B}.
\end{tikzcd}$$
where the top functor $t'$ sends $(M, \alpha\colon L_{B/\KK}[-1]\to f^*M)$ to $(f^*M,\alpha)$. By restriction, this induces a commuting square of $\infty$-categories of formal moduli problems and right adjoint functors between them. Using the equivalences from Example \ref{ex:def functor linear},  \Cref{thm:def functor vs Lie} and Remark \ref{rem:fmp slice}, this square of $\infty$-categories of formal moduli problems is equivalent to a commuting square of the form \eqref{diag:lie forget}. Indeed, the functor $\Gamma$ is equivalent to restriction along $t$ by definition, and one readily verifies that restriction along the functor $t'$ sending $(M, \alpha)$ to $(f^*M, \alpha)$ corresponds to $f_*$ at the level of pro-coherent modules.
\end{proof}

\newpage 

\section{Formally integrating  partition Lie algebroids}\label{sec:formal integration}
Let  $X$ be  a locally coherent qcqs derived  scheme over a coherent ground ring $\KK$.
In this section, we will formally integrate partition Lie algebroids $$ \mathfrak{g} \rightarrow T_{X/\KK}[1]$$ on locally coherent qcqs derived schemes $X$ over animated rings $\KK$.
The result of our integration procedure will be \textit{formal moduli stacks} $$ X \rightarrow Y$$ under $X$, where $Y$ can be thought of as a formal leaf space. 
More precisely,  we define:
	\begin{definition}[Formal moduli stacks] \label{DefModuliStack}
	A \textit{formal moduli stack} under $X$ is a map of $\KK$-prestacks $$ X \longrightarrow Y $$
	satisfying the following conditions:
	\begin{enumerate}
		\item $Y$ has deformation theory,  which means that it preserves limits of Postnikov towers and pullbacks along nilpotent extensions of $\KK$-algebras, see \Cref{hasdeftheory};
		\item $X\rightarrow Y $ is locally almost of finite presentation (`laft'), see \Cref{laftdef};
		\item $X \rightarrow Y$ is a nil-isomorphism. Equivalently, it restricts to an equivalence on all reduced affine schemes, see Remark \ref{rem:nil-iso laft}.
	\end{enumerate} 
Let $\modulistk_{X/\KK} \subset (\PrStk_\KK)_{X/}$ be the full subcategory spanned by all formal moduli \mbox{stacks under $X$.}
\end{definition} 

\begin{remark}
	Note that \Cref{DefModuliStack} also makes sense for $X$ a general prestack.
	\end{remark}

In \Cref{defvsstacks} below, we will link formal moduli stacks  
under \emph{affine} schemes to formal moduli problems in the sense of Definition \ref{def:formal moduli}, and hence to partition Lie algebroids. In \Cref{Formalintegration}, we will then prove our main integration equivalence in \Cref{mainintext}.  

\subsection{Formal moduli problems and formal moduli stacks}
\label{defvsstacks} Throughout this section, let $X=\Spec(B)$ be a coherent affine derived scheme over $\KK$. We will show:
\begin{theorem}\label{thm:FMP local to global}
Let $X=\Spec(B)$ be a coherent affine derived $\KK$-scheme. Then there is an equivalence between the $\infty$-category $\modulistk_{X/\KK}$ of formal moduli stacks under $X$ and the $\infty$-category of formal moduli problems
$$\begin{tikzcd}
F\colon \Art_{B/\KK}\arrow[r] & \sS
\end{tikzcd}$$
in the sense of Definition \ref{def:formal moduli}. 
If $X\rt Y$ is a formal moduli stack, then the corresponding formal moduli problem $\Art_{B/\KK}\rt \sS$ sends $A\rt B$ to the space of diagonal lifts
$$\begin{tikzcd}
X\arrow[r]\arrow[d] & Y\arrow[d]\\
\Spec(A)\arrow[r]\arrow[ru, dotted] & \Spec(\KK).
\end{tikzcd}$$
\end{theorem}
The proof closely follows (a simple version of) the argument in \cite{GR} and proceeds in two steps.
\begin{notation}\label{not:nilpotent cats}
Throughout, we write 
$$
X_{\prored}=\Spec(B_{\prored}), \qquad\qquad X_{\inf}:=(X/\KK)_{\inf}
$$
for the pro-reduction (\Cref{def:prored}) and infinitesimal stack of the affine derived scheme $X$. We will write $\sqcup$ for the pushout of (pro-)affine derived schemes (i.e.\ the fibre product of (ind-)rings) and consider the following categories:
\begin{enumerate}
\item Write $\Aff^{\mm{laft-nil}, +}_{X/\KK}$ for the $\infty$-category of affine $\KK$-schemes equipped with a nilpotent embedding $X\rt S$ that exhibits $X$ as an almost finitely presented $S$-scheme, such that $\mc{O}(S)\to \mc{O}(X)=B$ has eventually coconnective fibre. Equivalently, we simply have
$$
\Aff^{\mm{laft-nil}, +}_{X/\KK}=\Art_{B/\KK}^{\op}.
$$
In particular, each $S\in \Aff^{\mm{laft-nil}, +}_{X/\KK}$ is a coherent affine derived scheme as well (Corollary \ref{cor:artin remains coherent}).

\item Let $\Aff^{\mm{nil}, +}_{X/\KK}$ be the $\infty$-category of affine $\KK$-schemes equipped with a nilpotent embedding $X\rt S$ such that $\mc{O}(S)\to \mc{O}(X)$ has eventually coconnective fibre.

\item Let $\Aff_{/X_\mm{inf}}$ be the $\infty$-category of affine schemes over $X_{\mm{inf}}$, i.e.\ $S\in \Aff_{\KK}$ together with a map $S\hookleftarrow S_{\prored}\rt X$.
\end{enumerate}
\end{notation}
Our goal will be to first identify formal moduli problems with certain functors on the $\infty$-category $\Aff^{\mm{nil}, +}_{X/\KK}$. Next, we will relate functors on $\Aff^{\mm{nil}, +}_{X/\KK}$ with prestacks $X\to Y\to X_{\inf}$ in between $X$ and its infinitesimal prestack. The $\infty$-category of formal moduli stacks can then be identified with a natural subcategory of this $\infty$-category.
\begin{definition}
We will say that a functor $F\colon (\Aff^{\mm{nil}, +}_{X/\KK}\big)^\op\rt \sS$ \emph{has deformation theory} if:
\begin{enumerate}
\item its value on the initial object is contractible,
\item it sends pushouts along nilpotent embeddings to pullbacks of spaces.
\end{enumerate} 
\end{definition}
Let $u\colon \Aff^{\mm{laft-nil}, +}_{X/\KK}\hookrightarrow \Aff^{\mm{nil}, +}_{X/\KK}$ be the evident fully faithful embedding and consider the induced adjoint pair between $\infty$-categories of presheaves:
$$\begin{tikzcd}
u_!\colon \cat{P}\big(\Aff^{\mm{laft-nil}, +}_{X/\KK}\big)\arrow[r, hookrightarrow, yshift=1ex] & \cat{P}\big(\Aff^{\mm{nil}, +}_{X/\KK}\big)\colon u^*.\arrow[l, yshift=-1ex]
\end{tikzcd}$$
The main technical step in the proof of  \Cref{thm:FMP local to global} is the following:
\begin{proposition}\label{prop:LKE is deformation functor}
Let $F\colon \big(\Aff^{\mm{laft-nil}, +}_{X/\KK}\big)^{\op}\rt \sS$ be a formal moduli problem. Then the left Kan extension $u_!F\colon \big(\Aff^{\mm{nil}, +}_{X/\KK}\big)^{\op}\rt \sS$ has deformation theory.
\end{proposition}
The proof will require some preliminaries:
\begin{lemma}\label{lem:deformation functor artin case}
Let $F\colon \big(\Aff^{\mm{laft-nil}, +}_{X/\KK}\big)^{\op}\rt \sS$ be a formal moduli problem and consider a pushout diagram in $\Aff^{\mm{nil}, +}_{X/\KK}$
$$\begin{tikzcd}
S_0\arrow[d]\arrow[r, "f"] & S_0'\arrow[d]\\
S\arrow[r] & S'
\end{tikzcd}$$
such that $S_0$ and $S'_0$ are contained in $\Aff^{\mm{laft-nil}, +}_{X/\KK}$ and $S_0\rt S$ is a square zero extension. Then the induced map
$$\begin{tikzcd}
u_!F(S')\arrow[r] & u_!F(S)\times_{u_!F(S_0)} u_!F(S_0')
\end{tikzcd}$$
is an equivalence.
\end{lemma}
\begin{proof}
We will compare the fibres over a fixed $x\colon S'_0\rt u_!F$. Suppose $S_0\rt S$ is a square zero extension by an (necessarily eventually coconnective) ideal $I$ and let $S_{0, I[1]}$ be the trivial square zero extension of $S_0$ by its suspension. Since $S_0\in \Aff^{\mm{laft-nil}, +}_{X/\KK}$ is a coherent affine derived scheme (Corollary \ref{cor:artin remains coherent}), we can choose a filtered system $I_*\colon \cat{K}\rt \Coh(S_0)_{\geq 0}$ with $\colim_\alpha I_\alpha=I$.

Having fixed this, let $\cat{C}'$ be the $\infty$-category of factorizations
$$\begin{tikzcd}
S'_0\arrow[r, "="]\arrow[d] & S'_0\arrow[d, dotted]\arrow[rd, "x"]\\
S'\arrow[r, dotted] & T'\arrow[r, dotted] & u_!F
\end{tikzcd}$$
with $T'\in \Aff^{\mm{laft-nil}, +}_{X/\KK}$. A map is a natural transformation that restricts to the identity on the solid part of the diagram. The geometric realization of this $\infty$-category computes the fibre $u_!F(S')_{x}$. Likewise, let $\cat{C}$ be the $\infty$-category of factorizations 
$$\begin{tikzcd}
S_0\arrow[d]\arrow[r, "="] & S_0\arrow[d, dotted]\arrow[rd, "x|S_0"]\\
S\arrow[r, dotted] & T\arrow[r, dotted] & u_!F
\end{tikzcd}$$
which computes the fibre $u_!F(S)_{x}$. The restriction $u_!F(S')_{x}\rt u_!F(S)_{x}$ then arises as the realization of the functor $\pi\colon \cat{C}'\rt \cat{C}$ that restricts a factorization along $S\rt S'$.

We therefore have to prove that $\pi$ is a weak equivalence. To do this, we will replace $\cat{C}'$ and $\cat{C}$ be weakly equivalent categories $\cat{D}'$ and $\cat{D}$. First, let us consider the functor $\cat{K}\rt \Cat$ sending each $\alpha\in \cat{K}$ to the $\infty$-category $\cat{D}(\alpha)$ of diagrams of the form
\begin{equation}\label{diag:smaller square zero}\begin{tikzcd}
                                                   & {S_{0, I[1]}} \arrow[ld] \arrow[rr] \arrow[dd] &                                                           & S_0 \arrow[ld, equal] \arrow[dd] &      \\
{S_{0, I_\alpha[1]}} \arrow[rr, dotted] \arrow[dd] &                                                & S_0 \arrow[dd, dotted] \arrow[rrdd, "x|S_0", bend left] &                                                &      \\
                                                   & S_0 \arrow[ld, equal] \arrow[rr] &                                                           & S \arrow[ld, dotted]                           &      \\
S_0 \arrow[rr, dotted]                             &                                                & T \arrow[rr, dotted]                                              &                                                & u_!F
\end{tikzcd}
\end{equation}
where $T\in \Aff^{\mm{laft-nil}, +}_{X/\KK}$ and a morphism is a natural transformation which is the identity on the solid part of the diagram. This depends functorially on $\alpha$ by restriction along the map of trivial square zero extensions $S_{0, I[1]}\rt S_{0, I_\beta[1]}\rt S_{0, I_\alpha[1]}$ for each map of ideals $I_\alpha\to I_\beta$. We let $\cat{D}=\int_{\alpha\in \cat{K}} \cat{D}(\alpha)$ be the (cocartesian) unstraightening.

There is a natural functor $q\colon \cat{D}\rt \cat{C}$ restricting to the right face of the cube in \eqref{diag:smaller square zero}. This map is a cocartesian fibration (by postcomposition with $T\to T'$) whose fibre over a factorization $S\to T\to u_!F$ is given as follows. Note that the map $S\to T$ exhibits $S$ as a square zero extension of $S_0$ by $I$ \textit{relative to $T$}, i.e.\ it is classified by a map $L_{S_0/T}\to I[1]$. One can then identify
$$
q^{-1}(T) = \int_{\alpha\in \cat{K}} \Map_{\QC(S_0)/I[1]}\big(L_{S_0/T}, I_\alpha[1]\big)
$$
with the unstraightening of the functor sending each $\alpha$ to the space of lifts of the map $L_{S_0/T}\to I[1]$ to $I_\alpha[1]$. Since $S_0$ and $T$ are in $\Aff^{\mm{laft-nil}, +}_{X/\KK}$, the map $S_0\rt T$ is almost finitely presented (Corollary \ref{cor:maps of art are aft}). Using that $L_{S_0/T}$ is an almost perfect $S_0$-module, we find that
$$
|q^{-1}(T)|\simeq \colim_{\alpha\in K} \Map_{\QC(S_0)/I[1]}\big(L_{S_0/T}, I_\alpha[1]\big) \simeq \Map_{\QC(S_0)/I[1]}\big(L_{S_0/T}, \colim_\alpha I_\alpha[1]\big) = \ast.
$$
We conclude that $q\colon \cat{D}\rt \cat{C}$ is a weak equivalence.

Next, we define $\cat{D}'$ to the the $\infty$-category of diagrams of the form
\begin{equation}\label{diag:smaller square zero 2}\begin{tikzcd}
                                                   & {S_{0, I[1]}} \arrow[ld] \arrow[dd] \arrow[rr] &                        & S_0 \arrow[rr] \arrow[ld, equal] \arrow[dd] &                                                        & S'_0 \arrow[ld, equal] \arrow[dd] &      \\
{S_{0, I_\alpha[1]}} \arrow[dd] \arrow[rr, dotted] &                                                & S_0 \arrow[rr] &                                                           & S'_0 \arrow[dd, dotted] \arrow[rrdd, "x", bend left] &                                                 &      \\
                                                   & S_0 \arrow[ld, equal] \arrow[rr] &                        & S \arrow[rr]                                              &                                                        & S' \arrow[ld, dotted]                           &      \\
S_0 \arrow[rrrr, dotted]                           &                                                &                        &                                                           & T' \arrow[rr]                                          &                                                 & u_!F
\end{tikzcd}\end{equation}
where again, $ T'\in \Aff^{\mm{laft-nil}, +}_{X/\KK}$ and a morphism is a natural transformation which is the identity on the solid part of the diagram. Restriction to the right face determines a cocartesian fibration $q'\colon \cat{D}'\rt \cat{C}'$ whose fibre over $T'$ satisfies
$$|q'^{-1}(T')|\simeq \colim_{\alpha\in K} \Map_{\QC(S_0)/I[1]}\big(L_{S_0/T'}, I_\alpha[1]\big) \simeq \Map_{\QC(S_0)/I[1]}\big(L_{S_0/T'}, \colim_\alpha I_\alpha[1]\big) = \ast.
$$
We thus obtain a commuting square
$$\begin{tikzcd}
\cat{D}'\arrow[d, "{\pi'}"{swap}]\arrow[r] & \cat{C}'\arrow[d, "\pi"]\\
\cat{D}\arrow[r] & \cat{C}
\end{tikzcd}$$
where the horizontal functors are weak equivalences and $\pi'$ forgets the copies of $S_0'$ and $S'$ in Diagram \eqref{diag:smaller square zero 2}. It therefore remains to verify that $\pi'$ is a weak equivalence.

Now let $\cat{E}\hookrightarrow \cat{D}$ and $\cat{E}'\hookrightarrow \cat{D'}$ be the full subcategories of diagrams of the form \eqref{diag:smaller square zero}, resp.\ \eqref{diag:smaller square zero 2}, whose front face is a pushout in $\Aff^{\mm{laft-nil}, +}_{X/\KK}$. These inclusions admit left adjoints $L, L'$ that replace the object $T$ by the pushout of the front face; in particular, $\cat{E}$ and $\cat{E}'$ are weakly equivalent to $\cat{D}$ and $\cat{D}'$. Furthermore, there is a natural functor $\phi\colon \cat{E}\rt \cat{E}'$ sending each $T$ as in \eqref{diag:smaller square zero} to $T'=S'_0\sqcup_{S_0} T$, which fits naturally into a diagram of the form \eqref{diag:smaller square zero 2}. 

Now note that for any $T\in \cat{E}\subseteq \cat{D}$, there is a natural transformation $T\rt \pi'\circ\phi(T) = S'_0\sqcup_{S_0} T$. In particular, the map $\pi'\colon |\cat{D}'|\rt |\cat{D}|$ admits a right inverse (presented by $\phi$). Conversely, for every object $T'\in \cat{D}'$ as in \eqref{diag:smaller square zero 2}, there is a natural map 
$$
\phi\circ L\circ \pi'(T)=S'_0\sqcup_{S_{0, I_\alpha[1]}} S_0\rt T
$$
which shows that $\pi'\colon |\cat{D}'|\rt |\cat{D}|$ admits a left inverse (presented by $\phi\circ L$). We conclude that $\pi'$, and therefore the map $u_!F(S')_{x}\rt u_!F(S)_{x|S_0}$ is an equivalence.
\end{proof}
\begin{lemma}\label{lem:deformation functor half artin case}
Let $F\colon \big(\Aff^{\mm{laft-nil}, +}_{X/\KK}\big)^{\op}\rt \sS$ be a formal moduli problem and consider a pushout square
$$\begin{tikzcd}
S_0\arrow[d]\arrow[r, "f"] & S_0'\arrow[d]\\
S\arrow[r] & S'
\end{tikzcd}$$
in $\Aff^{\mm{nil}, +}_{X/\KK}$ such that $S'_0\in\Aff^{\mm{laft-nil}, +}_{X/\KK}$ and $S_0\rt S$ is a square zero extension. Then the induced map
$$\begin{tikzcd}
u_!F(S')\arrow[r] & u_!F(S)\times_{u_!F(S_0)} u_!F(S_0')
\end{tikzcd}$$
is an equivalence.
\end{lemma}
\begin{proof}
We fix $x\colon S\rt u_!F$ with $x_0=x\big|S_0$. Consider the $\infty$-category $\cat{D}$ of diagrams of the form
\begin{equation}\label{diag:fact1}\begin{tikzcd}
S_0\arrow[d]\arrow[r, dotted] & S'_0\arrow[d, dotted]\\
S\arrow[r, dotted]\arrow[rr, bend right=20, "x"{swap}] & S'\arrow[r, dotted, "x'"] & u_!F
\end{tikzcd}\end{equation}
where the square is cocartesian and $S_0'\in \Aff^{\mm{laft-nil}, +}_{X/\KK}$. A map is a natural transformation that is constant on the solid part of the diagram. Likewise, let $\cat{C}_0$ be the $\infty$-category of factorizations
$$\begin{tikzcd}
S_0\arrow[r, dotted]\arrow[rr, "x_0"{swap}, bend right=20] & S_0'\arrow[r, dotted, "x'_0"] & u_!F
\end{tikzcd}$$
where $S_0'\in \Aff^{\mm{laft-nil}, +}_{X/\KK}$. There is an evident right fibration $\pi\colon \cat{D}\rt \cat{C}_0$ sending $x'\mapsto x'\big|S'_0$. The lemma asserts that the fibres of $\pi$ are contractible.

By  \Cref{lem:deformation functor artin case}, the functor $\pi$ is also a left fibration. Furthermore, the $\infty$-category $\cat{C}_0$ is contractible by definition of $u_!F$ as a left Kan extension. Consequently, it will suffice to prove that the $\infty$-category $\cat{D}$ itself is contractible.

To see this, let $\cat{D}'$ be the $\infty$-category of diagrams of the form \eqref{diag:fact1} where the square is not necessarily cocartesian, but where $S'$ is contained in $\Aff^{\mm{laft-nil}, +}_{X/\KK}$. There is an evident functor $\pi\colon \cat{D}'\rt \cat{D}$ sending a factorization $x'\colon S'\rt u_!F$ with $S'\in \Aff^{\mm{laft-nil}, +}_{X/\KK}$ to $S\sqcup_{S_0} S'_0\to S'\to u_!F$. To prove the lemma it will thus suffice to show that:
\begin{enumerate}[label=(\alph*)]
\item $\cat{D}'$ is weakly contractible.
\item $\pi\colon \cat{D}'\rt \cat{D}$ is a weak equivalence.
\end{enumerate}
Let us start with (a): notice that restriction of the diagram \eqref{diag:fact1} to the bottom row defines a functor $\phi\colon \cat{D}'\rt \cat{C}'$ to the $\infty$-category of factorizations
$$\begin{tikzcd}
S\arrow[r, dotted]\arrow[rr, "x"{swap}, bend right=20] & S'\arrow[r, dotted, "x'"] & u_!F
\end{tikzcd}$$
with $S'\in \Aff^{\mm{laft-nil}, +}_{X/\KK}$. Note that $\cat{C}'$ is contractible by definition of $u_!F$ as a left Kan extension and that $\phi$ admits a right adjoint sending each factorization $S\to S'\to u_!F$ to
$$\begin{tikzcd}
S_0\arrow[d]\arrow[r, dotted] & S'\arrow[d, equal]\\
S\arrow[r, dotted]\arrow[rr, bend right=20, "x"{swap}] & S'\arrow[r, dotted, "x'"] & u_!F.
\end{tikzcd}$$
This implies that $\cat{D}'$ is weakly contractible.

For (b), pick a factorization $x'\colon S'=S\sqcup_{S_0} S'_0\rt u_!F$ in $\cat{D}$ and consider the comma $\infty$-category $x'/\cat{D}'$. Unraveling the definitions, one sees that $x'/\cat{D}'$ is just the $\infty$-category $\cat{D}'$ considered in our proof of (a), but with $S_0\rt S$ replaced $S'_0\to S'$. The proof of (a) then shows that $x'/\cat{D}'$ is contractible, so that $\pi$ is a weak equivalence.
\end{proof}
\begin{corollary}\label{cor:deformation functor}
Let $F\colon \big(\Aff^{\mm{laft-nil}, +}_{X/\KK}\big)^{\op}\rt \sS$ be a formal moduli problem and consider a pushout square in $\Aff^{\mm{nil}, +}_{X/\KK}$
$$\begin{tikzcd}
S_0\arrow[d]\arrow[r, "f"] & S_0'\arrow[d]\\
S\arrow[r] & S'
\end{tikzcd}$$
such that $S_0\rt S$ is a square zero extension. Then the induced map
$$\begin{tikzcd}
u_!F(S')\arrow[r] & u_!F(S)\times_{u_!F(S_0)} u_!F(S_0')
\end{tikzcd}$$
is an equivalence.
\end{corollary}
\begin{proof}
Let us fix a map $x\colon S'_0\rt u_!F$. It suffices to show that the induced map on fibres $u_!F(S')_{x}\rt u_!F(S)_{x|S_0}$ is an equivalence. Now $x$ factors as $S'_0\rt T_0\stackrel{t_0}{\rt} u_!F$ with $T_0\in \Aff^{\mm{laft-nil}, +}_{X/\KK}$. Letting $T=T_0\sqcup_{S_0} S$, we then have maps
$$\begin{tikzcd}
u_!F(T)_{t_0}\arrow[r] & u_!F(S')_{t_0|S'_0}\arrow[r] & u_!F(S)_{t_0|S_0}.
\end{tikzcd}$$
The left map and the composite are equivalences by  \Cref{lem:deformation functor half artin case}, so that the right map is an equivalence as well.
\end{proof}
\begin{proof}[Proof of  \Cref{prop:LKE is deformation functor}]
Every nilpotent embedding of eventually coconnective affine derived schemes can be decomposed into a finite sequence of square zero extensions, so this follows from Corollary \ref{cor:deformation functor}.
\end{proof}
\begin{definition}\label{def:laft def functor}
Let $f\colon G\rt G'$ be a natural transformation between functors $\big(\cat{Aff}_{X/\KK}^{\mm{nil}, +}\big)^{\op}\rt \sS$. We will say that $f$ is \textit{locally almost finitely presented} if for any $n\in \mathbb{N}$ and any cofiltered diagram $S_{*}\colon \cat{K}\rt \cat{Aff}_{X/\KK}^{\mm{nil}, +}$ with limit $S$ such that all $\mc{O}(S_\alpha)\to \mc{O}(X)=B$ have $(n-1)$-coconnective fibre, the map
$$
\colim_{\alpha\in \cat{K}^{\op}} G(S_\alpha)\rt G(S)\times_{G'(S)} \big(\colim_{\alpha\in \cat{K}^{\op}} G'(S_\alpha)\big)
$$
is an equivalence.
\end{definition}
\begin{remark}\label{rem:cofiltered as pushout}
Suppose that $S_{*}\colon \cat{K}\rt \cat{Aff}_{X/\KK}^{\mm{nil}, +}$ is a cofiltered diagram of affine schemes under $X$ such that all $\mc{O}(S_\alpha)\to \mc{O}(X)$ have $(n-1)$-coconnective fibre for some fixed $n$. Then
$$
S_\alpha\simeq S_{\alpha}^{\leq n}\sqcup_{X^{\leq n}} X
$$
where we write $X^{\leq n}=\Spec(\tau_{\leq n} B)$ for the corresponding truncated affine derived scheme.
\end{remark}
\begin{corollary}\label{cor:laft def functors}
The adjoint pair $(u_!, u^*)$ restricts to an equivalence between the full subcategories of:
\begin{enumerate}
\item formal moduli problems $F\colon \big(\cat{Aff}_{X/\KK}^{\mm{laft-nil}, +}\big)^{\op}\rt \sS$.
\item functors $G\colon \big(\cat{Aff}_{X/\KK}^{\mm{nil}, +}\big)^{\op}\rt \sS$ with deformation theory such that the canonical map $X\rt G$ is locally almost finitely presented.
\end{enumerate}
\end{corollary}
\begin{proof}
It will suffice to show that $X\to u_!(F)$ is indeed locally almost finitely presented. Assuming this, one has that $u_!$ is fully faithful because it takes the left Kan extension along the fully faithful functor $u$. A map $G\to G'$ between functors with deformation theory is an equivalence if and only if it for every square zero extension $X\to X_I$ with $I\in \QC(X)^+_{\geq 0}$, the induced map
$$\begin{tikzcd}
G(X_I) \simeq G(X)\times_{G(X_{I[1]})} G(X)\arrow[r] & G'(X)\times_{G'(X_{I[1]})} G'(X) = G'(X_I)
\end{tikzcd}$$
is an equivalence. If $X\to G$ and $X\to G'$ are both locally almost finitely presented, it suffices to verify this when $I\in \Coh(X)$, in which case $X_I\in \cat{Aff}_{X/}^{\mm{laft-nil}, +}$. It follows that $u^*$ detects equivalences such functors, so that $(u_!, u^*)$ restricts to the desired equivalence.

To see that $X\to u_!(F)$ is indeed locally almost finitely presented, let $S_{*}\colon \cat{K}\rt \cat{Aff}_{X/\KK}^{\mm{nil}, +}$ be a cofiltered diagram as in Definition \ref{def:laft def functor}, with limit $S$. By Remark \ref{rem:cofiltered as pushout}, we can assume that $S_\ast = S^{\leq n}_{\ast}\sqcup_{X^{\leq n}} X$ arises as the pushout of the cofiltered diagram of $n$-coconnective schemes $S_\alpha^{\leq n}$ under $X^{\leq n}$. Consequently, we have natural equivalences $S_\alpha\simeq S_{\alpha}^{\leq n}\sqcup_{S^{\leq n}} S$. We will need to prove that the square
$$\begin{tikzcd}
\colim_\alpha X(S_\alpha)\arrow[d]\arrow[r] & X(S)\arrow[d]\\
\colim_\alpha u_!F(S_\alpha)\arrow[r] & u_!F(S)
\end{tikzcd}$$
is cartesian. To this end, note that $u_!F(S)$ arises as the classifying space of the $\infty$-category $\cat{C}$ of factorizations $S\rt T\rt j_!F$ with $T\in \Aff^{\mm{laft-nil}, +}_{X/\KK}$. Likewise, for each $\alpha$ let $\cat{C}(\alpha)$ be the $\infty$-category of factorizations $S_\alpha\rt T\rt u_!F$. This determines a diagram of categories $\cat{K}\rt \Cat$ whose unstraightening
$$
\cat{C}'=\int_{\alpha\in \cat{K}} \cat{C}(\alpha)
$$
has the property that $|\cat{C}'|=\colim_\alpha u_!F(S_\alpha)$.

We now consider the functor $\pi\colon \cat{C}'\rt \cat{C}$ sending each $S_\alpha\to T\to u_!F$ to the composite $S\to S_\alpha\to T\to u_!F$. This is a cocartesian fibration (by postcomposition with $T\to T'$) whose fibre over $x=\big(S\to T\rt u_!F\big)\in \cat{C}$ is given by the $\infty$-category
$$
\pi^{-1}(x) = \int_{\alpha\in \cat{K}} \Map_{S/}\big(S_\alpha, T\big) \simeq \int_{\alpha\in \cat{K}} \Map_{S^{\leq n}/}\big(S^{\leq n}_\alpha, T\big).
$$
The classifying space of this $\infty$-category can be identified with $
|\pi^{-1}(x)| = \colim_{\alpha} \Map_{S^{\leq n}/}(S^{\leq n}_\alpha, T)$. For every map in $\cat{C}$ of the form 
$$
S\to T\stackrel{f}{\rt} T'\to u_!F
$$
the induced change-of-fibres map $f_*\colon \pi^{-1}(x)\rt\pi^{-1}(x')$ is then given at the level of classifying spaces by
$$\begin{tikzcd}
\colim_{\alpha} \Map_{S^{\leq n}/}(S^{\leq n}_\alpha, T)\arrow[r, "f\circ -"] & \colim_{\alpha} \Map_{S^{\leq n}/}(S^{\leq n}_\alpha, T').
\end{tikzcd}$$
This map is a weak equivalence since every $T\rt T'$ in $\Aff^{\mm{laft-nil}, +}_{X/\KK}$ is locally almost finitely presented (Corollary \ref{cor:maps of art are aft}). Consequently, given a point $x\in |\cat{C}|=u_!F(S)$, corresponding to a factorization $x\colon S\rt T\rt u_!F$, there is an equivalence
$$
\big(\colim_\alpha u_!F(S_\alpha)\big)\times_{u_!F(S)} \{x\}\simeq |\pi^{-1}(x)|.
$$
Applying this to each point $x\colon S\rt X\rt u_!F$, one sees that
$$
\big(\colim_\alpha u_!F(S_\alpha)\big)\times_{u_!F(S)} \{x\}\simeq \colim_\alpha \Map_{S/}(S_\alpha, X)\simeq \big(\colim_\alpha X(S_\alpha)\big)\times_{X(S)} \{x\}
$$
which yields the result.
\end{proof} 
Next, let $\PreSt_{X_{\inf}}^\mm{conv}$ be the $\infty$-category of convergent prestacks over $X_{\inf}$. Every $S\in \Aff^{\mm{nil}, +}_{X/\KK}$ in gives rise to a convergent prestack $S\to X_{\inf}$ over $X_{\inf}$, where the map is adjoint to the canonical map $S_{\prored} = X_{\prored} \to X$. Let us now consider the adjoint pairs
$$\begin{tikzcd}
v_!\colon \colon \cat{P}\big(\Aff^{\mm{nil}, +}_{X/\KK}\big)\arrow[r, yshift=1ex] & \big(\PreSt^\mm{conv}_{X_{\inf}}\big)\arrow[l, yshift=-1ex]\colon v^* & \tilde{v}_!\colon \cat{P}\big(\Aff^{\mm{nil}, +}_{X/\KK}\big)\arrow[r, yshift=1ex] & \big(\PreSt^\mm{conv}_{X_{\inf}}\big)_{X/}\arrow[l, yshift=-1ex]\colon \tilde{v}^*
\end{tikzcd}$$
where $v_!$ and $\tilde{v}_!$ are the unique colimit-preserving functors sending each nilpotent embedding $S\in \Aff^{\mm{nil}, +}_{X/\KK}$ to the diagrams of convergent prestacks $S\to X_{\inf}$ and $X\rt S\rt X_{\inf}$, respectively. At the level of the right adjoints, given a diagram of convergent $\KK$-prestacks $X\rt Y\rt X_{\inf}$, the presheaf $\tilde{v}^*Y$ sends each nilpotent embedding $X\hookrightarrow S$ in $\Aff^{\mm{nil}, +}_{X/\KK}$ to the space of dotted lifts
\begin{equation}\label{diag:v*}\begin{tikzcd}
X\arrow[r]\arrow[d, hookrightarrow] & Y\arrow[d]\\
S\arrow[r]\arrow[ru, dotted] & X_{\inf}
\end{tikzcd}\end{equation}
in $\PreSt^\mm{conv}_{\KK}$. Let us point out that, since $X\hookrightarrow S$ is a nilpotent embedding, this space is equivalent to the space of dotted lifts where we forget compatibility with the projection to $X_{\inf}$.
\begin{remark}\label{rem:pro-left adj}
The functor $v_!$ can be computed explicitly as follows: if $x\colon S\to X_{\inf}$ is a map from an \emph{eventually coconnective} affine derived scheme, then
$$
\Map_{/X_{\inf}}(S, v_!F) \simeq F(X\sqcup_{S_{\prored}} S)= \colim_\alpha F\big(S\amalg_{S_\alpha} X\big)
$$
with the colimit taken over the filtered $\infty$-category of nilpotent ideals in $\pi_0(\mc{O}(S))$. Indeed, both of these functors preserve colimits, as colimits of convergent prestacks are computed pointwise on eventually coconnective affine derived schemes. Furthermore, they are naturally equivalent for functors $F$ represented by $T\in \Aff^{\mm{nil}, +}_{X/\KK}$, because $\Map_{/X_{\inf}}(S, v_!T)$ is naturally equivalent to the diagram of dotted extensions in $\Aff_{\KK}$
$$\begin{tikzcd}
S_{\prored}\arrow[r, "x"]\arrow[hookrightarrow, d] & X\arrow[d, hookrightarrow] \\
S\arrow[r, dotted] & T.
\end{tikzcd}$$
\end{remark}
\begin{lemma}\label{lem:red1}
The adjoint pair $(\tilde{v}_!, \tilde{v}^*)$ restricts to an equivalence between the full subcategories of:
\begin{enumerate}
\item functors $F\colon\big(\Aff^{\mm{nil}, +}_{X/\KK}\big)^{\op}\rt \sS$ with deformation theory.
\item diagrams of nil-isomorphic convergent prestacks $X\rt Y\rt X_{\inf}$ where $Y$ has deformation theory.
\end{enumerate}
\end{lemma}
\begin{proof}
Let us first prove that $\tilde{v}_!$ and $\tilde{v}^*$ preserve these two subcategories. First, given $X\rt Y\rt X_{\inf}$ where $Y$ has deformation theory, the diagram \eqref{diag:v*} directly shows that $\tilde{v}^*(Y)$ sends pushouts in $\Aff^{\mm{nil}, +}_{X/\KK}$ along nilpotent embeddings to pullbacks of spaces and sends $X$ to a contractible space.

Next, we consider the functor
$$\begin{tikzcd}
\cat{P}\big(\Aff^{\mm{nil}, +}_{X/\KK}\big)\arrow[r, "\tilde{v}_!"] & \big(\PreSt^\mm{conv}_{X_{\inf}}\big)_{X/}\arrow[r, "\mm{forget}"] & \PreSt^\mm{conv}_{X_{\inf}}
\end{tikzcd}$$
and the unique colimit-preserving functor
$$\begin{tikzcd}
v_!\colon \cat{P}\big(\Aff^{\mm{nil}, +}_{X/\KK}\big)\arrow[r]  & \PreSt^\mm{conv}_{X_{\inf}}
\end{tikzcd}$$
that sends $X\hookrightarrow S$ to the map $S\to X_{\inf}$ adjoint to $S_{\prored}=X_{\prored}\to X$. Both functors preserve colimits indexed by contractible diagrams and are naturally equivalent on representables. Consequently, they are naturally equivalent on all presheaves arising as colimits of contractible diagrams of representables. Since every functor $F\colon \Aff^{\mm{nil}, +}_{X/\KK}\to \sS$ with deformation theory has $F(X)\simeq \ast$, it is such a colimit of a contractible diagram, i.e.\ we can identify $\tilde{v}_!(F)\rt X_{\inf}$ simply with $v_!(F)\rt X_{\inf}$. Using the formula from Remark \ref{rem:pro-left adj} and Lemma \ref{lem:convergent def thy}, one now readily verifies that if $F$ has deformation theory, then the convergent prestack $\tilde{v}_!(F)\simeq v_!(F)$ has deformation theory and is nil-isomorphic to $X_{\inf}$ (or equivalently, to $X$).

The adjoint pair $(\tilde{v}_!, \tilde{v}^*)$ therefore restricts to an adjoint pair between functors $F$ with deformation theory and nil-isomorphisms $X\rt Y\rt X_{\inf}$ where $Y$ has deformation theory. When $F$ has deformation theory, the unit $F\rt \tilde{v}^*\tilde{v}_! F$ is given at each nilpotent embedding $X\hookrightarrow S$ as follows. Write $x_0\colon X\to F$ for the canonical map and suppose that $\mc{O}(S)\to \mc{O}(X)$ has $(n-1)$-coconnective fibre for some $n$, so that $S=S^{\leq n}\sqcup_{X^{\leq n}} X$ is the pushout of the corresponding $n$-coconnective affine derived schemes. Then the unit map can be identified with
\begin{equation}\label{eq:v_!-v*-unit}\begin{tikzcd}
F(S)\simeq F(S^{\leq n}\sqcup_{X^{\leq n}} X)\arrow[r] & F(S^{\leq n}\amalg_{S_{\prored}} X)\times_{F(X^{\leq n}\amalg_{X_{\prored}} X)} \{\nabla^*x_0\}
\end{tikzcd}\end{equation}
where $f\colon X^{\leq n}\sqcup_{X_{\prored}} X\to X$ arises from the codiagonal map. Diagrammatically, this sends the dotted map $x\colon S\rt F$ in the following diagram to its restrictions:
$$\begin{tikzcd}
X^{\leq n}\amalg_{X_{\prored}}X\arrow[r, "\nabla"]\arrow[d] & X\arrow[rd, "x_0"]\arrow[d]\\
S^{\leq n}\amalg_{S_{\prored}}X\arrow[r] & S\arrow[r, dotted, "x"] & F
\end{tikzcd}$$
where $X^{\leq n}\amalg_{X_{\prored}}X$ is a pro-object in $\Aff^{\mm{nil}, +}_{X/\KK}$ using the inclusion of $X$ as the right component. Using that $X_{\prored}\rt S_{\prored}$ is an equivalence and that $S\simeq S^{\leq n}\sqcup_{X^{\leq n}} X$, one sees that the above square is a pushout square of pro-objects in $\Aff^{\mm{nil}, +}_{X/\KK}$. Consequently, the unit map \eqref{eq:v_!-v*-unit} is an equivalence, since $F$ had deformation theory.

On the other hand, for $X\stackrel{y_0}{\rt} Y\rt X_{\inf}$, the counit $\tilde{v}_!\tilde{v}^*Y\rt Y$ is given at each $x\colon S\to X_{\inf}$, where $S$ is an eventually coconnective affine derived scheme, by the map
\begin{equation}\label{eq:v_!-v*-counit}\begin{tikzcd}
\Map_{/X_{\inf}}\big(S\amalg_{S_{\prored}} X, Y\big)\times_{\Map_{/X_{\inf}}(X, Y)} \{y_0\}\arrow[r] & \Map_{/X_{\inf}}(S, Y).
\end{tikzcd}\end{equation}
In terms of diagrams, the domain can be identified with the space of dotted extensions in the following diagram
$$\begin{tikzcd}
S_\prored\arrow[r]\arrow[d] & X\arrow[r, "y_0"]\arrow[d] &  Y\arrow[d]\\
S\arrow[r] & S\sqcup_{S_\prored} X\arrow[ru, dotted, "y"{swap}] \arrow[r] & X_{\inf}.
\end{tikzcd}$$
The counit sends such $y$ to its restriction to $S$. Notice that the pushout $S\sqcup_{S_{\prored}} X$ arises as a pro-system of pushout squares $S\sqcup_{S_\alpha} X$, where each $S_\alpha\rt S$ is a nilpotent embedding. Since $Y$ has deformation theory, it follows that the map
$$\begin{tikzcd}
Y(X\sqcup_{S_\prored} S)\arrow[r] & Y(X)\times_{Y(S_\prored)} Y(S)
\end{tikzcd}$$
is an equivalence. Consequently, the counit map \eqref{eq:v_!-v*-counit} can be identified with the map
$$\begin{tikzcd}
Y(S)\times_{(Y(S_{\prored})\times_{X(S_{\prored})} \{x\})} \{y_0\big|_{S_{\prored}}\}\arrow[r] & Y(S)
\end{tikzcd}$$
projecting onto the first factor. The fact that $Y\rt X_{\inf}$ is a nil-isomorphism is equivalent to the fact that $Y(S_{\prored})\simeq X(S_{\prored})$, so that $Y(S_{\prored})\times_{X(S_{\prored})} \{x\}$ is contractible and the above map is indeed an equivalence as desired.
\end{proof}
Finally, we note:
\begin{lemma}\label{lem:form moduli stack over Xinf}
The functor $\mm{oblv}\colon \big(\PreSt_{/X_{\inf}}\big)_{X/}\to \PreSt_{X/\KK}$ sending each $X\to Y\to X_{\inf}$ to $X\to Y$ restricts to an equivalence between the full subcategories of:
\begin{enumerate}
\item Diagrams $X\to Y\to X_{\inf}$ where $Y$ has deformation theory and $X\to Y$ is a laft nil-isomorphism.

\item Formal moduli stacks $X\to Y$ in the sense of Definition \ref{DefModuliStack}.
\end{enumerate}
\end{lemma}
\begin{proof}
Note that if $f\colon X\to Y$ is locally almost finitely presented, then $f$ is a nil-isomorphism if and only if it restricts to an equivalence on reduced schemes, by Remark \ref{rem:nil-iso laft}. The forgetful functor then restricts to an equivalence because for any nil-isomorphism $X\to Y$, space $\Map_{X/}(Y, X_{\inf})$ is contractible, as all three functors are naturally equivalent on pro-reduced affine schemes $S_{\prored}$. 
\end{proof}
\begin{proof}[Proof of  \Cref{thm:FMP local to global}]
 \Cref{prop:LKE is deformation functor} and  \Cref{lem:red1} imply that the composite adjunction
$$\begin{tikzcd}
w_!\colon \cat{P}\big(\Aff^{\mm{laft-nil}, +}_{X/\KK}\big)\arrow[r, yshift=1ex, "u_!"]  & \cat{P}\big(\Aff^{\mm{nil}, +}_{X/\KK}\big) \arrow[r, yshift=1ex, "\tilde{v}_!"] \arrow[l, yshift=-1ex, "u^*"] & \big(\PreSt^{\mm{conv}}_{X_{\inf}}\big)_{X/}\arrow[l, yshift=-1ex, "\tilde{v}^*"]\colon w^*
\end{tikzcd}$$
restricts to an adjunction $(w_!^\mm{def}, w^*_{\mm{def}})$ between formal moduli problems on $\Aff^{\mm{laft-nil}, +}_{X/\KK}\simeq \Art_{B/\KK}^{\op}$ and nil-isomorphisms $X\rt Y\rt X_{\inf}$ where $Y$ has deformation theory. 

The resulting adjoint pair $(w_!^\mm{def}, w^*_{\mm{def}})$ has a fully faithful left adjoint: $u_!$ is fully faithful by Corollary \ref{cor:laft def functors} and $\tilde{v}_!$ was fully faithful on functors with deformation theory by  \Cref{lem:red1}. It therefore remains to show that the essential image of $w_!^\mm{def}$ consists precisely of nil-isomorphisms $X\rt Y\to X_{\inf}$ which exhibit $X$ as a laft $Y$-prestack. Lemma \ref{lem:form moduli stack over Xinf} then implies that $w^*_{\mm{def}}$ induces an equivalence between formal moduli stacks $X\to Y$ and formal moduli problems. Unravelling the definitions, one sees that $w^*_\mm{def}$ is precisely the functor from  \Cref{thm:FMP local to global}.

Let us first verify that $X\rt w_!^\mm{def}(F)$ is indeed laft, i.e.\ for each map $T\rt w_!^{\mm{def}}(F)$ from an eventually coconnective affine derived $\KK$-scheme $T$, the pullback $X\times_{w_!^{\mm{def}}(F)} T$ is a laft $T$-prestack. Note that $w_!^{\mm{def}}(F)$ is the colimit of a contractible diagram of nilpotent extensions $S\in \Aff^{\mm{laft-nil}, +}_{X/\KK}$, and that such colimits in $\PreSt^{\mm{conv}}_{\KK}$ are computed pointwise on eventually coconnective affine derived schemes. Hence it suffices to treat the case where $T\rt w_!^{\mm{def}}(F)$ fits into a diagram
$$\begin{tikzcd}
&  X\arrow[d, hookrightarrow]\arrow[rd] & \\
T\arrow[r] & S \arrow[r, "w_!(x)"] & w_!^{\mm{def}}(F)
\end{tikzcd}$$
where $X\to S$ is contained in $\Aff^{\mm{laft-nil}, +}_{X/\KK}$ and $x\in F(S)$. It therefore suffices to treat the case where $T=S$. In that case, 
$$
X\times_{w_!^{\mm{def}}(F)} S = w_!^\mm{def}\big(X\times_F S\big)
$$
arises as the image of a formal moduli problem over the base $S$, i.e.\ a formal moduli problem $\big(\Aff_{X/S}^{\mm{laft-nil}, +}\big)^{\op}\rt \sS$. This uses that $u_!$ preserves pullbacks of formal moduli problems (by Corollary \ref{cor:laft def functors}) and that $v_!$ preserves pullbacks by the formula from Remark \ref{rem:pro-left adj}.

Since $X\rt S$ is a nilpotent embedding which is almost finitely presented, each $S'\in \Aff_{X/S}^{\mm{laft-nil}, +}$ is almost finitely presented over $S$ (Corollary \ref{cor:maps of art are aft}). Consequently, the prestack $X\times_{w_!^{\mm{def}}(F)} T$ is a colimit of almost finitely presented affine derived $T$-schemes; this implies that it is a laft $T$-prestack, as required.

To conclude, it now suffices to verify that $w^*_{\mm{def}}$ detects equivalences between laft nil-isomorphisms $X\rt Y$ where $Y$ has deformation theory. This follows from the fact that $\tilde{v}^*$ detects weak equivalences by \Cref{lem:red1} and that $u^*$ detects weak equivalences by Corollary \ref{cor:laft def functors}.
\end{proof}

\subsection{Formal integration of partition Lie algebroids} 
\label{Formalintegration}

Let us again fix a coherent animated base ring $\KK$ and a locally coherent qcqs derived $\KK$-scheme $X$.
To prove our main theorem \Cref{mainintext}, we will proceed in two steps: first, we will show that the 
shifted tangent complex functor 
$$\begin{tikzcd}[row sep=0pc]
\modulistk_{X/\KK} \arrow[r] & (\QC^{\vee}_X)_{/T_{X/\KK}[1]}
\\
(X \rightarrow Y) \arrow[r, mapsto] & (T_{X/Y}[1] \rightarrow  T_{X/\KK}[1])
 \end{tikzcd}$$
is monadic; next, we will 
identify the resulting monad with the partition Lie algebroid monad constructed in \Cref{thm:pla scheme}.  
In our proof,  we will need to relate formal moduli stacks  on different schemes.

\begin{construction}[Functorial moduli stacks and their tangent fibres]\label{change_of_base_FMP}
	Let $X' \rightarrow X$ be a laft map of prestacks with deformation theory. Given a formal moduli stack $X \rightarrow Y$, consider the composite map $$X'\rightarrow X \rightarrow Y$$ and let $Y^{\wedge}_{X'}$ be the formal completion of $Y$ along  $X' \rightarrow Y$ as defined in Construction \ref{formal_completion}. Via the natural map $X' \rightarrow Y^{\wedge}_{X'}$, the prestack $Y^{\wedge}_{X'}$ becomes a formal moduli stack under $X'$ by \Cref{lem:formal completion obvious}. The construction $(X \rightarrow Y) \longmapsto (X'\rightarrow Y^{\wedge}_{X'})$ defines a functor 
	$$
	(-)^{\wedge}_{X'}\colon \modulistk_{X/\KK} \rightarrow \modulistk_{X'/\KK}.$$
	More formally, let us write
	$ \PrStk_{\KK}^{\mm{def}, \mm{laft-map}}$ for the $\infty$-category of $\KK$-prestacks with  deformation theory and laft maps between them.
	Consider the full subcategory 
	$$\modulistk_{\KK} \subset \Fun(\Delta^1, \PrStk_{\KK}^{\mm{def}, \mm{laft-map}})$$ whose objects are maps of $\KK$-prestacks $ X \rightarrow Y$ with $Y$ a formal moduli stack under $X$ (Definition \ref{DefModuliStack}) and whose morphisms are squares 
	$$\begin{tikzcd}
		X\arrow[r]\arrow[d] & Y\arrow[d]\\
		X'\arrow[r ] & Y'\end{tikzcd} 
	$$
	for which the vertical maps are laft and have deformation theory.
	Evaluation at $0$ defines a cartesian fibration 
	$$\modulistk_\KK \to \PrStk_{\KK}^{\mm{def}, \mm{laft-map}},$$ which we straighten to obtain a functor $(\PrStk_{\KK}^{\mm{def}, \mm{laft-map}})^{\op} \rightarrow  \Cat_{\infty}$. It  sends an object $X$ to $\modulistk_{X/\KK}$ and a morphism $X' \rightarrow X$ to the induced 
	$$	\begin{tikzcd}
	(-)^{\wedge}_{X'}\colon \modulistk_{X/\KK} \arrow[r] &  \modulistk_{X'/}\ ; \quad (X \rightarrow Y) \arrow[r, mapsto] &  (X'\rightarrow Y^{\wedge}_{X'}).
	\end{tikzcd}$$
	
	Construction \ref{tangent_complex_naive} and Proposition \ref{naturalityofT} then define (after suspending) the functorial shifted relative tangent complex  
	$$\begin{tikzcd}[row sep=0pc]
	T[1]\colon \modulistk_{\KK} \arrow[r] & \QC^{\vee}_{/T_{-/\KK}[1]}\\ \hspace{20pt} (X\rightarrow Y) \arrow[r, mapsto] & (T_{X/Y}[1] \rightarrow T_{X/\KK}[1])
	\end{tikzcd}$$
	This defines a map of cartesian fibrations by \Cref{naturalityofT} and \Cref{prop:tangent obvious properties}(4). In other words, we obtain functors $T_{X/-}[1]\colon \modulistk_{X/\KK}\to (\QC^\vee_X)_{/T_{X/\KK}[1]}$ that are natural in $X$ with respect to laft maps, i.e., each laft map 
	$f\colon X'\rightarrow X $ gives rise to a commutative square 
	$$\begin{tikzcd}[column sep=4em] 
		\modulistk_{X/\KK}  \arrow[r, "{T_{X/-}[1]}"] \arrow[d, "(-)^{\wedge}_{X}"{swap}] &  (\QC^\vee_{X})_{/T_{X/\KK}[1]}  \arrow[d, "f^{\sharp}"]\\
		\modulistk_{X'/\KK} \arrow[r, "{T_{X'/-}[1]}"{swap}] &  (\QC^\vee_{X'})_{/T_{X'/\KK}[1]}
	\end{tikzcd}$$ 
	where $f^{\sharp}(\mathcal{F}\rightarrow T_{X/\KK}[1]) \simeq f^{\ast}( \mathcal{F}) \times_{(f^*T^\mm{pre}_{X/\KK})^{\aft}[1]} T_{X'/\KK}[1]$. If $f$ is formally étale, then this simplifies to $f^\sharp(\mathcal{F}\rightarrow T_{X/\KK}[1])\simeq f^*\mathcal{F}$ by \Cref{prop:tangent obvious properties}(3).
\end{construction}

\begin{proposition}[Conservativity]\label{prop:conservativity}
Let $X$ be a prestack with deformation theory and $ Y_1 \rightarrow Y_2 $ a map of formal moduli stacks under $X$. If
$T_{X/Y_1} \rightarrow T_{X/Y_2}$ is an equivalence, then so is 
$Y_1 \rightarrow Y_2 $. 
\end{proposition}
\begin{proof}
Since $Y_1$ and $Y_2$ are both convergent, it suffices to verify that $Y_1(A)\to Y_2(A)$ is an equivalence for every eventually coconnective animated ring $A$. To see this, consider the ind-ring 
$$
A_{\prored}=``\colim_I" \pi_0(A)/I
$$
indexed by the filtered poset of nilpotent ideals of $\pi_0(A)$. We then have a commuting diagram
$$\begin{tikzcd}
& Y_1(A)\arrow[r]\arrow[d] & Y_2(A)\arrow[d]\\
X(A_{\prored})\arrow[r, "\sim"] & Y_1(A_{\prored})\arrow[r, "\sim"] & Y_2(A_{\prored}).
\end{tikzcd}$$
The bottom maps are equivalences since $X\to Y_1$ and $X\to Y_2$ are nil-isomorphisms and laft (Remark \ref{rem:nil-iso laft}). We therefore have to show that for each point $x\in X(A_{\prored})=\colim_I X(\pi_0(A)/I)$, the map on fibres $Y_1(A)_x\to Y_2(A)_x$ is an equivalence. To see this, it will suffice to show that for any nilpotent ideal $I\subseteq \pi_0(A)$ and $x\in X(\pi_0(A)/I)$, the map
\begin{equation}\label{eq:formally etale formal moduli}
Y_1(A)\times_{Y_1(\pi_0(A)/I)} \{x\}\rt Y_2(A)\times_{Y_2(\pi_0(A)/I)} \{x\}
\end{equation}
is an equivalence. Indeed, the map $Y_1(A)_x\to Y_2(A)_x$ is a filtered colimit of maps of the above form.

Now recall that $A\to \pi_0(A)/I$ decomposes as a finite sequence $A=A_n \to A_{n-1}\to \dots \to A_0=\pi_0(A)/I$ where each $A_k\to A_{k-1}$ is a square zero extension by an $A_0$-module $J_k$ (concentrated in a single degree, although we will not need this). To prove that \eqref{eq:formally etale formal moduli} is an equivalence, we will proceed by induction on this tower, the case of $A_0$ being evident. For the inductive step, note that there is a natural map of fibre sequences of spaces
$$\begin{tikzcd}
Y_1(A_k)\times_{Y_1(A_0)} \{x\}\arrow[r]\arrow[d] & Y_1(A_{k-1})\times_{Y_1(A_0)} \{x\} \arrow[r]\arrow[d] & Y_1(A_0\oplus J_k[1])\times_{Y_1(A_0)} \{x\}\arrow[d]\\
Y_2(A_k)\times_{Y_2(A_0)} \{x\}\arrow[r] & Y_2(A_{k-1})\times_{Y_2(A_0)} \{x\} \arrow[r] & Y_2(A_0\oplus J_k[1])\times_{Y_2(A_0)} \{x\}
\end{tikzcd}$$
since $Y_1$ and $Y_2$ have deformation theory. By inductive hypothesis the middle vertical map is an equivalence, so it suffices to verify that the right vertical map is an equivalence. Unravelling  \Cref{def:(co)tangent space} and using that the maps $X\to Y_i$ are laft, one sees that the right vertical map fits into a pullback square
$$\begin{tikzcd}
Y_1(A_0\oplus J_k[1])\times_{Y_1(A_0)} \{x\}\arrow[r]\arrow[d] & T_{X/Y_1, x}(J_k[2])\arrow[d]\\
Y_2(A_0\oplus J_k[1])\times_{Y_2(A_0)} \{x\}\arrow[r] & T_{X/Y_2, x}(J_k[2]).
\end{tikzcd}$$
The right vertical map is an equivalence by the assumption that $T_{X/Y_1}\to T_{X/Y_2}$ was an equivalence, so the result follows. 
\end{proof}

\begin{definition}[Zariski descent]
We say that a prestack $X \in \PrStk$ satisfies \textit{Zariski descent} if for all animated rings $A \in \CAlg^{\an}$ and all families of animated rings $$\{A\rightarrow A_i\} $$
	for which the maps $\Spec(A_i ) \rightarrow \Spec(A)$ are jointly surjective open immersions, we have 
	$$ X(A) \xrightarrow{\sim} \lim_{} X(S),$$
	where the limit is indexed by all $S \in \CAlg^{\an}_A$ such that $A\rightarrow S$ factors through some $A_i$.
\end{definition}

Theorem 7.3.5.2 in \cite{HTT} implies the following  `finite' characterisation of Zariski descent: 
\begin{proposition}
	A prestack $X \in \PrStk$ satisfies \textit{Zariski descent}	 if and only if the following conditions hold:
	\begin{enumerate}
		\item $	X	(0) \simeq \ast$ is contractible;
		\item Let us say that a square of animated rings
		$$\begin{tikzcd}
		A \arrow[d]\arrow[r] & 	A_1\arrow[d]\\
			A_2\arrow[r] & 	A_{12}
		\end{tikzcd}$$
		is a \emph{Zariski square} if it is cocartesian ($A_{12}\simeq A_1\otimes_A A_2$) and the maps $\Spec(A_i ) \rightarrow \Spec(A)$ are jointly surjective open immersions. Then $X$ sends each Zariski square to a pullback square.
	\end{enumerate}
\end{proposition}
\begin{remark}\label{rem:convergent zariski}
For any Zariski square, the induced square of $n$-truncations is a Zariski square as well, since $A\to A_1$ and $A\to A_2$ are flat. Consequently, each Zariski square is the limit of the Zariski squares of its $n$-truncations. A convergent prestack therefore satisfies Zariski descent if and only if it sends every Zariski square of eventually coconnective animated rings to a pullback.
\end{remark} 
\begin{proposition} \label{modulistacksdescent}
Let $X$ be a prestack with deformation theory that satisfies Zariski descent, and let $X\rightarrow Y$ be a formal moduli stack under $X$. Then the prestack $Y$ satisfies Zariski descent as well.
	\end{proposition}
	\begin{proof}
	By Remark \ref{rem:convergent zariski}, it suffices to verify that $Y(A)\simeq Y(A_1)\times_{Y(A_{12})} Y(A_2)$ for any Zariski square of eventually coconnective animated rings. Let us fix such a Zariski square, as well as a filtered system of nilpotent ideals $I\subseteq \pi_0(A)$ such that $\colim_I \pi_0(A)/I \cong A_{\red}$. For each such ideal, let us furthermore write $I_i=\pi_0(A_i)\otimes_{\pi_0(A)} I_i$. For $i=\emptyset, 1, 2, 12$, we then have that $\colim_I \pi_0(A_i)/I_i\cong A_i\otimes_A A_{\red}\cong (A_i)_{\red}$. Since $X\to Y$ is laft, we have a cartesian square
$$\begin{tikzcd}
\underset{I}{\colim} X(\pi_0(A_i)/I_i)\arrow[d]\arrow[r] & \underset{I}{\colim} Y(\pi_0(A_i)/I_i)\arrow[d]\\
X(A_{i, \red})\arrow[r] & Y(A_{i, \red})
\end{tikzcd}$$
for each $i=\emptyset, 1, 2, 12$. The bottom map is an equivalence because $X\to Y$ is a nil-isomorphism, so that the top map is an equivalence as well. 
 We therefore obtain a diagram of the form
$$\begin{tikzcd}[row sep=1pc]
Y(A)\arrow[d]\arrow[r] & Y(A_1)\times_{Y(A_{12})} Y(A_1)\arrow[d] \\
\underset{I}{\colim} Y(\pi_0A/I)\arrow[r,"\sim"] & \underset{I}{\colim} \big(Y(\pi_0A_1/I_1)\times_{Y(\pi_0A_{12}/I_{12})} Y(\pi_0A_2/I_2)\big)\\
\underset{I}{\colim} X(\pi_0A/I)\arrow[u, "\sim"] \arrow[r, "\sim"] & \underset{I}{\colim} \big(X(\pi_0A_1/I_1)\times_{X(\pi_0A_{12}/I_{12})} X(\pi_0A_2/I_2)\big)\arrow[u, "\sim"{swap}]
\end{tikzcd}$$
in which the bottom map is an equivalence because $X$ satisfies Zariski descent. Given a point $x\in \colim_I X(\pi_0(A)/I)$, we therefore need to check that the induced map on fibres $Y(A)_x\to Y(A_1)_{x_1}\times_{Y(A_{12})_{x_{12}}} Y(A_2)_{x_2}$ is an equivalence, where $x_i$ denotes the restriction of $x$ to the corresponding Zariski open. For this it suffices to show that for any nilpotent ideal $I$ and any point $x\in X(\pi_0(A)/I)$, the square
\begin{equation}\label{eq:descent fibre}\begin{tikzcd}
Y(A)\times_{Y(\pi_0A/I)} \{x\}\arrow[r]\arrow[d] & Y(A_1)\times_{Y(\pi_0A_1/I_1)} \{x_1\}\arrow[d]\\
 Y(A_2)\times_{Y(\pi_0A_2/I_2)} \{x_2\}\arrow[r] &  Y(A_{12})\times_{Y(\pi_0A_{12}/I_{12})} \{x_{12}\}
\end{tikzcd}\end{equation}
is cartesian. Indeed, one then finds that $Y(A)_x\simeq Y(A_1)_{x_1}\times_{Y(A_{12})_{x_{12}}} Y(A_2)_{x_2}$ by taking the filtered colimit over all $I$.

We now proceed as in the proof of  \Cref{prop:conservativity}: the map $A\to \pi_0(A)/I$ decomposes as a finite sequence $A=A(n) \to A(n-1)\to \dots \to A(0)=\pi_0A/I$ where each $A(k)\to A(k-1)$ is a square zero extension by an $A(0)$-module $J(k)$. Localising, we similarly have that $A_i\to \pi_0(A_i)/I_i$ decomposes as a finite sequence $A_i=A(n)_i\to \dots \to A(0)_i=\pi_0(A_i)/I_i$  of square zero extensions by the $A(0)_i$-modules $J(k)_i=J(k)\otimes_{A(0)} A(0)_i$. We will prove by induction on this tower that the square \eqref{eq:descent fibre} is cartesian; when $k=0$, all spaces are contractible and the assertion is evident.

For the inductive step, note that for each $i=\emptyset, 1, 2, 12$, there is a natural fibre sequence of spaces
$$\begin{tikzcd}[column sep=1.3pc]
Y(A(k)_i)\times_{Y(A(0)_i)} \{x_i\}\arrow[r] & Y(A(k-1)_i)\times_{Y(A(0)_i)} \{x_i\}\arrow[r] & Y(A(0)_i\oplus J(k)_i[1])\times_{Y(A(0)_i)} \{x_i\}
\end{tikzcd}$$
because $Y$ has deformation theory. By inductive hypothesis, the middle terms for $i=\emptyset, 1, 2, 12$ fit into a cartesian square. It therefore suffices to show that the right terms also form a cartesian square. Since $Y$ has deformation theory, there is a natural equivalence $$
Y(A(0)\oplus J(k)_i[1])\times_{Y(A(0))} \{x\}\simeq Y(A(0)_i\oplus J(k)_i[1])\times_{Y(A(0)_i)} \{x_i\}.
$$
and it suffices to verify that the square
$$\begin{tikzcd}
Y(A(0)\oplus J(k)[1])\times_{Y(A(0))} \{x\}\arrow[r]\arrow[d] & Y(A(0)\oplus J(k)_1[1])\times_{Y(A(0))} \{x\}\arrow[d]\\
Y(A(0)\oplus J(k)_2[1])\times_{Y(A(0))} \{x\}\arrow[r] & Y(A(0)\oplus J(k)_{12}[1])\times_{Y(A(0))} \{x\}
\end{tikzcd}$$
is cartesian. The functor $Y(A(0)\oplus -)\times_{Y(A(0))} \{x\}\colon \Mod_{A, \geq 0}\to \sS$ preserves all pullback diagrams preserved by the inclusion into $\Mod_A$, again because $Y$ has deformation theory. We are therefore left with showing that $J(k)[1]\simeq J(k)_1[1]\times_{J(k)_{12}[1]} J(k)_2[1]$. Unravelling the definitions, this is nothing but Zariski descent for (the quasi-coherent sheaf associated to) the $A(0)$-module $J(k)$.
	\end{proof}

\begin{theorem}[Zariski descent for formal moduli stacks] \label{modulizariski}
Given a derived scheme $X$, write $\mathfrak{U}$ for the poset of affine opens of $X$.
Then there is  canonical equivalence
$$\modulistk_{X/\KK} \xrightarrow{\ \sim \ } \lim_{U \in \mathfrak{U}^{{\op}}} \modulistk_{U/\KK}.$$
\end{theorem}
\begin{proof}
Let us consider the  cartesian fibration $$\modulistk_{ \mathfrak{U}, {\KK}} \rightarrow \mathfrak{U}$$ obtained by restricting the cartesian fibration $$\modulistk_{\KK} \rightarrow \PrStk_{\KK}^{\mm{def},\mm{laft-map}}$$ from \Cref{change_of_base_FMP}
along the evident functor $\mathfrak{U}\to \PrStk_{\KK}^{\mm{def},\mm{laft-map}}$.
By \cite[Proposition 3.3.3.1]{HTT}, there is an equivalence     $$\lim_{U \in \mathfrak{U}} \modulistk_{U/\KK} \simeq \Fun_{ \mathfrak{U}}^{\mm{cart}}(\mathfrak{U}, \modulistk_{ \mathfrak{U}, {\KK}} ), $$  where the right hand side denotes the $\infty$-category of  cartesian sections of $\modulistk_{\mathfrak{U},  {\KK}} \rightarrow \mathfrak{U}$.

We have a canonical functor $$\Psi\colon \modulistk_{X/\KK} \longrightarrow  \Fun_{ \mathfrak{U}}^{\mm{cart}}(\mathfrak{U}, \modulistk_{ \mathfrak{U}, \KK} )$$
sending $(X \rightarrow Y)$ to $$ \left\{U\longrightarrow Y^{\wedge}_U \right\}_{U \in \mathfrak{U}} .$$
Here the formal completion $Y^{\wedge}_U  \simeq Y \times_{Y^\sharp} U^\sharp$ is the prestack defined in \Cref{def:prored}, where we have used the notation $F^\sharp(-) := F(-_{\red})$.
Note that each map $U \rightarrow Y$ is laft  by \cite[Proposition 5.3.10]{lurie2004derived}, so $Y^{\wedge}_U$ \mbox{agrees with $(Y/U)_{\inf}$.}

The above functor is well-defined as each $U \longrightarrow Y^{\wedge}_U $  is a formal moduli stack under $U$.
Indeed, the fact that $Y^{\wedge}_U  \simeq (Y/U)_{\inf}$ has deformation theory follows immediately from the definitions.
Since $U \rightarrow Y$ is laft,  $U \rightarrow   Y^{\wedge}_{U}$ is laft by \Cref{lem:formal completion obvious} and clearly an equivalence on reduced objects.\vspace{3pt}
 
To construct an inverse  to $\Psi$, we  will need some notation. Write   $\PreSt_{\KK,<\infty}$
for the $\infty$-category of prestacks defined only on coconnective objects. Let 
$\Stk_\KK^{\Zar} \subset 	\PreSt_\KK$ ($\Stk_{\KK,<\infty}^{\Zar} \subset 	\PreSt_{\KK,<\infty}$) denote
the full subcategories  spanned by all prestacks $Y$  satisfying Zariski descent, i.e.\ satisfying $Y(0) \simeq \ast$ and sending Zariski squares (of eventually coconnective rings) to pullbacks.
By Remark \ref{rem:convergent zariski}, the functor 
$$\begin{tikzcd}
\conv\colon \PreSt_{\KK,<\infty} \arrow[r] & \PreSt_{\KK}; \quad X \arrow[r, mapsto] & \lim_n X(\tau_{\leq n}-)
\end{tikzcd}$$ 
sends sheaves to sheaves.

We then define $$\Phi\colon \Fun_{ \mathfrak{U}}^{\mm{cart}}(\mathfrak{U}, \modulistk_{ \mathfrak{U}, \KK} ) \longrightarrow  \modulistk_{X/\KK} $$ as the composite 
$$\Fun_{ \mathfrak{U}}^{\mm{cart}}(\mathfrak{U}, \modulistk_{ \mathfrak{U}, \KK} ) \subset \Fun (\mathfrak{U}, \Fun(\Delta^1,\Stk_{\KK,<\infty}^{\Zar}) )  \xrightarrow{\colim} 
\Fun(\Delta^1,\Stk_{\KK,<\infty}^{\Zar}) \xrightarrow{\conv}   \Fun(\Delta^1,\Stk_{\KK}^{\Zar}) . $$ 
 Note that $\colim$ is computed by sheafifying the pointwise colimit. We have used  \Cref{modulistacksdescent} to embed formal moduli stacks into Zariski sheaves.

We now want to verify that $\Phi$ takes values in formal moduli stacks under $X$. 
To this end, let us fix a family
$$\{U \rightarrow Y_U)\}_{U \in \mathfrak{U}} \in \Fun_{ \mathfrak{U}}^{\mm{cart}}(\mathfrak{U}, \modulistk_{ \mathfrak{U}, \KK} ) $$  
indexed by the poset $\mathfrak{U}$ of affine opens $U$ of $X$ and let $(X \rightarrow Y) \in 
\Fun(\Delta^1,\Stk_\KK^{\Zar})$ denote its image under $\Phi$. In a first step, we will prove that the natural map 
$$\{U \rightarrow Y_U\}_{U \in \mathfrak{U}} \rightarrow  \{U \rightarrow Y^{\wedge}_U\}_{U \in \mathfrak{U}}$$  
is an equivalence, i.e.\ that for each $U \in \mathfrak{U}$, the following square in $\Stk^{\Zar}_{\KK, <\infty}$
 is cartesian: 
$$\begin{tikzcd}
	Y_U\arrow[r]\arrow[d] & U^{\sharp}\arrow[d]\\
	\underset{U \in \mathfrak{U}}{\colim} Y_U \arrow[r] &  \big(\underset{U \in \mathfrak{U}}{\colim} Y_U\big)^{\sharp}.
\end{tikzcd}$$
Note that the functor $(-)^\sharp\colon \PrStk_{\KK, <\infty}\to \PrStk_{\KK, <\infty}$ preserves colimits and Zariski sheaves, and commutes with Zariski sheafification. For the last assertion, it suffices to observe that for any covering sieve $S\to \Spec(A)$ for the Zariski topology and a map $f\colon \Spec(B)\to \Spec(A)^\sharp$, corresponding to $f_{\red}\colon \Spec(B_{\red})\to \Spec(A_{\red})$, the base change $f^*S^\sharp\to \Spec(B)$ is a covering sieve as well: it is the unique covering sieve whose restriction to $\Spec(B_{\red})$ coincides with $f_{\red}^*S$. Consequently, the natural map of Zariski sheaves $\colim_{U\in \mathfrak{U}} Y^\sharp_U\to (\colim_{U\in \mathfrak{U}} Y_U)^\sharp$ is an equivalence. 

To see that the above square is cartesian, it then suffices to consider the natural transformation  $\alpha$  of $\mathfrak{U}$-indexed diagrams  in $\Stk^{\Zar}_{\KK, <\infty}$ given by $$ Y_U \rightarrow (Y_U)^{\sharp} $$ 
Write $\overline{\alpha}$ for the induced transformation between the  $\mathfrak{U}^{\rhd}$-indexed colimit cones.
The transformation $\alpha$ is cartesian as for each $U\rightarrow V$, we have $Y_U \simeq (Y_V)^{\wedge}_U \simeq  Y_V \times_{V^{\sharp}} U^{\sharp}$, which  implies by   \cite[Theorem 6.1.0.6]{HTT} that   the above square is also cartesian.

We will use this to show that $X\to Y$ is a formal moduli problem. The prestack $Y$ is convergent by construction and the map $X\rightarrow Y$ is laft by (the derived analogue of) \cite[17.4.3.3]{SAG}. To verify it has deformation theory, it suffices to check that for any pullback square of eventually coconnective animated rings 
$$\begin{tikzcd}
B'\arrow[r]\arrow[d] & B\arrow[d]\\
A' \arrow[r] &  A
\end{tikzcd}$$
with $B \rightarrow A $ a nilpotent extension and any $f\in Y(B)$ with image $f_A \in Y(A)$, the map $$ Y(\mathcal{O}_{B'}) \times_{Y(\mathcal{O}_{B}) } \{f\} \rightarrow Y(\mathcal{O}_{B} \otimes_{B'} A') \times_{Y(\mathcal{O}_{B} \otimes_{B} A) } \{f_A\} $$
of sheaves on the topological space $|\Spec(B)| \cong |\Spec(B')|$ is an equivalence on global sections. Taking global sections, this implies that the fibres $Y(B')\times_{Y(B)} \{f\}\to Y(A')\times_{Y(A)} \{f_A\}$ are equivalent, so that $Y(B')\simeq Y(B)\times_{Y(A)} Y(A')$.

Since $Y$ is a Zariski sheaf, it suffices to check this after restricting to  a \mbox{Zariski cover of $|\Spec(B)| $.} Let us pick a cover of $\Spec(B)$ by affine opens $V_i \subset \Spec(B)$ corresponding to affine opens \mbox{$V_{i, \red} \subset  \Spec(B_{\red})$} such that 
each composite $$f_{i, \red}\colon V_{i, \red} \subset  \Spec(B_{\red}) \rightarrow Y_\red \simeq X_{\red}$$ factors through some affine open $U_i \subset X$. This implies that each restriction $f_i\colon V_i \rightarrow Y$ of the map $f$ factors through some $Y^{\wedge}_{U_i}$.
The claim follows as  $Y^{\wedge}_{U_i}\simeq Y_{U_i}$ by the above argument, and $Y_{U_i}$ has deformation theory.

Finally, to see that $X\to Y$ is an equivalence on reduced affine schemes, we note that precomposing with $A \mapsto A_{\red}$ gives rise to a limit-preserving functor $\Fun(\mathrm{CR}^{\red},\mathcal{S}) \rightarrow \Fun(\CAlg^{\an},\mathcal{S}) $, $F \mapsto F(-_{\red})$ 
which sends sheaves to sheaves. 
Here $\mathrm{CR}^{\red} \subset \CAlg^{\an}$ denotes the full subcategory of reduced discrete rings.
Passing to left adjoints, we see that restriction to 
$\mathrm{CR}^{\red} $ commutes with sheafification. As each object $U \rightarrow Y_U$ induces an equivalence on reduced affine schemes, we deduce that $X=\colim_{U\in \mathfrak{U}} U \rightarrow \colim_{U\in \mathfrak{U}} Y_U= Y$ does so as well.

We may therefore conclude that  the functor $\Phi$ lands in formal moduli stacks, and the above argument also shows that $\Psi\circ \Phi  $ is equivalent to the identity.

To see that $\Phi \circ \Psi$ is equivalent to the identity, let us fix a formal moduli stack $X\rightarrow Y$ under $X$. It suffices to verify that the canonical map 
$$\colim_{U \in \mathfrak{U}} Y^{\wedge}_U  \rightarrow Y$$
is an equivalence of sheaves on eventually coconnective objects.
Since colimits in the $\infty$-topos $\Stk_{\KK,<\infty}^{\Zar} $ are universal in the sense of \cite[Definition 6.1.1.2]{HTT}, it suffices to check that for all maps $\Spec(A) \rightarrow Y$, the induced map 
$$\colim_{U \in \mathfrak{U}} \left(Y^{\wedge}_U  \times_{Y} \Spec(A)\right) \rightarrow  \Spec(A)$$
is an equivalence.
But $$Y^{\wedge}_U \times_{Y} \Spec(A) \simeq \Spec(A) \times_{Y^{\sharp}} U^{\sharp}\simeq  \Spec(A) \times_{X^{\sharp}} U^{\sharp}$$
and hence the above map is an equivalence since 
$ \colim_{U \in \mathfrak{U}} U ^\sharp  \rightarrow X^\sharp $
is an equivalence.
\end{proof}
\begin{proposition}[Monadicity of the tangent fibre]\label{monadicity} 
Let $X$ be a locally coherent qcqs derived scheme over an animated ring $\KK$.
Then the functor 
$$\begin{tikzcd}[row sep=0pc]
T_{X/-}[1]\colon \modulistk_{X/\KK} \arrow[r] & (\QC^{\vee}_X)_{/T_{X/\KK}[1]}\\
(X\rightarrow Y)\arrow[r, mapsto] & (T_{X/Y}[1] \rightarrow  T_{X/\KK}[1])
\end{tikzcd}$$
  in \Cref{change_of_base_FMP} is a monadic right adjoint.
\end{proposition}
The proof uses the following $\infty$-categorical result:
\begin{lemma}\label{lem:right adj into algebras}
Let $F\colon \cat{C}\leftrightarrows \cat{D}\colon U$ be a monadic adjunction and let $G\colon \cat{D}' \to \cat{D}$ be a functor from an $\infty$-category $\cat{D}' $ with geometric realisations. If $UG$ admits a left adjoint, then $G$ admits a left adjoint.
\end{lemma}
\begin{proof}
As in \cite[Lemma 4.7.3.13]{HA}, let $\cat{X}\subseteq \cat{D}$ be the full subcategory of objects $D\in \cat{D}$ such that $\Map_{\cat{D}}(D, G(-))$ is corepresentable and let $L\colon \cat{X}\to \cat{D}'$ be the corresponding partial left adjoint. For any simplicial diagram $D_\bullet$ in $\cat{X}$, the colimit $|L(D_\bullet)|$ corepresents $\Map_{\cat{D}}(|D_\bullet|, G(-))$, so that $\cat{X}$ is closed under geometric realisations. Since $UG$ admits a left adjoint, $\cat{X}$ contains all the free objects $F(C)$ with $C\in \cat{C}$. As every object in $\cat{D}$ is the geometric realisation of a ($U$-split) simplicial diagram of free objects, $\cat{X}=\cat{D}$.
\end{proof}
\begin{proof}[Proof of Proposition \ref{monadicity}]	
When $X=\Spec(B)$ is affine, the tangent functor $T_{X/-}[1]$ from \Cref{change_of_base_FMP} can be identified with the composite
$$\begin{tikzcd}[column sep=3pc]
T_{X/-}[1]\colon \modulistk_{X/\KK}\arrow[r, "\sim"] & \FMP_{B/\KK} \arrow[r, "{T_{B/}[1]}"] & (\QC^\vee_B)_{/T_{B/\KK}[1]}
\end{tikzcd}$$
where the first functor is the equivalence from Theorem \ref{thm:FMP local to global}, restricting formal moduli stacks to Artinian extensions of $B$, and the second functor takes the tangent complex of the corresponding formal moduli problem. It then follows from Proposition \ref{prop:fib anchor} and Theorem \ref{thm:def functor vs Lie} that $T_{X/-}[1]$ is a monadic right adjoint preserving sifted colimits.

	For non-affine $X$, let us first prove that $\modulistk_{X/\KK}$ has sifted colimits and that $T_{X/-}[1]$ preserves these. To this end, let us again write $\mathfrak{U}$ for the category of Zariski open affine subsets of $X$.
	By \Cref{naturalityofT}, have a commutative square 
	$$\begin{tikzcd}
		\modulistk_{X/\KK} \arrow[r]\arrow[d,"\simeq"{swap}] & (\QC^{\vee}_X)_{/T_{X/\KK}[1]}\arrow[d,"\simeq"]\\
		\lim_{U \in \mathfrak{U} }\modulistk_{U/\KK}\arrow[r] &  	\lim_{U \in \mathfrak{U} }	(\QC^{\vee}_U)_{/T_{U/\KK}[1]}
	\end{tikzcd}$$
	where the  vertical functors  are equivalences by \Cref{modulizariski} and \Cref{cor:descent}, respectively.
	Each functor $\modulistk_{U/\KK} \rightarrow (\QC^{\vee}_U)_{/T_{U/\KK}[1]}$ preserves sifted colimits, and this also holds true for the transition maps associated to inclusions $j\colon U\hookrightarrow V$ of Zariski open subsets of $X$
	$$\modulistk_{V/\KK} \xrightarrow{(-)^{\wedge}_U} \modulistk_{U/\KK}  \qquad (\QC^{\vee}_V)_{/T_{V/\KK}[1]} \xrightarrow{j^*} (\QC^{\vee}_U)_{/T_{U/\KK}[1]}.$$ 
 For the functor $(-)^{\wedge}_U$, we use that the functors $T_{U/}[1]\colon \modulistk_{U/\KK} \rightarrow (\QC^{\vee}_U)_{/T_{U/\KK}[1]}$ are conservative and hence \emph{detect} sifted colimits.
	It then follows from \cite[Corollary 5.3.6.10]{HTT} that $T_{X/-}[1]$ is also a sifted colimit preserving functor between $\infty$-categories with sifted colimits.
	
	Next, we will show that $T_{X/-}[1]$ admits a left adjoint. By Lemma \ref{lem:right adj into algebras}, it suffices to verify that
	$$\begin{tikzcd}
	\modulistk_{X/\KK} \arrow[r, "{T_{X/-}[1]}"] & (\QC^\vee_X)_{/T_{X/\KK}[1]}\arrow[r, "\mm{fib}"] & \QC^\vee_X
	\end{tikzcd}$$
	admits a left adjoint, since taking the fibre is a monadic right adjoint. This follows immediately from Proposition \ref{prop:left adjoint via sqz ext}, using that the left adjoint $F\mapsto X_F$ constructed there takes values in the full subcategory of formal moduli stacks since $X\to X_F$ is a nil-isomorphism.
	
	The claim now follows from the Barr--Beck--Lurie theorem \cite[Theorem 4.7.3.5]{HA}
as the tangent fibre is conservative by \Cref{prop:conservativity}.
\end{proof}

\begin{proposition}[Identification of the monad] \label{monadident}
The right adjoint in \Cref{monadicity} induces the partition Lie algebroid monad from \Cref{def:pla scheme} on the $\infty$-category $(\QC^{\vee}_X)_{/T_{X/\KK}[1]}$.
\end{proposition} 
\begin{proof}
Consider the commutative diagram
$$\begin{tikzcd}
\modulistk_{X/\KK}\arrow[r, "\sim"]\arrow[d, "{T_{X/-}[1]}"{swap}] & \underset{U\in \mathfrak{U}}{\lim} \modulistk_{U/\KK}\arrow[rd, "{T_{U/-}[1]}"{swap}]\arrow[r, "\sim"] & \underset{U\in \mathfrak{U}}{\lim} \FMP_{\mc{O}(U)/\KK}\arrow[d, "{T_{\mc{O}(U)/}[1]}"] \arrow[r, "\sim"] & \underset{U\in \mathfrak{U}}{\lim} \LieAlgd_{\mc{O}(U)/\KK)}\arrow[ld, "\mm{oblv}"]\\
(\QC^\vee_X)_{/T_{X/\KK}[1]}\arrow[rr, "\sim"] & & \underset{U\in \mathfrak{U}}{\lim} (\QC^\vee_U)_{/T_{U/\KK}[1]} 
\end{tikzcd}$$ 
	where the left square is defined as in   \Cref{monadicity}  and the two triangles arise from \Cref{thm:FMP local to global} and \Cref{thm:def functor vs Lie}. The claim then follows from the construction of the monad $\Lie_{\Delta, X/\KK}^{\pi}$ in the proof of  \Cref{thm:pla scheme}.
\end{proof}

We deduce:
 \begin{theorem}[Main theorem] \label{mainintext} 
	Let $X$ be a locally coherent qcqs derived $\KK$-scheme over an animated ring $R$.
	\begin{enumerate}
		\item 
		The shifted tangent fibre functor 
		$$\begin{tikzcd}[row sep=0pc]
		\modulistk_{X/R}\arrow[r] & (\QC^{\vee}_X)_{T_{X/\KK}[1]}\\
		(X \rightarrow Y)\arrow[r, mapsto] & (T_{X/Y}[1] \rightarrow  T_{X/\KK}[1])
		\end{tikzcd}$$
				lifts to an equivalence of $\infty$-categories $$T_{X/-}[1]\colon \modulistk_{X/R} \xrightarrow{\simeq} \LieAlgd_{X/\KK}.$$
		Hence every partition Lie algebroid on $X$ integrates uniquely to a  formal moduli \mbox{stack under $X$.}
		
		\item 	This equivalence is functorial with respect to almost finitely presented maps. More precisely, if 
		$X \rightarrow Y $ is a locally almost finitely presented map, there is a commuting diagram
		$$\begin{tikzcd}[column sep=3pc]
			\modulistk_{Y/}  \arrow[r, "{T_{Y/-}[1]}"] \arrow[d, "(-)^{\wedge}_{X}"{swap}] & \Lie_{\Delta,Y/R}^{\pi} \arrow[d, "f^{\sharp}"]\\
			\modulistk_{X/-} \arrow[r, "{T_{X/-}[1]}"{swap}] &  \Lie_{\Delta,X/R}^{\pi}
		\end{tikzcd}$$ 
		where $f^\sharp(\mf{g}\to T_{Y/\KK}[1]) = f^*\mf{g}\times_{(f^*L_{Y/\KK})^\vee[1]} T_{X/\KK}[1]$ as pro-coherent sheaves.
	\end{enumerate}
\end{theorem}

\appendix
\section{Derived algebra recollections}\label{sec:deralg}
The goal of this section is to introduce some elements of (homological) algebra that will be used throughout the text. We will start in Section \ref{sec:comm} by reviewing the theory of derived commutative rings, the non-connective analogues of animated (or simplicial) rings. We will work in a somewhat general setting that also includes the derived filtered and graded rings studied in Section \ref{section:affineKD}. Section \ref{sec:procoh affine} discusses the theory of \emph{pro-coherent modules} over an animated ring.

\subsection{Derived commutative rings}\label{sec:comm}
To set the stage, we will briefly recall the method of right-left extending monads from \cite[Section 3.2]{BM19} and \cite[Section 2.2]{BCN21} in the generality needed for us, see also \cite[Section 4.2]{R20} for the case of the commutative algebra monad.

\begin{definition}\label{def:proj gend}
Let $\cat{C}$ be a presentable stable $\infty$-category with a left complete  $t$-structure $(\cat{C}_{\geq 0}, \cat{C}_{\leq 0})$. We call $\cat{C}$ is \textit{stably projectively generated} if there exists a full additive subcategory $\Vect^{\ft}(\cat{C})\subseteq \cat{C}_{\geq 0}$ whose objects are compact in $\cat{C}$ and projective in $\cat{C}_{\geq 0 }$, generating $\cat{C}_{\geq 0}$ under sifted colimits. 

In this situation, there is an equivalence 
$$
\cat{C}\simeq \Fun^{\oplus}(\Vect^{\ft}(\cat{C})^{\op}, \Sp)
$$
identifying $t$-structures \cite[Remark 4.2.2]{R20}. We define the subcategory of \textit{flat objects} in $\cat{C}$ as the ind-completion of $\Vect^{\ft}(\cat{C})$:
$$
\Flat(\cat{C}):=\mm{Ind}(\Vect^{\ft}(\cat{C}))\subseteq \cat{C}.
$$
\end{definition}
\begin{example}
For $A \in \Alg(\Sp)$ 
a connective  algebra object in spectra, the stable $\infty$-category $\Mod_A$ of   (left) $A$-module spectra with its usual $t$-structure is stably projectively generated. In this case, we will take $\Vect^{\ft}_A$ to be the $\infty$-category consisting of finite direct sums of 
 copies of $A$. Then $\Flat_{A}$ is the $\infty$-category of flat $A$-modules in the sense of \cite[Definition 7.2.2.10]{HA}, by Lurie's version of Lazard's theorem (cf.\  Theorem 7.2.2.15 of [op.cit]).
\end{example} 
\begin{definition}\label{def:derived functor}
Let $\cat{C}$ be a projectively generated stable $\infty$-category with $t$-structure  $(\cat{C}_{\geq 0}, \cat{C}_{\leq 0}) $
and $\cat{D}$ a presentable $\infty$-category. We will say that a functor $F\colon \cat{C}\rt \cat{D}$ is \textit{right-left extended} if it preserves sifted colimits and  totalisations of finite cosimplicial diagrams with values in $\Vect^{\ft}(\cat{C})\subseteq \cat{C}$. 
\end{definition}
The following follows from \cite[Remark 2.45]{BCN21} and the definition of $\Flat(\cat{C})$ as an ind-completion: 
\begin{proposition}
 The restriction functor
$$\begin{tikzcd}
\Fun_\mm{RL}(\cat{C}, \cat{D})\arrow[r] & \Fun^{\omega}(\Flat(\cat{C}), \cat{D})\arrow[r, "\sim"] & \Fun(\Vect^{\ft}(\cat{C}), \cat{D})
\end{tikzcd}$$
is fully faithful. Here $\Fun_\mm{RL}(\cat{C}, \cat{D})\subseteq \Fun(\cat{C}, \cat{D})$ denotes the full subcategory of right-left extended functors and $\Fun^{\omega}(\Flat(\cat{C}), \cat{D}) \subset \Fun(\Flat(\cat{C}), \cat{D})$ is the full subcategory of functors preserving filtered colimits.
\end{proposition} 
 
\begin{definition}
Let us say that a functor $F\colon \Vect^{\ft}(\cat{C})\rt \cat{D}$, or equivalently, a filtered-colimit preserving functor $F\colon \Flat(\cat{C})\rt \cat{D}$, \textit{admits a right-left extension} or is \textit{right-left extendable}, if it is contained in the essential image. In that case, we will refer to its (unique) inverse image $\LL F\colon \cat{C}\rt \cat{D}$ as the \textit{derived functor}, or \textit{right-left extension}, of $F$. Note that this coincides with the  non abelian derived functor in the sense of Dold--Puppe on the connective part $\cat{C}_{\geq 0}\simeq \cat{P}_\Sigma(\Vect^{\ft}(\cat{C}))$. 
\end{definition}
\begin{example}\label{ex:polynomial}
When $\cat{D}$ is stable, every functor $F\colon \Vect^{\ft}(\cat{C})\rt \cat{D}$ which  is a filtered colimit of finite degree functors is right-left extendable. See \cite[Corollary 2.49]{BCN21} for details.
\end{example} 
\begin{example}\label{ex:derived monad}

Let $\End_{\mm{RL}}^{ \Flat(\cat C)}(\cat C) \subset \End_{\mm{RL}}(\cat C)$ be the full subcategory of endofunctors $\cat C \rightarrow \cat C$ which preserve the full subcategory $\Flat(\cat C) \subset \cat C$. Restriction $\End^{\Flat(\cat C)}_{\mm{RL}}(\cat C) \rt \End^{\omega}(\Flat(\cat C))$ is a monoidal equivalence onto the full subcategory of $\End^{\omega}(\Flat(\cat C))$ spanned by right-left extendable functors. Consequently, any right-left extendable and filtered-colimit preserving monad $T\colon \Flat(\cat{C})\rt \Flat(\cat{C})$ induces a \textit{derived monad} $\LL T\colon \cat{C}\rt \cat{C}$.
\end{example}
As an example of this kind of derived monad,  consider the derived symmetric algebra monad. To define this,  assume that $\cat{C}$ is a derived algebraic context in the sense of \cite[Section 4.2]{R20}:
\begin{definition}\label{def:alg context}
We will say that a stably projectively generated $\infty$-category $\cat{C}$ is a \emph{derived algebraic context} if it comes equipped with a closed symmetric monoidal structure such that:
\begin{enumerate}
\item $\Vect^{\ft}(\cat{C})$ is stable under tensor products and contains the monoidal unit.
\item $\Vect^{\ft}(\cat{C})$ is an ordinary category.
\item For each $M\in \Vect^{\ft}(\cat{C})$, the object $\sym^n(M)=\tau_{\leq 0}((M^{\otimes n}_{h\Sigma_n}))$ is contained in $\Vect^{\ft}(\cat{C})$. 
\end{enumerate}
\end{definition}
Condition (1) implies that $\cat{C}_{\geq 0}$ is stable under the tensor product. We include condition (2) for simplicity; it implies that $\Flat(\cat{C})$ is an ordinary category and that $\cat{C}^{\heartsuit}$ is a symmetric monoidal abelian category with enough projectives. Extending the functors $\sym^n\colon \Vect^{\ft}(\cat{C})\to \Vect^{\ft}(\cat{C})$ by filtered colimits and taking the sum over all $n\geq 0$, one then obtains a monad $\sym\colon \Flat(\cat{C})\to \Flat(\cat{C})$ preserving filtered colimits. Since each $\sym^n\colon \Vect^{\ft}(\cat{C})\to \Vect^{\ft}(\cat{C})$ is of degree $\leq n$, this admits a derived monad $\LL\sym\colon \cat{C}\to \cat{C}$. 
\begin{definition}\label{def:derived comm}
If $\cat{C}$ is a derived algebraic context, we will refer to $\LL\sym$-algebras as \emph{derived commutative algebras} and to connective $\LL\sym$-algebras as \emph{animated commutative algebras} in $\cat{C}$, and write $\SCR(\cat{C}_{\geq 0})\subseteq \DAlg(\cat{C})$ for the corresponding $\infty$-categories. 
\end{definition}
The (ordinary) category $\DAlg(\cat{C})^{\heartsuit}$ of discrete derived algebras is equivalent to the category of commutative algebras in $\cat{C}^{\heartsuit}$.
\begin{example}\label{ex:derived rings}
Let $\cat{C}=\Mod_{\ZZ}$ with the usual $t$-structure, so that $\Vect^{\ft}(\cat{C})=\Vect^{\ft}_{\ZZ}$ is the (ordinary) category of finite free abelian groups and $\Flat(\cat{C})=\Flat_{\ZZ}$ is the (ordinary) category of flat abelian groups. This is a derived algebraic context, and the resulting monad $\sym\colon \Flat_{\ZZ}\to \Flat_{\ZZ}$ is the usual symmetric algebra monad. One thus obtains the $\infty$-categories $\SCR=\SCR(\Mod_{\ZZ})$ and $\DAlg=\DAlg(\Mod_{\ZZ})$ of animated and derived rings.
\end{example}
\begin{example}\label{ex:diagrams}
Let $\cat{C}$ be a derived algebraic context and let $\cat{I}$ be a small symmetric monoidal (ordinary) category. Then $\Fun(\cat{I}, \cat{C})$ is a derived algebraic context as well with respect to the Day convolution product and the $t$-structure in which a diagram is (co)connective if it is pointwise (co)connective in $\cat{C}$.
\end{example}
\begin{remark}\label{rem:coproduct derived algebras}
There is a natural equivalence $\LL\sym(V)\otimes \LL\sym(W)\simeq \LL\sym(V\oplus W)$ for all $V, W\in \cat{C}$: this follows by right-left extension, using that there is such a natural equivalence for $V, W\in \Flat(\cat{C})$ by construction. Consequently, the coproduct of two derived algebras is simply given by their tensor product. Since $\DAlg(\cat{C})\to \cat{C}$ preserves sifted colimits by construction, it follows that pushouts of derived algebras are computed as relative tensor products, as usual. 
\end{remark}
\begin{remark}\label{rem:Eoo}
For every $M\in \Vect^{\ft}(\cat{C})$, there is a natural map $(M^{\otimes n})_{h\Sigma_n}\to \sym^n(M)$ from the coinvariants in $\cat{C}$. By right-left extension, this induces a map $\free_{\EE_\infty}\to \LL\sym$ from the $\EE_\infty$-monad. In particular, any derived commutative algebra has an underlying $\EE_\infty$-algebra and a module over a derived commutative algebra is simply a module over its underlying $\EE_\infty$-algebra.
\end{remark} 
\begin{definition}
For $A\in \DAlg(\cat{C})$, we define the $\infty$-category of \emph{derived (commutative) $A$-algebras} to be the under-category
$$
\DAlg_{A}(\cat{C})=\DAlg(\cat{C})_{A/}.
$$
The forgetful functor $\DAlg_A(\cat{C})\to \Mod_A(\cat{C})$ is monadic and we will write $\LL\sym_A\colon \Mod_A(\cat{C})\to \Mod_A(\cat{C})$ for the associated monad.  
\end{definition}
\begin{remark}\label{rem:derived algebras over animated}
If $A\in \DAlg(\cat{C})$ is connective, then $\Mod_A(\cat{C})$ is stably projectively generated for the $t$-structure inherited from $\cat{C}$. The monad $\LL\sym_A$ is then derived from the restricted monad $\LL\sym_A\colon \Flat(\Mod_A(\cat{C}))\to \Flat(\Mod_A(\cat{C}))$.
\end{remark}
The usual cotangent complex formalism for animated rings carries over to this setting. Below we will recall one way to see this, using an abstract method for constructing functors out of the $\infty$-category of $T$-algebras over some monad $T$ (see also \cite{holeman2023derived}). 
\begin{proposition}\label{prop:left extending functors on algebras}
Let $\cat{C}$ be a stably projectively generated stable $\infty$-category. Suppose that $T\colon \Flat(\cat{C})\rt \Flat(\cat{C})$ is a filtered colimit preserving and right-left extendable monad, and let $\LL T\colon \cat{C}\rt \cat{C}$ be the derived monad $T$. Let $\mm{Kl}_{T}(\Flat(\cat C))$ denote the full subcategory of $\mm{Alg}_T(\Flat(\cat{C}))$ spanned by the free algebras. Assume that
$$\begin{tikzcd}
F\colon \mm{Kl}_T(\Flat(\cat{C}))\arrow[r] & \cat{D}
\end{tikzcd}$$
is a functor such that
\begin{enumerate}
\item the composite $F\circ \mm{Free}_T\colon \Flat(\cat{C})\rt \cat{D}$ preserves filtered colimits and is right-left extendable,
\item $F$ preserves geometric realizations of simplicial objects in $\mm{Kl}_{T}(\Flat(\cat C))$ that become split after applying $\mm{Forget}_T$.
\end{enumerate}
Then there exists a unique sifted colimit preserving functor $\LL F\colon \Alg_T(\cat{C})\rt \cat{D}$ that extends $F$, together with a natural equivalence $\LL F\circ \mm{Free}\simeq \LL(F\circ \mm{Free})$. 
\end{proposition}

\begin{proof}
Let $\Fun^{\omega}_{\mm{RL}}(\Flat(\cat{C}), \cat{D})$ denote the $\infty$-category of filtered-colimit preserving functors that are right-left extendable, and notice that there is a right action $$\Fun^{\omega}_{\mm{RL}}(\Flat(\cat{C}), \cat{D})\curvearrowleft \Fun^{\omega}_{\mm{RL}}(\Flat(\cat{C}), \Flat(\cat{C})).$$ In the above situation, we have an algebra object $T\in \Fun^{\omega}_{\mm{RL}}(\Flat(\cat{C}), \Flat(\cat{C}))$, and the functor $F\circ \mm{Free}_T\in \Fun^{\omega}_{\mm{RL}}(\Flat(\cat{C}), \cat{D})$ is a right module for $T$. Using the same logic as \cite[Proposition 2.2.15]{holeman2023derived}, taking right-left extension is lax monoidal and consequently, the functor $\LL(F\circ \mm{Free}_T)\colon \cat{C}\rt \cat{D}$ is a right module for the derived monad $\LL T$. On the other hand, $\mm{Forget}_T\colon \Alg_T(\cat{C})\rt \cat{C}$ is a left $T$-module. Since $\cat{D}$ admits sifted colimits, we can define $\LL F$ by the formula:
$$\begin{tikzcd}
\LL F:= \big|\mm{Bar}\big(\LL(F\circ \mm{Free}_T), T, \mm{Forget}_T\big)\big|\colon \Alg_T(\cat{C})\arrow[r] & \cat{D}.
\end{tikzcd}$$
One readily verifies all the desired properties.
\end{proof}
\begin{construction}[Cotangent complex as a derived functor]\label{cons:cotangent complex}
Let $\cat{C}$ be a derived algebraic context and write $\Mod(\cat{C})$ for the $\infty$-category of tuples $(A, M)$ consisting of a derived commutative algebra $A$ and an $A$-module $M$ in $\cat{C}$. Notice that $\Mod$ arises as the $\infty$-category of algebras over a derived monad $T_{\mathrm{mod}}$ on $\cat{C}$, such that $\mm{Kl}_{T_{\mathrm{mod}}}(\cat{C})$ is simply the ordinary $\infty$-category of free derived algebras $A=\sym(V)$ generated by $V\in \Flat(\cat{C})$, together with a free $A$-module $A\otimes W$ on $W\in \Flat(\cat{C})$. We now have:
\begin{enumerate}
\item A functor between ordinary categories $\triv\colon \mm{Kl}_{T_{\mathrm{mod}}}(\Flat(\cat{C}))\rt \DAlg(\cat{C})$ sending each $(A, M)$ to the trivial square zero extension $A\oplus M$.

\item A functor between ordinary categories $\Omega^1\colon \mm{Kl}_{\sym}(\Flat(\cat{C}))\rt \mm{Kl}_{T_{\mathrm{mod}}}(\cat{C})$ defined by the universal property that $\Map(\Omega^1_A, (B, M))\simeq \Map(A, B\oplus M)$. One readily verifies that such a universal object exists for each free algebra $A=\sym(V)$ and is given by the tuple $(\sym(V), \sym(V)\otimes V)$, which is indeed in $\mm{Kl}_{T_{\mathrm{mod}}}(\cat{C})$.
\end{enumerate}
Both of these functors satisfy the conditions of  \Cref{prop:left extending functors on algebras} and hence admit derived functors $L\colon \DAlg(\cat{C})\rt \Mod(\cat{C})$ and $\triv\colon \Mod(\cat{C}) \rt \DAlg(\cat{C})$. The universal property of $\Omega^1$ then implies that $L$ is left adjoint to $\triv$. We refer to $L_{A}\in \Mod_{A}(\cat{C})$ as the \emph{cotangent complex} of $A$.
\end{construction}
As usual, every map $f\colon A\rt B$ in $\DAlg(\cat{C})$ induces a map $(A, L_{A})\to (B, L_{B})$ in $\Mod(\cat{C})$. Since the projection $\Mod(\cat{C}) \rt \DAlg(\cat{C})$ is a cocartesian fibration, as $\cat{C}\in \CAlg(\cat{Pr}^{\mathrm{L}})$, this is equivalent to a map $f^*L_{A}\rt L_{B}$ in $\Mod_{B}$, whose cofibre is the \emph{relative cotangent complex} $L_{B/A}$:
$$
\begin{tikzcd}
f^*L_{A}\arrow[r]& L_{B} \arrow[r]& L_{B/A}.
\end{tikzcd}
$$
If $B=\LL\sym_A(M)$ is the free derived $A$-algebra on $M\in \Mod_A(\cat{C})$, then $L_{B/A}\simeq B\otimes_A M$.

We conclude with some observations about finite type conditions on maps between animated commutative algebras, well-known for animated rings.
\begin{definition}\label{def:almost finite}
If $\cat{C}$ is a derived algebraic context and $A\in \SCR(\cat{C})$, then we will say that:
\begin{enumerate}
\item a map $A\to B$ in $\SCR(\cat{C})$ is \emph{almost of finite presentation} if $B$ is an almost compact object of $\SCR_A(\cat{C})$, that is, each $\tau_{\leq n}B$ is a compact object of $\tau_{\leq n}\SCR_A(\cat{C})$ \cite[Definition 7.2.4.8]{HA}.
\item an $A$-module $M\in \Mod_A(\cat{C})$ is \emph{almost perfect} if $M\in \Mod_A(\cat{C})_{\geq k}$ and $M$ defines an almost compact object of $\Mod_A(\cat{C})_{\geq k}$.
\end{enumerate} 
\end{definition}
\begin{proposition}\label{prop:aft conditions}
Let $\cat{C}$ be a derived algebraic context and $f\colon A\to B$ be a map of animated commutative algebras in $\cat{C}$. Then the following are equivalent:
\begin{enumerate}
\item $f$ is almost of finite presentation.
\item The map $\pi_0(A)\to \pi_0(B)$ of commutative algebras in $\cat{C}^{\heartsuit}$ is of finite presentation and $L_{B/A}$ is an almost perfect $B$-module.
\end{enumerate}
In addition, the following assertions hold:
\begin{enumerate}[resume]
\item If $B$ is almost perfect as an $A$-module, then $f$ is almost of finite presentation.

\item If $f$ is almost of finite presentation and $\pi_0(A)\to \pi_0(B)$ is surjective, then $B$ is almost perfect as an $A$-module.
\end{enumerate}
\end{proposition}
The proof of  \Cref{prop:aft conditions} requires some preliminaries about cell decompositions and the resulting filtrations.
\begin{construction}\label{con:cell}
Let $A\rt B$ be a map of animated commutative algebras in $\cat{C}$. By a \textit{cell decomposition} of $B$, we will mean a sequence of animated filtered $A$-algebras $A=B_{-1}\to B_0\to B_1\to \dots \to B$ constructed inductively according to the following procedure:
\begin{itemize}
\item Having constructed $B_n$, choose a map $\chi_n\colon M_n\rt (B/B_n)[-1]$ from a connective $B_n$-module $M_n$ such that the cofibre is $(n+1)$-connective.

\item Define $B_{n+1}$ as the pushout of derived filtered  algebras
$$\begin{tikzcd}
\LL\sym_{B_n}(M_n)\arrow[d, "\chi"{swap}]\arrow[r, "0"] & B_n\arrow[d]\\
B_n\arrow[r, "\iota"] & B_{n+1}
\end{tikzcd}$$
where the top map sends $M_n$ to zero. The map $B_n\rt B$ then factors canonically over $\iota\colon B_n\rt B_{n+1}$.
\end{itemize}
Let us record some properties of such cell decompositions:
\begin{enumerate}[label=(\alph*)]
\item If $\pi_0(A)\to \pi_0(B)$ is surjective, one can take $B_0=B_{-1}=A$.

\item\label{it:filtration on cell attachment} Each $B_{n+1}$ admits an increasing filtration whose associated graded is $\LL\sym_{B_n}(M_n[1])$: this filtration is the base change along $\chi$ of the filtration on $B_n$ obtained by applying $\LL\sym_{B_n}$ to the zero object with increasing filtration given by $F_0(0)=M_n$ and $F_i(0)=0$ for $i>0$. Here we use Example \ref{ex:diagrams} in the case $\cat{I}=(\mathbb{Z}, \leq)$. 

\item Each cofibre $B/B_n$ is $(n+1)$-connective and hence each $M_n$ is $n$-connective. To see this, we proceed by induction, the case $n=-1$ being evident. For the inductive step, note that the map $\chi_n$ factors as 
$$
\chi_n\colon M_n[1]\rt B_{n+1}/B_n\rt B/B_n.
$$
Since the cofibre of the composite is $(n+2)$-connective by construction, to show that the cofibre $B/B_{n+1}$ of the second map is $(n+2)$-connective it will suffice to show that $M_n[1]\rt B_{n+1}/B_n$ has $(n+2)$-connective cofibre. Using the filtration from \ref{it:filtration on cell attachment}, this map decomposes as
$$
M_n[1]\simeq F_1(B_{n+1})/F_0(B_{n+1})\rt F_2(B_{n+1})/F_0(B_{n+1})\rt \dots \rt B_{n+1}/F_0(B_{n+1}).
$$
It therefore suffices to verify that the cofibre of each of these maps is $(n+2)$-connective. These cofibres are given by $\LL\sym^m_{B_n}(M_n[1])$ with $m\geq 2$, which fit into a cofibre sequence
$$
(M_n[1]\otimes_{B_n}\dots \otimes_{B_n} M_n[1])_{h\Sigma_m}\to \LL\sym^m_{B_n}(M_n[1])\to F_m(M_n[1]).
$$ 
The result now follows because the outer terms are both $(n+2)$-connective. For the right term, this follows because $F_m$ is right-left extended from a functor on $\Vect^{\ft}(\cat{C})$ with $1$-connective values.

\item For each $n$, consider the natural $B_{n+1}$-linear maps 
$$
B_{n+1}\otimes_{B_n} M_n[1]\rt B_{n+1}\otimes_{B_n} (B_{n+1}/B_n)\rt L_{B_{n+1}/B_n}
$$
where the first map is induced by $\chi$ and the last map arises from the universal derivation $B_{n+1}\rt B_{n+1}\oplus L_{B_{n+1}/B_n}$. The composite map is an equivalence of $(n+1)$-connective modules and the cofibre of the first map is $(n+2)$-connective, so that the cofibre of the natural map $B_{n+1}\otimes_{B_n} (B_{n+1}/B_n)\rt L_{B_{n+1}/B_n}$ is $(n+3)$-connective. An inductive argument then shows that each $B\otimes_{B_n} (B/B_n)\rt L_{B/B_n}$ has $(n+3)$-connective cofibre as well.
\end{enumerate}
\end{construction}
\begin{lemma}\label{lem:graded nakayama}
Let $\cat{C}$ be a derived algebraic context and let $\Fun((\mathbb{Z}_{\geq 0}, =), \cat{C})$ be the $\infty$-category of non-negatively graded objects in $\cat{C}$, as in Example \ref{ex:diagrams}. Given a graded algebra $A:=A^{\bullet}$ and a graded module $M:=M^{\bullet}$, the following hold:
\begin{enumerate}
\item If $A^{0} \otimes_A M\simeq 0$, then $M\simeq 0$.
\item If $A$ and $M$ are connective and $A^{0} \otimes_A M$ is an almost perfect $A^{0}$-module, then $M$ is almost perfect.
\end{enumerate}
\end{lemma}
\begin{proof}
For (1), we prove the claim by induction on the weight. Clearly, $M^{0}=0$. Assuming that $M^{i}=0$ for $i<n$, we have $M^{n} \simeq (A^{0} \otimes_A M)^{n}$, so that $M^{n}=0$ as well.  
For (2), note that 
$$
\pi_0(A^{0})\otimes_{\pi_0(A)} \pi_0(M) \cong \pi_0(A^{0} \otimes_A M)
$$
is the module of indecomposables associated to the graded $\pi_0(A)$-module $\pi_0(M)$. Since this module is finitely generated, it follows that $\pi_0(M)$ is a finitely generated $\pi_0(A)$-module. Lifting the generators provides a map $P\rt M$ from a finite free graded $A$-module, which is surjective on $\pi_0$. The cofibre $N$ is then a $1$-connective $A$-module such that $A^{0} \otimes_A N$ is an almost perfect $A^{0}$-module. Proceeding by induction, one finds that $M$ is indeed almost perfect.
\end{proof}
\begin{lemma}\label{lem:aperf filtered}
Let $\cat{C}$ be a derived algebraic context and let $\Fun((\mathbb{Z}, \leq), \cat{C})$ be the $\infty$-category of objects in $\cat{C}$ with an unbounded increasing filtration, as in Example \ref{ex:diagrams}. Let $A=F_\star A$ be a connective filtered algebra which is complete, i.e.\ $\lim_{n\to -\infty} F_nA=0$. Then a filtered $A$-module $M=F_\star M$ is almost perfect if and only if it is complete and $\gr(M)$ is an almost perfect left $\gr(A)$-module.
\end{lemma}
\begin{proof}
If $M$ is almost perfect, then $\gr(M)$ is almost perfect because taking the associated graded preserves compact objects and is $t$-exact up to a shift. To see that $M$ is complete, let $n\geq 0$ and pick a perfect filtered $A$-module $N_n$ with a map $N_n\rt M$ whose cofibre is $n$-connective. Since $A$ is complete, $N_n$ is complete and hence $\lim F^iM \simeq \lim F^i(M/N_n)$. By the Milnor sequence, the latter is $(n-1)$-connective. Since $n$ was arbitrary, we conclude that $\lim F^iM=0$.

For the converse, it suffices that if $M$ is $n$-connective, there exists a cofibre sequence
$$\begin{tikzcd}
\bigoplus_{i=1}^k A\otimes F_{\geq w_i}V_i[n]\arrow[r, "f"] & M\arrow[r] & M'
\end{tikzcd}$$
where $M'$ is $(n+1)$-connective and the first term consists of free filtered $A$-modules with $V_i\in \Vect^{\ft}(\cat{C})$, concentrated in filtration weight $\geq w_i$. Indeed, this implies that $M'$ is a complete $(n+1)$-connective module such that $\gr(M')$ is almost perfect; it follows by induction that each $\tau_{\leq m}(M)$ is equivalent to the $m$-truncation of a perfect $A$-module, so that $M$ is almost perfect.

Since $\gr(M)$ is an almost perfect $n$-connective $\gr(A)$-module, the desired cofibre sequence exists at the graded level. In particular, there exists a map $f'\colon \bigoplus_{i=1}^n \gr(A)\otimes V_i(w_i)[n]\to \gr(M)$, where the domain is a sum of free $\gr(A)$-modules in various weights $w_i$. Since the $V_i\in \Vect^{\ft}(\cat{C})$ are projective in $\cat{C}_{\geq 0}$, we can lift $f'$ to a map $f$ of filtered modules. The cofibre $M'$ of $f$ is then a complete module with $(n+1)$-connective associated graded, and is thus itself $(n+1)$-connective by the Milnor sequence.
\end{proof}
\begin{proof}[Proof of  \Cref{prop:aft conditions}]
For (1) $\Rightarrow$ (2), note that $\pi_0(B)$ is a compact object in $\pi_0(A)$-algebras and hence finitely presented. The cotangent complex $L_{B/A}$ is almost perfect because for any filtered system of $k$-coconnective $B$-modules $M_\alpha$, the map $\colim \Map_B(L_{B/A}, M_\alpha)\rt \Map_B(L_{B/A}, \colim M_\alpha)$ is the base change of the equivalence
$$
\colim \Map_{\SCR_A(\cat{C})}\big(B, B\oplus M_\alpha\big)\rt\Map_{\SCR_A(\cat{C})}\big(B, \colim(B\oplus M_\alpha)\big).
$$
For (2) $\Rightarrow$ (1), we claim that there exists a cell decomposition for $B$ with the property that each $B_n$ is almost of finite presentation over $A$. It then follows that $B$ is the colimit of a sequence of almost finitely presented algebras that stabilises upon $k$-truncation for any $k$. This implies that $A\to B$ almost finitely presented as well.

To construct this cell decomposition, we fix a finite presentation of $\pi_0(A)\rt \pi_0(B)$. We can then choose a map $M_0=A\otimes V_0\to B$ from a free $A$-module on $V_0\in \Vect^{\ft}(\cat{C})$ such that $\pi_0(M_0)\to \pi_0(B)$ picks out the generators of $\pi_0(B)$. We then set $B_0=\LL\sym_A(M_0)$. Next, we choose a map $M_1=A\otimes V_1\to \fib(B_0\to B)$ from a free $B_0$-module on $V_1\in \Vect^{\ft}(\cat{C})$ such that $\pi_0(M_1)\to \pi_0(B_0)=\sym_{\pi_0(A)}(\pi_0(M_0))$ picks out the relations of this presentation. The resulting $B_1$ is almost finitely presented and $\pi_0(B_1)=\pi_0(B)$. Now suppose we have constructed the desired cell decomposition up to $B_n$, for some $n\geq 1$. Since $A\to B_n$ is a almost of finite presentation, the first part of the proof shows that $L_{B_n/A}$ is an almost perfect $B_n$-module. Consequently, $L_{B/B_n}$ is an almost perfect $B$-module. Since $$B\otimes_{B_n} (B/B_n)\rt L_{B/B_n}$$ has $(n+3)$-connective cofibre, it follows that we can choose an $n$-connective free $B_n$-module $M_n=B_n\otimes V_n[n]$ with $V_n\in \Vect^{\ft}(\cat{C})$, together with a map $\chi\colon M_n\rt (B/B_n)[-1]$ that is surjective on $\pi_n$. The resulting map $B_n\to B_{n+1}$ is then almost finitely presented, so that the composite $A\to B_{n+1}$ is almost finitely presented as well.

For assertion (4),  suppose that $f$ is almost finitely presented and that $\pi_0(A)\to \pi_0(B)$ is surjective, and consider the above cell decomposition. Then $B_0=A$, and we claim that each $B_n$ is an almost perfect $A$-module. The colimit $B$ is then an almost perfect $A$-module as well, because the cofibre of each $B_n\to B_{n+1}$ is $(n+1)$-connective. Note that by part \ref{it:filtration on cell attachment} of Construction \ref{con:cell}, $B_{n+1}$ comes with an increasing filtration by $B_n$-modules whose associated graded $\LL\sym_{B_n}(M_n[1])$ is almost perfect: indeed, each $\LL\sym^k_{B_n}(M_n[1])$ is almost perfect and (at least) $k$-connective, since $M_n$ was a connective almost perfect $B_n$-module. It follows that $B_{n+1}$ is an almost perfect $B_n$-module, so that an inductive argument shows that $B_n\in \APerf_A(\cat{C})$ (the case $n=0$ being evident).

Finally, let us prove assertion (3). Consider the cell decomposition for $B$ defined inductively by setting $B_{-1}=A$ and $$M_n=(B/B_n)[-1].$$ We claim that $B$ is almost perfect as a $B_n$-module for each $B_n$. Assuming this, it follows by induction that each $B_n\to B_{n+1}$ is almost finitely presented, so that $A\to B_{n+1}$ is almost finitely presented as well. Since the sequence of truncations $\tau_{\leq k}B_n$ is eventually constant, the colimit $B$ is then an almost finitely presented $A$-algebra as well.

We prove our claim by induction, the case $n=-1$ being evident. Assuming $B$ is an almost perfect $B_n$-module,  we endow $B_{n+1}$ with the increasing filtration from part \ref{it:filtration on cell attachment} of Construction \ref{con:cell} and $B$ with the increasing filtration where $F_0B=B_n$ and $F_iB=B$ for $i\geq 1$. Then $F_\star B$ is an $F_\star B_{n+1}$-module, and it it suffices to verify that it is almost perfect as a filtered module. By  \Cref{lem:graded nakayama}, it suffices to verify that the associated graded is an almost perfect module over $\gr(B_{n+1})=\LL\sym_{B_n}(B/B_n)$. We then have a cofibre sequence of $\gr(B_{n+1})$-modules $B_n(0)\to \gr(B)\to (B/B_n)(1)$ of where the outer terms are concentrated in weight $0$ and $1$. By  \Cref{lem:graded nakayama}, it therefore suffices to verify that the graded $B_n$-modules
\begin{align*}
B_n\otimes_{\LL\sym_{B_n}(B/B_n)} B_n &\simeq \LL\sym_{B_n}(B/B_n[1])\\
B_n\otimes_{\LL\sym_{B_n}(B/B_n)} (B/B_n)&\simeq \LL\sym_{B_n}(B/B_n[1])\otimes_{B_n} B/B_n
\end{align*}
are almost perfect $B_n$-modules. This follows from the fact that $B/B_n$ is an almost perfect $B_n$-module and that $\LL\sym_{B_n}(B/B_n[1])$ is almost perfect, because each $\LL\sym^k_{B_n}(B/B_n[1])$ is almost perfect and $k$-connective.
\end{proof}

\subsection{Pro-coherent modules}\label{sec:procoh affine}
As highlighted in \cite{GR}, the usual $\infty$-category of quasi-coherent sheaves on a derived scheme (or prestack) is not very well-adapted to deformation theory. For schemes that are locally almost finitely presented over a field, it is better to consider their categories of ind-coherent sheaves instead. In this situation, Serre duality provides a natural identification between ind-coherent sheaves and the dual notion of pro-coherent sheaves: on an almost finitely presented scheme $X$ over a field $k$, this is simply given by $\mm{Ind}(\Coh_X)\simeq \mm{Ind}(\Coh_X^{\op})$.
To also deal with situations without Serre duality, where we cannot directly use the machinery of ind-coherent sheaves, let us introduce the following provisional modification of the $\infty$-category of pro-coherent sheaves over an affine derived scheme:
\begin{definition}\label{def:procoh affine}
Given an animated ring $A$, we will write $\QC^\vee_A$ for the $\infty$-category of exact functors $F\colon \Mod_A\rt \Sp$ satisfying the following two conditions:
\begin{enumerate}
\item They are \textit{convergent}, in the sense that they preserve limits of Postnikov towers in $\Mod_A$.
\item They are \textit{almost finitely presented}, in the sense that for each $n\in\mathbb{N}$ and each filtered diagram of $n$-coconnective $A$-modules $M_\alpha$, the map $\colim_\alpha F(M_\alpha)\rt F(\colim M_\alpha)$ is an equivalence.
\end{enumerate}
\end{definition}
\begin{warning}
For a general animated ring, the $\infty$-category $\QC^\vee_A$ is not very manageable and one should view this as an ad-hoc definition. The definition will essentially only become useful when $A$ is a \textit{coherent} animated ring: in that case $\QC^\vee_A\simeq \mm{Ind}(\Coh_A^\op)$ can be identified with the compactly generated $\infty$-category of \textit{pro-coherent sheaves} on $A$. 
 
\end{warning}
Keeping this warning in mind, let us nonetheless record some formal properties of $\QC^\vee_A$.
\begin{notation}
If $\cat{C}$ is a stable $\infty$-category with a $t$-structure, we let $\cat{C}^+=\bigcup_n \cat{C}_{\leq n}$, $\cat{C}^-=\bigcup_n \cat{C}_{\geq n}$ and $\cat{C}^b=\cat{C}^+\cap \cat{C}^-$ be the full subcategories of eventually coconnective, eventually connective and bounded objects \cite{HA}.
\end{notation}
\begin{lemma}\label{lem:pro-coh equivalent}
Let $A$ be an animated ring. Then the following $\infty$-categories are naturally equivalent:
\begin{enumerate}
\item The $\infty$-category $\QC^\vee_A$.
\item The $\infty$-category of exact functors $\Mod^+_{A}\rt \Sp$ that are almost finitely presented.
\item The $\infty$-category of left exact functors $F\colon \Mod^+_{A}\rt \sS$ that are almost finitely presented.

\item The $\infty$-category of functors $F\colon \Mod^+_{A, \geq 0}\rt \sS$ that are locally almost finitely presented and preserve the terminal object and pullbacks along a $\pi_0$-surjection.

\item The $\infty$-category of functors $F\colon \Mod_{A, \geq 0}\rt \sS$ preserving limits of almost eventually constant towers and satisfying the conditions from (4).
\end{enumerate}
\end{lemma}
\begin{proof}
The equivalence between (1) and (2) is given by restriction and right Kan extension along $j\colon \Mod^+_{A}\hookrightarrow \Mod_A$. Recall that the right Kan extension is given explicitly by $j_*F(M)=\lim F(\tau_{\leq n} M)$. The equivalence between (2) and (3) is given by postcomposition with $\Omega^\infty$. The equivalence between (3) and (4) is given by restriction along the inclusion $i_1\colon \Mod^b_{A}\hooklongrightarrow \Mod^+_{A}$ and $i_2\colon \Mod^+_{A, \geq 0}\hooklongrightarrow \Mod_{A}^b$ of bounded (resp.\ bounded connective) objects. The inverse is given by right Kan extension along $i_2$ followed by left Kan extension along $i_1$. The equivalence between (4) and (5) follows by restriction and right Kan extension along $\Mod^+_{A, \geq 0}\hooklongrightarrow \Mod_{A, \geq 0}$.
\end{proof}
\begin{lemma}\label{lem:pro-coh kappa-presentable}
Let $A$ be an animated ring and let $\kappa$ be a regular cardinal such that $\tau_{\leq n}\colon \Mod_A\rt \Mod_A$ preserves $\kappa$-compact objects. Then restriction and left Kan extension along $i\colon \Mod_{A}^{+, \kappa}\hookrightarrow \Mod^+_{A}$ induces an equivalence between $\QC^\vee_A$ and the $\infty$-category of exact functors $F\colon \Mod^{+, \kappa}_{A}\rt \Sp$ preserving colimits of $\kappa$-small filtered diagrams of $n$-coconnective objects, for each $n$.
\end{lemma}
\begin{proof}
Since $\tau_{\leq n}$ preserves $\kappa$-small objects, for any $n$-coconnective $A$ module $M$ the inclusion $\big(\Mod_{A, \leq n}^\kappa\big)_{/M}\hookrightarrow \big(\Mod^{+, \kappa}_{A}\big)_{/M}$ is cofinal. Using that $\big(\Mod_{A, \leq n}^\kappa\big)_{/M}$ is filtered, one readily sees that $i_!i^*F\simeq F$ for any $F\in \QC^\vee_A$. The identification of its essential image follows from the fact that any filtered colimit can be written as a $\kappa$-filtered colimit of $\kappa$-small filtered colimits.
\end{proof}
\begin{corollary}
Let $A$ be an animated ring $A$. Then $\QC^\vee_A$ is a stable presentable $\infty$-category. 
\end{corollary}
\begin{proof}
By  \Cref{lem:pro-coh kappa-presentable}, $\QC^\vee_A\subseteq \Fun\big(\Mod^{+, \kappa}_A, \Sp\big)$ is the full subcategory of functors preserving certain types of colimits. It is hence presentable by the proof of \cite[Proposition 5.5.3.8]{HTT}.
\end{proof}
\begin{corollary}\label{cor:aft approx}
Let $F\colon \Mod_A\rt \Sp$ be an exact functor which is accessible and convergent and consider the $\infty$-category $\big(\QC^\vee_A\big)_{/F}$ of functors in $\QC^\vee_A$ equipped with a natural transformation to $F$. This $\infty$-category admits a terminal object $F^{\aft}$, which we will refer to as the \emph{almost finitely presented approximation} to $F$.
\end{corollary}
\begin{proof}
Since $F$ is convergent and $\kappa$-accessible for some uncountable $\kappa$, it suffices to verify that there exists a universal object $F^{\aft}\in \QC^{\vee}_A$ equipped with a natural transformation to the restriction $F\colon \Mod_{A}^{+, \kappa}\rt \Sp$. The previous proof shows that for large enough $\kappa$, the inclusion $\QC^{\vee}_A\subseteq \Fun\big(\Mod_{A}^{+, \kappa}, \Sp\big)$ is accessible and preserves colimits, and hence admits a right adjoint.
\end{proof}
For any $A$, there is an adjoint pair
$$\begin{tikzcd}
\upiota\colon \Mod_A\arrow[r, yshift=1ex] & \QC^\vee_A\cocolon \upupsilon\arrow[l, yshift=-1ex]
\end{tikzcd}$$
where the left adjoint sends an $A$-module $M$ to the exact functor $M\otimes_A (-)\colon \Mod^+_A\rt \Sp$ and the right adjoint sends any $F\colon \Mod^+_A\rt \Sp$ to the $A$-module $\upupsilon(F)=\lim_n F(\tau_{\leq n} A)$.

\begin{observation}\label{obs:t-structure}
The $\infty$-category $\QC^\vee_A$ comes equipped with a left complete $t$-structure whose connective part consists of those $F\colon \Mod^+_A\rt \Sp$ that are right $t$-exact. This induces an equivalence 
$$
\Mod_{A, \geq 0}\simeq \QC^\vee_{A, \geq 0}.
$$
Indeed, the equivalences of  \Cref{lem:pro-coh equivalent} identify $\QC^\vee_{A, \geq 0}$ with the $\infty$-category of colimit-preserving functors $F\colon \Mod_{A, \geq 0}\rt \sS$, which is equivalent to $\Mod_{A, \geq 0}$ via sending such a functor $F$ to the $A$-module $F(A) \in \Mod_{A,\geq 0}$.

Consequently, $\Mod_A$ is the right completion of $\QC^\vee_A$. In particular, we will typically identify \textit{eventually connective} $A$-modules with the corresponding object in $\QC^\vee_A$. If $A$ is eventually coconnective, then $\QC_A\rt \QC^\vee_A$ is fully faithful on all modules (not just the eventually connective ones).
\end{observation}
Let $A$ be an animated ring and $F\in \QC^\vee_A$. For each $N\in \Mod^+_A$, the spectrum $F(N)$ inherits a natural $A$-module structure from the $A$-module structure on $N$. More precisely, one can this describe this $A$-module structure on $F(N)$ by the exact functor $\perf(A)^{\op}\rt \Sp$ sending $M$ to $F(M^\vee\otimes_A N)$.
Using this, the $\infty$-category $\QC^\vee_A$ is tensored over $\Mod_A$ via
\begin{equation}\label{diag:tensor over Mod}\begin{tikzcd}
\otimes_A\colon \QC^\vee_A\times \Mod_A\ar[r] & \QC^\vee_A; \quad (M\otimes_A F)(N) = M\otimes_A (F(N))
\end{tikzcd}\end{equation}
for each $N\in \Mod^+_A$. This preserves colimits in each variable and the functor $\upiota\colon \Mod_A\rt \QC^\vee_A$ is compatible with the tensoring. 
\begin{definition}\label{def:pro-coh dual}
Let $M\in \Mod_A$ be an $A$-module. Then the tensoring $M\otimes_A -\colon \QC^\vee_A\rt \QC^\vee_A$ admits a right adjoint. We define the \emph{pro-coherent dual} $M^\vee$ of $M$ to be the value of this right adjoint on $A$.
\end{definition}

\begin{example}\label{ex:pro-coh dual}
If $M$ is a perfect $A$-module, then the adjoint to $M\otimes_A -\colon \QC^\vee_A\rt \QC^\vee_A$ is given by tensoring with the dual perfect module. Consequently, the pro-coherent dual of a perfect $A$-module is simply its usual $A$-linear dual. In terms of functors, one can also identify this with the functor $\hom_A(M, -)\colon \Mod_A\rt \Sp$. Using this and the fact that $(-)^\vee$ sends colimits in $\Mod^-_{A}$ to limits in $\QC^\vee$, it follows that for any eventually connective $M\in \Mod^-_{A}$, its pro-coherent dual $M^\vee$ is the almost finitely presented approximation to $\hom_A(M, -)$ in the sense of Corollary \ref{cor:aft approx}.
\end{example}
We will be particularly interested in the pro-coherent duals of \emph{almost perfect} $A$-modules: in this case, $M^\vee$ is given by $\hom_A(M, -)\colon \Mod_A\rt \Sp$, which is already almost finitely presented.
\begin{definition}\label{def:dually almost perfect}
We will say that a pro-coherent module $F\in \QC^\vee_A$ is \emph{dually almost perfect} if $F\colon \Mod_A\rt \Sp$ is corepresented by an almost perfect $A$-module. Taking pro-coherent duals therefore determines an equivalence
$$\begin{tikzcd}
(-)^\vee\colon \APerf_A^{\op}\arrow[r, "\sim"] & \dAPerf_A
\end{tikzcd}$$
between almost perfect $A$-modules and dually almost perfect pro-coherent $A$-modules.
\end{definition}
\begin{lemma}\label{lem:slice duality}
Let $M\in\Mod_A$ be an $A$-module and let $M^\vee$ be its pro-coherent $A$-linear dual. Then pro-coherent duality determines an equivalence
\begin{equation}\label{diag:slice duality}\begin{tikzcd}
(-)^\vee\colon \big((\APerf_A)_{M/}\big)^{\op}\arrow[r, "\sim"] & (\dAPerf_A)_{/M^\vee}.
\end{tikzcd}\end{equation}
\end{lemma}
\begin{proof}
The functor \eqref{diag:slice duality} defines a map of right fibrations covering the equivalence $\APerf_A^{\op}\simeq \dAPerf_A$. It therefore suffices to verify that it induces an equivalence on fibres, i.e.\ that the map $\Map_{\Mod_A}(M, N)\rt \Map_{\QC^\vee_A}(N^\vee, M^\vee)$ is an equivalence for every $N\in \APerf_A$. To see this, recall that $N^\vee\colon \Mod_A\rt \Sp$ is corepresented by $N$ and $M^\vee\colon \Mod_A\rt \Sp$ is the almost finitely presented approximation to the functor corepresented by $M$, so that $\Map_{\QC^\vee_A}(N^\vee, M^\vee)$ is equivalent to the space of maps between the functors corepresented by $N$ and $M$. Unravelling the definitions, one then sees that $\Map_{\Mod_A}(M, N)\rt \Map_{\QC^\vee_A}(N^\vee, M^\vee)$ is an equivalence by the Yoneda lemma.
\end{proof}

\begin{definition}\label{def:tor-ampl}
Let $A$ be an animated ring. For all $-\infty\leq a\leq b\leq \infty$, we will say that an object $F\in \QC^\vee_A$ has \textit{tor-amplitude} in $[a, b]$ if for any discrete $A$-module $M\in \Mod_A^\heartsuit$, 
$$
\pi_nF(M)\neq 0 \qquad\Longrightarrow\qquad a\leq n\leq b.
$$
Write $\QC^{\vee, [a, b]}_{A}\subseteq \QC^\vee_A$ for the full subcategory on the objects with tor-amplitude in $[a, b]$. 
\end{definition}
\begin{example}
We have the following special cases:
\begin{enumerate}
\item A pro-coherent module $F$ has tor-amplitude in $[0, \infty]$ if and only if $F$ is connective. In particular, $\QC^{\vee, [0, \infty]}_A\simeq \Mod_{A, \geq 0}$ is equivalent to the ordinary $\infty$-category of connective $A$-modules.

\item $\QC^{\vee, [0, 0]}_{A}\simeq \mm{Flat}_A$ coincides with the full subcategory of $\Mod_{A, \geq 0}$ spanned by the \textit{flat} $A$-modules.

\item Let $F=M^\vee$ be dually almost perfect. Then $M^\vee$ has tor-amplitude in $[-b, -a]$ if and only if $M$ has tor-amplitude in $[a, b]$. 
\end{enumerate}
\end{example}

Let us now turn to the functoriality of $\QC^\vee_A$ in the animated ring $A$.
\begin{definition}\label{def:procoh functoriality}
Let $f\colon A\rt B$ be a map of rings. We will write $f^*\colon \QC^\vee_A\leftrightarrows \QC^\vee_B \cocolon f_*$ for the adjoint pair whose left adjoint is given by precomposition with $f_*\colon \Mod^+_{B}\rt \Mod^+_A$.
\end{definition}
Note that $f^*$ fits in a commuting square
$$\begin{tikzcd}
\Mod_A\arrow[r, "f^*"]\arrow[d, "\upiota"{swap}] & \Mod_B\arrow[d, "\upiota"]\\
\QC^\vee_A\arrow[r, "f^*"] & \QC^\vee_B.
\end{tikzcd}$$
In the presence of Serre duality, $f^*$ is Serre dual to the $f^!$-functor on ind-coherent sheaves \cite[Section 9.2.3]{GaitsgoryIndcoh}
\begin{lemma}\label{lem:procoh base change}
If $f\colon A\rt B$ has finite tor-amplitude, then the right adjoint $f_*$ preserves colimits. Furthermore, for any pullback square
$$\begin{tikzcd}
\Spec(B')\arrow[r, "f'"]\arrow[d, "g'"{swap}] & \Spec(B)\arrow[d, "g"]\\
\Spec(A')\arrow[r, "f"] & \Spec(A).
\end{tikzcd}$$
where $f$ has finite tor-amplitude, the Beck--Chevalley map $g^*f_*\rt f'_*g'^*$ is an equivalence.
\end{lemma}
\begin{proof}
If $f$ has finite tor-amplitude, then $f_*$ is simply given by restriction along $f^*\colon \Mod^+_A\rt \Mod^+_{B}$. Using this, one sees that for every $F\in \QC^\vee_{A'}$ and $M\in \Mod^+_{B}$, the Beck--Chevalley map $(g^*f_*F)(M)\rt (f'_*g'^*F)(M)$ is given by the image under $F$ of the Beck--Chevalley map $g^*f_*M\to f'_*g'^*M$, which is an equivalence.
\end{proof}
\begin{corollary}\label{cor:descent}
The functor $\QC^\vee\colon \SCR_{\KK}\rt \cat{Pr^L}$ satisfies descent with respect to universal descent morphisms of finite tor-amplitude, in the sense of \cite[Definition D.3.1.1]{SAG}. In particular, it satisfies fppf descent.
\end{corollary}
\begin{proof}
Using that $A\mapsto \Mod_{A}$ preserves products, one readily sees that $\QC^\vee$ preserves finite products of animated rings. Next, let $f\colon A\rt B$ be a universal descent morphism of finite tor-amplitude. Let $B^\bullet$ be the corresponding \v{C}ech nerve and consider the map $\theta\colon \QC^\vee_{A}\rt \lim \QC^\vee_{B^\bullet}$. Using  \Cref{lem:procoh base change} and \cite[Corollary 4.7.5.3]{HA}, we see that $\theta$ is an equivalence if $f^*\colon \QC^\vee_A\rt \QC^\vee_B$ is conservative. To see this, take $F\in \QC^\vee_A$ such that $f^*F\simeq 0$. The objects $M\in \Mod_A$ such that $F(M)\simeq 0$ form a stable subcategory of $\Mod_A$ that contains $f_*(\Mod_B)$ and is closed under retracts; since $f$ was a universal descent morphism, this means that it contains every $A$-module.
\end{proof}

We now specialise to the case where $A$ is a coherent animated ring.
\begin{definition}\label{def:coherent ring}
An animated ring $A$ is said to be \textit{coherent} if the standard $t$-structure on $\Mod_A$ restricts to a $t$-structure on $\APerf_A$. Equivalently, $\pi_0(A)$ is coherent and each $\pi_n(A)$ is a finitely presented $\pi_0(A)$-module.  
\end{definition}

\begin{definition}\label{def:coherent modules}
Assume $A$ is a coherent animated ring. An almost perfect $A$-module $M$ is \textit{coherent} if it is eventually coconnective. Equivalently, an $A$-module $M$ is coherent if all homotopy groups $\pi_{n}M$ are finitely presented over $\pi_{0}A$, and only finitely many of $\pi_{n}M$ are non-zero. We denote $\Coh_{A} \subset \APerf_{A}$ the full subcategory of coherent modules.
\end{definition}

\begin{lemma}\label{lem:pro-coh over coherent ring}
If $A\in \SCR$ is coherent, then $\QC^\vee_A$ is equivalent to the categories of:
\begin{enumerate}
\item Exact functors $F\colon \Coh_A\rt \cat{Sp}$.

\item Functors $F\colon \Coh_{A,\geq 0}\rt \sS$ that preserve the terminal object and pullbacks along $\pi_0$-surjections.

\item\label{it:pro-coh coherent 3} Functors $F\colon \APerf_{A,\geq 0}\rt \sS$ that preserve almost eventually constant towers, the terminal object and pullbacks along $\pi_0$-surjections.
\end{enumerate}
In particular, $\QC^\vee_A\simeq \mm{Ind}(\Coh_A^{\op})$ is compactly generated.
\end{lemma}
\begin{proof}
This follows from  \Cref{lem:pro-coh kappa-presentable} with $\kappa=\omega$, as well as the argument from  \Cref{lem:pro-coh equivalent}.
\end{proof}
\begin{remark}
Let $A$ be a coherent animated ring with dualising complex $\omega_A$. Then the $t$-structure on $\QC^\vee_A$ from Observation \ref{obs:t-structure} is Serre dual to the $t$-structure on $\mm{Ind}(\Coh_A)$ whose connective part is generated by $\omega_A$.
\end{remark}
\begin{corollary}
The following holds for a map of coherent animated rings $f\colon A\rt B$:
\begin{enumerate}
\item If $f$ has finite tor-amplitude, then the functor $f^*\colon \QC^\vee_A\rt \QC^\vee_B$ coincides with $\mm{Ind}(f^*)$, for $f^*\colon \Coh_A^{\op}\rt \Coh_B^{\op}$.
\item If $f$ is finite, then $f^*$ has a left adjoint $f_!\colon \QC^\vee(B)\rt \QC^\vee(A)$ given by the ind-completion of $f_*\colon \Coh_B^{\op}\rt \Coh_A^{\op}$.
\end{enumerate}
\end{corollary}
\begin{corollary}\label{cor:procoh monoidal}
Let $A$ be a coherent animated ring. Then $\QC^\vee_A$ carries a unique closed symmetric monoidal structure whose restriction to $\dAPerf_A\simeq \APerf_A^{\op}$ is equivalent to the usual symmetric monoidal structure on almost perfect $A$-modules, i.e.\ $M^\vee\otimes_A N^\vee=(M\otimes_A N)^\vee$ for all $M, N\in \APerf_A$. Every map $f\colon A\rt B$ between coherent animated rings induces a symmetric monoidal functor $f^*\colon \QC^\vee_A\rt \QC^\vee_B$.
\end{corollary}
\begin{proof}
Using part \eqref{it:pro-coh coherent 3} of  \Cref{lem:pro-coh over coherent ring}, it suffices to verify that $\QC^\vee_A\subseteq \Fun(\APerf_{A, \geq 0}, \sS)$ is a monoidal left Bousfield localisation with respect to the Day convolution product. This follows from the fact that if $F$ is a functor as in \eqref{it:pro-coh coherent 3} and $M\in \APerf_{A, \geq 0}$, then $F(M\otimes_A -)$ satisfies the conditions from \eqref{it:pro-coh coherent 3} as well.
\end{proof}
\begin{remark}\label{rem:procoh dual coherent case}
One readily verifies that $\upiota\colon \Mod_A\rt \QC^\vee_A$ is symmetric monoidal and that the tensoring from \eqref{diag:tensor over Mod} is induced by this symmetric monoidal functor. It follows that the pro-coherent dual of an $A$-module $M$ from  \Cref{def:pro-coh dual} coincides with the dual of $\upiota(M)$ with respect to the symmetric monoidal structure from Corollary \ref{cor:procoh monoidal}.
\end{remark}

\section{Derived algebraic geometry recollections}\label{sec:dag}

We recall the following terminology: 
 
\begin{definition}[Prestacks]
A \textit{prestack} is an accessible functor $X\colon \SCR \rt \sS$ from animated rings to spaces. We will write $\PrStk$ for the $\infty$-category of prestacks. Given a prestack $S$, we let 
\begin{enumerate}
\item $\PrStk_{/S}$ be the $\infty$-category of \emph{$S$-prestacks}, that is, prestacks $X$ endowed with a map $X\rt S$.

\item $\Aff_{/S}$ be the $\infty$-category of \emph{affine derived schemes} with a map to $S$, i.e.\ all $\Spec(A)\rt S$ where $\Spec(A)$ is the functor corepresented by an animated ring $A$.
\end{enumerate}
\end{definition}
The $\infty$-category of $S$-prestacks can then also be identified with the $\infty$-category of accessible functors $(\Aff_{/S})^\op\rt \sS$.
\begin{example}
When $S=\Spec(\KK)$ is affine, $\Aff_{/S}\simeq \SCR_{\KK}^{\op}$ can be identified with the opposite of the $\infty$-category of animated commutative $\KK$-algebras. Consequently, a prestack $X$ over $\KK$ is simply an accessible functor $X\colon \SCR_{\KK}\rt \sS$; we will  write $\PrStk_\KK$ for the resulting $\infty$-category of $\KK$-prestacks.
\end{example}

The purpose of this section is to briefly recall some of the results and constructions from \cite{GR} in this setting.

\subsection{Finiteness and deformation-theoretic properties of prestacks}
 
We begin with a recollection of some of the conditions that one can impose on a prestack, following \cite{GR} and \cite[Section 17]{SAG}. As in \cite[Volume 1, Chapter 2]{GR} and \cite[Section 17.4]{SAG}, the following finiteness condition will play a central role in this text:
\begin{definition} \label{laftdef}
A map of prestacks $X \rightarrow S$ is said to be 
\begin{itemize}
\item \textit{convergent} if for all  $A \in \SCR$, the square
$$\begin{tikzcd}
X(A)\arrow[r]\arrow[d] & \lim X(\tau_{\leq n}A)\arrow[d]\\
S(A)\arrow[r] & \lim S(\tau_{\leq n}A)
\end{tikzcd}$$
is cartesian.

\item \textit{locally almost finitely presented} if it is convergent and satisfies the following condition for each $n\geq 0$: for each filtered diagram $A_\alpha\colon I\rt \SCR_{\leq n}$ of $n$-truncated animated rings with colimit $A$, the square
$$\begin{tikzcd}
\colim X(A_\alpha)\arrow[r]\arrow[d] & X(A)\arrow[d]\\
\colim S(A_\alpha)\arrow[r] & S(A)
\end{tikzcd}$$
is cartesian. Equivalently, $X$ is convergent and for any $\Spec A\rt S$ and each $n\geq 0$, the base change $X\times_S \Spec(A)$ restricts to a functor $\SCR_{A, \leq n}\rt \sS$ that preserves filtered colimits.
\end{itemize}
We will write $\PrStk_S^{\laft}\hooklongrightarrow \PrStk^\conv_S\hooklongrightarrow \PrStk_S$ for the full subcategories spanned by the $S$-prestacks that are locally of finite presentation, resp.\ convergent.
\end{definition}
Next, let us recall the conditions on a prestack that guarantee that it has a well-behaved infinitesimal structure, using the following algebraic terminology:
\begin{definition}
A map of animated rings $A\rt B$ is said to be a \textit{nilpotent extension} if $\pi_0(A)\rt \pi_0(B)$ is surjective and its kernel is a nilpotent ideal.
\end{definition}
\begin{lemma}[{\cite[Volume 2, Chapter 1, Proposition 5.5.3]{GR}}]\label{lem:decompose nilpotent extensions}
Let $f\colon A\rt B$ be a nilpotent extension. Then $f$ can be decomposed as the limit of an almost eventually constant tower
$$\begin{tikzcd}
A=A_\infty\arrow[r] & \dots\arrow[r] & A_n\arrow[r] & A_{n-1}\arrow[r] & \dots \arrow[r] & A_0=B
\end{tikzcd}$$
where each $A_n\rt A_{n-1}$ is a square zero extension by a $\pi_0(B)$-module $M_n$ and for each $N$.
\end{lemma}
\begin{definition}[{\cite[Volume 2, Chapter 1, Definition 7.1.2]{GR}}] \label{hasdeftheory}
An $S$-prestack $X$ is said to \textit{have deformation theory} (relative to $S$) if it is convergent (relative to $S$) and for each pullback diagram of animated rings on the left in which $B\rt A$ (and hence $B'\rt A'$) is a nilpotent extension, the square on the right is cartesian
$$\begin{tikzcd}
B'\arrow[d]\arrow[r] & B\arrow[d]\\
A'\arrow[r] & A
\end{tikzcd}\hspace{60pt} \begin{tikzcd}
X(B')\arrow[d]\arrow[r] & X(B)\times_{X(A)} X(A')\arrow[d]\\
S(A')\arrow[r] & S(B)\times_{S(A)} S(A').
\end{tikzcd}$$
Equivalently, for each $s\colon \Spec(A)\rt S$, the fibre $X_s\colon \SCR_A\rt \sS$ has deformation theory, i.e.\ preserves limits of Postnikov towers and pullbacks along nilpotent extensions of animated $A$-algebras. 
\end{definition}

\begin{lemma}\label{lem:convergent def thy}
Let $X$ be a convergent $S$-prestack. Then $X$ has deformation theory if and only if one of the following two equivalent conditions holds:
\begin{enumerate}
\item For each square in $(\Aff_{/S})^{\op}$ opposite to a pullback square of animated rings
$$\begin{tikzcd}
B'\arrow[r]\arrow[d] & A\arrow[d, "{(\mm{id}, 0)}"]\\
B\arrow[r] & A\oplus I[1]
\end{tikzcd}$$
with $I\in \QC(A)_{\geq 0}$, its image under $X\colon (\Aff_{/S})^{\op} \rt \sS$ is a pullback diagram of spaces.

\item[(1')] The condition of (1) holds for all squares such that $A, B$ and $I$ are eventually coconnective.
\end{enumerate}
\end{lemma}
\begin{proof}
The equivalence between $X$ having deformation theory and condition (1) follows from  \Cref{lem:decompose nilpotent extensions} by decomposing a general nilpotent extension as a tower of successive square-zero extensions. The equivalence between (1) and (1') follows from $X$ being convergent. 
\end{proof}

Finally, let us introduce two versions of the formal completion of a prestack, based on the following algebraic constructions:

\begin{definition}[Reduction and pro-reduction]\label{def:prored}
	Let $A\in \SCR_{\KK}$ be an animated $\KK$-algebra.
	\begin{enumerate}
		\item The \textit{reduction} $A_\mm{red}$ is the quotient of $\pi_0(A)$ by its nilradical.
		\item The \textit{pro-reduction} $A_\prored$ to be the ind-ring
		$$
		A_\prored = ``\colim_{I}" \pi_0(A)/I,
		$$
		where the colimit is taken over all nilpotent ideals $I$ of $\pi_0(A)$.
	\end{enumerate}
	
	For $S=\Spec(A)$    affine, we   write $S_\mm{red} = \Spec(A_\mm{red})$ for the underlying reduced scheme. We moreover define a pro-scheme $S_\prored$    over $R$   as the formal limit
	  $$\Spec(A_\mm{prored}) = ``\mathrm{lim}_{I}" \Spec(\pi_0(A)/I), $$ where $I$  again ranges over  all nilpotent ideals   of $\pi_0(A)$. Given another  $R$-prestack $X$, we set
	$$X(	A_\prored) :=\Map_{\mathrm{Pro}(\PreSt)} (S_\prored, X) = \colim_{I } X( \pi_0(A)/I).$$

\end{definition}

\begin{definition}\label{def:inf-stack} \label{formal_completion}
	Let $X\rt  Y$ be a map of  prestacks over $R$. 	\begin{enumerate}
		\item The  \textit{formal completion} $Y^\wedge_X$ is the $R$-prestack sending an animated $\KK$-algebra $A$ to 
	$$
		Y^\wedge_X(A) = X(A_\mm{red})\times_{Y(A_\mm{red})} Y(A).
	$$ 
	\item
		 The \textit{formal infinitesimal completion} of $Y$ at $X$ is the $R$-prestack sending  $A$ to 
		$$
		(X/Y)_{\inf}(A) = X(A_{\prored})\times_{Y(A_\prored)} Y(A).
		$$
	\end{enumerate}
	  
\end{definition} 
\begin{remark}\label{rem:formal completion vs infinitesimal completion}
	When $X \rightarrow Y$ is laft, then the canonical map $	(X/Y)_{\inf} \rightarrow 		Y^\wedge_X$ is an equivalence, as $X(A_{\red})\simeq X(A_{\prored})\times_{Y(A_{\prored})} Y(A_{\red})$.  
\end{remark}
\begin{notation}
	For $Y= \ast$ the terminal prestack, we write   $X_{\inf} = (X/\ast)_{\inf}$. 	In characteristic $0$, the prestack $X_{\inf}$ is often called the \textit{de Rham stack} of $X$.
\end{notation}

\begin{definition}\label{def:nil-iso}
	We will say that a map of $\KK$-prestacks $f\colon X\rt Y$ is a \textit{nil-isomorphism} if for every $A\in \SCR_{\KK}$, the map
	$$
	X(A_{\prored})\rt Y(A_{\prored})
	$$
	is an equivalence.
\end{definition}
\begin{remark}\label{rem:nil-iso laft}
	When $X \rightarrow Y$ is laft, then it is a nil-isomorphism if and only if $X(A ) \rightarrow Y(A)$ is an equivalence for all reduced commutative rings.
\end{remark}

\begin{lemma}\label{lem:formal completion obvious}
Completion satisfies the following properties:
\begin{enumerate}
\item For any $X\rt Y$, the map $Y^\wedge_X\rt Y$ has deformation theory. If $X\rt Y$ is laft, then $X \rt Y^\wedge_X $ and  $Y^\wedge_X\rt Y$ are laft.

\item Let $X'\rt X$ be an étale map between derived schemes and let $X\rt Y$ be a nil-isomorphism. Then the square
$$\begin{tikzcd}
X'\arrow[r]\arrow[d] & X\arrow[d]\\
Y^\wedge_{X'}\arrow[r] & Y
\end{tikzcd}$$
is cartesian.
\end{enumerate}
\end{lemma}
\begin{proof}  
The first part of (1) follows from the fact that for any nilpotent extension $A'\to A$ of animated rings, $Y^\wedge_X(A')\simeq Y^\wedge_X(A)\times_{Y(A)} Y(A')$. For the second part, it is enough to show that $Y^\wedge_X\rt Y$ is laft. This is the base change of the map $\ast^\wedge_X\to \ast^\wedge_Y$, which is laft because for any filtered diagram of $n$-coconnective animated rings $A_\alpha$ with colimit $A$, $\colim X(A_{\alpha, \red})\simeq \colim Y(A_{\alpha, \red})\times_{Y(A_{\red})}X(A_{\red}) $. For (2), the fact that $X'\to X$ is laft implies that $Y^\wedge_{X'}\to Y$ is the base change of $X'_{\inf}\to X_{\inf}$ (Remark \ref{rem:formal completion vs infinitesimal completion}). It now suffices to verify that the map $X'\rt X'_{\inf}\times_{X_{\inf}} X$ is an equivalence. This follows from $X'\rt X$ being étale, so that $X'(A)\simeq X(A)\times_{X(\pi_0A/I)} X'(\pi_0A/I)$ for any nilpotent ideal $I$.
\end{proof}

\subsection{Pro-coherent sheaves} \label{procohappendix} 
The theory of pro-coherent sheaves on affine derived schemes from Section \ref{procohappendix} has an evident global analogue:
\begin{definition}
For a prestack $X$, we define the stable presentable $\infty$-category $\QC^\vee_X$ as
$$
\QC^\vee_X=\lim_{S\in \Aff_{/X}} \QC^\vee_{\mc{O}(S)}
$$
where the limit is taken with respect to the $f^*$-functoriality. Every map $f\colon Y\rt X$ between prestacks induces an adjoint pair $f^*\colon \QC^\vee_X\leftrightarrows \QC^\vee_Y\colon f_*$.
\end{definition}
Let us point out that in the presence of Serre duality, this is Serre dual to the $\infty$-category $\QC^!_X$ of ind-coherent sheaves on $X$ defined using $f^!$-functoriality \cite[9.2.3]{GaitsgoryIndcoh}.
The main features of $\QC^\vee_A$ carry over to pro-coherent sheaves on prestacks by naturality:
\begin{enumerate}
\item Since each $f^*\colon \QC^\vee_A\rt \QC^\vee_B$ is right $t$-exact, for any prestack $X$ the $\infty$-category $\QC^\vee_X$ admits a right complete $t$-structure such that all $f^*$-functors are right $t$-exact. Since the functor $\upiota\colon \Mod_A\to \QC^\vee_A$ intertwines the $f^*$-functors, there is a natural transformation $\upiota\colon \QC_X\rt \QC^\vee_X$ which induces an equivalence on connective objects.

\item Each $\QC^\vee_X$ is tensored and cotensored over $\QC_X$ and $\upiota$ is compatible with the tensoring.

\item If $M$ is an almost perfect $A$-module, then the natural map $f^*(M^\vee)\to (f^*M)^\vee$ is an equivalence. Using this, one sees that for any prestack $X$, pro-coherent duality induces an equivalence
$$\begin{tikzcd}
(-)^\vee\colon \APerf_X^{\op}\arrow[r, "\sim"] & \dAPerf_X
\end{tikzcd}$$
between the full subcategories of pro-coherent sheaves $F$ that are (\emph{dually}) \emph{almost perfect} in the sense that $f^*F$ is (dually) almost perfect for each $f\colon \Spec(A)\to X$.
\end{enumerate}
The $\infty$-category $\QC^\vee_X$ is particularly well-behaved when $X$ is locally coherent. In particular, we have the following generalisation of  \Cref{lem:pro-coh over coherent ring} to derived schemes that are \textit{locally coherent}, i.e.\ which admit an Zariski cover coherent affine open subschemes or equivalently, for which every affine open subscheme is coherent \cite[Theorem 2.4.2, Corollary 2.4.5]{glaz1989coherent}.
\begin{proposition}\label{prop:pro-coh over scheme}
Let $X$ be a qcqs derived scheme which is locally coherent and write write $\mathfrak{U}$ for the poset of affine opens of $X$. Then the natural map 
$$\begin{tikzcd}
\Coh_X^{\op}=\lim_{U\in \mathfrak{U}}\Coh_U^{\op}\arrow[r] & \lim_{U\in \mathfrak{U}}\QC_U^\vee= \QC^\vee_X
\end{tikzcd}$$
induces an equivalence $\mm{Ind}(\Coh_X^{\op})\simeq \QC^\vee_X$.
\end{proposition}
\begin{proof}
By the usual reduction principle for qcqs schemes, it suffices to verify that for $U_0\subseteq X\supseteq U_1$ an open cover such that the assertion holds on $U_0$, $U_1$ and $U_{01}=U_0\cap U_1$, the assertion holds also for $X$. Since $\QC^\vee$ satisfies Zariski descent by Corollary \ref{cor:descent}, it suffices to verify that $\mm{Ind}(\Coh^{\op})$ also satisfies Zariski descent. 

To see this, let us write $j^*\colon \mm{Ind}(\Coh_X^{\op})\leftrightarrows \mm{Ind}(\Coh_U^{\op})\colon j_*$ for the adjoint pair associated to an open inclusion $j\colon U\hookrightarrow X$. Since $j_*$ preserves colimits, this is an adjoint pair in $\cat{Pr^L}$. As such, it is the dual with respect to the Lurie tensor product of the adjoint pair $j^*\colon \mm{Ind}(\Coh_X)\leftrightarrows \mm{Ind}(\Coh_U)\colon j_*$ considered in \cite{GaitsgoryIndcoh}. Using this duality, one sees the counit $j^*j_*\to \mm{id}$ is an equivalence \cite[Lemma 4.1.1]{GaitsgoryIndcoh} and that $\mm{Ind}(\Coh^{\op})$ satisfies base change for open inclusions \cite[Lemma 3.6.9]{GaitsgoryIndcoh}. It follows that
$$\begin{tikzcd}
(j_0^*, j_{01}^*, j_1^*)\colon \mm{Ind}(\Coh_X^{\op})\arrow[r] & \mm{Ind}(\Coh_{U_0}^{\op})\times_{\mm{Ind}(\Coh_{U_{01}}^{\op})} \mm{Ind}(\Coh_{U_1}^{\op})
\end{tikzcd}$$
is an equivalence as in the proof of \cite[Proposition 4.2.1]{GaitsgoryIndcoh}: this functor has a colimit-preserving right adjoint sending a compatible triple $(F_0, F_1, F_{01})$ to the fibre product $j_{0*}F_0\times_{j_{01*}F_{01}} j_{1*}F_1$. Using base change and the fact that $j_i^*j_{i*}\simeq \mm{id}$, one sees that the counit is an equivalence. To verify that the unit $F\rt j_{0*}j_0^*F\times_{j_{01*}j_{01}^*F} j_{1*}j_{1}^*F$ is an equivalence, it suffices to treat the case where $F\in \Coh(X)^{\op}$. In that case, the result follows from descent for coherent sheaves.
\end{proof}

\subsection{Tangent complex}
In this section, we will discuss the tangent complex of a map of prestacks $X\to S$ (if it exists). The main goal will be to establish a form of functoriality of the tangent complex, which is used in Section \ref{Formalintegration}.

Let us start by recalling the pointwise definition of the tangent and cotangent complex of an $S$-prestack. Let $\pi\colon X\rt S$ be an $S$-prestack and $x\colon \Spec(A)\rt X$ a point, and consider the functor 
\begin{equation}\label{diag:tangent}
\Mod_{A,\geq 0}\rt \sS; \quad M\longmapsto\left\{\begin{tikzcd} \Spec(A)\arrow[rr, "x"]\arrow[d] & & X\arrow[d]\\
\Spec(A\oplus M)\arrow[r]\arrow[rru, dotted] & \Spec(A)\arrow[r, "\pi(x)"] & S\end{tikzcd}\right\}.
\end{equation}
We then have the following:
\begin{definition}\label{def:(co)tangent space}
Let $X$ be an $S$-prestack and let $x\colon \Spec(A)\rt X$ be a point.
\begin{enumerate}
\item We will say that $X$ admits a \textit{cotangent complex} at $x$ (relative to $S$) if there exists a (necessarily unique) eventually connective module $L_{X/S, x}\in \Mod_{A, >\infty}$ corepresenting the functor \eqref{diag:tangent}.

\item We will say that $X$ admits a (pro-coherent) \textit{tangent complex at $x$} (relative to $S$) if the functor \eqref{diag:tangent} defines an object $T_{X/S, x}\in \QC^\vee_A$.
\end{enumerate} 
\end{definition}
If $X$ is an $S$-prestack with deformation theory, then each map of rings $f\colon A\rt B$ induces a natural equivalence
$$
X(A\oplus M)\times_{X(A)}\{x\}\simeq X(B\oplus f_*M)\times_{X(B)} \{f^*x\}
$$
for each connective $A$-module $M$. This implies that $L_{X/S, f^*x}\simeq f^*L_{X/S, x}$ and $T_{X/S, f^*x}\simeq f^*T_{X/S, x}$ whenever $X$ admits a (co)tangent complex at the point $x$. Consequently, if $X$ has deformation theory and admits a (co)tangent complex at each of its points, then these (co)tangent complexes glue to global objects on $X$:
\begin{definition}\label{def:tangent complex}
Let $X$ be an $S$-prestack with deformation theory.
\begin{enumerate}
\item If $X$ admits a cotangent complex at each point, then the \textit{cotangent complex} $L_{X/S}\in \QC_X$ is the unique quasi-coherent sheaf such that $L_{X/S, x}\simeq x^*L_{X/S}$ for each point of $X$. 

\item If $X$ admits a tangent complex at each point, then the \textit{tangent complex} $T_{X/S}\in \QC^\vee_X$ is the unique pro-coherent sheaf such that $T_{X/S, x}\simeq x^*T_{X/S}$ for each point of $X$. 
\end{enumerate}
\end{definition}
\begin{example}\label{ex:laft implies tangent at each point}
Suppose that $X$ is a laft $S$-prestack with deformation theory. Then $X$ admits a relative tangent complex at each point.
\end{example}
Having a tangent complex at each point requires a finiteness condition on $X\to S$ similar to Example \ref{ex:laft implies tangent at each point} (but at the linear level). However, there are many prestacks for which one can make sense of their tangent complex, without requiring the existence of a tangent complex at each point. For example, for every affine derived $\KK$-scheme $X$, we can define the tangent complex $T_{X/\KK}=L_{X/\KK}^\vee$ to be the pro-coherent dual of its cotangent complex. In order to define the tangent complex in such cases, and to make the functoriality of the tangent complex more explicit, we will need a somewhat technical argument.
\begin{construction}\label{con:qcoh etc}
Let $\Mod_{\geq 0, *}\rt \Aff$ be the cocartesian fibration classifying the functor sending $f\colon \Spec(A)\rt \Spec(B)$ to the direct image $f_*\colon \Mod_{A,\geq 0}\rt \Mod_{B, \geq 0}$ and let us consider the following pullback
$$\begin{tikzcd}
\cat{M}\arrow[rr]\arrow[d] & & \Fun([1], \PrStk)\arrow[d, "\mm{ev}_0"]\\
\Mod_{\geq 0, *}\arrow[r] & \Aff\arrow[r] & \PrStk
\end{tikzcd}$$
An object of $\cat{M}$ corresponds a map of prestacks $x\colon \Spec(A)\rt X$ together with a connective $A$-module $M$. The functor $\cat{M}\rt \PrStk$ sending $(\Spec(A)\to X, M)\longmapsto X$ is a cocartesian fibration and each fibre $\cat{M}_X$ comes equipped with a cocartesian fibration $\cat{M}_X\rt \Aff_{/X}$ parametrising the $\infty$-category of connective modules at each point of $X$.

Let $\Fun_{/\PrStk}(\cat{M}, \sS)\rt \PrStk$ be the relative functor $\infty$-category over $\PrStk$; by \cite{HTT} this is a cartesian fibration, classified by the functor $\PrStk^{\op} \rightarrow \cat{Cat}_{\infty} $ sending a prestack $X$ to $\Fun(\cat{M}_X, \sS)$ and sending $f\colon X\rt Y$ to restriction along $\cat{M}_X\rt \cat{M}_Y$. Let us write
$$
\cat{D}\hookrightarrow \Fun_{/\PrStk}(\cat{M}, \sS)
$$
for the full subcategory of tuples $(X, F)$ consisting of a prestack $X$ and a functor $F\colon \cat{M}_X\rt \sS$ satisfying the following two properties:
\begin{enumerate}
\item\label{it:coherence} $F$ sends cocartesian arrows with respect to $\cat{M}_X\rt \Aff_{/X}$ to equivalences.
\item\label{it:pro-rep} For each $x\colon \Spec(A)\rt X$, the restriction of $F$ to the fibre 
$$
F_x\colon \cat{M}_{X, x}\simeq \Mod_{A, \geq 0}\rt \sS
$$
is accessible and preserves the terminal object, pullbacks along $\pi_0$-surjections and limits of almost eventually constant towers.
\end{enumerate}
Let us write $\cat{D}^{\aft}\subseteq \cat{D}$ and $\cat{D}^\mm{corep}\subseteq \cat{D}$ for the subcategories of tuples $(X, F)$ where, in addition to condition \eqref{it:pro-rep}, each restriction $F_x$ is almost finitely presented, respectively corepresentable by an eventually connective $A$-module.

One readily verifies that for any map of prestacks $f\colon X\rt Y$, restriction along $\cat{M}_X\rt \cat{M}_Y$ preserves these conditions, so that for any cartesian arrow $(X, G)\rt (Y, F)$ with $(Y, F)\in \cat{D}$, $(X, G)$ is in $\cat{D}$ as well (and similarly in the other cases). Consequently, we obtain a diagram of cartesian fibrations (and maps preserving cartesian arrows)
$$\begin{tikzcd}
\cat{D}^{\aft}\arrow[r, hook]\arrow[rd] & \cat{D}\arrow[d] & \cat{D}^{\mm{corep}}\arrow[ld]\arrow[l, hook']\\
& \PrStk
\end{tikzcd}$$
\end{construction}
\begin{proposition}
Each of the above three cartesian fibrations classifies a functor $\PrStk^{\op}\rt \Cat_{\infty}$ preserving limits. Furthermore, the cartesian fibration $\cat{D}^{\aft}\rt \PrStk$ is classified by the functor $X\longmapsto \QC^\vee_X$ and the cartesian fibration $\cat{D}^{\mm{corep}}\rt \PrStk$ is classified by the functor $X\longmapsto (\QC^{-}_{X})^{\op}$.
\end{proposition}
\begin{proof}
Let $X$ be a prestack and consider the cocartesian fibration $\cat{M}_X\rt \Aff_{/X}$ that classifies the diagram of categories $\Aff_{/X}\rt \Cat_\infty$ sending a map of affines $f\colon \Spec(A)\rt \Spec(B)$ over $X$ to $f_*\colon \Mod_{A,  \geq 0}\rt \Mod_{B,  \geq 0}$. The localisation of $\cat{M}_X$ at the cocartesian arrows classifies the colimit of this diagram of categories.  Condition \eqref{it:coherence} therefore implies that $\cat{D}_X$ is a full subcategory of the limit
$$
\lim_{\Spec(A)\in \Aff_{/X}} \Fun\big(\Mod_{A,  \geq 0 }, \sS\big)
$$
where each morphism $f\colon \Spec(A)\rt \Spec(B)$ in $\Aff_{/X}$ is sent to restriction along $f_*$. Using this, Condition \eqref{it:pro-rep} then identifies
$$
\cat{D}_X \simeq \lim_{\Spec(A)\in \Aff_{/X}} \Fun_{\mm{rex}, \mm{acc}}(\Mod^+_{A, \geq 0}, \sS)
$$
with the limit of the $\infty$-categories of accessible functors that are right exact (or equivalently, that preserve the terminal object and pullbacks along $\pi_0$ -surjections). This identification is natural in $X$, so that $X\mapsto \cat{D}_X$ evidently preserves limits. The categories $\cat{D}^{\aft}_X$ and $\cat{D}^{\mm{corep}}_X$ are the limits over $\Aff_{/X}$ of the diagram of full subcategories of $\Fun_{\mm{rex}, \mm{acc}}(\Mod^+_{A, \geq 0}, \sS)$ spanned the by almost finitely presented and corepresentable functors. It follows that there are natural equivalences $\cat{D}^{\aft}\simeq \QC^\vee_X$ and $\cat{D}_X^\mm{corep}\simeq (\QC^{-}_{X})^{\op}$.
\end{proof}
\begin{observation}\label{obs:aft approx}
Given a prestack $X$, the inclusion
$$\begin{tikzcd}
\QC^\vee_X=\lim_{\Spec(A)\to X} \QC^\vee_A\arrow[r, hookrightarrow] & \lim_{\Spec(A)\to X} \Fun_{\mm{rex}, \mm{acc}}(\Mod^+_{A, \geq 0}, \sS)\simeq \cat{D}_X
\end{tikzcd}$$
preserves colimits. Since $\QC^\vee_X$ is presentable, it follows that there is a right adjoint $\cat{D}_X\rt \QC^\vee_X$. Given $F\in \cat{D}_X$, we will refer to its image under this right adjoint $F^{\aft}$ as the \emph{almost finitely presented approximation} to $F$. Let us point out that the right adjoints $\cat{D}_X\to \QC^\vee_X$ do \emph{not} assemble into a right adjoint $\cat{D}\to \cat{D}^{\aft}$.
\end{observation}
Let us now use this construction to give a functorial description of the tangent complex.
\begin{construction}\label{tangent_complex_naive}
Let us write $\Fun(\Delta^1, \PrStk)^{\mm{def}}\subseteq \Fun(\Delta^1, \PrStk)$ for the full subcategory of maps $X\rt S$ that have deformation theory. We define a functor
$$
\Phi\colon \cat{M}\times_{\PrStk} \Fun(\Delta^1, \PrStk)^{\mm{def}}\rt \sS;
$$
as follows. The domain of $\Phi$ is the fibre product of the cocartesian fibration $\cat{M}\rt \PrStk$ from Construction \ref{con:qcoh etc} and the functor $\mm{ev}_0\colon \Fun(\Delta^1, \PrStk)^\mm{def}\rt \PrStk$ taking the domain. We then define $\Phi$ by
$$
\Phi\Big(x\colon \Spec(A)\to X,\ M, \ X\to S\Big) = \left\{\begin{tikzcd} \Spec(A)\arrow[rr, "x"]\arrow[d] & & X\arrow[d]\\
\Spec(A\oplus M)\arrow[r]\arrow[rru, dotted] & \Spec(A)\arrow[r, "\pi(x)"] & S\end{tikzcd}\right\}.
$$
The functor $\Phi$ is adjoint to
$$\begin{tikzcd}
\Fun(\Delta^1, \PrStk)^{\mm{def}}\arrow[rd, "\mm{ev}_0"{swap}]\arrow[rr] & & \Fun_{/\PrStk}(\cat{M}, \sS)\arrow[ld]\\
& \PrStk
\end{tikzcd}$$
sending $X\to S$ to the restriction $\Phi(-, X\to S)\colon \cat{M}_X\rt \sS$. The fact that $X\to S$ has deformation theory implies that $\Phi(-, X\to S)$ satisfies conditions \eqref{it:coherence} and \eqref{it:pro-rep} from Construction \ref{con:qcoh etc}. Consequently, we obtain a functor that we will denote by
\begin{equation}\label{eq:Tnaive}\begin{tikzcd}
\Fun(\Delta^1, \PrStk)^{\mm{def}}\arrow[rd, "\mm{ev}_0"{swap}]\arrow[rr, "T^\mm{pre}"] & & \cat{D}\arrow[ld]\\
& \PrStk.
\end{tikzcd}\end{equation}
If $X\to S$ is laft, then $T^\mm{pre}(X/S)$ is contained in $\cat{D}^{\aft}$. Under the equivalence $\cat{D}^{\aft}_X\simeq \QC^\vee(X)$, this corresponds to the tangent complex $T_{X/S}$ from  \Cref{def:tangent complex}. If $X\to S$ admits cotangent complexes at all points of $X$, then $T^\mm{pre}(X/S)\in \cat{D}^\mm{corep}$ and corresponds under the equivalence $\cat{D}^\mm{corep}_X\simeq (\QC^{-}_{X})^{\op}$ to the cotangent complex $L_{X/S}$ from \Cref{def:tangent complex}.
\end{construction}
From the formula for $\Phi$, one readily sees:
\begin{lemma}\label{lem:Tpre pullbacks}
The functor $T^\mm{pre}$ preserves pullbacks.
\end{lemma}
Since $\cat{D}^{\mm{corep}}\subseteq \cat{D}$ is closed under pullbacks, this reproduces for example the well-known behaviour of the cotangent complex with respect to base change and composition.

Let us now use the functor $T^\mm{pre}$ to define the tangent complex of a prestack with deformation theory:
\begin{definition}\label{def:approximate tangent}
Let $X$ be an $S$-prestack with deformation theory. We define the (pro-coherent) \textit{tangent complex of $X$} over $S$ to be the almost finitely presented approximation $T_{X/S}=T_{X/S}^{\mm{pre}, \aft}\in \QC^\vee_X$ from Observation \ref{obs:aft approx}.
\end{definition}
\begin{example}\label{ex:finite case tangent}
If $X$ is a \emph{laft} $S$-prestack with deformation theory, this coincides with the pointwise definition of tangent complex from Definition \ref{def:tangent complex}. Indeed, since $X\to S$ is laft, $T_{X/S}^\mm{pre}$ is already contained in $\cat{D}_X^\mm{aft}\subseteq \cat{D}_X$.
\end{example}
\begin{example}\label{ex:dual cotangent}
Suppose that $X$ is a (not necessarily laft) $S$-prestack with deformation theory and a cotangent complex at each point. Then there is an equivalence $T_{X/S}=L_{X/S}^\vee$ in $\QC^\vee_X$, where $L_{X/S}^\vee$ is the pro-coherent linear dual of $L_{X/S}\in (\QC^{-}_{X})^{\op}$ (see \Cref{ex:pro-coh dual}). Since the right adjoint $\QC^\vee_X\rt \QC_X$ preserves the linear dual of a quasi-coherent sheaf, one can think of $T_{X/S}\in \QC^\vee_X$ as a pro-coherent refinement of the \emph{coherator} of the tangent sheaf on $X$, in the sense of \cite[Tag 08D6]{stacks-project}.
\end{example}
The tangent complex $T_{X/S}$ is much less well-behaved when $X\to S$ is not laft. For example, given a map of $S$-prestacks with deformation theory $X\to Y$, there is no direct map from $T_{X/S}$ to $f^*T_{Y/S}$. For instance, if $X$ and $Y$ admit a cotangent complex at each point, there is only a zig-zag of maps $L_{X/S}^\vee\to (f^*L_{Y/S})^\vee\leftarrow f^*(L_{Y/S}^\vee)$ because base change does not commute with duality without finiteness assumptions. Let us nonetheless record the following properties:
\begin{proposition}\label{prop:tangent obvious properties}
Let $f\colon X\to Y$ be a map of $S$-prestacks with deformation theory. Then the following hold.
\begin{enumerate}
\item There is a fibre sequence $T_{X/Y}\to T_{X/S}\to (f^*T^\mm{pre}_{Y/S})^{\aft}$.

\item There is a zig-zag of maps $T_{X/S}\to (f^*T^\mm{pre}_{Y/S})^{\aft} \leftarrow f^*T_{Y/S}$ in $\QC^\vee_X$.

\item If $f$ is a formally étale, then there is a map $f^*T_{Y/S}\to T_{X/S}$.

\item $T_{X/S}\simeq T_{X/S^\wedge_X}$.
\end{enumerate}
If $X$ and $Y$ are furthermore \emph{laft} $S$-prestacks with deformation theory, then:
\begin{enumerate}[resume]
\item There is natural map $T_{X/S}\to f^*T_{Y/S}$ that fits into a fibre sequence $T_{X/Y}\rt T_{X/S}\rt f^*T_{Y/S}$ in $\QC^\vee_X$.
\item Given a morphism $S'\rt S$, there is a natural equivalence $T_{X\times_S S'/S'}=p^*T_{X/S}$ where $p\colon X\times_S S'\rt X$ is the evident projection.
\end{enumerate}
\end{proposition}
\begin{proof}
The first three assertions follow from the natural fibre sequence $T^\mm{pre}_{X/Y}\to T^\mm{pre}_{X/S}\to f^*T^\mm{pre}_{Y/S}$ in $\cat{D}_X$ by taking almost finitely presented approximations, using that $T^\mm{pre}_{X/Y}\simeq 0$ if $X\to Y$ is formally étale. Likewise, (4) follows from the fact that $T^\mm{pre}_{X/S}\simeq T^\mm{pre}_{X/S^\wedge_X}$. Assertion (5) follows from the fact that $T_{Y/S}=T^\mm{pre}_{Y/S}$ and (6) follows from the equivalence $T^\mm{pre}_{X\times_S S'/S'}=p^*T^\mm{pre}_{X/S}$, which follows from Lemma \ref{lem:Tpre pullbacks}.
\end{proof}

The somewhat problematic behaviour of $T_{X/S}$ when $X\to S$ is not laft will not be a problem for us, because we are mainly interested in the $\infty$-category $(\QC^\vee_X)_{/T_{X/S}}$ of pro-coherent sheaves over $T_{X/S}$. This $\infty$-category is equivalent to the $\infty$-category $(\QC^\vee_X)_{/T_{X/S}^\mm{pre}}$ of pro-coherent sheaves equipped with a map $F\to T^\mm{pre}_{X/S}$ in $\cat{D}_X$, and $T^\mm{pre}_{X/S}$ is very well-behaved.

For instance, each laft map $X\to Y$ of $S$-prestacks with deformation theory gives rise to a natural object $T_{X/Y}\in (\QC^\vee_X)_{/T_{X/S}}$, or equivalently, in $(\QC^\vee_X)_{/T_{X/S}^\mm{pre}}$. Let us conclude with a functorial description of this construction:
\begin{construction}
Let us fix a base prestack $S$ and write $\PrStk^{\mm{def}, \mm{laft-map}}_{S}$ for the subcategory of $\PrStk_{S}$ whose objects are prestacks with deformation theory over $S$ and whose morphisms are the laft morphisms between such $S$-prestacks. Let us write
$$
\cat{E}=\cat{D}\times_{\PrStk} \PrStk^{\mm{def}, \mm{laft-map}}_{S}, \qquad \qquad \QC^\vee=\cat{D}^{\aft}\times_{\PrStk} \PrStk^{\mm{def}, \mm{laft-map}}_S
$$
for the evident restrictions. The cartesian fibration $\cat{E}\rt \PrStk^{\mm{def}, \mm{laft-map}}_{S}$ then admits a section $T^\mm{pre}$ sending each $X\longmapsto T^\mm{pre}_{X/S}$. Taking fibrewise over-categories then produces a diagram that we will denote by
$$\begin{tikzcd}
\QC^\vee_{/T} \defeq \QC^\vee\times_{\cat{E}} \cat{E}_{/T^\mm{pre}}\arrow[rr, hook] \arrow[rd] & & \cat{E}_{/T^\mm{pre}}\arrow[ld]\\
& \PrStk^{\mm{def}, \mm{laft-map}}_{S}.
\end{tikzcd}$$
At the level of the fibres over a prestack $X$, this is given by
$$\begin{tikzcd}
(\QC^\vee_X)_{/T_{X/S}}\simeq (\QC^\vee_X)\times_{\cat{D}_X} (\cat{D}_X)_{/T^\mm{pre}_{X/S}}\arrow[r, hook] & (\cat{D}_X)_{/T^\mm{pre}_{X/S}}.
\end{tikzcd}$$
\end{construction}
\begin{lemma}\label{lem:anchor functorial}
The projection $\QC^\vee_{/T}\rt \PrStk^{\mm{def}, \mm{laft-map}}_{S}$ is a cartesian fibration.
\end{lemma}
\begin{proof}
Let us first consider the cartesian fibration $\pi\colon \cat{E}\rt \PrStk^{\mm{def}, \mm{laft-map}}_S$. This classifies a functor sending $X$ to $\cat{E}_X=\lim_{\Spec(A)\to X} \Fun_{\mm{ex}, \mm{acc}}(\Mod^+_A, \Sp)$ and $f\colon X\rt Y$ to the evident functor between limits
$$
f^*\colon \cat{E}_Y=\lim_{\Spec(A)\to Y} \Fun_{\mm{ex}, \mm{acc}}(\Mod^+_A, \Sp)\rt \lim_{\Spec(A)\to X} \Fun_{\mm{ex}, \mm{acc}}(\Mod^+_A, \Sp)=\cat{E}_X.
$$
Note that each $\cat{E}_X$ is a stable $\infty$-category and that $f^*$ is exact. This implies that $\cat{E}$ admits finite limits and that $\pi$ admits a fully faithful right adjoint (taking the terminal object fibrewise). 
The projection $\cat{E}_{/T^\mm{pre}}\rt \PrStk^{\mm{def}, \mm{laft-map}}_{S}$ is then a cartesian fibration as well: indeed, given a map $f\colon X\rt Y$ and $F\rt T^\mm{pre}_{Y/S}$ in $\cat{E}_Y$, the cartesian lift of $f$ in $\cat{E}_{/T^\mm{pre}}$ corresponds to the following pullback square in $\cat{E}$:
$$\begin{tikzcd}
T^\mm{pre}_{X/S}\times_{T^\mm{pre}_{Y/X}} F\arrow[r, "\tilde{f}"]\arrow[d] & F\arrow[d]\\
T^\mm{pre}_{X/S}\arrow[r] & T^\mm{pre}_{Y/X}.
\end{tikzcd}$$
More explicitly, this is given by the pullback $T^\mm{pre}_{X/S}\times_{f^*T^\mm{pre}_{Y/X}} f^*F$ in the fibre $\cat{E}_X$.

Finally, it suffices that the full subcategory $\QC^\vee_{/T}\subseteq \cat{E}_{/T^\mm{pre}}$ is stable under these cartesian arrows. For this, suppose that $F\in \QC^\vee_Y\subseteq \cat{E}_Y$ and consider the fibre sequence in $\cat{E}_X$
$$
T^\mm{pre}_{X/Y}\rt T^\mm{pre}_{X/S}\times_{f^*T^\mm{pre}_{Y/X}} f^*F \rt f^*F.
$$
The base change functor $f^*$ sends $\QC^\vee_Y$ to $\QC^\vee_X$, so $f^{\ast}F\in \QC^\vee_X$. Since $f\colon X\rt Y$ was assumed to be laft, $T^\mm{pre}_{X/Y}=T_{X/Y}\in \QC^\vee_X$ as well so that the middle term is contained in $\QC^\vee_X$, as desired.
\end{proof}
\begin{remark}\label{rem:fsharp}
Informally, Lemma \ref{lem:anchor functorial} asserts that there is a functor sending each $S$-prestack $X$ with deformation theory to the $\infty$-category $(\QC^\vee_X)_{/T_{X/S}}$, and each map $f\colon X'\to X$ to a functor
$$\begin{tikzcd}
f^\sharp\colon (\QC^\vee_X)_{/T_{X/S}}\arrow[r] & (\QC^\vee_{X'})_{/T_{X'/S}}; \quad (F\to T_{X/S})\arrow[r, mapsto] & T_{X'/S}\times_{f^*(T^\mm{pre}_{X/S})^\mm{aft}} F.
\end{tikzcd}$$
In particular, when $f$ is formally étale, then $f^\sharp(F\to T_{X/S})= f^*F$ by Proposition \ref{prop:tangent obvious properties}.

If $f\colon X\rt Y$ is a map of $S$-prestacks with deformation theory and a cotangent complex, Example \ref{ex:dual cotangent} identifies the resulting base change functor with $f^\sharp(F)= T_{X/S}\times_{(f^*L_{Y/S})^\vee} f^*F$, using the natural map $f^*(T_{Y/S})=f^*(L_{Y/S}^\vee)\rt (f^*L_{Y/S})^\vee$.
\end{remark}
\begin{proposition} \label{naturalityofT}
There is a map of cartesian fibrations preserving cartesian arrows
$$\begin{tikzcd}
\Fun(\Delta^1, \PrStk_S^{\mm{def}, \mm{laft-map}})\arrow[rr, "T"]\arrow[rd, "\mm{ev}_0"{swap}] & & \QC^\vee_{/T}\arrow[ld]\\ 
& \PrStk_S^{\mm{def}, \mm{laft-map}}
\end{tikzcd}$$
sending each $X\to Y$ to $T_{X/Y}\to T_{X/S}$. Furthermore, the induced functors between fibres preserve limits.
\end{proposition}
In other words, for every square of laft maps between $S$-prestacks with deformation theory
$$\begin{tikzcd}
X'\arrow[r]\arrow[d] & X\arrow[d]\\
Y'\arrow[r] & Y
\end{tikzcd}$$
there is a natural map $T_{X'/Y'}\to f^{\sharp}(T_{X/Y})$ where $f^\sharp$ is as in Remark \ref{rem:fsharp}. The fact that $T$ preserves cartesian arrows corresponds to this map being an equivalence if $Y'=Y$. We are mostly interested for in the situation of (locally coherent) derived schemes $X$ (over a base $S$): the above then provides a functor sending $X\to Y$ a map $T_{X/Y}\to L_{X/Y}^\vee$ in $\QC^\vee_X$, functorial with respect to laft maps of schemes $X'\to X$.
\begin{proof}
The desired functor arises formally from the functor $T^\mm{pre}$ in \eqref{eq:Tnaive}, using that $T_{X/Y}=T^\mm{pre}_{X/Y}$ if $X\to Y$ is laft (Example \ref{ex:finite case tangent}). To see that is preserves cartesian arrows, note that $T^\mm{pre}$ sends
$$\begin{tikzcd}
X'\arrow[r, "f"]\arrow[d] & X\arrow[d] \arrow[rr, mapsto, yshift=-4ex, shorten=1.8pc] &[20pt] & T^\mm{pre}_{X'/Y}\arrow[r]\arrow[d] & T^\mm{pre}_{X/Y}\arrow[d]\\
Y\arrow[r, "\mm{id}"] & Y & & 
T^\mm{pre}_{X'/S}\arrow[r] & T^\mm{pre}_{X/S}
\end{tikzcd}$$
Here the left square defines a cartesian arrow covering $f$ in the domain. To see that its image is cartesian, it suffices to verify that the right square is cartesian in $\cat{E}$; this follows from the fact that the functor $T^\mm{pre}$ from \eqref{eq:Tnaive} preserves pullbacks by Lemma \ref{lem:Tpre pullbacks}.
\end{proof}
We conclude by recording the following property of the induced functors $T_{X/-}$ between the fibres.
\begin{proposition}\label{prop:left adjoint via sqz ext}
Let $X$ be an $S$-prestack with deformation theory and consider the composite
$$\begin{tikzcd}
\Psi\colon (\PrStk_S^{\mm{def}, \mm{laft-map}})_{X/-} \arrow[r, "T_{X/-}"] & (\QC^\vee_X)_{/T_{X/S}}\arrow[r, "\mm{cofib}"] & \QC^\vee_X
\end{tikzcd}$$
where the first functor is induced by the functor $T$ from Proposition \ref{naturalityofT} and the second functor sends $F\to T_{X/S}$ to its cofibre. This functor admits a left adjoint, sending each $F\in \QC^\vee_X$ to a map $X\to X_F$ that is furthermore a nil-isomorphism. In particular, the functor $T_{X/-}$ preserves limits.
\end{proposition}
\begin{proof}
Unravelling the definitions, $\Psi(f\colon X\to Y)= (f^*T^\mm{pre}_{Y/S})^{\aft}$. To describe its left adjoint, recall from Construction \ref{con:qcoh etc} that each $F\in \QC^\vee_X$ defines a functor $F\colon \mc{M}_X\to \sS$ on the $\infty$-category of tuples $(x\colon \Spec(A)\to X, \ M\in \Mod_{A, \geq 0}^+)$. Let $\cat{F}\to \mc{M}_X$ be the (cocartesian) unstraightening of this functor. We then define
$$
X_F = \colim_{(A, M)\in \cat{F}^{\op}} \Spec(A\oplus M)
$$
as the colimit in convergent $S$-prestacks. There are natural maps $X\to X_F\to X$. The first map is the natural map from the colimit indexed by the subcategory of $\mc{F}^{\op}$ on those $(A, M)$ where $M\simeq 0$ is the zero module. The second map is the colimit over the objects of $\mc{F}$, corresponding to tuples $$(x\colon \Spec(A)\to X, \ M\in \Mod_{A, \geq 0}^+,\ y\in F(M)),$$ of the maps $\Spec(A\oplus M)\to \Spec(A)\xrightarrow{x} X$. Given a map $f\colon X\to Y$ to an $S$-prestack with deformation theory, a map $X_F\to Y$ of $S$-prestacks under $X$   corresponds  to a natural transformation $F\to f^*T^\mm{pre}_{Y/S}$ between functors $\mc{M}_X\to \sS$, i.e.\ a map $F\to (f^*T^\mm{pre}_{Y/S})^{\aft}=\Psi(f\colon X\to Y)$ in $\QC^\vee_X$. 

Hence, it remains to show that $X\to X_F$ is indeed a laft nil-isomorphism and that $X_F$ has deformation theory. For this it suffices to check that for each $x\colon \Spec(A)\to X$ with $A$ eventually coconnective, the base change $x^*X_F$ is a laft $A$-prestack with deformation theory that is nil-isomorphic to $\Spec(A)$. This base change can be identified as follows. Let $\mc{F}_x$ denote the unstraightening of the restriction $x^*F\colon \Mod_{A, \geq 0}^+\to \sS$ of $F$ to the fibre of $(\mc{M}_X)\times_{\Aff_{/X}} \{x\}$. Then
\begin{equation}\label{eq:colim X_F}
x^*X_F = \colim_{M\in \mc{F}^{\op}_x} \Spec(A\oplus M)
\end{equation}
where the colimit is taken in convergent $A$-prestacks. This is clearly nil-isomorphic to $\Spec(A)$, as every $\Spec(A\oplus M)$ is nil-isomorphic to $\Spec(A)$. The category $\mc{F}^{\op}_x$ is filtered, so that $x^*X_F$ has deformation theory, being a filtered colimit of $A$-prestacks with deformation theory.

To see that $x^*X_F$ is laft, observe that for any eventually coconnective animated $A$-algebra $B$, one can either compute the value $x^*X_F(B)$ using the colimit \eqref{eq:colim X_F}, or as follows. Let us write $\cat{Sqz}_{/B}=\Mod_{A, \geq 0}^+\times_{\SCR_A} \SCR_{A/B}$ for the $\infty$-category of $A$-modules $M\in \Mod_{A, \geq 0}^+$ equipped with a map of animated $A$-algebras $A\oplus M\to B$. Then
$$
x^*X_F(B) = \colim_{M\in \cat{Sqz}_{/B}} x^*F(M).
$$
If $B$ is $n$-coconnective, it suffices to take the colimit over the cofinal subcategory $\cat{Sqz}_{\leq n/B}\subseteq \cat{Sqz}_{/B}$ of $n$-coconnective $A$-modules $M$ equipped with a map $A\oplus M\to B$. Note that $\cat{Sqz}_{\leq n/B}$ is compactly generated with compact objects given by $$(M, A\oplus M\to B),$$ where $M$ is the $n$-truncation of a perfect connective $A$-module; let us refer to such $n$-coconnective modules as \emph{$n$-compact}. Since $x^*F$ preserves colimits of filtered diagrams of $n$-coconnective $A$-modules, one can thus reduce the above colimit to the colimit over the full subcategory $\cat{Sqz}^\omega_{\leq n/B}\subseteq \cat{Sqz}_{\leq n/B}$ of $n$-compact $A$-modules $M$ with a map $A\oplus M\to B$. Equivalently, this means that
$$
x^*X_F(B) = \colim_{M\in \mc{F}_{x, \leq n}^{\omega, \op}} \Map_{\SCR_A}(A\oplus M, B)\simeq \colim_{M\in \mc{F}_{x, \leq n}^{\omega, \op}} \Map_{\tau_{\leq n}\SCR_A}(\tau_{\leq n}A\oplus M, B)
$$
where $\mc{F}_{x, \leq n}^{\omega}\subseteq \mc{F}_x$ is the full subcategory of tuples $(M, y\in x^*F(M))$ where $M$ is   $n$-compact.

We conclude that the restriction of $x^*X_F$ to $\tau_{\leq n}\SCR_A$ is equivalent to a colimit of representables $\Spec(\tau_{\leq n}A\oplus M)$ where $M$ is an $n$-compact $A$-module. By Lemma \ref{lem:sqz is aft}, each square zero extension $\tau_{\leq n}A\oplus M$ by an $n$-compact $A$-module defines a compact object in $\tau_{\leq n}\SCR_A$. It follows that $x^*X_F$ restricts to a functor on $\tau_{\leq n}\SCR_A$ preserving filtered colimits, so that $x^*X_F$ is indeed laft.
\end{proof}

\bibliographystyle{amsalpha}

\bibliography{There}

\providecommand{\bysame}{\leavevmode\hbox to3em{\hrulefill}\thinspace}
\providecommand{\MR}{\relax\ifhmode\unskip\space\fi MR }
\providecommand{\MRhref}[2]{%
  \href{http://www.ams.org/mathscinet-getitem?mr=#1}{#2}
}
\providecommand{\href}[2]{#2}
\begin{thebibliography}{{Sta}19}

\bibitem[BCN]{BCN21}
D.~Lukas~B. Brantner, Ricardo Campos, and Joost~J. Nuiten, \emph{P{D} operads
  and explicit partition lie algebras}, Mem. Amer. Math. Soc., to appear.

\bibitem[BM]{BM19}
D.~Lukas~B. Brantner and Akhil Mathew, \emph{Deformation theory and partition
  {L}ie algebras}, Acta Math., to appear.

\bibitem[BW20]{brantner2020purely}
D.~Lukas~B. Brantner and Joe Waldron, \emph{Purely inseparable {G}alois theory
  {I}: the fundamental theorem}, arXiv preprint arXiv:2010.15707 (2020).

\bibitem[Del]{deligne1986letter}
Pierre Deligne, \emph{Letter to {J}. {M}illson and {W}. {G}oldman (1986)}.

\bibitem[DG02]{dwyer2002complete}
William~G. Dwyer and John P.~C. Greenlees, \emph{Complete modules and torsion
  modules}, Amer. J. Math. \textbf{124} (2002), no.~1, 199--220. \MR{1879003}

\bibitem[Dri]{drinfeld1988letter}
Vladimir~G. Drinfeld, \emph{A letter from {K}harkov to {M}oscow}, EMS Surv.
  Math. Sci. \textbf{1}, no.~2, 241--248, Translated from Russian by Keith
  Conrad. \MR{3285856}

\bibitem[Eke87]{Ekedahl}
Torsten Ekedahl, \emph{Foliations and inseparable morphisms}, Algebraic
  geometry, {B}owdoin, 1985 ({B}runswick, {M}aine, 1985), Proc. Sympos. Pure
  Math., vol. 46, Part 2, Amer. Math. Soc., Providence, RI, 1987, pp.~139--149.
  \MR{927978}

\bibitem[Fu24]{FuThesis}
Jiaqi Fu, \emph{A duality between {L}ie algebroids and infinitesimal
  foliations}, arXiv preprint arXiv:2410.04950 (2024).

\bibitem[Gai11]{GaitsgoryIndcoh}
Dennis Gaitsgory, \emph{Ind-coherent sheaves}, arXiv preprint arXiv:1105.4857
  (2011).

\bibitem[GHN17]{GHN17}
David Gepner, Rune Haugseng, and Thomas Nikolaus, \emph{Lax colimits and free
  fibrations in {$\infty$}-categories}, Doc. Math. \textbf{22} (2017),
  1225--1266. \MR{3690268}

\bibitem[Gla89]{glaz1989coherent}
Sarah Glaz, \emph{Commutative coherent rings}, Lecture Notes in Mathematics,
  vol. 1371, Springer-Verlag, Berlin, 1989. \MR{999133}

\bibitem[GR17a]{GR}
Dennis Gaitsgory and Nick Rozenblyum, \emph{A study in derived algebraic
  geometry}, no. 221, American Mathematical Soc., 2017.

\bibitem[GR17b]{GRII}
\bysame, \emph{A study in derived algebraic geometry. {V}ol. {II}.
  {D}eformations, {L}ie theory and formal geometry}, Mathematical Surveys and
  Monographs, vol. 221, American Mathematical Society, Providence, RI, 2017.
  \MR{3701353}

\bibitem[Har66]{hartshorne1966residues}
Robin Hartshorne, \emph{Residues and duality}, vol.~20, Springer, 1966.

\bibitem[Hen18]{Hennion}
Benjamin Hennion, \emph{Tangent {L}ie algebra of derived {A}rtin stacks}, J.
  Reine Angew. Math. \textbf{741} (2018), 1--45. \MR{3836141}

\bibitem[Hin01]{hinich2001dg}
Vladimir Hinich, \emph{D{G} coalgebras as formal stacks}, J. Pure Appl. Algebra
  \textbf{162} (2001), no.~2-3, 209--250. \MR{1843805}

\bibitem[Hin16]{hinich2013dwyer}
\bysame, \emph{Dwyer-{K}an localization revisited}, Homology Homotopy Appl.
  \textbf{18} (2016), no.~1, 27--48. \MR{3460765}

\bibitem[Hol23]{holeman2023derived}
Adam Holeman, \emph{{D}erived $\delta$-rings and relative prismatic
  cohomology}, arXiv preprint arXiv:2303.17447 (2023).

\bibitem[Kap07]{kapranov2007free}
Mikhail Kapranov, \emph{Free {L}ie algebroids and the space of paths}, Selecta
  Math. (N.S.) \textbf{13} (2007), no.~2, 277--319. \MR{2361096}


\bibitem[HTT]{HTT}
\bysame, \emph{Higher topos theory}, Annals of Mathematics Studies, vol. 170,
  Princeton University Press, Princeton, NJ, 2009. \MR{2522659}

\bibitem[HA]{HA}
\bysame, \emph{Higher algebra}, Preprint from the author's web page (2017).

\bibitem[SAG]{SAG}
\bysame, \emph{Spectral algebraic geometry}, Preprint available from the
  author's web page (2016).

\bibitem[DAG]{lurie2004derived}
Jacob~A. Lurie, \emph{Derived algebraic geometry}, Ph.D. thesis, Massachusetts
  Institute of Technology, 2004.

\bibitem[DAG X]{DAG-X}
\bysame, \emph{Derived algebraic geometry {X}: Formal moduli problems},
  Preprint from the author's web page (2011).



\bibitem[Man09]{manetti2009differential}
Marco Manetti, \emph{Differential graded {L}ie algebras and formal deformation
  theory}, Algebraic geometry---{S}eattle 2005. {P}art 2, Proc. Sympos. Pure
  Math., vol.~80, Amer. Math. Soc., Providence, RI, 2009, pp.~785--810.
  \MR{2483955}

\bibitem[Mao21]{M21}
Zhouhang Mao, \emph{Revisiting derived crystalline cohomology}, arXiv preprint
  arXiv:2107.02921 (2021).

\bibitem[Nui19a]{nuiten2019homotopical}
Joost~J. Nuiten, \emph{Homotopical algebra for {L}ie algebroids}, Applied
  Categorical Structures \textbf{27} (2019), no.~5, 493--534.

\bibitem[Nui19b]{nuiten2019koszul}
\bysame, \emph{{K}oszul duality for {L}ie algebroids}, Advances in Mathematics
  \textbf{354} (2019), 106750.

\bibitem[Pra67]{pradines1967theorie}
Jean Pradines, \emph{Theorie de {L}ie pour les groupoides differentiable}, CR
  Acad. Sci. Paris \textbf{264} (1967), 245--248.

\bibitem[Pri10]{pridham2010unifying}
Jon~P. Pridham, \emph{Unifying derived deformation theories}, Advances in
  Mathematics \textbf{224} (2010), no.~3, 772--826.

\bibitem[Rak20]{R20}
Arpon Raksit, \emph{Hochschild homology and the derived de {R}ham complex
  revisited}, arXiv preprint arXiv:2007.02576 (2020).

\bibitem[{Sta}19]{stacks-project}
The {Stacks project authors}, \emph{The {S}tacks project},
  \url{https://stacks.math.columbia.edu}, 2019.

\bibitem[TV23]{toen2023infinitesimal}
Betrand To{\"e}n and Gabriele Vezzosi, \emph{Infinitesimal derived foliations},
  arXiv preprint arXiv:2305.13010 (2023).

\end{thebibliography}
\end{document}